\documentclass[12pt, oneside]{book} 
\usepackage{latexsym}
\usepackage{amssymb}
\usepackage[toc,page]{appendix}
\usepackage{enumerate}
\usepackage{chngcntr}
\usepackage{color}
\usepackage{float}
\usepackage{xcolor}
\usepackage{commath}
\usepackage[shortlabels]{enumitem}
\usepackage{esint}
\usepackage{enumitem}
\usepackage{multicol}
\usepackage{comment}
\usepackage{amsfonts}
\usepackage[labelfont=bf]{caption}
\usepackage{amsmath,amsthm}
\usepackage{cases}
\usepackage{algorithm}
\usepackage{algorithmic}
\usepackage{framed}
\usepackage{boites}
\usepackage{subcaption}
\usepackage{url}
\usepackage[export]{adjustbox}[2011/08/13]
\usepackage{graphicx}
\usepackage{here}
\usepackage{tabularx, cellspace}
\usepackage[paper=a4paper,left=30mm,right=20mm,top=25mm,bottom=30mm]{geometry}

\usepackage{array,arydshln}

\newenvironment{@abssec}[1]{%
\if@twocolumn

\section*{#1}%
\else

\vspace{.05in}\footnotesize
\parindent .2in
{\upshape\bfseries #1. }\ignorespaces
\fi}

{\if@twocolumn\else\par\vspace{.1in}\fi}

\newcommand\keywordsname{Key words}
\newcommand\AMSname{AMS subject classifications}
\newcommand\AMname{AMS subject classification}
\newcommand\restr[2]{{
\left.\kern-\nulldelimiterspace 
#1 
\vphantom{|} 
\right|_{#2} 
}}
\newtheorem{theorem}{Theorem}[chapter]
\newtheorem{lemma}[theorem]{Lemma}

\newtheorem{proposition}[theorem]{Proposition}
\newtheorem{remark}[theorem]{Remark}
\newtheorem{definition}[theorem]{Definition}
\newtheorem{problem}{Problem}[section]
\newtheorem{example}[theorem]{Example}

\counterwithout{figure}{chapter}

\newcommand{\NN}{\mathbb{N}}
\newcommand{\ZZ}{\mathbb{Z}}

\newcommand{\RR}{\mathbb{R}}
\newcommand{\CC}{\mathbb{C}}

\def\XXint#1#2#3{{\setbox0=\hbox{$#1{#2#3}{\int}$}
\vcenter{\hbox{$#2#3$}}\kern-.5\wd0}}

\newcommand{\link}{\mathop{\circ\kern-.35em -}}

\newcommand{\ol}{\overline}
\newcommand{\pa}{\partial}

\newcommand{\tr}{\mathop{\mathrm{tr}}}

\newcommand{\dv}{{\mathop{\mathrm{div}}}}

\newcommand{\gr}{{\nabla}}
\newcommand{\ali}{\infty}

\newcommand{\al}{\alpha}

\newcommand{\ve}{\varepsilon}
\newcommand{\e}{\varepsilon}
 
\newcommand{\la}{\lambda}

\newcommand{\om}{\omega}
\newcommand{\Om}{\Omega}
\newcommand{\rn}{{{\mathbb{R}}^N}}

\newcommand{\vol}{{\rm Vol}}

\providecommand{\norm}[1]{\l#1\|}

\newcommand{\hi}{H^1}

\begin{document}

\begin{titlepage}
    \begin{center}
        \vspace*{1cm}
        
        \LARGE
        \textbf
        {The homogenization method for topology optimization of structures: old and new}
        \vspace{0.5cm}
       \large
        \vfill
        \large 
        
        \vspace{1.8cm}
        
        
        \vfill
   \LARGE 
   Gr\'egoire Allaire\\
    Lorenzo Cavallina\\
	Nobuhito Miyake\\
    Tomoyuki Oka\\
	Toshiaki Yachimura
        \vspace{0.8cm}
        
        
    \end{center}
\end{titlepage}

\newpage
\pagestyle{plain}
\noindent
 \begin{minipage}[t]{.55\textwidth}
{Gr\'egoire Allaire\\
CMAP \\Ecole Polytechnique,\\
$91128$ Palaiseau, France\\
\texttt{gregoire.allaire@polytechnique.fr}
}
\end{minipage}
 \begin{minipage}[t]{.45\textwidth}
{Lorenzo Cavallina\\
RCPAM, Graduate School of Information Sciences, Tohoku University,\\
980-8579 Sendai, Japan\\
\texttt{cava@ims.is.tohoku.ac.jp}
\hfill
}
\end{minipage}
\vskip0.2\baselineskip
\vspace*{1cm}
\noindent
\begin{minipage}[t]{.55\textwidth}
{Nobuhito Miyake\\
Mathematical Institute,\\ Tohoku University,\\
980-8578 Sendai, Japan\\
\texttt{nobuhito.miyake.t2@dc.tohoku.ac.jp}
}
\end{minipage}
 \begin{minipage}[t]{.45\textwidth}
{Tomoyuki Oka\\
Mathematical Institute, \\Tohoku University,\\
980-8578 Sendai, Japan\\
\texttt{tomoyuki.oka.q3@dc.tohoku.ac.jp}
\hfill
}
\end{minipage}
\vskip0.2\baselineskip
\vspace*{.4cm}
\noindent
 \begin{minipage}[t]{.55\textwidth}
{Toshiaki Yachimura\\
RCPAM, Graduate School of Information\\ Sciences, Tohoku University,\\
980-8579 Sendai, Japan\\
\texttt{yachimura@ims.is.tohoku.ac.jp}
}
\end{minipage}

\tableofcontents
\vspace*{-0.5cm}

\pagestyle{plain}
\thispagestyle{plain}
\chapter*{Preface and acknowledgments}
\addcontentsline{toc}{chapter}{Preface}
These are the lecture notes of a short course on the homogenization method for topology optimization of structures, given by one of us, Gr\'egoire Allaire, during the ``GSIS International Summer School 2018'' at Tohoku University (Sendai, Japan). 
Based on the slides of this course, the four other authors, Lorenzo Cavallina, Nobuhito Miyake, Tomoyuki Oka, Toshiaki Yachimura, have written the present lecture notes, which have been proofread by Gr\'egoire Allaire. 
Each chapter of these lecture notes corresponds to one class, except the two first ones which were taught together. 

Topology optimization of structures is nowadays a well developed field with many different approaches and a wealth of applications. 
One of the earliest method of topology optimization was the homogenization method, introduced in the early eighties. 
It became extremely popular in its over-simplified version, called SIMP (Solid Isotropic Material with Penalization), which retains only the notion of material density and forgets about true composite materials with optimal (possibly non isotropic) microstructures. 
However, the appearance of mature additive manufacturing technologies which are able to build finely graded microstructures (sometimes called lattice materials) drastically changed the picture and one can see a resurrection of the homogenization method for such applications. 
Indeed, homogenization is the right technique to deal with microstructured materials where anisotropy plays a key role, a feature which is absent from SIMP. 
Homogenization theory allows to replace the microscopic details of the structure (typically a complex networks of bars, trusses and plates) by a simpler effective elasticity tensor describing the mesoscopic properties of the structure. 

The goal of this course is to review the necessary mathematical tools of homogenization theory and apply them to topology optimization of mechanical structures. 
The ultimate application, targeted in this course, is the topology optimization of structures built with lattice materials. 
Practical and numerical exercises are given, based on the finite element free software FreeFem++.\\

Finally, the authors would like to express their gratitude to the organizers of the Summer School: Reika Fukuizumi, Kei Funano, Jun Masamune, Jinhae Park, Ruo Li, Shigeru Sakaguchi, Kenjiro Terada, Takayuki Yamada and Lei Zhang. The meeting was partially supported by a grant from the JSPS A3 Foresight Program, JSPS KAKENHI Grant Numbers $26287020$ and $26400062$, GSIS and RCPAM.

\chapter{Introduction}\label{introduction}
\section{Optimal design of structures}
A problem of optimal design (material, shape and topology optimization) of structures is defined by three ingredients (see \cite{Allaire2,BS,HM,HP,KPTZ,SK}):
\begin{enumerate}[(a)]
\item a {\bf model} (typically a partial differential equation) to evaluate (or analyze) the mechanical behavior of a structure, 

\item an {\bf objective function} which has to be minimized or maximized, or sometimes several objectives (also called cost functions or criteria), 

\item a {\bf set of admissible designs} which precisely defines the optimization variables, including possible constraints.
\end{enumerate}
The kind of optimal design problems which we focus on in these lecture notes can be roughly divided into three categories, from the “easiest” to the “most difficult” one: 
\begin{enumerate}
\item {\bf Parametric or sizing optimization}, for which designs are parametrized by a few variables (for example, thickness or member sizes), implying that the set of admissible designs is considerably simplified (see Figure~\ref{three categories}-1, where the variable parameters, the thickness of the two boxes in this case, are symbolized by arrows),

\item {\bf Shape (geometric) optimization}, for which all designs are obtained from an initial guess by moving its boundary without change of its topology due to the generation of new boundaries (see Figure~\ref{three categories}-2, where an admissible shape is drawn with a broken line), 

\item {\bf Topology optimization} where both the shape and the topology of the admissible designs can vary without any explicit or implicit restrictions (see Figure~\ref{three categories}-3, where the broken lines show removable holes).
\end{enumerate}

\begin{figure}[h]
\centering
\includegraphics[width=1\linewidth,center]{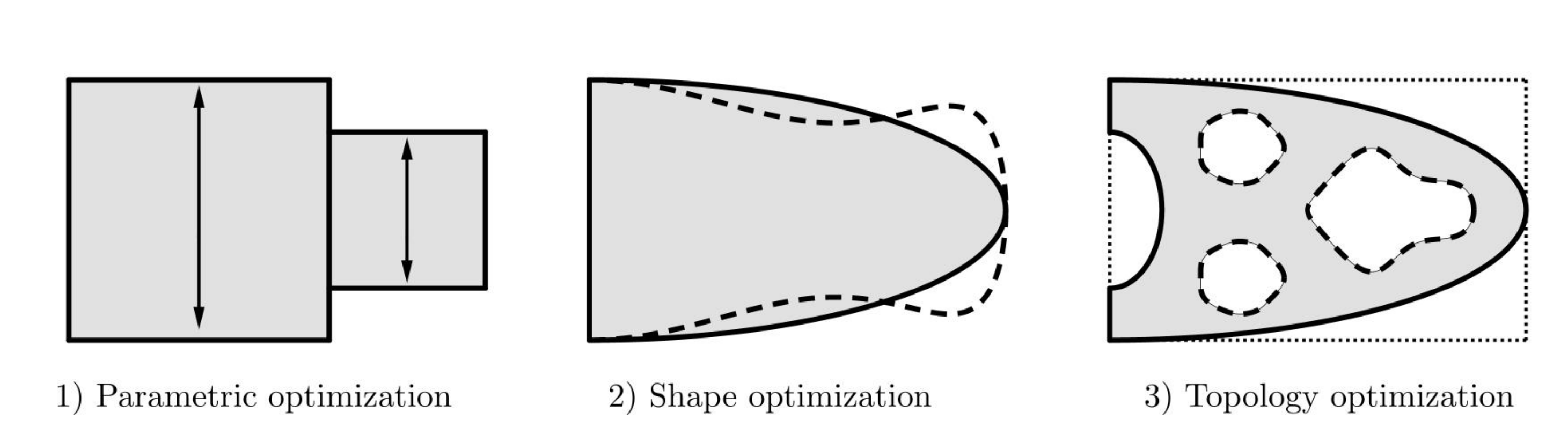}
\caption{Three categories of optimal design problems.} 
\label{three categories}
\end{figure}
The last category in the above is, of course, the most general but also the most difficult. 
We recall that two shapes share the same topology if there exists a continuous deformation from one to the other. In dimension $2$, topology is completely characterized by the number of holes (or, equivalently, of connected components of the boundary). In dimension $3$ it is quite more complicated. Indeed, the topology of a set in dimension $3$ is not only determined by the number of holes, but it also depends on the number and intricacy of “handles” or “loops”. 

First of all, one could ask 
theoretical questions concerning existence, uniqueness, and qualitative properties of the solutions of these shape optimization problems. One could also study the necessary and/or sufficient conditions satisfied by the optimal shapes. Such ``optimality conditions'' are very important both from a theoretical and a numerical point of view. They are often the basis for numerical algorithms of gradient method type. Furthermore one can investigate the numerical computation of approximate optimal shapes. All these questions will be addressed in the following chapters.

\section{Example of sizing or parametric optimization}
First of all, we show some examples of sizing or parametric optimization. Let us consider the thickness optimization of a membrane, where $\Omega$ is a mean surface of a (plane) membrane and $h$ is the thickness in the normal direction to the mean surface $\Omega$ (see Figure~\ref{plaque}). 

\begin{figure}[h]
\centering
\includegraphics[width=0.7\linewidth,center]{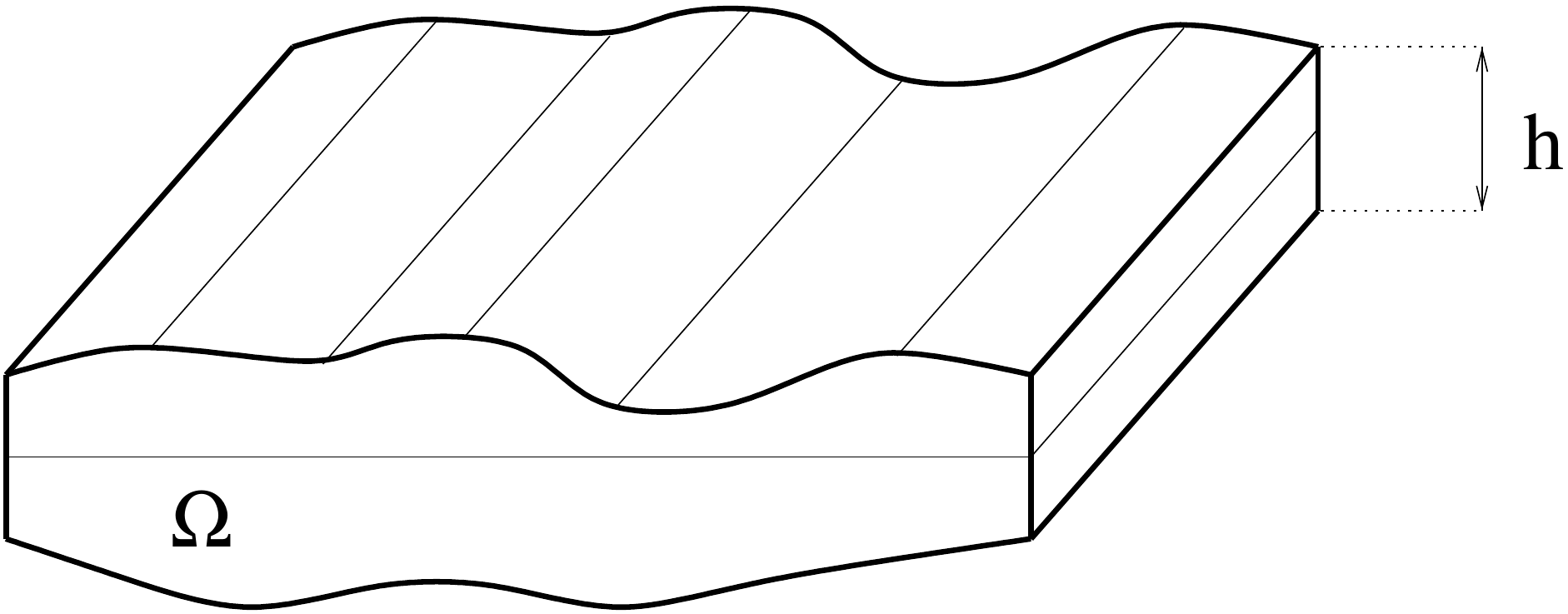}
\caption{Membrane with variable thickness $h$.} 
\label{plaque}
\end{figure}

In what follows, we consider our membrane to be pre-stressed at its boundary and subject to some vertical force $f$. Moreover, for small displacements, small deformations and negligible bending effects in the elasticity, the membrane deformation can be modeled by its vertical displacement $u: \Omega \to \RR$, solution of the following partial differential equation, the so-called membrane model (see also \cite{Allaire2, kats}),
\begin{equation*}
\left\{
\begin{aligned}
-\dv(h\gr u)&=f &&\text{ in } \Omega, \\
u&=0 &&\text{ on } \partial\Omega, 
\end{aligned}
\right.
\end{equation*}
where the thickness $h$ is bounded by some given minimum and maximum values: 
\begin{equation*}
0 < h_{\min} \le h(x) \le h_{\max} < \infty. 
\end{equation*}
The thickness $h$ is the optimization variable. Notice that we are dealing with a sizing or parametric optimal design problem here, because the computational domain $\Omega$ does not change.

Let us define the set of admissible thickness as follows: 
\begin{equation*}
\mathcal{U}_{\rm ad}=\left\lbrace h\in L^\infty(\Omega)\;:\; 0<h_{\rm min}\le h(x)\le h_{\rm max} \text{ a.e. in } \Omega, \int_\Omega h(x) dx= h_0|\Omega|  \right\rbrace, 
\end{equation*}
where $h_0$ is an imposed average thickness. 
\begin{remark}[Possible additional ``feasibility" constraints]
According to the production process of membranes, the thickness $h(x)$ can be discontinuous, or on the contrary continuous. A uniform bound can be imposed on its first derivative $h'(x)$ (molding-type constraint) or on its second order derivative $h''(x)$, linked to the curvature radius (milling-type constraint).
\end{remark}
The optimization criterion is linked to some mechanical property of the membrane, evaluated through its displacement $u$, solution of the PDE, 
\begin{equation*}
J(h) = \int_{\Omega} j(u) \, dx, 
\end{equation*}
where, of course, $u$ depends on $h$. For example, the global rigidity of a structure is often measured by its compliance, or work done by the load $f$: the smaller the work, the larger the rigidity (compliance = $-$ rigidity). In such a case, we set 
\begin{equation*}
j(u) = fu. 
\end{equation*}
Another example amounts to achieve (at least approximately) a target
displacement $u_{0}(x)$, which is modeled by taking
\begin{equation*}
j(u) = \abs{u - u_0}^2. 
\end{equation*}
Those two criteria are the typical examples studied in this course. 
Then, a parametric optimization problem is
\begin{equation*}
\inf_{h\in\mathcal{U}_{\rm ad}} J(h)  \, . 
\end{equation*}

Other examples of objective functions are the following: 
\begin{itemize}
\item Introducing the stress vector $\sigma(x) = h(x) \nabla u(x)$, we can minimize the maximum stress norm 
\begin{equation*}
J(h) = \sup_{x \in \Omega} \abs{\sigma(x)}
\end{equation*}
or more generally, for any $p \ge 1$, the following $p$-norm
\begin{equation*}
J(h) = \left( \int_{\Omega} \abs{\sigma(x)}^p \, dx \right)^{1/p}. 
\end{equation*}
\item For a vibrating structure, introducing the first eigenfrequency $\omega$, defined by 
\begin{equation*}
\left\{
\begin{aligned}
-\dv(h\gr u) &= \omega^2 u &&\text{ in } \Omega, \\
u &= 0 &&\text{ on } \partial\Omega.  
\end{aligned}
\right.
\end{equation*}
We consider $J(h) = - \omega$ to maximize it. 

\item Multiple loads optimization: for $n$ given loads $(f_i)_{1\le i \le n}$ the independent displacements $u_{i}$ are solutions of 
\begin{equation*}
\left\{
\begin{aligned}
-\dv(h\gr u_{i}) &= f_{i} &&\text{ in } \Omega, \\
u_{i} &= 0 &&\text{ on } \partial\Omega.  
\end{aligned}
\right.
\end{equation*}
We then introduce an aggregated criterion 
\begin{equation*}
J(h) = \sum^n_{i=1} c_i \int_{\Omega} j(u_i) \, dx, 
\end{equation*}
with given coefficients $c_i$, or 
\begin{equation*}
J(h) = \max_{1 \le i \le n} \left\{ \int_{\Omega} j(u_i) \, dx \right\}. 
\end{equation*}
\end{itemize}

\section{Example of shape optimization}
In this section, we show two examples of shape optimization. At first let us consider a shape optimization of a membrane's shape. A reference domain for the membrane is denoted by $\Omega$, with a boundary made of three disjoint parts 
\begin{equation*}
\pa \Omega = \Gamma \cup \Gamma_D \cup \Gamma_N, 
\end{equation*}
where $\Gamma$ is the variable part, $\Gamma_D$ is the Dirichlet (clamped) part and $\Gamma_N$ is the
Neumann part (loaded by $g$). 

\begin{figure}[h]
\centering
\includegraphics[width=0.7\linewidth,center]{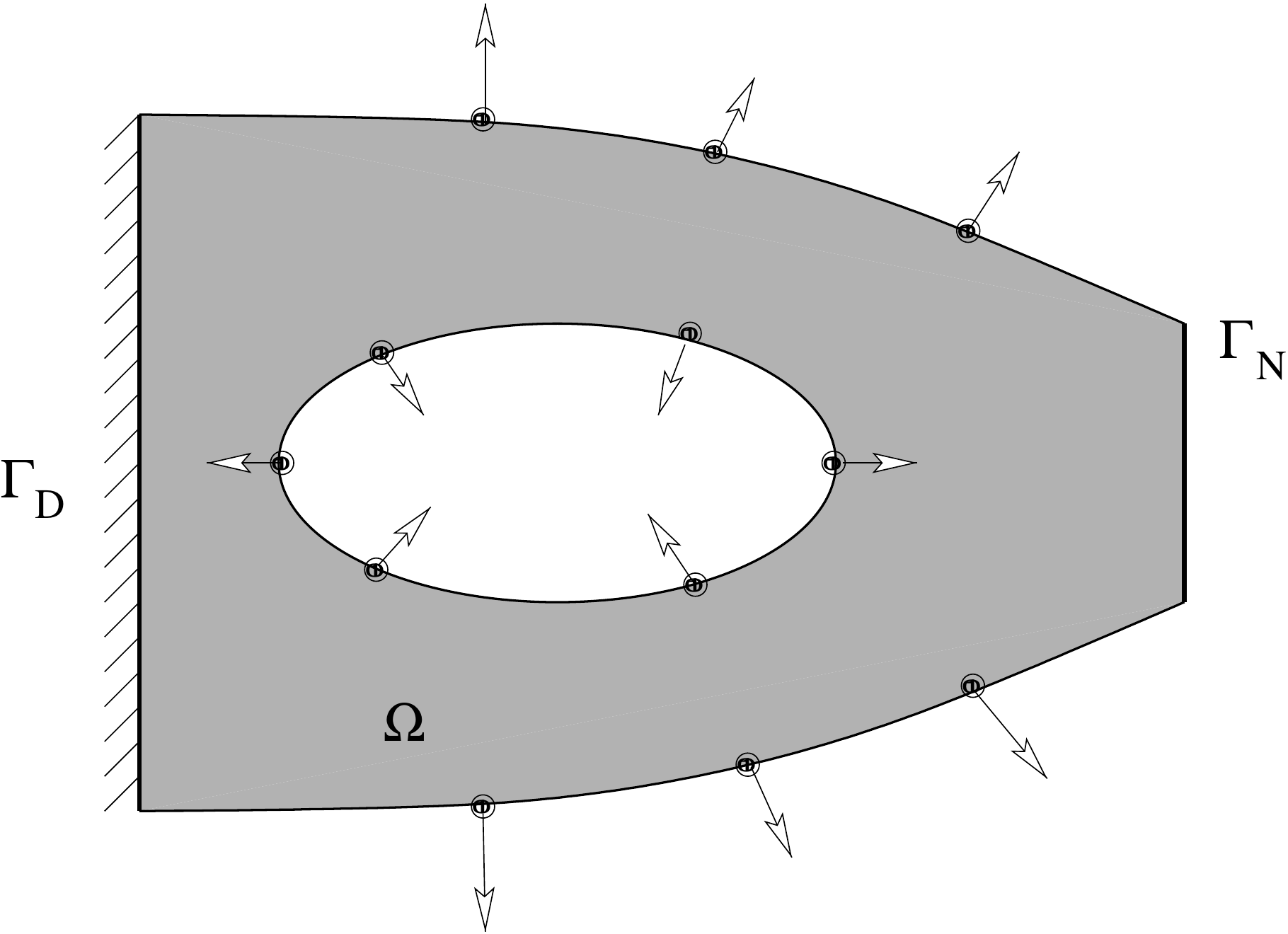}
\caption{Shape optimization of a membrane's shape.} 
\label{hadamard}
\end{figure}

The vertical displacement $u$ is the solution of the following membrane model 
\begin{equation*}
\left\{
\begin{aligned}
-\Delta u &= 0 &&\text{ in } \Omega, \\
u &= 0 &&\text{ on } \Gamma_D, \\
\dfrac{\pa u}{\pa n} &= g &&\text{ on } \Gamma_N, \\
\dfrac{\pa u}{\pa n} &= 0 &&\text{ on } \Gamma. \\
\end{aligned}
\right.
\end{equation*}
From now on the membrane thickness is fixed, equal to $1$. Moreover, we consider the parts $\Gamma_D$ and $\Gamma_N$ to be given. Thus the set of admissible shapes is 
\begin{equation*}
\mathcal{U}_{\rm ad}=\left\lbrace \Omega \subset \RR^N  \;:\; \Gamma_D \cup \Gamma_N \subset \pa \Omega \,\, \text{and} \,\, \abs{\Omega} = V_0  \right\rbrace, 
\end{equation*}
where $V_0>0$ is a given volume. The shape optimization problem reads 
\begin{equation*}
\inf_{\Omega \in \mathcal{U}_{\rm ad}} J(\Omega), 
\end{equation*}
with, as a criterion, the compliance 
\begin{equation*}
J(\Omega) = \int_{\Gamma_N} gu \, ds, 
\end{equation*}
or a least-square functional to achieve a target displacement $u_0(x)$ 
\begin{equation*}
J(\Omega) = \int_{\Omega} \abs{u - u_0}^2 \, dx. 
\end{equation*}
Notice that the true optimization variable is only the free boundary $\Gamma$, and therefore the topology of the shape does not change. 

Another example is a shape optimization in the elasticity setting. The model of linearized elasticity gives the displacement vector field $u: \Omega \to \RR^N$ as the solution of the system of equations
\begin{equation*}
\left\{
\begin{aligned}
-\dv(Ae(u)) &= 0 &&\text{ in } \Omega, \\
u &= 0 &&\text{ on } \Gamma_D, \\
(Ae(u))\cdot n &= g &&\text{ on } \Gamma_N, \\
(Ae(u))\cdot n &= 0 &&\text{ on } \Gamma, \\
\end{aligned}
\right.
\end{equation*}
with $e(u) = (\nabla u + (\nabla u)^t)/2$ and $A \xi = 2\mu \xi + \lambda(\tr \xi)\mathrm{Id}$, where $\mu$ and $\lambda$ are the Lam\'{e} coefficients, and $n$ is the outer unit normal to $\Omega$. 
The boundary $\partial\Omega$ is again divided into three disjoint parts
\begin{equation*}
\pa \Omega = \Gamma \cup \Gamma_D \cup \Gamma_N, 
\end{equation*}
where $\Gamma$ is the free boundary, the true optimization variable. The set of admissible shapes is again
\begin{equation*}
\mathcal{U}_{\rm ad}=\left\lbrace \Omega \subset \RR^N  \;:\; \Gamma_D \cup \Gamma_N \subset \pa \Omega \,\, \text{and} \,\, \abs{\Omega} = V_0  \right\rbrace, 
\end{equation*}
where $V_0$ is a given imposed volume. The objective function chosen is either the compliance 
\begin{equation*}
J(\Omega) = \int_{\Gamma_N} g \cdot u \, ds, 
\end{equation*}
or a least-square criterion for the target displacement $u_0(x)$ 
\begin{equation*}
J(\Omega) = \int_{\Omega} \abs{u - u_0}^2 \, dx. 
\end{equation*}
As before, the shape optimization problem reads 
\begin{equation*}
\inf_{\Omega \in \mathcal{U}_{\rm ad}} J(\Omega). 
\end{equation*}

\section{Topology optimization and the homogenization method}

\begin{figure}[h]
\centering
\includegraphics[width=0.7\linewidth,center]{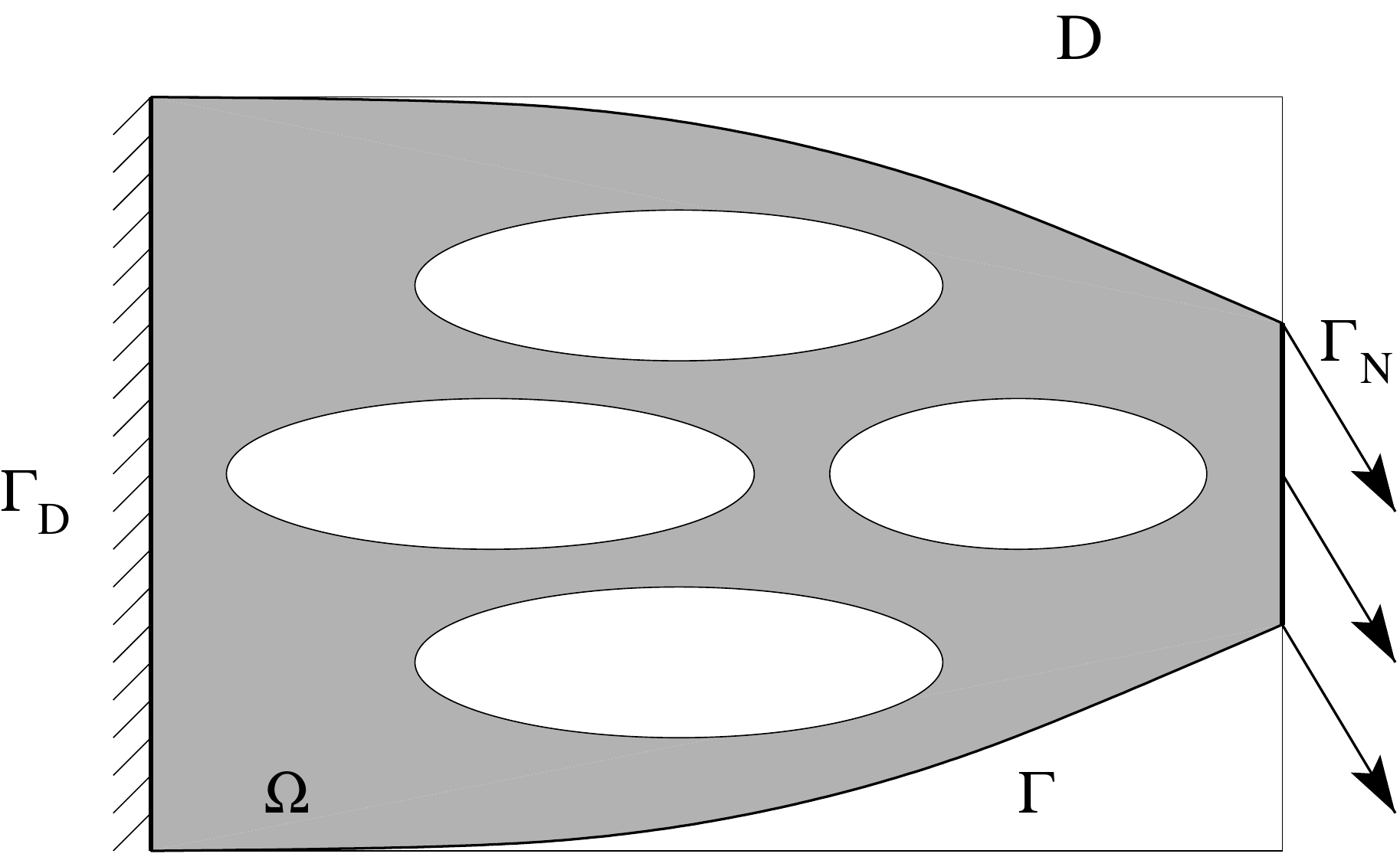}
\caption{Topology optimization of a membrane's shape.} 
\label{omega}
\end{figure}

In topology optimization, not only the connected components of the boundary $\Gamma$ are allowed to move but also new connected components (holes in $2$-d) of $\Gamma$ can appear or disappear. Topology is now optimized too. In order to solve this task, we introduce the homogenization method. The homogenization method is a kind of averaging methods for partial differential equations, and is commonly used to determine the averaged (or effective, or homogenized, or equivalent, or macroscopic) parameters of a heterogeneous medium \cite{Allaire1,BLP,cherkaev,CD,JKO,MT,TA}. 

\begin{figure}[h]
\centering
\includegraphics[width=0.8\linewidth,center]{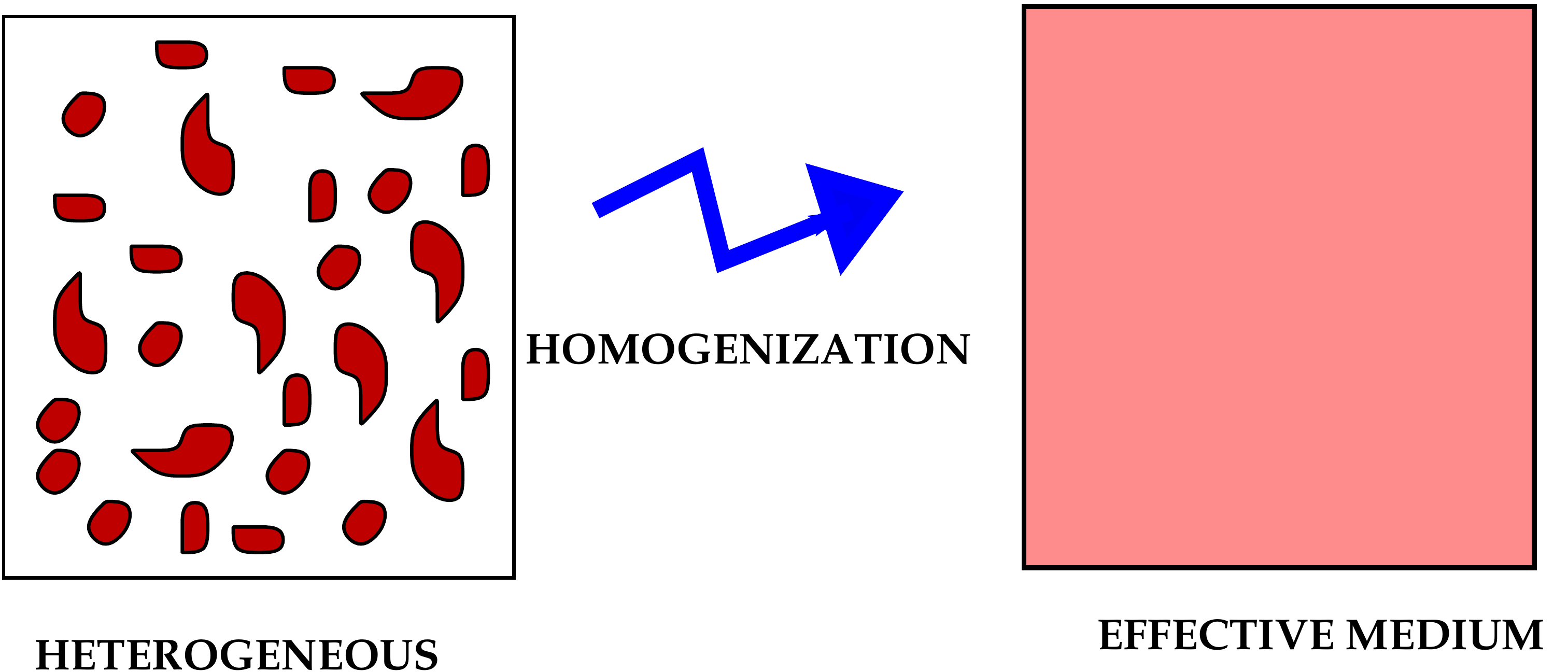}
\caption{Homogenization in a nutshell.} 
\label{homog2}
\end{figure}

How does homogenization apply to optimal design ?
The homogenization method is based on the concept of “relaxation”: it makes ill-posed problems well-posed by enlarging the space of admissible ``shapes''.
It is crucial to introduce “generalized” shapes, that are ``not too generalized''. In the homogenization method, we think of generalized shapes as “limits” of minimizing sequences of classical shapes.
We can then say that homogenization allows, as admissible shapes, composite materials obtained by micro-perforation of the original material (fine mixtures of material and void). 

\section{Lattice materials in additive manufacturing}
Additive manufacturing, also known as 3D printing, is a process that creates physical structures built layer by layer from a digital design by using metallic powder melted by a laser or an electron beam \cite{gibson}. One of the main advantages of additive manufacturing is that, a priori, there are no limitations on the structures that can be built (unfortunately, in practice there are some limitations of manufacturability, like overhangs or the possibility of thermal residual stresses). 
Moreover, one can even build microstructures or lattice materials. 

\begin{figure}[htbp]
    \begin{minipage}[t]{.45\textwidth}
        \centering
        \includegraphics[width=\textwidth]{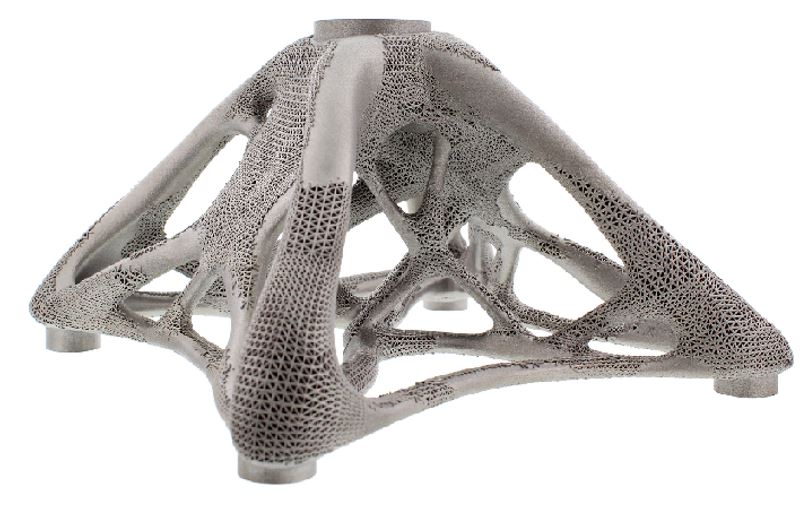}
    \end{minipage}
    \hfill
    \begin{minipage}[t]{.45\textwidth}
        \centering
        \includegraphics[width=\textwidth]{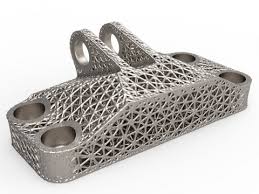}
    \end{minipage}
    \hfill
    \caption{Some examples of lattice structures.
    Left: an architectural spider bracket
    (\protect\url{https://altairenlighten.com/wp-content/uploads/2017/03/architectural-spider-bracket.jpg}).
    Right: crystallon, lattice structures in Rhino and Grasshopper (\protect\url{https://noizear.com/crystallon-lattice-structures-in-rhino-and-grasshopper})
}
\label{lattice struc}
\end{figure}

However, it is impossible to describe all the fine details of a lattice structure in a finite element model for optimization purposes. Therefore, homogenization theory is the right tool for dealing with lattice materials and related optimal design problems. In chapter \ref{chap of additive manufacturing}, we will tackle the problem of optimizing lattice structures by using the homogenization method. 

\section{Goals of these lecture notes} 
The main goal of these lecture notes is to introduce the homogenization method for topology optimization of structures. The rest of these lecture notes are organized as follows. In chapter \ref{chap1}, we show some tools in optimization and describe numerical algorithms for computing optimal designs. In chapter \ref{Parametric optimal design}, we consider parametric optimization problem and compute gradients of objective functions (by an optimal control approach) for further use in gradient-type algorithms. A representative example of parametric optimization is that of a membrane's thickness. In chapter \ref{Homogenization theory}, we provide a brief survey on homogenization theory. In chapter \ref{Topo opti homo method}, we apply the homogenization method to topology optimization. In chapter \ref{chap of additive manufacturing}, we present a resurrection of the homogenization method for the design of lattice materials in additive manufacturing. 

Through these lecture notes, numerical exercizes are proposed with the FreeFem++ code (\url{http://www.freefem.org}). FreeFem++ is a free software for solving partial differential equations by the finite element method \cite{Hecht}. Moreover, you can find some scripts of FreeFem++ for shape optimization in the web site (\url{http://www.cmap.polytechnique.fr/~allaire/freefem_en.html}) and in the corresponding educational paper  \cite{Allaire3}. 

\section{Exercises}
\begin{problem}\label{ex 1}
Solve (with FreeFem++) the elasticity equations for the following test cases: cantilever, bridge, MBB beam and L-beam (see Figure~\ref{four shapes}).
\end{problem}
\begin{figure*}[h]
        \centering
        \begin{subfigure}[b]{0.475\textwidth}
            \centering
            \includegraphics[width=\textwidth]{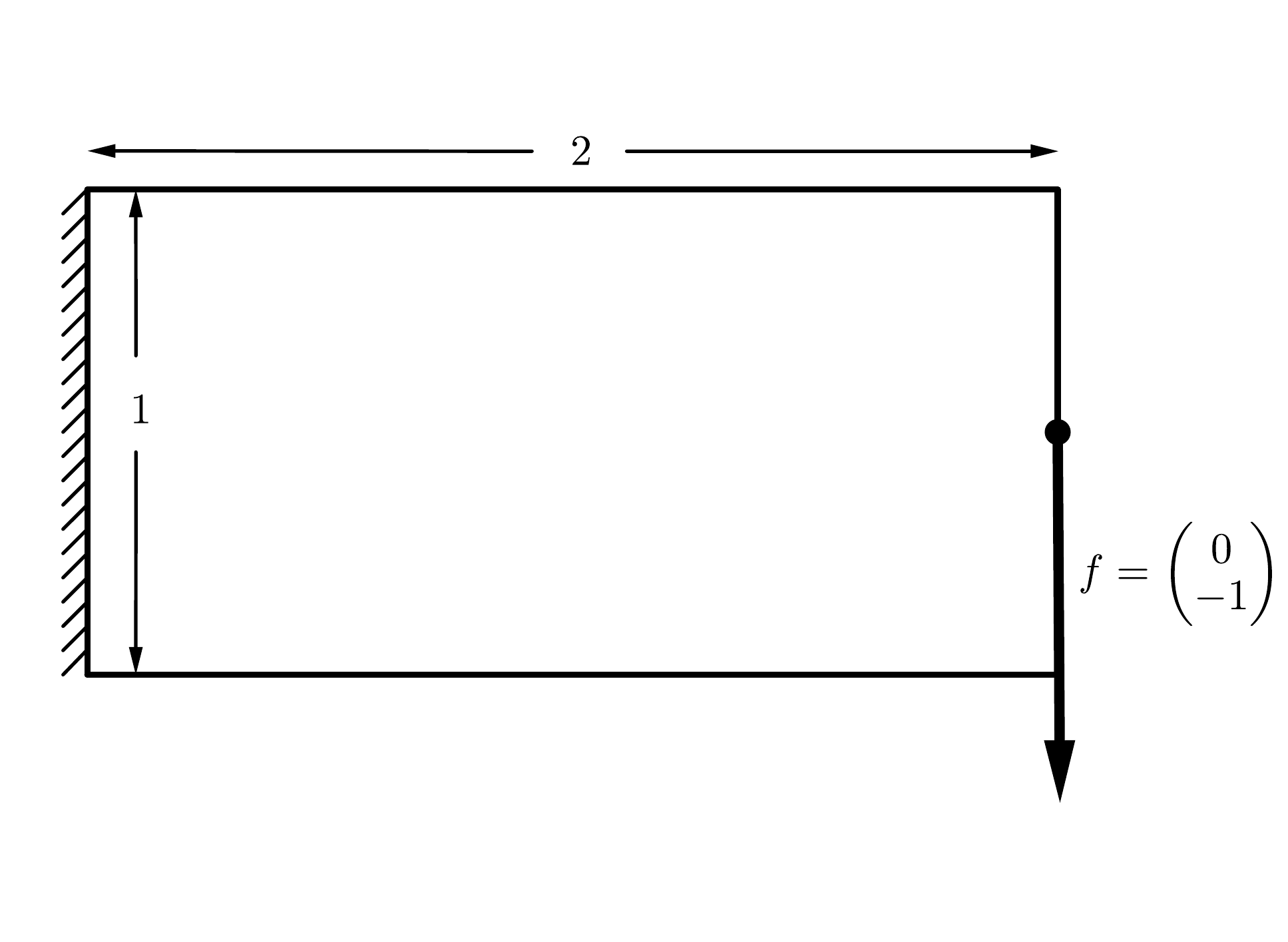}
            \caption[]%
            {{\small Cantilever}}    
        \end{subfigure}
        \hfill
        \begin{subfigure}[b]{0.475\textwidth}  
            \centering 
            \includegraphics[width=\textwidth]{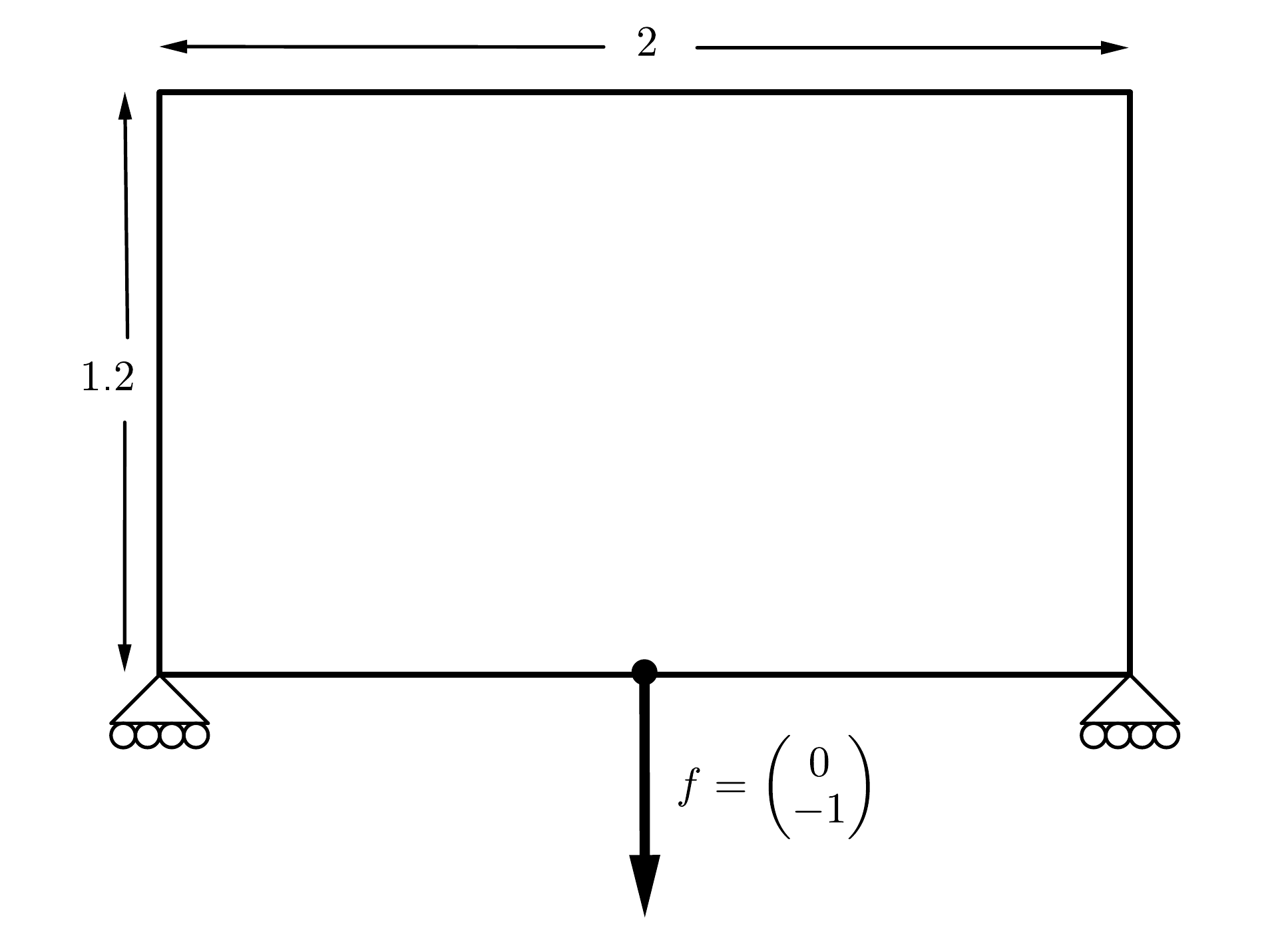}
            \caption[]%
            {{\small Bridge}}    
        \end{subfigure}
        \vskip\baselineskip
        \begin{subfigure}[b]{0.6\textwidth}   
            \centering 
            \includegraphics[width=\textwidth]{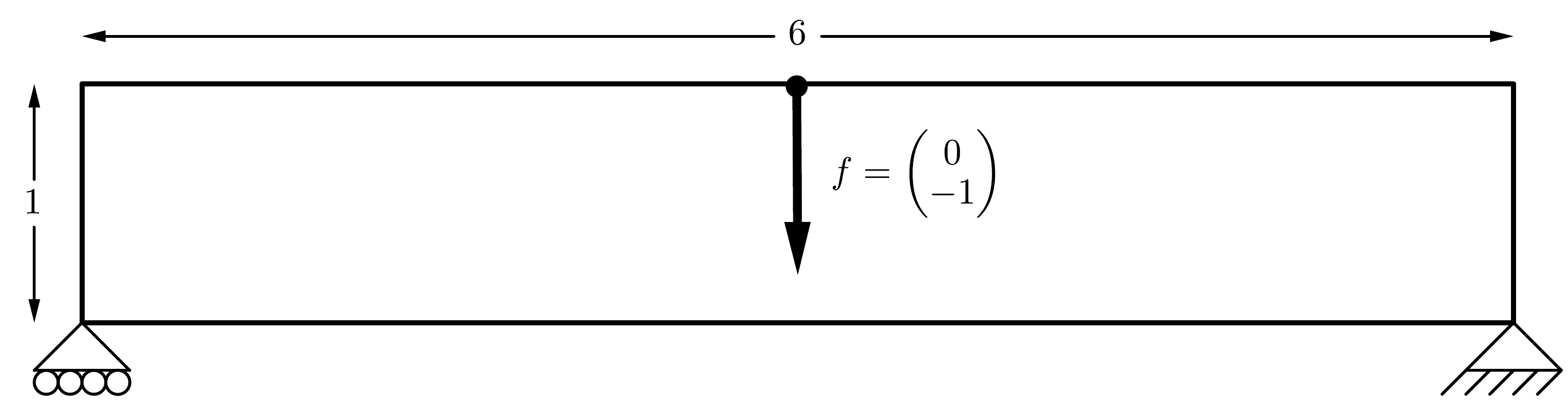}
            \caption[]%
            {{\small MBB beam}}    
        \end{subfigure}
        \quad
        \begin{subfigure}[b]{0.35\textwidth}   
            \centering 
            \includegraphics[width=\textwidth]{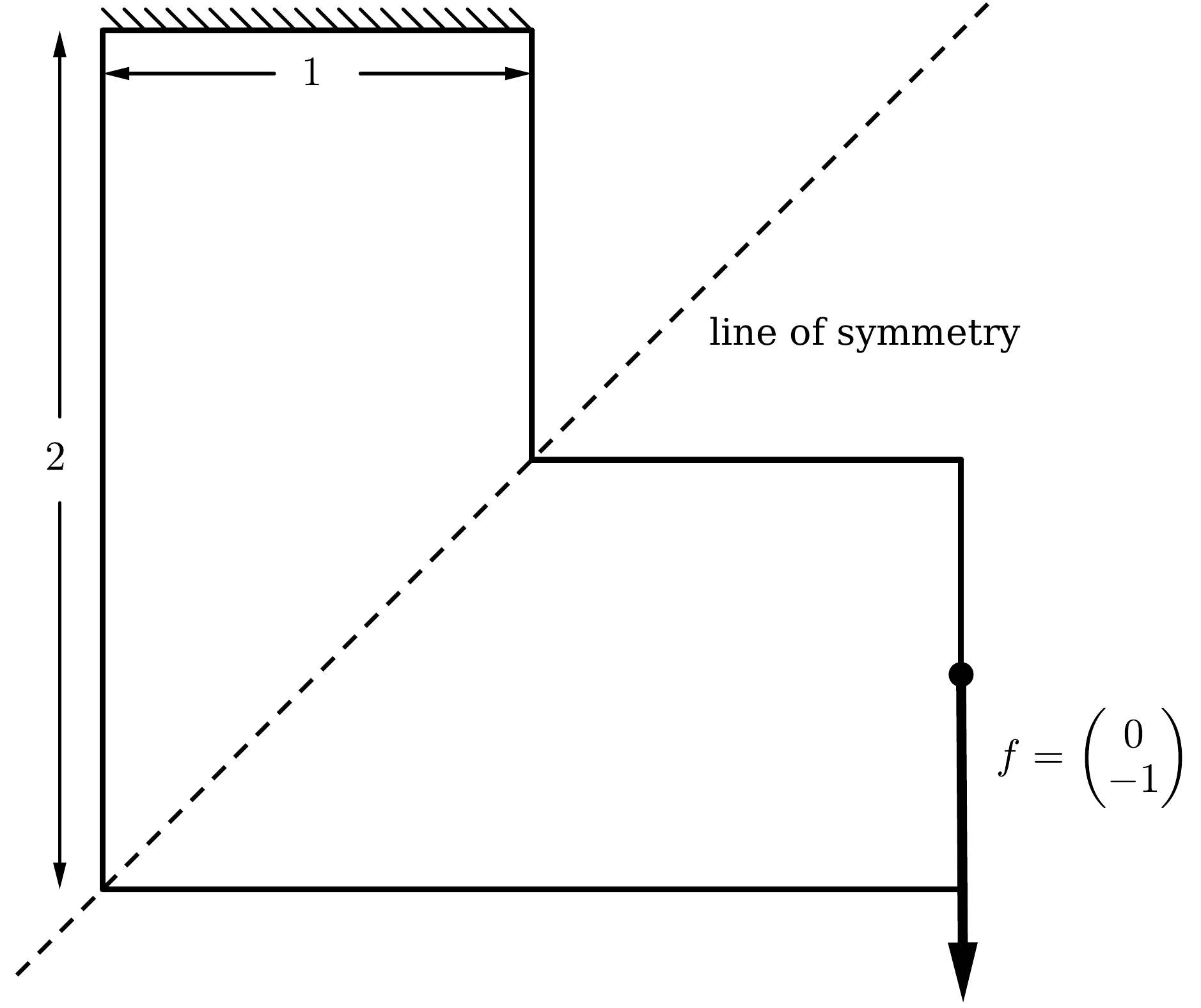}
            \caption[]%
         {{\small L-beam}}    
        \end{subfigure}
        \caption[]
        {{\small The various boundary conditions of Problem~\ref{ex 1}. Here, the loads are to be intended as acting on a very small region of the boundary, around the points represented by the full black dots.}}
      
       \label{four shapes}
    \end{figure*}

\chapter{Some tools in optimization}\label{chap1}

We review some classical result in optimization theory. 
More details can be found in textbooks like \cite{Allaire4,bonnans,ET,nocedal}.

\section{Generalities}
Let $V$ be a Banach space and $K \subset V$ be a non-empty subset. Let $J:V \to \mathbb{R}$. We consider the following minimization problem
\begin{equation*}
\inf_{v \in K 
} J(v).
\end{equation*}
Let us specify some basic definitions.
\begin{definition}
An element $u$ is called a local minimizer of $J$ on $K$ if  
\begin{equation*}
u \in K \,\, \text{and} \,\, \exists \delta > 0, \,\, \forall v \in K, \,\, \norm{v-u} < \delta \Longrightarrow J(v) \geq J(u).
\end{equation*}
Moreover, an element $u$ is called a global minimizer of $J$ on $K$ if 
\begin{equation*}
u \in K \,\, \text{and} \,\, J(v) \geq J(u) \,\, \forall v \in K. 
\end{equation*}
\end{definition}

\begin{definition}
A minimizing sequence of a function $J$ on the set $K$ is a sequence $(u^{n})_{n \in \mathbb{N}}\subset K$ such that 
\begin{equation*}
\lim_{n \to +\infty} J(u^{n}) = \inf_{v \in K} J(v).
\end{equation*}
\end{definition}
By definition of the infimum value of $J$ on $K$ there always exists at least one minimizing sequence for $J$ on $K$. 

Let us consider the existence of minima for optimization problems in finite dimension. The following result guarantees the existence of a minimum. 
\begin{theorem}\label{ex minimizer for cont J over closed K}
Let $K$ be a non-empty closed subset of $\mathbb{R}^N$ and $J$ a continuous function from $K$ to $\mathbb{R}$ satisfying the so-called ``infinite at infinity'' property, i.e., 
\begin{equation*}
\forall (u^{n})_{n\in\NN} \,\, \text{sequence in} \,\, K, \,\, \lim_{n \to +\infty} \norm{u^{n}} = +\infty \Longrightarrow \lim_{n \to +\infty} J(u^{n}) = +\infty. 
\end{equation*}
Then there exists at least one minimizer of $J$ on $K$. 
Furthermore, from each minimizing sequence of $J$ over $K$ one can extract a subsequence which converges to a minimum of $J$ on $K$. 
\end{theorem}
\begin{proof}
Let $(u^n)_{n\in\NN}$ be a minimizing sequence for $J$ over $K$. In particular, since $J$ is infinite at infinity and the sequence $\left(J(u^n)\right)_{n\in\NN}$ is bounded, we conclude that $(u^n)_{n\in\NN}$ must be bounded as well. Therefore, since closed bounded sets are compact in finite dimension, there exists a subsequence $\left(u^{n_k}\right)_{k\in\NN}$ that converges to a point $u\in \rn$. Now, $u\in K$ because $K$ is closed, and $J(u^{n_k})$ converges to $J(u)$ by continuity. We conclude that 
$\displaystyle J(u)=\lim_{k\to\infty}J(u^{n_k})=\inf_{K} J$.

\end{proof}

\begin{remark}\label{no minimum in inf dim space}
In an infinite dimensional vector space, a continuous function on a closed bounded set does not necessarily attain its minimum. For example, let $H^{1}(0,1)$ be the usual Sobolev space with the norm $\displaystyle\norm{v} = \left(\int^{1}_{0} (v'(x)^{2} + v(x)^{2}) dx \right)^\frac{1}{2}$. Let 
\begin{equation*}
J(v) = \int^{1}_{0} \left( (|v'(x)|-1)^{2} + v(x)^{2} \right) dx. 
\end{equation*}
One can check that $J$ is continuous and ``infinite at infinity''. 
Nevertheless the minimization problem 
\begin{equation*}
\inf_{v \in H^{1}(0,1)} J(v)
\end{equation*}
does not admit a minimizer. Indeed, there exists no $v \in H^{1}(0,1)$ such that $J(v) = 0$ but, still, 
\begin{equation*}
\inf_{v \in H^{1}(0,1)} J(v) = 0. 
\end{equation*}
To obtain it, we construct a minimizing sequence $(u^{n})_{n\in\NN}$ defined for, $n \geq 1$, by 
\begin{equation*}
u^{n}(x) = 
\left\{
\begin{aligned}
  &x - \frac{k}{n} && \text{ if } \ \frac{k}{n} \leq x \leq \frac{2k+1}{2n}, \\
  &\frac{k+1}{n} - x && \text{ if } \ \frac{2k+1}{2n} \leq x \leq\frac{k+1}{n}
\end{aligned}
\right.
\ \text{ for } \  0\le k\le n-1, 
\end{equation*}
as Figure \ref{zigzag}. 
\begin{figure}[h]
\centering
\includegraphics[width=0.9\linewidth,center]{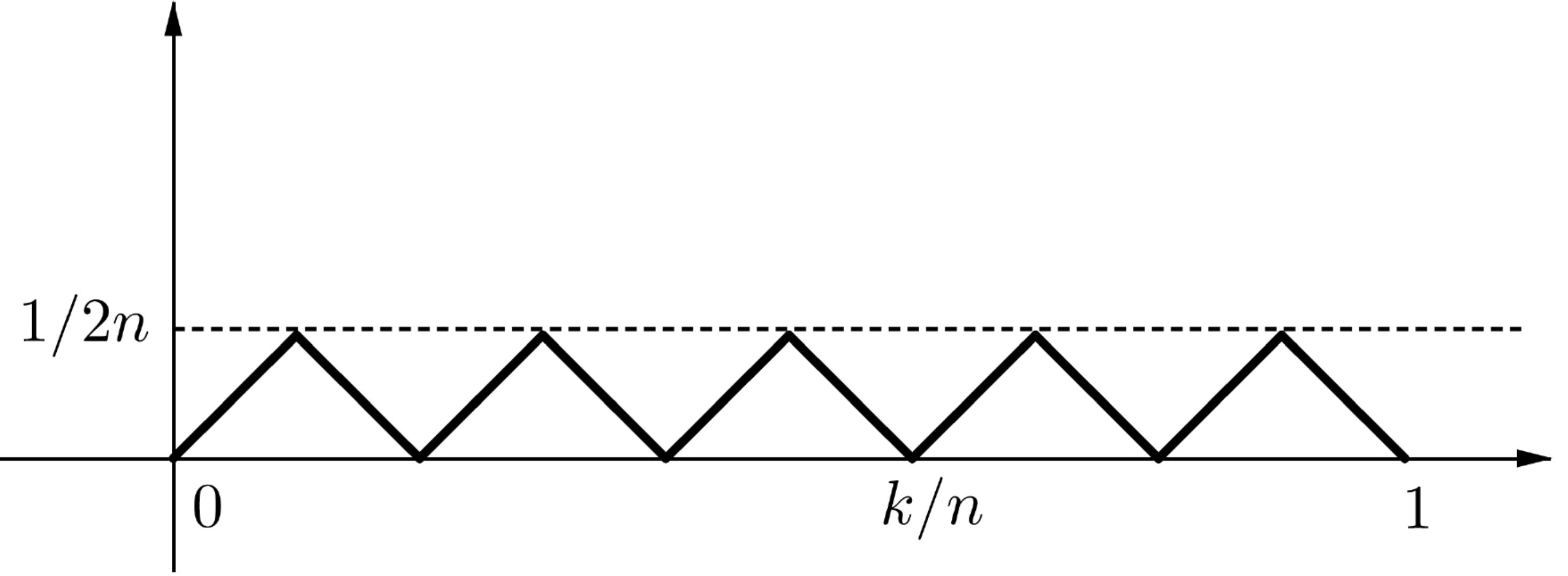}
\caption{The function $u^n$ for $n=5$.
} 
\label{zigzag}
\end{figure}

We can easily check that $u^{n} \in H^{1}(0,1)$ and $(u^{n})' = \pm 1$. Consequently, 
\begin{equation*}
J(u^{n}) = \int^{1}_{0} u^{n}(x)^{2} dx = \frac{1}{12 n^2} \to 0. 
\end{equation*}
We clearly see in this example that the minimizing sequence $(u^{n})_{n\in\NN}$ is ``oscillating'' more and more and it is not compact in $H^{1}(0,1)$ despite being bounded in the same space. 
\end{remark}

\section{Convex analysis}
As we have seen in Remark \ref{no minimum in inf dim space}, continuous functions do not necessarily attain their minimum on a bounded closed set. In order to extend the result of Theorem \ref{ex minimizer for cont J over closed K} to the case of an infinite dimensional Hilbert space, we shall work in a convex framework.

\begin{definition}
A set $K \subset V$ is said to be convex if, for any $x,y \in K$ and for any $\theta \in [0,1]$, the linear combination $\theta x + (1-\theta)y$ belongs to $K$. 
\end{definition}

\begin{definition}
A function $J$, defined from a non-empty convex set $K \subset V$ into $\mathbb{R}$ is convex on $K$ if 
\begin{equation}\label{eq strictly convex}
J(\theta u + (1-\theta) v) \leq \theta J(u) + (1-\theta) J(v) \quad \forall u,v \in K, \forall \theta \in [0,1]. 
\end{equation}
Furthermore, $J$ is said to be strictly convex if the inequality above is strict whenever $u \neq v$ and $\theta \in (0,1)$. 
\end{definition}

\begin{theorem}\label{existence of min in infinite dim convex space}
Let $K$ be a non-empty closed convex set in a reflexive Banach space $V$ (i.e. the dual of $V'$ is $V$ itself), and $J$ be a convex continuous function on $K$, which is ``infinite at infinity", i.e., 
\begin{equation*}
\forall (u^{n})_{n \in\NN} \,\, \text{sequence in} \,\, K, \lim_{n \to +\infty} \norm{u^{n}} = +\infty \Longrightarrow \lim_{n \to +\infty} J(u^{n}) = +\infty. 
\end{equation*}
Then, there exists a minimizer of $J$ in $K$. 
\end{theorem}
\begin{remark}
Theorem \ref{existence of min in infinite dim convex space} remains true if $V$ is just the dual of some separable Banach space. In particular it holds true when $V=L^p(\Omega)$ with $1<p\le \infty$.
\end{remark}
The proof will follow along the same lines as that of Theorem \ref{ex minimizer for cont J over closed K}. However, the infinite dimensional case is much more delicate, since it relies on \emph{weak convergence} and its relations with convexity. We refer to \cite[Theorem 9.2.7 and Remark 9.2.9]{Allaire4} for a complete proof.

\begin{proposition}\label{prop uniq minimizer for str convex J}
Under the hypotheses of Theorem \ref{existence of min in infinite dim convex space}, suppose that $J$ is strictly convex. Then $J$ has at most one minimizer. 
\end{proposition}
\begin{proof}
Suppose by contradiction that $u_1\neq u_2$ are two distinct minimizers of $J$ over the closed strictly convex set $K$. If we take $\theta=1/2$ in \eqref{eq strictly convex}, then we get 
\begin{equation*}
J\left(\frac{u_1+u_2}{2}\right)<\frac{1}{2}J(u_1)+\frac{1}{2}J(u_2)=\min_{v\in K} J(v),
\end{equation*}
which is contradicts the definition of minimum.
\end{proof}

\begin{proposition} \label{prop:global minimizer}
If $J$ is convex on the convex set $K$, then any local minimizer of $J$ on $K$ is a global minimizer. 
\end{proposition}

\begin{proof}
Let $u $ be a local minimizer of $J$ on $K$.  
Thus there exists $\delta > 0$ such that $J(v) \ge J(u)$ for any $v\in K\cap B(u,\delta)$.
 Let $w\in K\setminus B(u,\delta)$ . Our aim is to show that $J(w)\ge J(u)$. 
 Let $\theta\in (0,1]$ be such that $u+\theta (w-u)\in B(u,\delta)$.
 For example we can take $\theta = \frac{\delta}{\|w-u\|_K}$ . 
 Since $u+\theta (w-u)\in B(u,\delta)$, 
 it follows that $J(u)\le J(u+\theta(w-u))$, and  by Jensen's inequality
\[J(u)\le J(u+\theta(w-u))\le(1-\theta)J(u)+\theta J(w).\]
Thus, $J(u )\le J(w) $ follows.
\end{proof}

Convexity is not the only tool to prove existence of minimizers. Another method is, for example, compactness. 

\section{Optimality conditions}

In this section, we discuss optimality conditions for objective functions.

\begin{definition}
Let $V$ be a Banach space. A function $J$, defined from a neighborhood of $u \in V$ into $\mathbb{R}$, is said to be differentiable in the sense of Fr\'echet at $u$ if there exists $L \in V'$ such that 
\begin{equation*}
J(u + w) = J(u) + L(w) + o(w) \,\, \text{with} \,\, \lim_{w \to 0} \frac{\abs{o(w)}}{\norm{w}}=0.
\end{equation*}
We call $L$ the differential (or derivative, or gradient) of $J$ at $u$ and we denote it by 
\begin{equation}\label{L is dare?}
L = J'(u), \quad \text{ or }\quad L(w) = \langle J'(u), w\rangle_{V',V}. 
\end{equation}
\end{definition}

\begin{remark}
If $V$ is a Hilbert space, its dual $V'$ can be identified with $V$ itself thanks to the Riesz representation theorem. Thus, there exists a unique $p \in V$ such that $\langle p, w \rangle = L(w)$. We also write $p = J'(u)$. We use this identification $V = V'$ if $V = \mathbb{R}^n$ or $V = L^{2}(\Omega)$.
In practice, it is often easier to compute the directional derivative $j'(0) = \langle J'(u), w \rangle_{V',V}$ with $j(t) = J(u+tw)$. 
\end{remark}

Consider the variational formulation 
\begin{equation}\label{varaitional formulation}
\text{find} \,\, u \in V \,\, \text{such that} \,\, a(u,w) = L(w) \quad \forall w \in V, 
\end{equation}
where $a$ is a symmetric coercive continuous bilinear form and $L$ is a continuous linear form. By the Lax--Milgram theorem we know that the variational formulation \eqref{varaitional formulation} admits a unique solution. Let us now define the energy 
\begin{equation*}
J(v) = \frac{1}{2}a(v,v) - L(v). 
\end{equation*}
The following lemma tells us the relationship between the energy $J$ and the variational formulation \eqref{varaitional formulation}. 
\begin{lemma}
Let $u \in V$ be the unique solution of the variational formulation \eqref{varaitional formulation}. Then $u$ is the unique minimizer of $J$, that is, 
\begin{equation*}
J(u) = \min_{v \in V} J(v). 
\end{equation*}
Conversely, if $u \in V$ is a point giving an energy minimum of $J(v)$, then $u$ is the unique solution of the variational formulation \eqref{varaitional formulation}. 
\end{lemma}

\begin{proof}
If $u$ is the solution of the variational formulation \eqref{varaitional formulation}, then thanks to the symmetry of $a$ we have 
\begin{equation*}
J(u + v) = J(u) + \frac{1}{2}a(v,v) + a(u,v) - L(v) = J(u) + \frac{1}{2} a(v,v) \geq J(u). 
\end{equation*}
As $u + v$ is arbitrary in $V$, $u$ minimizes the energy $J$ in $V$. 

Conversely, let $u \in V$ be such that 
\begin{equation*}
J(u) = \min_{v \in V} J(v). 
\end{equation*}
For $v \in V$ we define $j(t) = J(u+tv)$. Then \begin{equation*}
j(t) = \frac{t^2}{2} a(v,v) + t \left( a(u,v) - L(v) \right) + J(u). 
\end{equation*}
We differentiate $t \mapsto j(t)$, 
\begin{equation*}
j'(t) = ta(v,v) + \left( a(u,v) - L(v) \right). 
\end{equation*}
By definition, $j'(0) = \langle J'(u), v \rangle_{V',V}$, thus 
\begin{equation*}
\langle J'(u), v \rangle_{V',V} = a(u,v) - L(v). 
\end{equation*} 
Since $t = 0$ is a minimum point of $j$, we have $a(u,v) = L(v)$ for all $v \in V$.  
\end{proof}

\begin{remark}
When computing the Fr\'echet differential of a given functional $J$ at $u$ (see the definition of $L$ and $w\mapsto L(w)$ in \eqref{L is dare?}), there is not always an obvious way to deduce a formula for $J'(u)$, nevertheless most of the time it is enough to know the mapping $w\mapsto \langle J'(u),w\rangle$.
\end{remark}
\begin{example}\label{ex some derivatives}
\begin{enumerate}
\item For fixed $f\in L^2(\Om)$, define
\begin{equation*}
J(v)=\int_\Om \left( \frac{1}{2} v^2-fv\right)\,dx, \quad v\in L^2(\Omega).
\end{equation*}
 We have
\begin{equation*}
\langle J'(u),w\rangle= \int_\Omega (uw-fw) \,dx.
\end{equation*}
Thus 
\begin{equation*}
J'(u)=u-f \in L^2(\Omega).
\end{equation*}
Notice that here we identified $L^2(\Omega)$ with its dual.
\item For fixed $f\in L^2(\Omega)$ define 
\begin{equation*}
J(v)= \int_\Om \left( \frac{1}{2} |\gr v|^2-fv\right)\,dx,\quad  v\in H_0^1(\Omega).
\end{equation*}
We have
\begin{equation*}
\langle J'(u),w\rangle = \int_\Omega \left( \gr u \cdot \gr w- fw\right) dx.
\end{equation*}
Therefore, by the usual definition of the duality pairing between $H_0^1(\Omega)$ and $H^{-1}(\Omega)$ (that comes from a formal integration by parts) we get
\begin{equation*}
J'(u)=-\Delta u - f \in H^{-1}(\Omega)=\left( H_0^1(\Omega)\right)'.
\end{equation*}
Notice that here the space $H_0^1(\Omega)$ is {\bf not} identified with its dual.
\end{enumerate}
\end{example}

\begin{remark}\label{change of scalar prod}
If instead of the ``usual'' scalar product in $L^2$ we rather use the $H^1$ scalar product in the second part of Example \ref{ex some derivatives}, then we have to identify $J'(u)$ with a different function (in other words, the definition of $J'(u)$ depends on the scalar product used).
From the directional derivative
\begin{equation*}
\langle J'(u),w\rangle= \int_\Omega (\gr u\cdot \gr w - fw)\, dx, 
\end{equation*}
using the $H^1$ scalar product $\displaystyle\langle\phi,w\rangle= \int_\Omega (\gr \phi\cdot \gr w + \phi w) \,dx$, we deduce that 
\begin{equation*}
-\Delta J'(u) + J'(u)= -\Delta u - f, \quad J'(u)\in H_0^1(\Om)
\end{equation*}
in the distributional sense.
Here we identify $H_0^1(\Omega)$ with its dual.
\end{remark}


\begin{theorem}[Euler inequality]\label{thm Euler ineq}
Let $K$ be a convex Banach space. Take $u\in K$ and let $J: K\to \RR$ be differentiable at $u$. If $u$ is a local minimizer of $J$ in $K$, then 
\begin{equation}\label{Euler inequality}
\langle J'(u),v-u \rangle\ge0 \quad \forall v\in K.
\end{equation}
On the other hand, if $u\in K$ satisfies this inequality and $J$ is convex, then $u$ is a global minimizer of $J$ in $K$.
\end{theorem}
\begin{proof}
For $v\in K$ and $\delta\in(0,1]$, we have $u+\delta(v-u)\in K$. Thus, if $\delta$ is sufficiently small, since $u$ is a local minimizer of $J$ in $K$, we have
\begin{equation*}
\frac{J(u+\delta(v-u))-J(u)}{\delta}\ge 0.
\end{equation*}
We obtain inequality \eqref{Euler inequality} by letting $\delta\to 0$ in the above.  

We will now prove the second claim of the theorem. Since, by hypothesis $J$ is convex on $K$, then, the graph of $J$ always lies above its tangent plane at any point $w\in K$. In other words, the following inequality holds true for all $v\in K$:
\begin{equation*}
J(v)\ge J(w)+ \left\langle J'(w), v-w\right\rangle.
\end{equation*}
The conclusion follows by taking $w=u$.
\end{proof}
\begin{remark}
If $u$ belongs to the interior of $K$, then we deduce the Euler equation $J'(u)=0$.
\end{remark}
\begin{remark}
The Euler inequality is usually just a necessary condition (for instance, it is verified also if $u$ is a local maximizer). It becomes a necessary and sufficient condition under the further assumption that the functional $J$ is also convex. 
\end{remark}

\subsection{Minimization with equality constraints}

We consider the following problem

\begin{equation}\label{constrained min problem with equality}
\inf_{v\in V, F(v)=0} J(v),
\end{equation}
where $F= \left( F_1, \dots, F_M\right)$ is a differentiable function from $V$ into $\RR ^M$.

Notice that the set $K=\left\{ v\in V:\; F(v)=0\right\}$ is not necessarily convex. We will therefore need a generalized version of the Euler inequality as stated in Theorem \ref{thm Euler ineq}. To this end we introduce the set of admissible directions for our constrained optimization problem. 
\begin{definition}\label{def cone of directions}
At every point $v\in K$, the set
\begin{equation*}
K(v)=\left\{w\in V \;:\; 
\begin{aligned}
\exists (v^n)_{n\in\NN} \subset K,\; \exists (\varepsilon^n)_{n\in \NN} \subset (0,\infty),\\ \lim_{n\to\infty}v^n=v,\; \lim_{n\to\infty}\varepsilon^n=0,\; \lim_{n\to\infty}(v^n-v)/\varepsilon^n=w 
\end{aligned}
\right\}
\end{equation*}
is called the cone of admissible directions at the point $v$.
\end{definition}
In other words, $K(v)$ is the set of all vectors that are tangent at $v$ to a curve in $K$ that passes through $v$ (hence, if $K$ is a regular manifold, $K(v)$ coincides with the tangent space to $K$ at $v$). Moreover, notice that, as the name suggests, the set $K(v)$ is a \emph{cone} in the sense of convex analysis: namely, for all $\lambda\ge0$ and $w\in  K(v)$, then also $\lambda w \in K(v)$.

\begin{proposition}[Euler inequality, general case]\label{prop euler general}
Let $u$ be a local minimum of $J$ over $K$. Then, if $J$ is differentiable at $u$, we have
\begin{equation*}
\left\langle J'(u), w\right\rangle\ge 0 \quad \forall w\in K(v).
\end{equation*}
\end{proposition}
\begin{proof}
With the same notations of Definition \ref{def cone of directions}, set $w^n=(v^n-v)/\varepsilon^n$. By definition, we have that $w\in K(v)$ if and only if there exists a sequence $(w^n)_{n\in\NN}$ in $V$ and a sequence of positive real numbers $(\varepsilon^n)_{n\in\NN}$ such that 
\begin{equation*}
\lim_{n\to\infty} w^n = w, \quad \lim_{n\to \infty} \varepsilon^n=0, \quad \textrm{ and } v+\varepsilon^n w^n \in K\quad \forall n\in\NN.
\end{equation*}
Now, since $u$ is a local minimum of $J$ over $K$, we get
\begin{equation*}
\frac{J(u+\varepsilon^n w^n) -J(u)}{\varepsilon^n} \ge 0 \quad\text{ for $n$ large enough}.
\end{equation*}
Passing to the limit as $n\to\infty$ yields
\begin{equation*}
\left\langle J'(u), w \right\rangle \ge 0 \quad \forall w\in K
\end{equation*}
as claimed.
\end{proof}

\begin{definition}
We call Lagrangian of problem \eqref{constrained min problem with equality}, the function
\begin{equation*}
\mathcal{L}(v,\mu)= J(v)+\sum_{i=1}^M \mu_i F_i(v) = J(v)+\mu\cdot F(v) \quad \forall(v,\mu)\in V\times \RR^ M.
\end{equation*}
The new variable $\mu\in \RR^M$ is called Lagrange multiplier for the constraint $F(v)=0$.
\end{definition}

\begin{lemma}\label{lem min pb eq eq}
The constrained minimization problem \eqref{constrained min problem with equality} admits the following equivalent formulation using the Lagrangian:
\begin{equation*}
\inf_{v\in V, F(v)=0} J(v)= \inf_{v\in V} \sup_{\mu\in \RR^M} \mathcal{L}(v,\mu).
\end{equation*}
\end{lemma}
\begin{proof}
The proof is done by cases. Notice that, if $F(v)=0$, then $J(v)=\mathcal{L}(v,\mu)$ for all $\mu\in \RR^M$. On the other hand, if $F(v)\ne 0$, then $\displaystyle\sup_{\mu\in\RR^M} \mathcal{L}(v,\mu)=+\infty$. 
Putting the two together yields
\begin{equation*}
\begin{aligned}
\inf_{v\in V}\sup_{\mu\in \RR^M} \mathcal{L}(v,\mu) &= \min\left( \inf_{v\in V, F(v)=0} \sup_{\mu\in\RR^M} \mathcal{L}(v,\mu),\, \inf_{v\in V, F(v)\ne 0} \sup_{\mu\in\RR^M} \mathcal{L}(v,\mu) \right) \\
&= \min \left( \inf_{v\in V, F(v)=0} J(v), \,+\infty   \right) = \inf_{v\in V, F(v)=0} J(v).
\end{aligned}
\end{equation*}
\end{proof}

\begin{theorem}[Stationarity of the Lagrangian]\label{stationarity of the lag 1}
With the same notation of \eqref{constrained min problem with equality}, assume that $J$ and $F$ are continuously differentiable in a neighborhood of $u\in V$ such that $F(u)=0$. If $u$ is a local minimizer and if the vectors $\left( F_i'(u)\right)_{1\le i\le M}$ are linearly independent, then there exist a Lagrange multiplier $\lambda\in\RR^M$ such that 
\begin{equation}\label{eq stat of the lag}
\frac{\pa\mathcal{L}}{\pa v}(u,\lambda)= J'(u)+ \lambda\cdot F'(u)=0 \quad \text{and}\quad \frac{\pa\mathcal{L}}{\pa\mu}(u,\la)=F(u)=0.
\end{equation}
\end{theorem}
\begin{proof}
First define $K=\left\{ v\in V : \; F(v)=0\right\}$ and then the corresponding cone of admissible directions $K(u)$ by Definition \ref{def cone of directions}. Now, since the vectors $\left(F_i'(u)\right)_{1\le i\le M}$ are linearly independent by hypothesis, we can use the implicit function theorem in a standard way 
 to deduce that 
\begin{equation*}
K(u)=\left\{w\in V \;:\; \left\langle F'_i(u), w\right\rangle =0 \text{ for } i=1, \dots, M   \right\},
\end{equation*}
or equivalently 
\begin{equation*}
K(u)=\bigcap_{i=1}^M \left[F'_i(u)\right]^\perp.
\end{equation*}
In particular $K(u)$ is a vector space (it is indeed the tangent space to the variety $K$ at the point $u$). Thus we can successively take $w$ and $-w$ in Proposition \ref{prop euler general} to get 
\begin{equation*}
\left\langle J'(u),w\right\rangle =0 \quad \forall w \in \bigcap_{i=1}^M \left[ F_i'(u) \right]^\perp.
\end{equation*}
This implies that $J'(u)$ is generated by $\left( F_i'(u) \right)_{1\le i\le M}$ (moreover, since the $F_i'(u)$ are linearly independent, the Lagrange multipliers $\mu_i$ are uniquely defined).  
\end{proof}
\subsection{Minimization with inequality constraints}
We consider the following minimization problem with inequality constraints
\begin{equation}\label{minim pb ineq constr}
\inf_{v\in V, F(v)\le 0} J(v),
\end{equation}
where $F(v)\le 0$ here means that $F_i(v)\le 0$ for $1\le i \le M$, with $F=(F_1,\dots, F_M): V\to \RR^M$ differentiable.
\begin{definition}
Let $u$ be such that $F(u)\le 0$. The set 
\begin{equation*}
I(u)=\Big\lbrace i \in \{1,\dots, M\} \;:\; F_i(u)=0 \Big\rbrace
\end{equation*}
is called the set of active constraints at $u$. The inequality constraints are said to be qualified at $u\in K$ if the vectors $\left( F_i' (u)\right)_{i\in I(u)}$ are linearly independent.
\end{definition}

There are other (more general) definitions of constraints qualification \cite{bonnans}.

\begin{definition}
We call Lagrangian of the previous problem the function
\begin{equation*}
\mathcal{L}(v,\mu)=J(v)+\sum_{i=1}^M \mu_i F_i(v) = J(v)+\mu\cdot F(v)\quad \forall (v,\mu)\in V\times \left( \RR_{\ge0}\right)^M.
\end{equation*}
The new {\bf non negative} variable $\mu\in \left( \RR_{\ge0}\right)^M$ is called Lagrange multiplier for the constraint $F(v)\le 0$. 
\end{definition}
The proof of the result below is analogous to that of Lemma \ref{lem min pb eq eq} and thus will be omitted.
\begin{lemma}
The constrained minimization problem \eqref{minim pb ineq constr} is equivalent to 
\begin{equation*}
\inf_{v\in V, \; F(v)\le 0} J(v)= \inf_{v\in V} \sup_{\mu\in \left( \RR_{\ge0}\right)^M} \mathcal{L}(v,\mu).
\end{equation*}
\end{lemma}

The existence of (non negative) Lagrange multipliers, analogous to Theorem \ref{stationarity of the lag 	1}, can be proved also for a minimization problem subject to inequality constraints. We refer to \cite[Theorem 10.2.15]{Allaire4} for a proof.

\begin{theorem}[Stationarity of the Lagrangian for the inequality constraint]
We assume that the constraints are qualified at $u$ satisfying $F(u)\le 0$. If $u$ is a local minimizer, there exist Lagrange multipliers $\lambda_1,\dots,\lambda_M\ge 0$ such that 
\begin{equation}\label{stationarity of the lag 2}
J'(u)+\sum_{i=1}^M \lambda_i F'_i(u)=0,\quad \lambda_i\le 0,\quad\lambda_i=0 \ \text{  if  }F_i(u)<0\quad \forall i\in\{1,\dots, M\}.
\end{equation}
\end{theorem}

The condition \eqref{stationarity of the lag 2} is indeed the stationarity of the Lagrangian since
\begin{equation*}
\frac{\pa\mathcal{L}}{\pa v}(u,\lambda)=J'(u)+ \lambda\cdot F'(u)=0,
\end{equation*}
and the condition $F(u)\le 0$ and $\lambda\cdot F(u)=0$ for $\lambda\ge 0$, is equivalent to the Euler inequality (Theorem \ref{thm Euler ineq}) associated to the maximization problem $\sup \mathcal{L}(u,\mu)$ with respect to the variable  $\mu$ in the closed convex set $\left(\RR_{\ge0}\right)^M$. Indeed
\begin{equation*}
\frac{\pa \mathcal{L}}{\pa\mu} (u,\la)\cdot (\mu-\lambda)= F(u)\cdot (\mu-\la)\le 0 \quad\forall\mu\in\left(\RR_{\ge0}\right)^M,
\end{equation*}
and thus $F(u)\cdot \mu \le F(u)\cdot \lambda=0$ for all $\mu\in (\RR_{\ge0})^M$ as claimed.
\subsection{Interpreting the Lagrange multipliers}

Define the Lagrangian for the minimization of $J(v)$ under the constraint $F(v)=c$ as follows:
\begin{equation*}
\mathcal{L}(v,\mu,c)= J(v)+\mu\cdot \left( F(v)-c\right).
\end{equation*}
We claim that the value of the Lagrange multiplier represents the sensitivity of the minimal value with respect to variations of the constraint $c$. To this end, let $u(c)$ and $\lambda(c)$ denote the minimizer and the optimal Lagrange multiplier respectively. Moreover, we assume that they are differentiable with respect to $c$. Then 
\begin{equation*}
\gr_c J(u(c))= - \lambda(c).
\end{equation*}
In other words, $\lambda$ gives the derivative of the minimal value with respect to $c$ without any further calculation. 
Indeed 
\begin{equation*}
\gr_c J(u(c))  = \gr_c \mathcal{L}\left(u(c),\lambda(c),c\right) 
= \frac{\pa \mathcal{L}}{\pa v} \gr_c u(c) + \frac{\pa \mathcal{L}}{\pa \mu} \gr_c \mu(c) 
+ \frac{\pa \mathcal{L}}{\pa c} = -\lambda(c),
\end{equation*}
where, in the last equality we used 
\begin{equation*}
\frac{\pa \mathcal{L}}{\pa v}(u(c), \lambda(c),c)=0 \quad\textrm{ and }\quad\frac{\pa\mathcal{L}}{\pa \mu}(u(c),\lambda(c),c)=0,
\end{equation*}
which are a consequence of Theorem \ref{stationarity of the lag 1} and the constraint $F(u(c))=c$ respectively.

\section{Dual energy} \label{sec:dual energy}

In this section, we shall associate to a minimizing problem with a maximizing problem, so called dual problem.
To simplify the argument, we will assume that $V$ and $Y$ are two Banach spaces.
Let $V'$ and $Y'$ be the corresponding dual spaces.
The following argument is according to \cite{ET} and see the book for the more general setting.
For $J\colon V\to\RR\cup\{\infty\}$, 
we consider the following minimizing problem:
\begin{equation} \label{eq:primal problem}
\inf_{v\in V}J(v).
\end{equation}
For given problem \eqref{eq:primal problem}, we are now able to define a dual problem.
We shall consider a 
function $\Phi\colon V\times Y\to\RR\cup\{\infty\}$ such that
\begin{equation*}
\Phi(v, 0)=J(v), \quad v\in V.
\end{equation*}
We define the conjugate function $\Phi^*\colon V'\times Y'\to\RR\cup\{\infty\}$ as
\begin{equation*}
\Phi^*(v^*, p^*):=\sup_{(v, p)\in V\times Y}\left\{\langle v^*, v\rangle+\langle p^*, p\rangle-\Phi(v, p)\right\}, \quad (v^*, p^*)\in V'\times Y'.
\end{equation*}
We call the problem
\begin{equation} \label{eq:dual problem}
\sup_{p^*\in Y'}\left\{-\Phi^*(0, p^*)\right\}
\end{equation}
the dual problem of \eqref{eq:primal problem}.

In the following, we will mention the relationship between \eqref{eq:primal problem} and \eqref{eq:dual problem} in a special case.
Let $\Lambda\colon V\to Y$ be a continuous linear operator.
Assume that $J$ can be rewritten as
\begin{equation*}
J(v)=\tilde{J}(v, \Lambda v), \quad v\in V,
\end{equation*}
where $\tilde{J}$ is a function of $V\times Y$ into $\RR\cup\{\infty\}$.
In this case, the function $\Phi$ will be
\begin{equation*}
\Phi(v, p):=\tilde{J}(v, \Lambda v-p), \quad (v, p)\in V\times Y.
\end{equation*}
Then the conjugate function $\Phi^*$ becomes
\begin{equation*}
\begin{aligned}
\Phi^*(0, p^*)
&=\sup_{(v, p)\in V\times Y}\left\{\langle p^*, p\rangle-\tilde{J}(v, \Lambda v-p)\right\}\\
&=\sup_{v\in V}\sup_{q\in Y}\left\{\langle p^*, \Lambda v\rangle-\langle p^*, q\rangle-\tilde{J}(v, q)\right\}\\
&=\sup_{(v, q)\in V\times Y}\left\{\langle \Lambda^*p^*, v\rangle-\langle p^*, q\rangle-\tilde{J}(v, q)\right\}.
\end{aligned}
\end{equation*}

For this case, we can see the following relationship.
\begin{theorem} \label{thm:dual energy}
Assume that $\tilde{J}$ is convex and \eqref{eq:primal problem} is finite.
We also assume that there exists $v_0\in V$ such that $\tilde{J}(v_0, \Lambda v_0)<\infty$ and the function $p\mapsto \tilde{J}(v_0, p)$ is continuous at $\Lambda v_0$.
Then
\begin{equation*}
\inf_{v\in V}J(v)=\sup_{p^*\in Y'}\left\{-\Phi^*(0, p^*)\right\}
\end{equation*}
and maximizing problem \eqref{eq:dual problem} has at least one solution.
\end{theorem}
To show Theorem~\ref{thm:dual energy}, we will use convex analysis.
For the details of the proof, see \cite[Chapter III, Theorem~4.1]{ET}.

\begin{example} \label{ex:dual energy}
We show an application of Theorem~\ref{thm:dual energy}.
Let $\Omega\subset\RR^N$ be a smooth domain.
We consider the Dirichlet problem
\begin{equation*}
\left\{
\begin{aligned}
-{\rm div}\,(h\nabla u)&=f &&\text{ in } \Omega,\\
u&=0 &&\text{ on } \partial\Omega,
\end{aligned}
\right.
\end{equation*}
where $f\in L^2(\Omega)$ and $h\colon \Omega\to\RR$ is a positive given function.
The solution $u$ of the problem above is the minimizer of
\begin{equation*}
\dfrac{1}{2}\int_\Omega h|\nabla v|^2\,dx-\int_\Omega fv\,dx, \quad v\in H^1_0(\Omega).
\end{equation*}
We can apply Theorem~\ref{thm:dual energy} with
\begin{equation*}
V=H^1_0(\Omega), \quad Y=L^2(\Omega)^N, \quad \Lambda=\nabla, \quad \tilde{J}(v, p)=\dfrac{1}{2}\int_\Omega h|p|^2\,dx-\int_\Omega fv\,dx.
\end{equation*}
In the case, we see that 
\begin{equation*}
\begin{aligned}
\Phi^*(0, p^*)
&=\sup_{v\in H^1_0(\Omega)}\sup_{q\in L^2(\Omega)^N}\left\{\int_\Omega \left(p^*\cdot\nabla v + fv - \dfrac{1}{2} h|q|^2- p^*\cdot q\right)\,dx\right\}\\
&=
\left\{
\begin{aligned}
&\dfrac{1}{2}\int_\Omega h^{-1}|p^*|^2\,dx &&\quad {\rm if} \ -{\rm div}\,p^*=f,\\
&\infty &&\quad {\rm otherwise}
\end{aligned}
\right.
\end{aligned}
\end{equation*}
and hence
\begin{equation*}
\sup_{p^*\in Y'}\{-\Phi^*(0, p^*)\}=-\inf_{\substack{p^*\in L^2(\Omega)^N \\ -{\rm div}\,p^*=f}}\int_\Omega h^{-1}|p^*|^2\,dx.
\end{equation*}
\end{example}

\section{Numerical algorithms}
In this section we present some numerical algorithms in order to solve the kind of minimization problems that were treated in this chapter. All these algorithms are of iterative nature: starting from a give initial value $u^0$, we construct a sequence $(u^n)_{n\in\NN}$, which can be shown to converge to the solution $u$ of the given minimization problem under some hypotheses.  
\subsection{A gradient-type algorithm (non-constrained case)}
Suppose that $V=\RR^N$ (or, more generally, a Hilbert space, that we will identify with its dual $V'$). We consider the following minimization problem without constraints:
\begin{equation}\label{inf v in V of J(v)}
\inf_{v\in V} J(v). 
\end{equation}
We initialize the algorithm by choosing some initial value $u^0\in V$ and iteratively update it as follows:
\begin{equation}\label{gradient type alg no const}
u^{n+1}=u^n-\mu J'(u^n),
\end{equation}
where $\mu$ is a positive parameter that we choose in advance (a more sophisticate algorithm involving the optimal choice of $\mu=\mu^n$ for each iteration is discussed in \cite[Theorem 3.38]{Allaire2}).

\begin{theorem}
Let $V$ be a Hilbert space and suppose that the functional $J: V\to \RR$ is strongly convex, that is, for some $\alpha>0$
\begin{equation*}
\langle J'(u)-J'(v),u-v\rangle\ge \alpha\norm{u-v}^2\quad \forall u,v\in V.
\end{equation*}
Moreover, assume that $J$ is differentiable with Lipschitz continuous derivative $J'$. Then, if $\mu$ is small enough (depending on $\al$ and on the Lipschitz constant of $J'$), the gradient-type algorithm described above converges. In other words, for all $u^0$, the sequence $(u^n)_{n\in\NN}$ defined in \eqref{gradient type alg no const} converges to the solution $u$ of \eqref{inf v in V of J(v)}.
\end{theorem}

For a proof, see \cite{Allaire4}.

\begin{remark}
Choosing the right step length is not an easy task. 
Let us use the line search strategy as follows: 
start with a given step $\mu^0>0$. Now, at each iteration, increase the current step, $\mu_{n+1}=1.1\times \mu_n$, if $J$ decreases, and reduce it, $\mu_{n+1}=0.5\times \mu_n$ if $J$ increases. 
\end{remark}

\subsection{A gradient-type algorithm (constrained case)}
Suppose that $J$ is a real valued strictly convex differentiable functional defined on a nonempty closed convex subset $K$ of the Hilbert space $V$. The set $K$ represents the imposed constraints. We consider the following minimization problem
\begin{equation}\label{inf v in K of J(v)}
\inf_{v\in K} J(v).
\end{equation}
Theorem \ref{ex minimizer for cont J over closed K} ensures the existence of a minimizer $u$ for \eqref{inf v in K of J(v)} (which is unique by Proposition \eqref{prop uniq minimizer for str convex J}). Moreover, according to Theorem \ref{thm Euler ineq}, the minimizer $u$ is characterized by the condition
\begin{equation*}
\langle J'(u), v-u\rangle\ge 0 \quad\forall v\in K.
\end{equation*}
Notice that the condition above can be rephrased as follows. For all $\mu>0$
\begin{equation}\label{it was a projection!}
\langle u-\left( u-\mu J'(u) \right), v-u\rangle\ge0 \quad \forall v\in K.
\end{equation}
Let $P_K: V\to K$ denote the projection operator onto the convex subset $K$. Then \eqref{it was a projection!} just states that $u$ is the orthogonal projection of $u-\mu J'(u)$ onto $K$. In other words
\begin{equation*}
u=P_K\left( u-\mu J'(u) \right)\quad\forall \mu>0.
\end{equation*}
Therefore we devise a (projected) gradient-type algorithm, defined by the following iteration 
\begin{equation}\label{proj grad algorithm}
u^{n+1}=P_K\left( u^n-\mu J'(u^n) \right),
\end{equation}
where $\mu$ is a fixed positive parameter.

\begin{theorem}
Let $J$ be a differentiable strongly convex functional, with derivative $J'$ Lipschitz continuous on $V$. Then, if $\mu$ is small enough, the projected gradient algorithm with fixed step defined above converges. In other words, for all initial values $u^0\in K$, the sequence $(u^n)_{n\in\NN}$ defined by \eqref{proj grad algorithm} converges to the solution $u$ of \eqref{inf v in K of J(v)}.
\end{theorem}
We refer to \cite[Theorem 10.5.8]{Allaire4} for a proof.

\begin{remark}
Another possibility is to penalize the constraints, i.e., for small $\e$ we replace the problem 
\begin{equation*}
\inf_{v\in V,\,F(v)\le0} J(v) \quad \text{ by }\quad \inf_{v\in V} \left\lbrace J(v) + \frac{1}{\e}  \sum_{i=1}^M \big(\max(F_i (v),0\big)^2    \right\rbrace.
\end{equation*}
\end{remark}

\begin{example}[Some projection operators $P_K$]
Here we present some projection operators that can be computed explicitly.
\begin{itemize}
\item If $V=\RR^M$ and $\displaystyle K=\prod_{i=1}^M \,[a_i,b_i]$, then for $x=(x_1,\dots, x_M)\in \RR^M$
we have
\begin{equation*}
P_K(x)=y\quad \text{ with }\quad y_i= \min\left(\max(a_i,x_i),b_i\right) \quad \text{ for } 1\le i \le M.
\end{equation*}
\item If $V=\RR^M$ and $\displaystyle K=\Big\{x\in \RR^M:\; \sum_{i=1}^M x_i=c_0\Big\}$, then 
\begin{equation*}
P_K(x)=y\quad\text{ with } \quad y_i=x_i-\lambda,\quad \lambda=\frac{1}{M}\left(-c_0 + \sum_{i=1}^M x_i \right).
\end{equation*}
\item Similarly, if $V=L^2(\Omega)$ and $K=\{ \phi \in V :\; a(x)\leq \phi(x) \leq b(x)  \}$ or $\displaystyle K=\left\{ \phi\in V : \; \int_\Omega \phi \, dx=c_0 \right\}$ the corresponding projection operators $P_K$ can be obtained by replacing finite sums with integrals in the two examples above. 
\end{itemize}
\end{example}
For more general closed convex sets $K$, the corresponding projection operator $P_K$ can be very hard to determine. In such cases one can use the so called Uzawa algorithm \cite{Allaire4} which looks for a saddle point of the Lagrangian. 

\section{Exercises}
\begin{problem}
For a given $f\in L^2(\Omega)$, $\Omega$ being a rectangle in 2D, solve the following optimization problem numerically under the constraints $0\le u(x)\le 1$ and $\displaystyle\int_\Omega u\,dx=|\Omega|/2${\rm :} 
\begin{equation*}
\min_{u\in L^2(\Omega)} \int_\Omega |u-f|^2\, dx. 
\end{equation*}
\end{problem}

\begin{problem}
For a given $f\in L^2(\Omega)$, $\Omega$ being a rectangle in 2D, and $\e>0$, solve the following optimization problem numerically under the constraints $0\le u(x)\le 1${\rm :} 
\begin{equation*}
\min_{u\in L^2(\Omega)} \int_\Omega \left( |u-f|^2 + {\e}^2|\gr u|^2\right) \, dx. 
\end{equation*}
\end{problem}

\chapter{Parametric optimal design}\label{Parametric optimal design}
\section{Optimization of a membrane's thickness}
In this section, we consider a parametric optimal design problem of a membrane. Let $\Omega$ be a bounded domain in $\rn$ ($N\ge 2$) and $f\in L^2(\Omega)
$ be external forces. Let us consider the displacement $u \in H^1_0(\Omega)$, defined as the solution of  
\begin{equation}\label{membrane with thickness h equation}
\left\{
\begin{aligned}
-\dv(h\gr u)&=f &&\text{ in } \Omega, \\
u&=0 &&\text{ on } \partial\Omega, 
\end{aligned}
\right.
\end{equation}
where $h = h(x)$ is the thickness of membrane. Note that Lax--Milgram theorem ensures that there exists a unique solution $u \in H^1_0(\Omega)$ of \eqref{membrane with thickness h equation} if $f \in L^2(\Omega)$. 
For some given constants $0<h_{\rm min}<h_{0}<h_{\rm max}$, we seek to optimize the membrane by varying its thickness $h(x)$ in the admissible set defined by 
\begin{equation*}
\mathcal{U}_{\rm ad}=\left\lbrace h\in L^\infty(\Omega) \;:\; 0<h_{\rm min}\le h(x)\le h_{\rm max} \text{ a.e. in } \Omega,\; \int_\Omega h(x) dx= h_0|\Omega|  \right\rbrace.
\end{equation*}

We consider the following parametric shape optimization problem:
\begin{equation*}
\inf_{h\in\mathcal{U}_{\rm ad}} \left\{J(h)=\int_\Omega j(u)\, dx\right\},
\end{equation*}
where $u$ depends on $h$ through the state equation \eqref{membrane with thickness h equation}, and $j$ is a $C^1$ function from $\RR$ to $\RR$ such that $|j(u)|\le  C(u^2+1)$ and $|j'(u)|\le C(|u|+1)$. As examples of function $j$, we can take $j(u) = f u$ if we want to minimize the compliance (maximize the rigidity of the membrane), or $j(u) = |u-u_0|^2$ if we want to minimize the least-square criterion to reach a target displacement $u_0 \in L^2(\Omega)$. 

Before studying the existence of an optimal thickness, we show the continuity of the cost function. 
\begin{proposition}\label{h mapsto J(h) is cont from Uad}
The application
\begin{equation*}
h\mapsto J(h)=\int_\Omega j(u) \,dx
\end{equation*}
is a continuous mapping from $\mathcal{U}_{\rm ad}$ into $\RR$. 
\end{proposition}
\begin{proof}
The result follows immediately by composition of the two continuous functions that appear in the following lemmas: Lemma \ref{u to int ju} and Lemma \ref{h to u}. 
\end{proof}

\begin{lemma}\label{u to int ju}
The map $\displaystyle v\mapsto \int_\Omega j(v) \,dx$ is continuous from $L^2(\Omega)$ into $\RR$.
\end{lemma}
\begin{proof}
The result follows by the Lebesgue dominated convergence theorem.
\end{proof}

\begin{lemma}\label{h to u}
The map $h\mapsto u$, where $u\in H_0^1(\Omega)$ is the solution of \eqref{membrane with thickness h equation}, is a continuous function from $\mathcal{U}_{\rm ab}$ into $H_0^1(\Omega)$. 
\end{lemma}
\begin{proof}
Let $(h_n)_{n\in\NN}\subset\mathcal{U}_{\rm ab}$ be a sequence converging in the $L^\infty$-norm to some $h_\infty\in L^\infty(\Omega)$. Let $u_n\in H_0^1(\Omega)$ denote the unique solution of the membrane equation with associated thickness $h_n$:
\begin{equation*}
\left\lbrace
\begin{aligned}
-\dv (h_n \gr u_n)&= f &&\text{ in }\Omega,\\
u_n&=0 && \text{ on }\partial\Omega,
\end{aligned}
\right.
\end{equation*}
or the equivalent weak formulation
\begin{equation}\label{membrane equation for u_n weak}
\int_\Omega h_n \gr u_n \cdot \gr \phi\, dx= \int_\Omega f \phi\, dx \quad \forall\phi\in H_0^1(\Omega). 
\end{equation}
We will prove that $u_n$ is a Cauchy sequence in $H_0^1(\Omega)$ and thus  it converges. Take $n,m\in \NN$ and subtract the variational formulation for $u_n$ \eqref{membrane equation for u_n weak} from that of $u_m$ for fixed $\phi\in H_0^1(\Omega)$ to be chosen later. 
We get
\begin{equation*}
\int_\Omega h_m \gr (u_m-u_n)\cdot \gr \phi\, dx = \int_\Omega (h_n-h_m) \gr u_n \cdot \gr \phi\, dx \quad \forall \phi\in H_0^1(\Omega). 
\end{equation*}
Choosing $\phi=u_m-u_n$ we deduce
\begin{equation*}
\norm{\gr (u_m-u_n)}_{L^2(\Omega)}\le \frac{C}{h_{\rm min}^2} \norm{f}_{L^2(\Omega)} \norm{h_m-h_n}_{L^\ali(\Omega)},
\end{equation*}
which proves the claim.
\end{proof}

\section{Existence theories} \label{exist thm}
The question of the existence of optimal shapes is far from simple. We cannot apply the results of Chapter \ref{chap1} directly since $J(h)$ is not generally convex function. In fact, there exists no optimal shape in general. General counter-examples have been found by Murat \cite{Murat}. It is an important issue because this non-existence phenomenon has dramatic consequences for the numerical computations. Thus the definition of the set $\mathcal{U}_{\rm ab}$ of admissible designs has to be modified in order to obtain existence of optimal shapes. The main strategies employed to gain the existence of optimal shapes are discretization (when the admissible set is made finite dimensional), regularization (when the admissible set is made compact), and sometimes a miracle (when the given optimization problem happens to be convex).   

\subsection{Definition of a counter-example}
First, let us show a counter-example to the existence of optimal design for the membrane problem. For simplicity, let $N = 2$ and $\Omega = (0,1) \times (0,1)$. 
\begin{figure}[h]
\centering
\includegraphics[width=.8\linewidth,center]{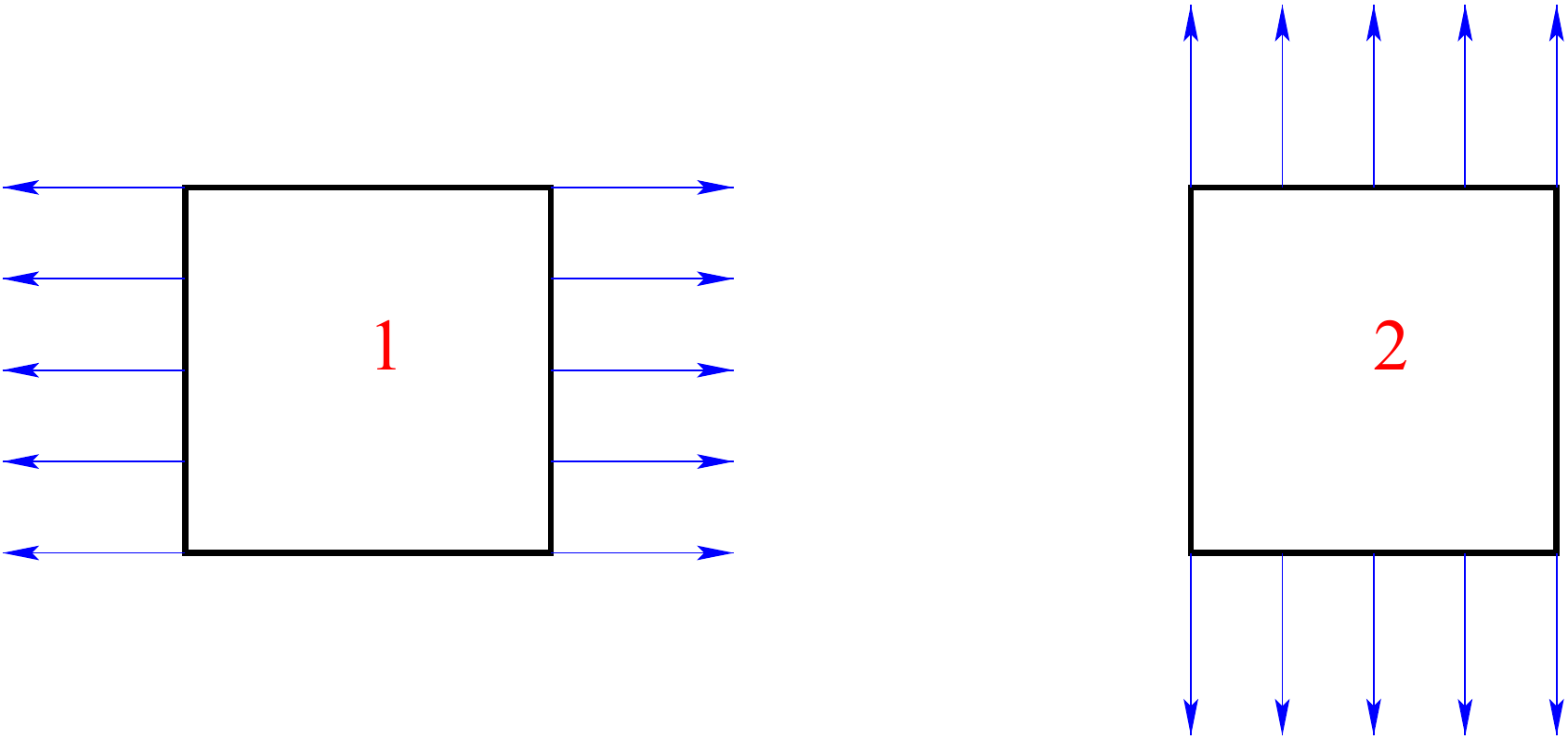}
\caption{The setting of the counter-example: we seek a membrane that is strong for horizontal loading $(1)$ and weak for vertical loading $(2)$.} 
\label{cexepais}
\end{figure}

We want to minimize the following objective function for $h\in\mathcal{U}_{\rm ad}$:  
\begin{equation}\label{existence counter-example definition}
J(h)= \int_{\pa\Omega} e_1\cdot n u_1 \,ds - \int_{\pa\Omega} e_2 \cdot n u_2\, ds, 
\end{equation}
where $e_1$, $e_2$ are the horizontal and vertical directions $(1,0)$, $(0,1)$ respectively and $u_1$, $u_2$ are the solutions of the following membrane problems: 
\begin{equation*}
    \left\{
    \begin{aligned}
        -\dv (h \gr u_1)&=0 &&\text{ in }\Omega,\\
        h\gr u_1\cdot n&=e_1\cdot n &&\text{ on }\pa\Omega,
    \end{aligned}
    \right.
    \quad\quad
    \left\{
    \begin{aligned}
        -\dv (h \gr u_2)&=0 &&\text{ in }\Omega,\\
        h\gr u_2\cdot n&=e_2\cdot n &&\text{ on }\pa\Omega.
    \end{aligned}
    \right.
\end{equation*}
When we minimize \eqref{existence counter-example definition}, we want the membrane to be strong for horizontal loading (we minimize compliance in the $e_1$ direction), and at the same time weak for vertical loading (we maximize the compliance in the direction $e_2$). This property of the objective function makes the problem ill-posed in the following sense.

\begin{theorem}\label{the counter-example is really a counter-example}
The infimum of \eqref{existence counter-example definition} is not attained by any $h\in \mathcal{U}_{\rm ad}$. 
\end{theorem}
Since the rigorous proof of Theorem \ref{the counter-example is really a counter-example} is a little bit technical, here we will only explain the main ideas by means of a ``hand-waving argument". First of all, notice that if $h$ is uniform (i.e. $h$ is a constant function), then by definition the membrane is isotropic. Therefore, also the domain $\Omega$ is isotropic, that is to say that it shows the same mechanical behavior in all direction. However, it is better to build horizontal layers of alternating small and large thicknesses in order to minimize the objective function \eqref{existence counter-example definition} (see Figure \ref{cexepais2}). In other words, we are building a laminated structure that is horizontally strong but vertically weak. In order to intuitively justify this statement consider the following. 
Vertically, the lines of forces must cross the layers of minimal thickness: this means that the structure is thus weak with respect to vertical stress. On the other hand, horizontally, the lines of forces follow the layers of maximal thickness: this means that the structure is thus strong with respect to horizontal stress. 
However, since the boundary conditions are uniform, the membrane is
horizontally stronger if the layers are finer, as the lines of forces are deviating from the horizontal to a lesser extent.
If $h$ oscillates at a small scale, we obtain an anisotropic composite material. To reach the minimum, the oscillation scale must go to $0$. Therefore, there does not exist any real optimal design that does not involve a microstructure at an infinitely small scale. We refer the interested reader to Section 5.2 in \cite{Allaire2} for the details. 
\begin{figure}[h]
\centering
\includegraphics[width=0.5\linewidth,center]{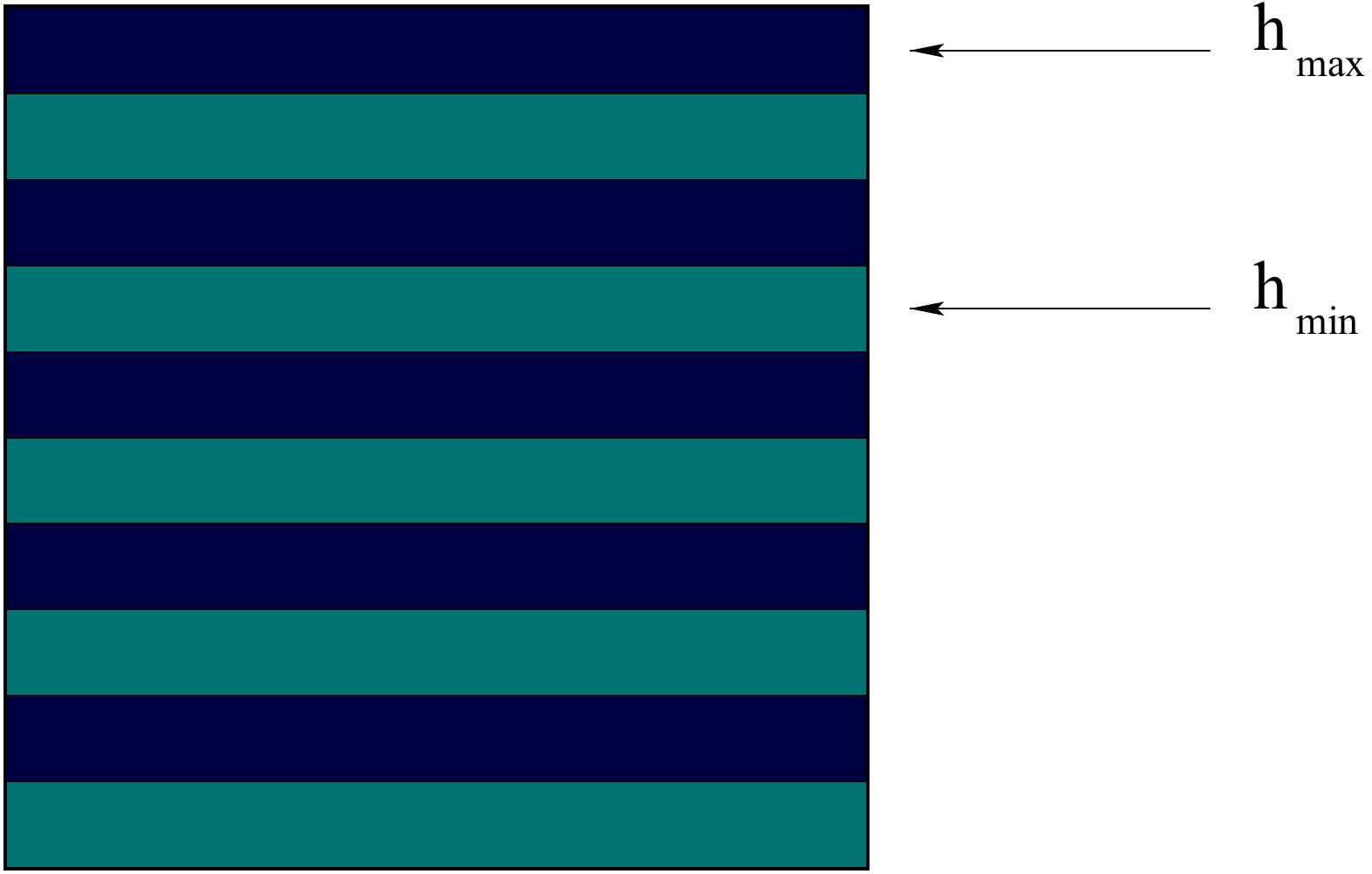}
\caption{Horizontal layers of alternating small and large thicknesses.} 
\label{cexepais2}
\end{figure}






\subsection{Existence for a discretized model}
One way to avoid non-existence due to a loss of compactness consists in working with a discretized (and hence finite-dimensional) model. 
Let $(\omega_i)_{1\le i\le n}$ be a partition of $\Omega$ such that
\begin{equation*}
\ol{\Omega}=\bigcup_{i=1}^n \ol{\om}_i, \quad \omega_i\cap\omega_j=\emptyset \quad\text{ for } i\ne j.
\end{equation*}
We introduce the subset $\mathcal{U}_{\rm ad}^n$ of $\mathcal{U}_{\rm ab}$ defined by
\begin{equation*}
\mathcal{U}_{\rm ab}^n=\left\lbrace 
h\in\mathcal{U}_{\rm ab} \;:\; h(x)\equiv h_i\in\RR \text{ in }\om_i,\quad 1\le i\le n
\right\rbrace.
\end{equation*}
In other words, any function $h\in \mathcal{U}_{\rm ab}^n$ is uniquely determined by the choice of the vector $(h_i)_{1\le i\le n}\in \RR^n$ and thus $\mathcal{U}_{\rm ad}^n$ is identified with a closed subset of $\RR^n$.

\begin{theorem}[Existence in finite dimension]
The discretized optimization problem 
\begin{equation*}
\inf_{h\in \mathcal{U}_{\rm ad}^n} J(h)
\end{equation*}
admits at least one minimizer.
\end{theorem}
\begin{proof}
Since $\mathcal{U}_{\rm ab}^n$ is a compact subset of $\rn$ and $J(h)$ is a continuous function on $\mathcal{U}_{\rm ab}^n$, the existence of a minimizer of $J$ in $\mathcal{U}_{\rm ab}^n$ follows from Theorem \ref{ex minimizer for cont J over closed K}. 
\end{proof}


\subsection{Existence with a regularity constraint}
Another classical way of ensuring the existence of minimizers relies in imposing additional regularity. For example, consider the space $C^1(\ol\Omega)$ which is a Banach space with the norm
\begin{equation*}
\norm{\phi}_{C^1(\ol{\Omega})}= \max_{x\in\ol{\Omega}}\left(|\phi(x)|+ |\gr \phi(x)| \right).
\end{equation*}
Take a given constant $R>0$ and introduce the subspace $\mathcal{U}_{\rm ad}^{\rm reg}$: 
\begin{equation*}
\mathcal{U}_{\rm ad}^{\rm reg}=\left\lbrace 
h\in \mathcal{U}_{\rm ad}\cap C^1(\ol\Omega) \;:\;  
\norm{h}_{C^1(\ol\Omega)}\le R
\right\rbrace.
\end{equation*}
The upper bound on the $C^1$-norm of $h$ in the definition above can be interpreted as a ``feasibility" (or ``manufacturability") constraint, as, in practice, the thickness cannot vary too rapidly. Then the following theorem holds: 
\begin{theorem}\label{thm regularized opt pb}
The regularized optimization problem 
\begin{equation*}
\inf_{h\in\mathcal{U}_{\rm ad}^{\rm reg}} J(h)
\end{equation*}
admits at least one minimizer.
\end{theorem}
\begin{proof}
Consider a minimizing sequence $(h_n)_{n\in\NN}\subset \mathcal{U}_{\rm ad}^{\rm reg}$ such that
\begin{equation*}
\lim_{n\to\infty} J(h_n)= \inf_{h\in\mathcal{U}_{\rm ad}^{\rm reg}} J(h).
\end{equation*}
By definition, the sequence $(h_n)_{n\in\NN}$ is bounded uniformly in $n$ in the space $C^1(\ol\Omega)$. We then apply a variant of Rellich theorem which states that one can extract a subsequence (still denoted by $h_n$ for simplicity) that converges in $C^0(\ol\Omega)$ to a limit function $h_\infty$ (furthermore, we know that $h_\ali\in C^1(\ol\Omega)$). We already know that $h\mapsto J(h)$ is a continuous mapping from $\mathcal{U}_{\rm ad}$ into $\RR$ by Proposition \ref{h mapsto J(h) is cont from Uad}, therefore
\begin{equation*}
\lim_{n\to\infty} J(h_n)=J(h_\ali),
\end{equation*}
which proves that $h_\ali$ is a global minimizer of $J$ in $\mathcal{U}_{\rm ad}^{\rm reg}$ as claimed.
\end{proof}
\begin{remark}
Theorem \ref{thm regularized opt pb} is actually a theorem of limited practical interest for the following reasons.
\begin{itemize}
\item In the practical cases, it is not clear how to choose the upper bound $R$ in the definition of $\mathcal{U}_{\rm ad}^{\rm reg}$.
\item Usually we do not have convergence as $R$ goes to infinity.
\item It is not clear whether, numerically, we have global or local minimizers.
\item Numerically, an upper bound on the $H^1$-norm is preferred instead:
\begin{equation*}
\norm{h}_{H^1(\Omega)}\le R.
\end{equation*}
\end{itemize}
\end{remark}

\section{Computation of a continuous gradient} \label{sec:Computation of a continuous gradient}
In this section, we will calculate the gradient of the objective function $J(h)$. This tells us the necessary conditions for optimality of the optimal shape and allows us to establish a numerical algorithm for calculating the optimal shape. 

First, we consider the boundary value problem
\begin{equation}\label{eq of h to u(h)}
	\left\{
	\begin{aligned}
		-\dv (h \gr u)&= f && \text{ in } \Omega,\\
			u&=0 && \text{ on } \partial\Omega,
	\end{aligned}
	\right.
\end{equation}
where $h$ belong to the following convex set which is larger than $\mathcal{U}_{\rm ad}$$\colon$
\begin{equation*}
	\mathcal{U}=
	\left\{ 
			h\in L^\infty(\Omega) \;:\; \exists h_0>0 \ {\rm  such \ that} \ h(x)\ge h_0 \,\, {\rm a.e.\;in } \ \Omega
	\right\}.
\end{equation*}
\begin{lemma} \label{Lemma:der u}
The application  $h\mapsto u(h)$, which gives the solution $u(h)\in H^1_0(\Omega)$ of \eqref{eq of h to u(h)} for $h\in \mathcal{U}$, is differentiable and its directional derivative at $h$ in the direction $k\in L^\infty(\Omega)$ is given by 
\begin{equation*}
	\langle u'(h), k\rangle =v,
\end{equation*}
where $v$ is the unique solution in $H^1_0(\Omega)$ of
\begin{equation}\label{eq of v}
	\left\{
	\begin{aligned}
		-\dv (h \gr v) &= \dv (k\gr u) && \text{ in } \Omega,\\
		u&=0 && \text{ on } \partial\Omega.
	\end{aligned}
	\right.
\end{equation}
\end{lemma}

\begin{proof}
Formally, one simply differentiates equation \eqref{eq of h to u(h)} with respect to $h$.
However, to be mathematically rigorous one should rather work at the level of the variational formulation.
To compute the directional derivative with respect to $k\in L^\infty(\Omega)$, we define $h(t)=h+tk$ for $t>0$. 
For $t>0$, let $u(t)$ be the solution for the thickness $h(t)$.
Differentiating with respect to $t$ leads to
\begin{equation*}
	\left\{
	\begin{aligned}
		-\dv (h(t) \gr u'(t)) &= \dv (h'(t)\gr u(t)) && \text{ in } \Omega,\\
u'(t)&=0 && \text{ on } \partial\Omega,
	\end{aligned}
	\right.
\end{equation*}
and, since $h'(0)=k$, we deduce $u'(0)=v$.

Let us justify the above calculation by showing that the map $h \mapsto u(h)$ is differentiable in the sense of Fr\'echet. First, there exists a unique solution $v$ of \eqref{eq of v} in $H^1_0(\Omega)$ thanks to the Lax--Milgram Theorem applied to the variational formulation 
\begin{equation}\label{variform of v}
\int_{\Omega} h \nabla v \cdot \nabla \phi \, dx = -  \int_{\Omega} k \nabla u \cdot \nabla \phi \, dx \quad \forall \phi \in H^1_0(\Omega). 
\end{equation}
We combine \eqref{variform of v} with the following variational formulation for $u(t)$
\begin{equation}\label{variform of u}
\int_{\Omega} h(t) \nabla u(t) \cdot \nabla \phi \, dx = \int_{\Omega} f \phi \, dx \quad \forall \phi \in H^1_0(\Omega). 
\end{equation}
Since $u(1) = u(h+k)$ and $u(0) = u(h)$, we obtain by difference 
\begin{equation*}
\int_{\Omega} h \nabla \left( u(h+k) - u(h) - v \right) \cdot \nabla \phi \, dx = - \int_{\Omega} k \nabla \left( u(h+k) - u(h) \right) \cdot \nabla \phi \, dx. 
\end{equation*}
Taking $\phi = u(h+k) - u(h) - v$ as a test function in the above yields 
\begin{equation}\label{estimate for diff1}
\begin{aligned} 
&\norm{\nabla \left( u(h+k) - u(h) - v \right)}_{L^2(\Omega)}^2 \\
&\quad= - \int_{\Omega} k \nabla \left( u(h+k) - u(h) \right) \cdot \nabla \left( u(h+k) - u(h) - v \right) \, dx 
\end{aligned}
\end{equation}
which implies
\begin{equation}\label{estimate for diff1bis}
\norm{\nabla \left( u(h+k) - u(h) - v \right)}_{L^2(\Omega)} \le C \norm{k}_{L^{\infty}(\Omega)} \norm{\nabla \left( u(h+k) - u(h) \right)}_{L^2(\Omega)}, 
\end{equation}
where we used Cauchy--Schwarz's inequality and the $H^1_0$ boundedness of $v$. Furthermore, by \eqref{variform of u} we have  
\begin{equation}\label{simple calculation1}
\int_{\Omega} (h + k) \nabla \left( u(h+k) - u(h) \right) \cdot \nabla \phi \, dx = - \int_{\Omega} k \nabla u(h) \cdot \nabla \phi \, dx.  
\end{equation}
Taking the test function as $\phi = u(h+k) - u(h)$ in \eqref{simple calculation1}, we obtain the following estimate: 
\begin{equation}\label{estimate for diff2}
\norm{\nabla \left( u(h+k) - u(h) \right)}_{L^2(\Omega)} \le C \norm{k}_{L^{\infty}(\Omega)}. 
\end{equation}
Combining \eqref{estimate for diff1} with \eqref{estimate for diff2}, we have 
\begin{equation*}
\norm{\nabla \left( u(h+k) - u(h) - v \right)}_{L^2(\Omega)} \le C \norm{k}^2_{L^{\infty}(\Omega)}. 
\end{equation*}
Therefore we obtain $u(h+k) = u(h) + v + o(k)$ as $\norm{k}_{L^{\infty}(\Omega)} \to 0$, which proves the claim.  
\end{proof}

\begin{lemma} \label{Lemma:J'}
For $h\in \mathcal{U}$, let $u(h)\in H^1_0(\Omega)$ be the solution to \eqref{eq of h to u(h)} and
\begin{equation*}
	J(h)=\int_\Omega j(u(h))\,dx,
\end{equation*}
where $j$ is a $C^1$ function from $\RR$ into $\RR$ such that $|j(u)|\le C(u^2+1)$ and $|j'(u)|\le C(|u|+1)$ for any $u\in\RR$.
The application $J(h)$, from $\mathcal{U}$ into $\RR$, is differentiable and its directional derivative at $h$ in the direction $k\in L^\infty(\Omega)$ is given by
\begin{equation*}
	\langle J'(h), k\rangle =\int_\Omega j'(u(h))v\,dx,
\end{equation*}
where $v=\langle u'(h), k\rangle$ is the unique solution in $H^1_0(\Omega)$ of
\begin{equation*}
	\left\{
	\begin{aligned}
		-\dv (h \gr v) &= \dv (k\gr u) && \text{ in } 	\Omega,\\
		u&=0 && \text{ on } \partial\Omega.
	\end{aligned}
	\right.
\end{equation*}
\end{lemma}

\begin{proof}
By simple composition of differentiable applications. To justify it, one only has to check that all the terms are well defined. We omit the details of the proof.
\end{proof}

\subsection{Adjoint state}
In order to treat the derivative of the objective function $J(h)$, we introduce the adjoint state $p$, defined as the unique solution in $H^1_0(\Omega)$ of 
\begin{equation}\label{eq adjoint equation}
	\left\{
	\begin{aligned}
		-\dv (h \gr p) &= -j'(u) && \text{ in } 	\Omega,\\
		p&=0 && \text{ on }  \partial\Omega.
	\end{aligned}
	\right.
\end{equation}

\begin{theorem} \label{Theorem:adjoint}
The cost function $J(h)$ is differentiable on $\mathcal{U}$ and
\begin{equation*}
J'(h)=\gr u\cdot\gr p.
\end{equation*}
If $h\in\mathcal{U}_{\rm ad}$ is a local minimizer of $J$ in $\mathcal{U}_{\rm ad}$, then it satisfies the necessary optimality condition
\begin{equation*}
\int_\Omega \gr u\cdot\gr p(k-h)\,dx\ge0
\end{equation*}
for any $k\in\mathcal{U}_{\rm ad}$.
\end{theorem}

\begin{proof}
To make explicit $J'(h)$ from Lemma~\ref{Lemma:J'}, we must eliminate $v=\langle u'(h), k\rangle$.
To this end, we employ the use of the adjoint state, solution of \eqref{eq adjoint equation}. multiplying the equation for $v$ by $p$ and that for $p$ by $v$, we integrate by parts
\begin{equation*}
\int_\Omega h\gr p\cdot\gr v\,dx=-\int_\Omega j'(u)v\,dx,
\end{equation*}
\begin{equation*}
\int_\Omega h\gr v\cdot\gr p\,dx=-\int_\Omega k\gr u\cdot\gr p\,dx,
\end{equation*}
Comparing these two equalities we deduce
\begin{equation*}
\langle J'(h), k\rangle=\int_\Omega j'(u)v\,dx=\int_\Omega k\gr u\cdot\gr p\,dx
\end{equation*}
for any $k\in L^\infty(\Omega)$.
Since $\gr u\cdot\gr p$ belongs to $L^1(\Omega)$, we check that $J'(h)$ is continuous on $L^\infty(\Omega)$. To obtain the condition of optimality, it suffices to apply Theorem \ref{thm Euler ineq} since $\mathcal{U}_{\rm ad}$ is a closed non-empty convex subset of $L^{\infty}(\Omega)$.
\end{proof}

\begin{remark}[How to find the adjoint state]
For independent variable $(\hat{h}, \hat{u}, \hat{p})\in L^\infty(\Omega)\times H^1_0(\Omega)\times H^1_0(\Omega)$, we introduce the Lagrangian
\begin{equation*}
	\mathcal{L}(\hat{h}, \hat{u}, \hat{p})=\int_\Omega j(\hat{u})\,dx+\int_\Omega\hat{p}\left(-\dv \left(\hat{h}\gr\hat{u}\right)-f\right)\,dx,
\end{equation*}
where $\hat{p}$ is a Lagrange multiplier (a function) for the constraint which connects $u$ to $h$.
By integration by parts we get
\begin{equation*}
\mathcal{L}(\hat{h}, \hat{u}, \hat{p})=\int_\Omega j(\hat{u})\,dx+\int_\Omega\left( \hat{h}\gr\hat{p}\cdot\gr\hat{u}-f\hat{p}\right)\,dx.
\end{equation*}
The partial derivative of $\mathcal{L}$ at $\hat u = u$ in the direction $\phi\in H^1_0(\Omega)$ is given by
\begin{equation*}
\left\langle\dfrac{\partial\mathcal{L}}{\partial \hat u}(\hat{h}, {u}, \hat{p}), \phi\right\rangle=\int_\Omega j'(u)\phi\,dx+\int_\Omega\left(\hat{h}\gr{p}\cdot\gr\phi\right)\,dx.
\end{equation*}
Notice that, requiring that $\left\langle\dfrac{\partial\mathcal{L}}{\partial \hat u}({h}, {u}, {p}), \phi\right\rangle=0$ for all directions $\phi$ is nothing else than the variational formulation of the adjoint equation \eqref{eq adjoint equation}.
\end{remark}

\subsection{A simple formula for the derivative}

It is possible to compute the derivative of $J$ by means of the Lagrangian in the following way:
\begin{equation*}
J'(h)=\dfrac{\partial\mathcal{L}}{\partial h}(h, u, p),
\end{equation*}
where $u$ is the state function (solution to \eqref{eq of h to u(h)}) and $p$ is the adjoint state (solution to problem \eqref{eq adjoint equation}).
Indeed, we have
\begin{equation*}
J(h)=\mathcal{L}(h, u, \hat{p}) \quad \forall\hat{p}\in H^1_0(\Omega)
\end{equation*}
by definition of the state function $u$.
Thus, if the map $h\mapsto u(h)$ is differentiable, we get for $k\in L^\infty(\Omega)$
\begin{equation*}
\left\langle J'(h), k\right\rangle=\left\langle \dfrac{\partial\mathcal{L}}{\partial h}(h, u, \hat{p}), k\right\rangle+\left\langle \dfrac{\partial\mathcal{L}}{\partial u}(h, u, \hat{p}), 
\dfrac{\partial u}{\partial h}(k)\right\rangle.
\end{equation*}
Then, taking $\hat{p}=p$, the adjoint we obtain
\begin{equation*}
\left\langle J'(h), k\right\rangle=\left\langle \dfrac{\partial\mathcal{L}}{\partial h}(h, u, p), k\right\rangle.
\end{equation*}

By the above discussion, we obtain the following theorem.

\begin{theorem}
Let $\mathcal{L}(\hat{h}, \hat{u}, \hat{p})$ be the Lagrangian defined as the sum of the objective function and the variational formulation of the state equation, i.e.,
\begin{equation*}
\mathcal{L}(\hat{h}, \hat{u}, \hat{p})=\int_\Omega j(\hat{u})\,dx+\int_\Omega\left( \hat{h}\gr\hat{p}\cdot\gr\hat{u}-f\hat{p}\right)\,dx.
\end{equation*}
Let $p$ be the solution of the adjoint equation
\begin{equation*}
\left\langle \dfrac{\partial\mathcal{L}}{\partial u}(h, u, p), \phi\right\rangle=0 \quad \forall\phi\in H^1_0(\Omega).
\end{equation*}
Assume that the solution $u=u(h)$ of the state equation \eqref{eq of h to u(h)} is differentiable with respect to $h$. Then the objective function $J$ is differentiable and
\begin{equation*}
J'(h)=\dfrac{\partial\mathcal{L}}{\partial h}(h, u, p).
\end{equation*}
\end{theorem}
This theorem is the practical method for computing $J'(h)$. Once the gradient of the cost function has been obtained, it is natural and quite easy to implement a gradient method to minimize $J(h)$ numerically. In Section \ref{Numerical algorithm for optimal thickness}, we provide numerical algorithms to compute the optimal thickness.

\section{A discrete approach}

One can wonder whether the such optimal design problems get simpler after discretization.
Unfortunately, the answer is ``no''.
In this section, we consider a discrete approach to the problems.
Applying a finite element method, the equation becomes a linear system of order $n$
\begin{equation*}
	K(h)y(h)=b,
\end{equation*}
where $K(h)$ is the rigidity matrix of the membrane (which depends on $h$), $b$ is a vector representing the forces $f$, and $y(h)$ the vector of the coordinates of the solution $u$ in the finite element basis (of dimension $n$).
We also discretize the admissible set as follows:
\begin{equation*}
	\mathcal{U}^{\rm disc}_{\rm ad}=
   \left\{
   	h\in\RR^N \;:\; h_{\rm max}\ge h_i\ge h_{\rm min}>0,\; \sum^n_{i=1}c_ih_i=h_0|\Omega|
   \right\},
\end{equation*}
where the finite sum
\begin{equation*}
	\sum^n_{i=1}c_ih_i
\end{equation*}
is an approximation of 
\begin{equation*}
	\int_\Omega h(x)\,dx.
\end{equation*}
Approximating the cost function, the discrete problem becomes
\begin{equation*}
	\inf_{h\in\mathcal{U}^{\rm disc}_{\rm ad}}
    \left\{
    	J^{\rm disc}(h)=j^{\rm disc}(y(h))
    \right\},
\end{equation*}
where $j^{\rm disc}$ is a smooth approximation of $j$ from $\RR^N$ into $\RR$.
In the case of the compliance we have:
\begin{equation*}
	j^{\rm disc}(y(h))=b\cdot y(h)=K(h)^{-1}b\cdot b.
\end{equation*}
In the case of a least-square criterion for a target displacement we have:
\begin{equation*}
	j^{\rm disc}(y(h))=B(y(h)-y_0)\cdot(y(h)-y_0) ,
\end{equation*}
where $B$ is a mass matrix. 
In practice, we need a way to compute the gradient of $J^{\rm disc}(h)$.
This can be applied to both finding the optimality condition and the implementation of a numerical method of minimization.

First, we consider the following ``naive idea".
Since $y(h)=K(h)^{-1}b$, we have 
\begin{equation}\label{j disc prime}
	(J^{\rm disc})'(h)=y'(h)(j^{\rm disc})'(y(h))  \quad\text{with} \quad y'(h)=-K(h)^{-1}K(h)'K(h)^{-1}b,
\end{equation}
where we used the notation $f'(h)=\left(\partial f(h)/ \partial h_i\right)_{1\le i \le n}$ and the second identity in \eqref{j disc prime} is a direct application of the formula for the derivative of a matrix.
We remark that this method is not practically useful because one must solve $n+1$ linear systems with respect to the matrix $K(h)$ in order to obtain all components of $y'(h)$. 
Recall that $K(h)$ is a very large matrix (of size $n\times n$) and its inverse is never explicitly computed as it would take too long.
As a consequence, we do not use the explicit formula $y(h)=K(h)^{-1}b$.
We rather use an adjoint method.

\subsection{Adjoint state}
\begin{definition}
	We define the adjoint state $p\in\RR^N$ as the solution of
    \begin{equation}\label{Kp=-j'}
    	K(h)p(h)=-(j^{\rm disc})'(y(h)).
    \end{equation}
\end{definition}
By rearranging the second equality of \eqref{j disc prime} we get
\begin{equation}\label{Ky'=-K'y}
K(h)y'(h)=-K'(h)y(h).
\end{equation}
Now, taking the scalar product of \eqref{Ky'=-K'y} with $p(h)$ and that of \eqref{Kp=-j'} with $y'(h)$, we obtain, for each component $i=1,\dots, n$:
\begin{equation*}
	K(h)p(h)\cdot \dfrac{\partial y}{\partial h_i}(h)=-\dfrac{\partial K}{\partial h_i}(h)y(h)\cdot p(h)=-(j^{\rm disc})'(y(h))\cdot \dfrac{\partial y}{\partial h_i}(h),
\end{equation*}
from which we deduce 
\begin{equation*}
	(J^{\rm disc})'(h)=K'(h)y(h)\cdot p(h)=\left(\dfrac{\partial K}{\partial h_i}(h)y(h)\cdot p(h)\right)_{1\le i\le n}.
\end{equation*}
In practice, this is the very formula that we use for evaluating the gradient $(J^{\rm disc})'(h)$ since it requires only to solve two linear systems.

There is no simplification in using a discrete approach rather than a continuous one.
Some authors prefer to discretize first and optimize afterwards.
This approach guarantees a perfect compatibility between the gradient and the cost function, but it requires a deep knowledge of the numerical solver. Here, we follow another philosophy, ``first optimize in a continuous framework, then discretize". It is much simpler, and no precision is lost if the finite element spaces are adequately chosen.

\section{Numerical algorithms}\label{Numerical algorithm for optimal thickness}
In this section, we show numerical algorithms to seek the optimal thickness of $h$. First, we consider the following projected gradient algorithm. 
\begin{algorithm}[H]
\caption{Projected gradient algorithm} 
\begin{enumerate}
	\item Initialization of the thickness $h_0\in\mathcal{U}_{\rm ad}$ (for example, a constant function which satisfies the constraints);
	\item Iterations until convergence, for $n\ge0$ set
	\begin{equation*}
		h_{n+1}=P_{\mathcal{U}_{\rm ad}}\left(h_n-\mu J'(h_n)\right),
	\end{equation*}
    where $\mu>0$ is a small descent step, $P_{\mathcal{U}_{\rm ad}}$ is the projection operator on the closed convex set $\mathcal{U}_{\rm ad}$ and the derivative of $J$ is given by
    \begin{equation*}
    	J'(h_n)=\gr u_n\cdot\gr p_n 
    \end{equation*}
    with state $u_n$ and adjoint $p_n$ (both defined with respect to the thickness $h_n$).
\end{enumerate}
\end{algorithm}

To make the algorithm fully explicit, we have to specify how to compute the projection operator $P_{\mathcal{U}_{\rm ad}}$.

We define the projection operator $P_{\mathcal{U}_{\rm ad}}$ as follows: 
\begin{equation*}
	\left(P_{\mathcal{U}_{\rm ad}}(h)\right)(x)=\max \left(h_{\min}, \min(h_{\max}, h(x)+\ell)\right),\quad x\in\Omega.
\end{equation*}
where $\ell$ is the unique Lagrange multiplier such that
\begin{equation*}
	\int_\Omega P_{\mathcal{U}_{\rm ad}}(h)\,dx=h_0|\Omega|.
\end{equation*}
The determination of the constant $\ell$ is not explicit but based on an iterative algorithm. 
First, notice that the function
\begin{equation*}
	h\longmapsto F(\ell)=\int_\Omega\max\left(h_{\min}, \min(h_{\max}, h(x)+\ell)\right)\,dx
\end{equation*}
is strictly increasing on the interval $[\ell^-, \ell^+]$, the inverse image of the closed interval $[h_{\min}|\Omega|, h_{\max}|\Omega|]$.
Thanks to this monotonicity property, we propose a simple iterative algorithm$\colon$ we first bracket the root by an interval $[\ell^1, \ell^2]$ such that
\begin{equation*}
	F(\ell^1)\le h_0|\Omega|\le F(\ell^2),
\end{equation*}
then we proceed by dichotomy to find the root $\ell$.

\begin{remark}
	\mbox{} 
	\begin{itemize}
		\item[{\rm 1}.]
        In practice, we rather use a projected gradient algorithm with a variable step (not optimal) which guarantees the decrease of the functional $J(h_{n+1})<J(h_n)$.
       \item[{\rm 2}.]
       The algorithm is rather slow.
       A possible acceleration is based on the quasi-Newton algorithm.
       \item[{\rm 3}.]
       The overhead generated by the adjoint computation is very modest$\colon$one has to build a new right-hand-side (using the state) and solve the corresponding linear system (with the same rigidity matrix).
       \item[{\rm 4}.]
       Convergence is detected when the optimality condition is satisfied with a threshold $\e>0$
       \begin{equation*}
       	|h_n-\max\left(h_{\min}, \min(h_{\max}, h_n-\mu_n J'(h_n)+\ell_n)\right)|\le\e\mu_n h_{\max}.
       \end{equation*}
	\end{itemize}
\end{remark}

\subsection{Another numerical algorithm for the compliance}

When $j(u)=fu$, we find $p=-u$ since $j'(u)=f$.
This particular case is said to be self-adjoint.
We use the dual or complementary energy (see Section~\ref{sec:dual energy}) 
\begin{equation*}
	\int_\Omega fu\,dx=\min_{\substack{\tau\in L^2(\Omega)^N, \\ -\dv\, \tau=f \ {\rm in} \ \Omega}}
   \int_\Omega h^{-1}|\tau|^2\,dx
\end{equation*}
in order to rewrite the original optimization problem as a double minimization problem:
\begin{equation*}
	\inf_{h\in\mathcal{U}_{\rm ad}}\min_{\substack{\tau\in L^2(\Omega)^N, \\ -\dv\, \tau=f \ {\rm in} \ \Omega}}
   \int_\Omega h^{-1}|\tau|^2\,dx,
\end{equation*}
and the order of minimization is irrelevant. This problem is convex and therefore it admits a minimizer. 

By elementary calculation, we can show that the following lemma holds. 
\begin{lemma}\label{simple calculation for phi}
	The function $\phi(a, \sigma)=a^{-1}|\sigma|^2$, defined from $\RR_{\ge0}\times\RR^N$ into $\RR$, satisfies 
    \begin{equation}\label{phi(a,sigma)=}
    	\phi(a, \sigma)=\phi(a_0, \sigma_0)+\phi'(a_0, \sigma_0)\cdot(a-a_0, \sigma-\sigma_0)+\phi\left(a, \sigma-\dfrac{a}{a_0}\sigma_0\right),
    \end{equation}
    where the derivative is given by 
    \begin{equation*}
    	\phi'(a_0, \sigma_0)\cdot(b, \tau)=-\dfrac{b}{a_0^2}|\sigma_0|^2+\dfrac{2}{a_0}\sigma_0\cdot\tau.
    \end{equation*}
In particular, since by \eqref{phi(a,sigma)=}, the graph of  $\phi(a,\sigma)$ lies above its linear approximation at each point $(a_0,\sigma_0)$, then $\phi$ is convex.
\end{lemma}

As a result, we obtain the following.
\begin{lemma}[Optimality conditions]\label{opti con for compliance pr}
	For a given $\tau\in L^2(\Omega)^N$, the problem
    \begin{equation*}
    	\min_{h\in\mathcal{U}_{\rm ad}}\int_\Omega h^{-1}|\tau|^2\,dx
    \end{equation*}
    admits a minimizer $h(\tau)$ in $\mathcal{U}_{\rm ad}$ given by
    \begin{equation}\label{h(tau)(x)=}
       h(\tau)(x)=
       \left\{
        \begin{aligned}
        &h^*(x) &&\text{ if } \ h_{\min}<h^*(x)<h_{\max},\\
        &h_{\min} &&\text{ if } \ h^*(x)\le h_{\min},\\
        &h_{\max} &&\text{ if } \ h^*(x)\ge h_{\max}
        \end{aligned}
        \quad {\rm with} \ h^*(x)=\dfrac{|\tau(x)|}{\sqrt{\ell}},
        \right.
    \end{equation}
    where $\ell$ is the Lagrange multiplier such that
    \begin{equation*}
    	\int_\Omega h(\tau)(x)\,dx=h_0|\Omega|.
    \end{equation*}
\end{lemma}
\begin{proof}[Sketch of the proof]
By Lemma \ref{simple calculation for phi} we obtain that the map $h\mapsto \int_\Omega h^{-1}|\tau|^2 dx$ is convex in $\mathcal{U}_{\rm ad}$. Therefore, Theorem \ref{existence of min in infinite dim convex space} ensures the existence of a minimum point $h$. This point is then characterized by the optimality condition given by Theorem \ref{thm Euler ineq}. We refer to \cite[Lemma 5.2.25]{Allaire2} for more details.  
\end{proof}

Lemma \ref{opti con for compliance pr} tells us the following numerical algorithm for the compliance: 

\begin{algorithm}[H]
\caption{Optimality criteria method} 
\begin{itemize}
\item[{\rm 1}.]
	Initialization of the thickness $h_0\in\mathcal{U}_{\rm ad}$.
 \item[{\rm 2}.]
 	Iterations until convergence, for $n\ge0$, 
   \begin{itemize}
   		\item[{\rm (a)}]
        	Computation of the state $\tau_n$, unique solution of
           \begin{equation}\label{min 2a}
           		\min_{\substack{\tau\in L^2(\Omega)^N, \\ -\dv \,\tau=f \ {\rm in} \ \Omega}}
   \int_\Omega h_n^{-1}|\tau|^2\,dx.
           \end{equation}
       \item[{\rm (b)}]
       		Update of the thickness$\colon$
           \begin{equation*}
           		h_{n+1}=h(\tau_n),
           \end{equation*}
        	where $h(\tau)$ is the minimizer defined by \eqref{h(tau)(x)=}. Finally, the Lagrange multiplier $\ell$ is computed by dichotomy.
\end{itemize}
\end{itemize}
\end{algorithm}

Remark that, by the dual energy approach introduced in Section \ref{sec:dual energy},  minimizing \eqref{min 2a} in $\tau$ is equivalent to solving the equation
\begin{equation*}
	\left\{
    \begin{aligned}
    	-\dv(h_n\gr u_n) & =f &&\text{ in } \Omega,\\
       u_n & =0 && \text{ on } \partial\Omega,
    \end{aligned}
	\right.
\end{equation*}
and then recovering $\tau_n$ by the formula
\begin{equation*}
	\tau_n=h_n\gr u_n.
\end{equation*}
The algorithm can be interpreted as an alternate minimization in $\tau$ and $h$ of the objective function.
In particular, we deduce that the objective function always decreases through the iterations. Indeed, for all $n\ge0$, 
\begin{equation*}
	J(h_{n+1})=\int_\Omega h_{n+1}^{-1}|\tau_{n+1}|^2\,dx\le \int_\Omega h_{n+1}^{-1}|\tau_{n}|^2\,dx\le\int_\Omega h_{n}^{-1}|\tau_{n}|^2\,dx=J(h_n),
\end{equation*}
where, for fixed $h_{n+1}$ we minimized in $\tau$ and then, for fixed $\tau_n$ we minimized in $h$.
This algorithm can also be interpreted as an optimality criteria method.

\section{Thickness optimization of an elastic plate}
We consider the following elasticity problem for an elastic plate $\Omega$
\begin{equation*}
	\left\{
    \begin{aligned}
    	-\dv \, \sigma &=f && \text{ in } \Omega,\\
      \sigma&=2\mu h e(u)+\lambda h {\rm tr}(e(u)){\rm Id} && \text{ in } \Omega,\\
      u&=0 &&\text{ on } \Gamma_D, \\
      \sigma\cdot n&=g &&\text{ on } \Gamma_N
    \end{aligned}
	\right.
\end{equation*}
with strain tensor $e(u)=(\gr u+(\gr u)^t)/2$.
The set of admissible thicknesses is
\begin{equation*}
	\mathcal{U}_{\rm ad}=
    \left\{
    	h\in L^\infty(\Omega) \; :\;h_{\max}\ge h(x)\ge h_{\min} >0 \text{ a.e. in } \Omega, \int_\Omega h(x)\,dx=h_0|\Omega|
    \right\}.
\end{equation*}

\begin{figure}[H]
    \begin{minipage}[h]{.45\textwidth}
        \centering
        \includegraphics[width=\textwidth]{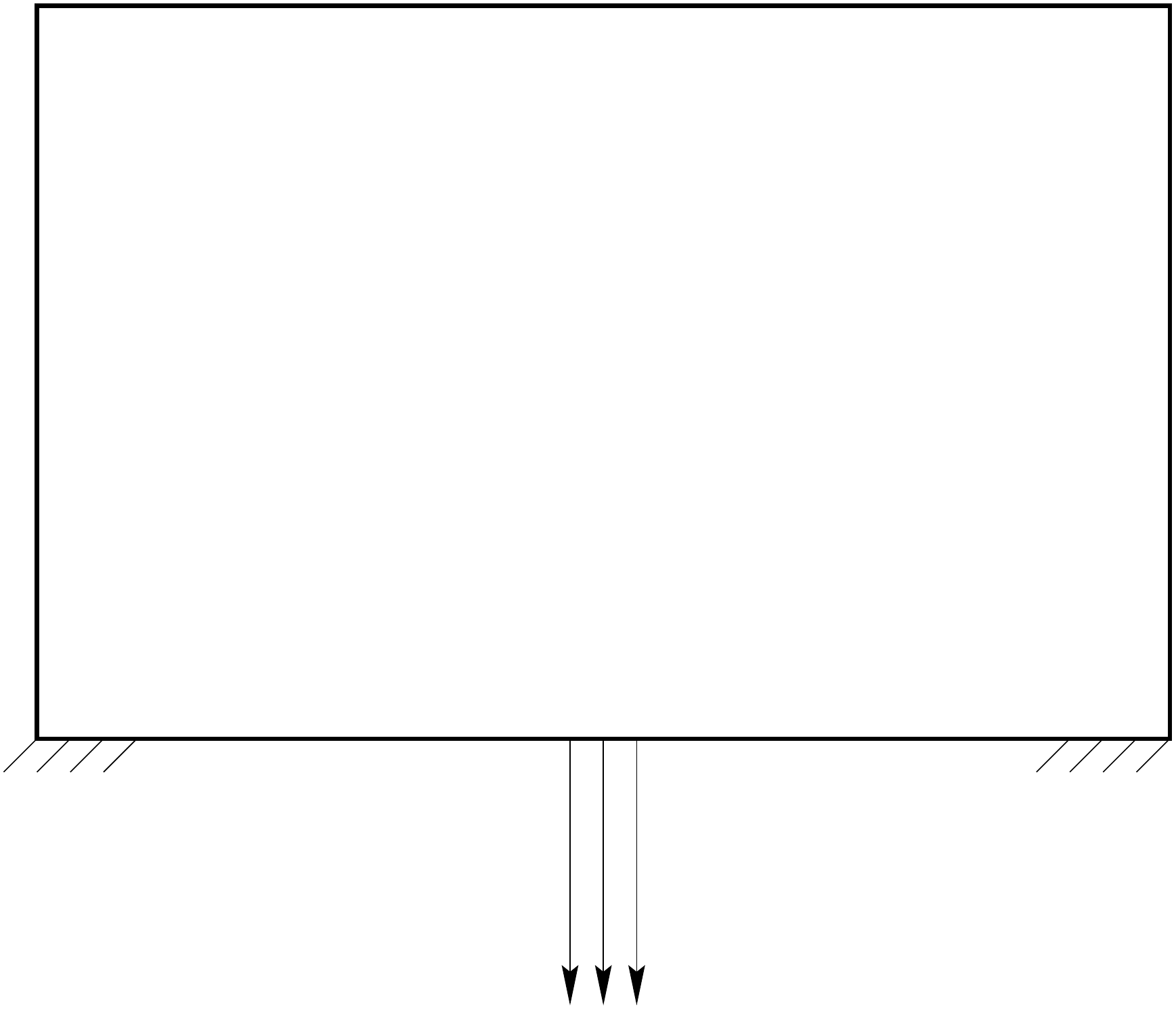}
    \end{minipage}
    \hfill
    \begin{minipage}[h]{.50\textwidth}
        \centering
        \includegraphics[width=\textwidth]{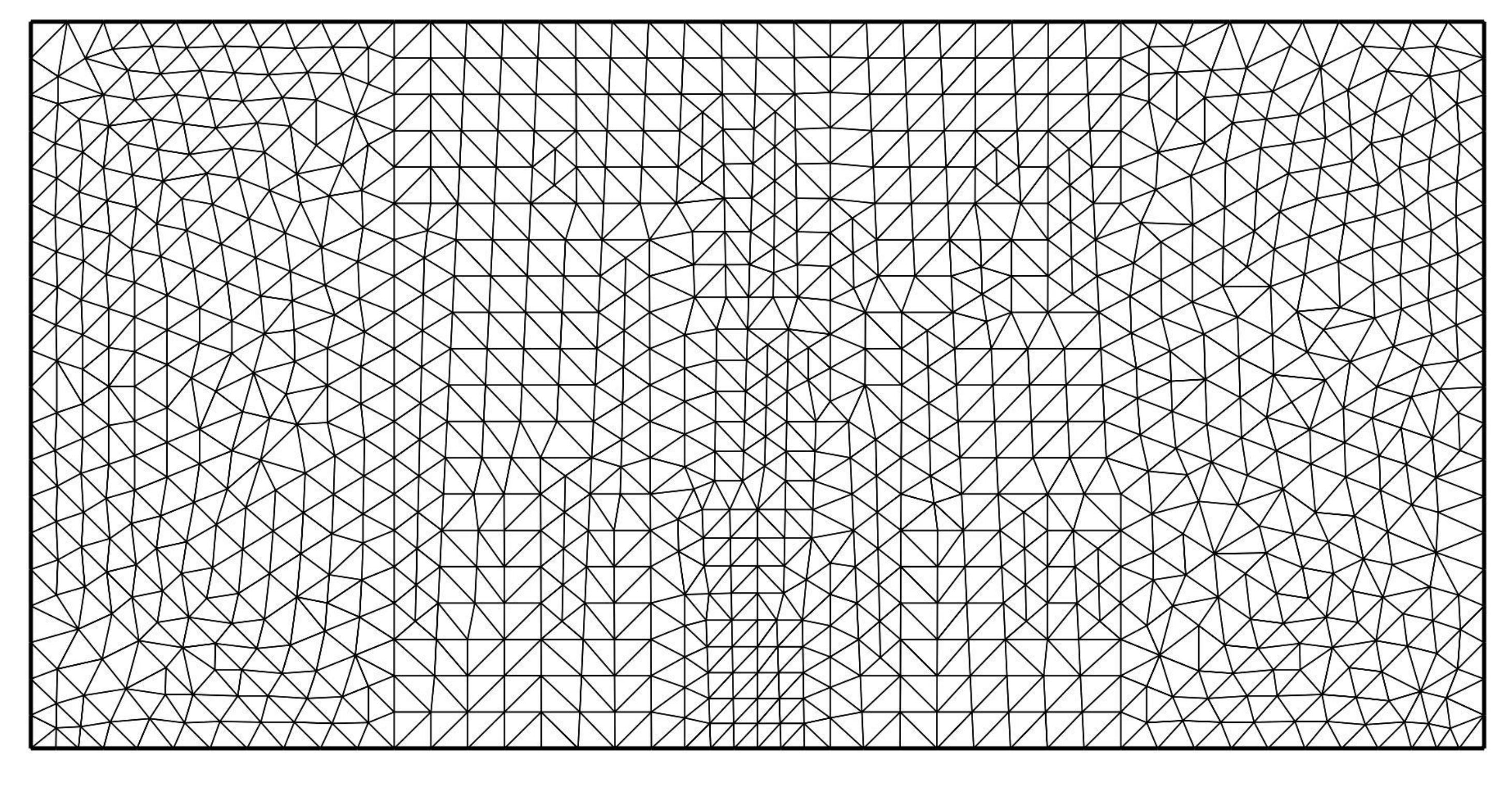}
    \end{minipage}
    \caption{Boundary conditions and mesh for an elastic plate.}
\label{Boundary conditions and mesh for an elastic plate}
\end{figure}

The compliance optimization reads
\begin{equation}\label{compliance optimizarion prob0}
	\inf_{h\in\mathcal{U}_{\rm ad}}\left\{J(h)=\int_\Omega f\cdot u\,dx+\int_{\Gamma_N} g\cdot u\,ds\right\}.
\end{equation}
The theoretical results are the same of previous sections.
We apply the optimality criteria method for the compliance optimization \eqref{compliance optimizarion prob0}. In order to compute \eqref{compliance optimizarion prob0}, we use FreeFem++. You can see its scripts on the web page \url{http://www.cmap.polytechnique.fr/~allaire/freefem_en.html}.

 \begin{figure}[H]
    \begin{minipage}[h]{.45\textwidth}
        \centering
        \includegraphics[width=\textwidth]{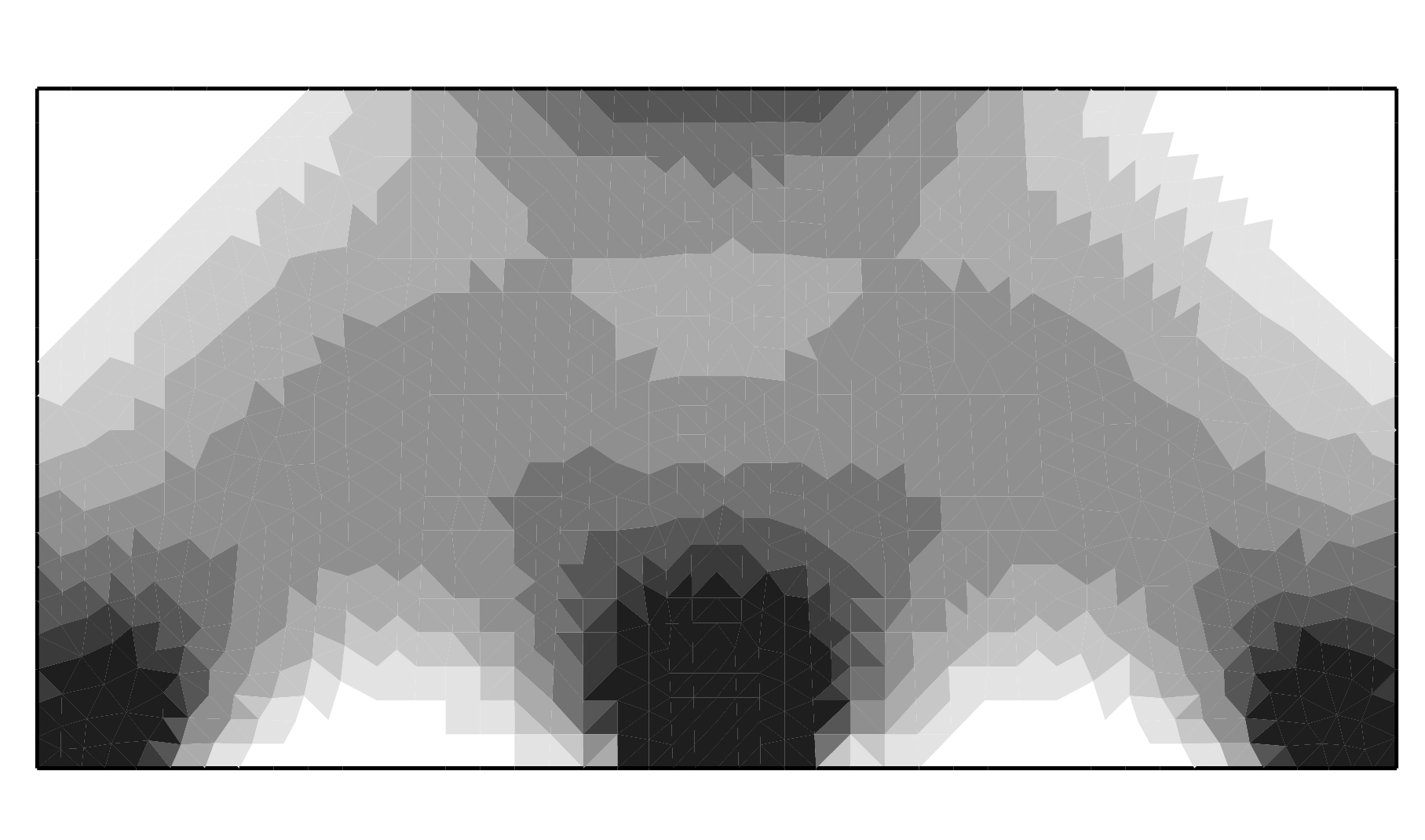}
    \end{minipage}
    \begin{minipage}[h]{.45\textwidth}
        \centering
        \includegraphics[width=\textwidth]{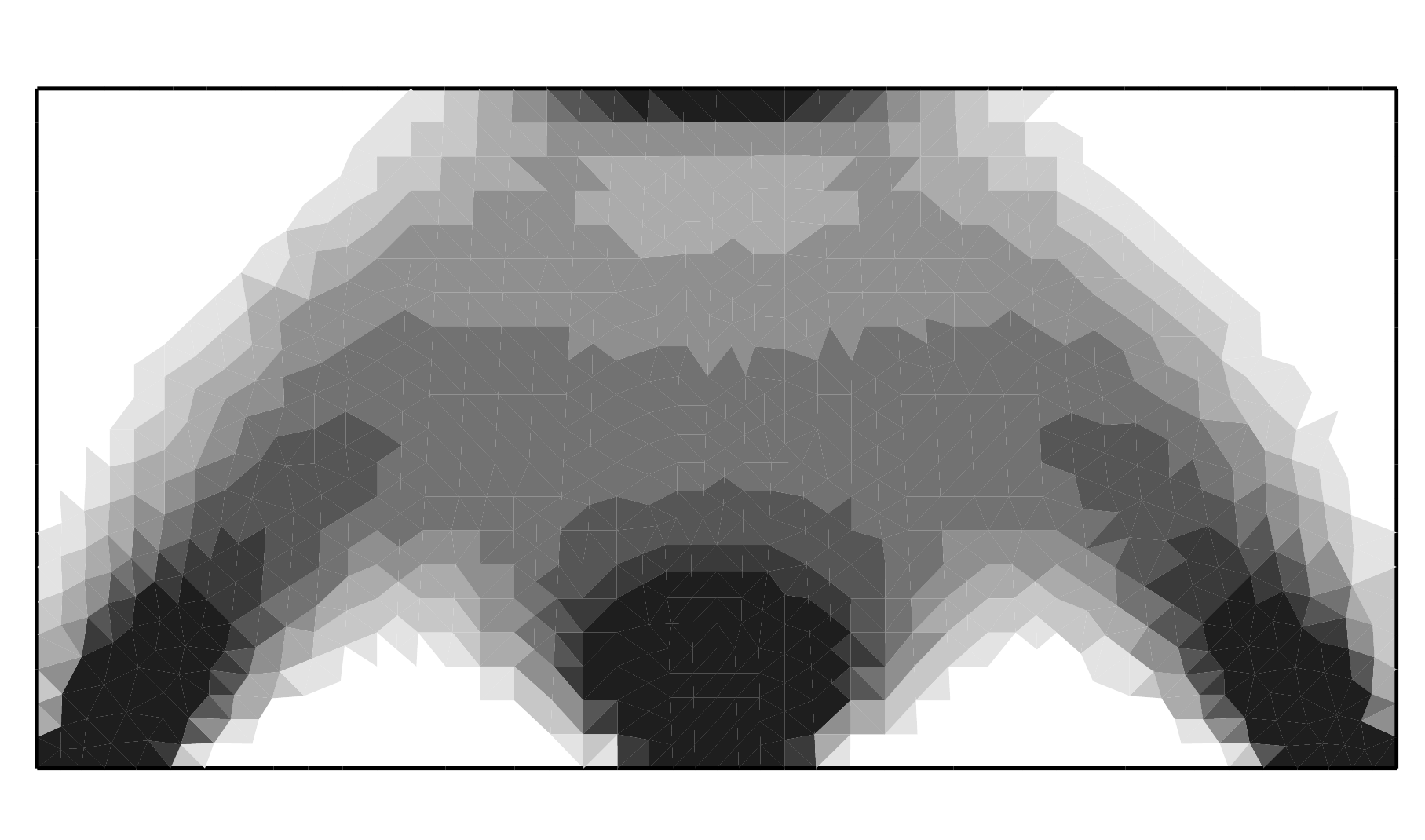}
    \end{minipage}
    
     \begin{minipage}[h]{.45\textwidth}
        \centering
        \includegraphics[width=\textwidth]{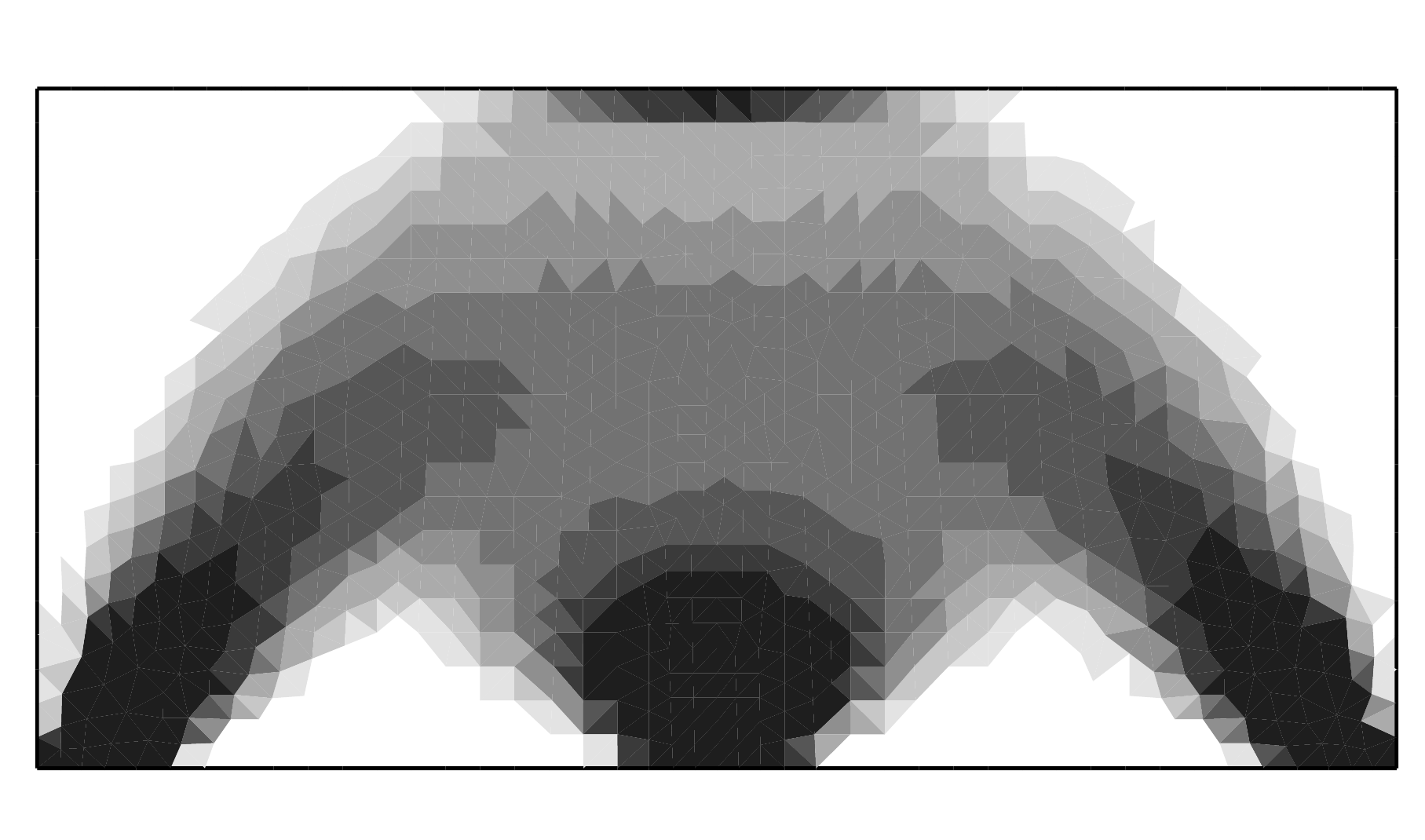}
    \end{minipage}
     \begin{minipage}[h]{.45\textwidth}
        \centering
        \includegraphics[width=\textwidth]{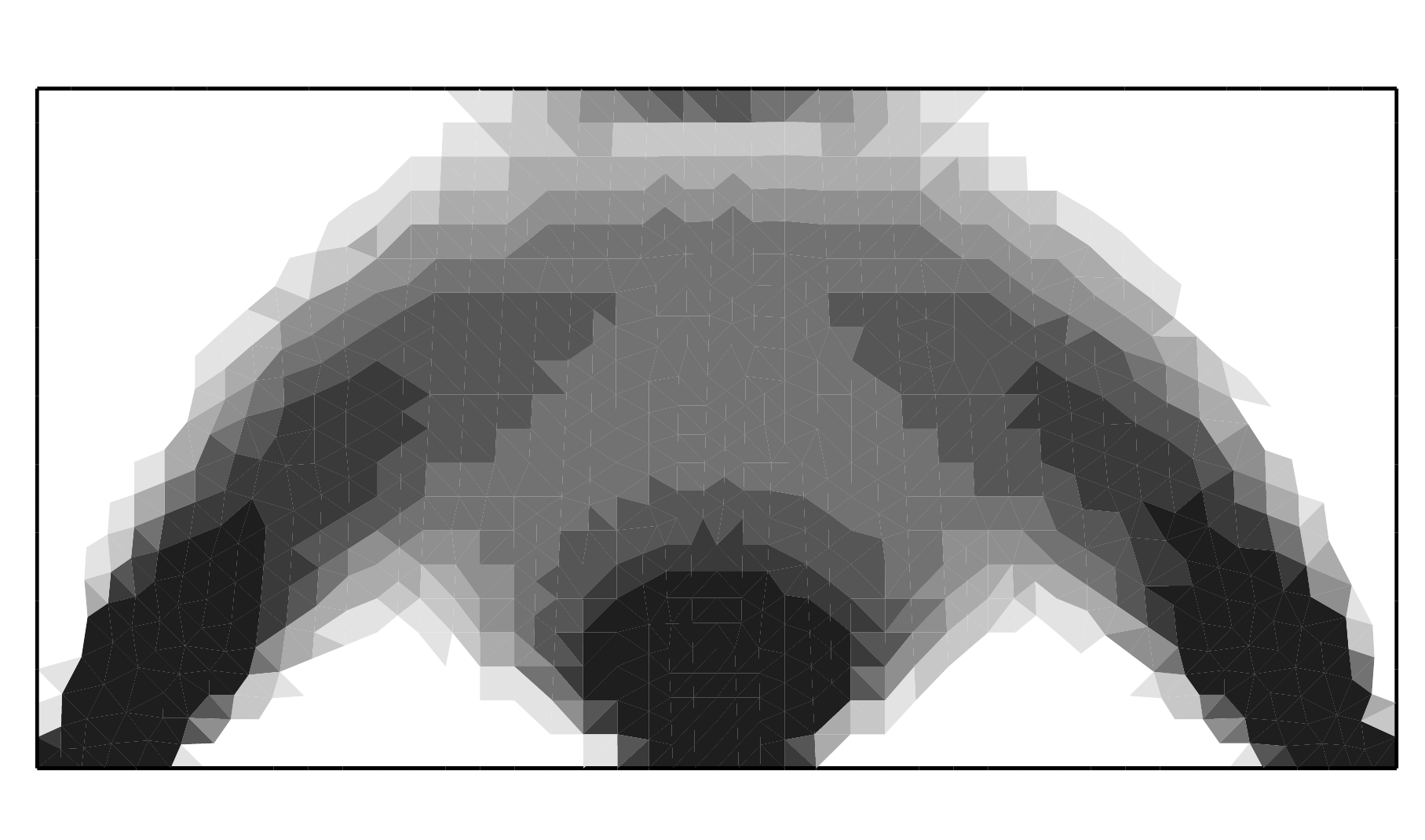}
    \end{minipage}
    \caption{Thickness at iterations 1, 5, 10, 30 (uniform initialization), where $h_{\min} = 0.1, h_{\max} = 1.0, h_0 = 0.5$ (increasing thickness from white to black). }
\label{Thickness at iterations 1, 5, 10, 30 (uniform initialization)}
\end{figure}

\begin{figure}[H]
    \begin{minipage}[h]{.45\textwidth}
        \centering
        \includegraphics[width=\textwidth]{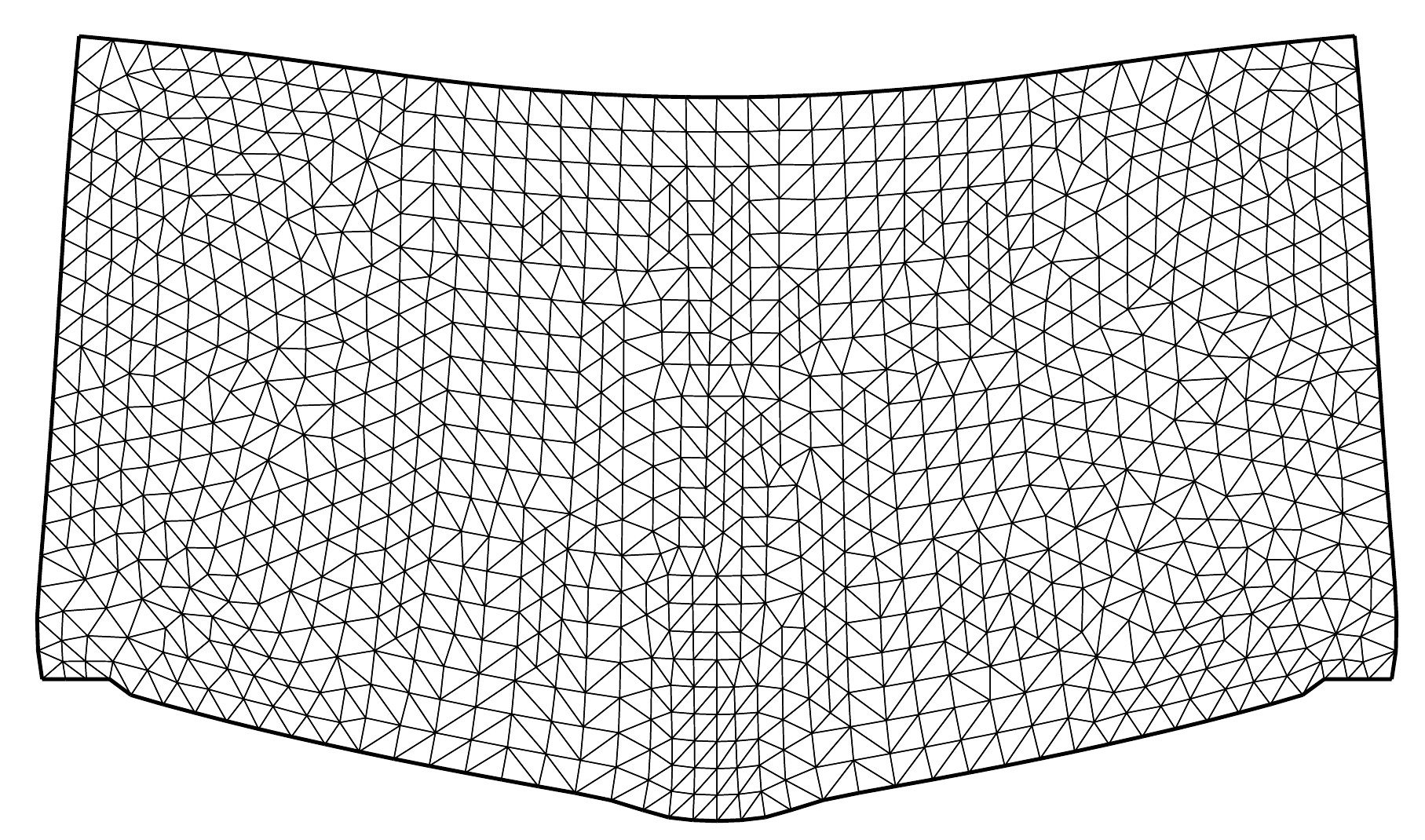}
    \end{minipage}
    \begin{minipage}[h]{.45\textwidth}
        \centering
        \includegraphics[width=\textwidth]{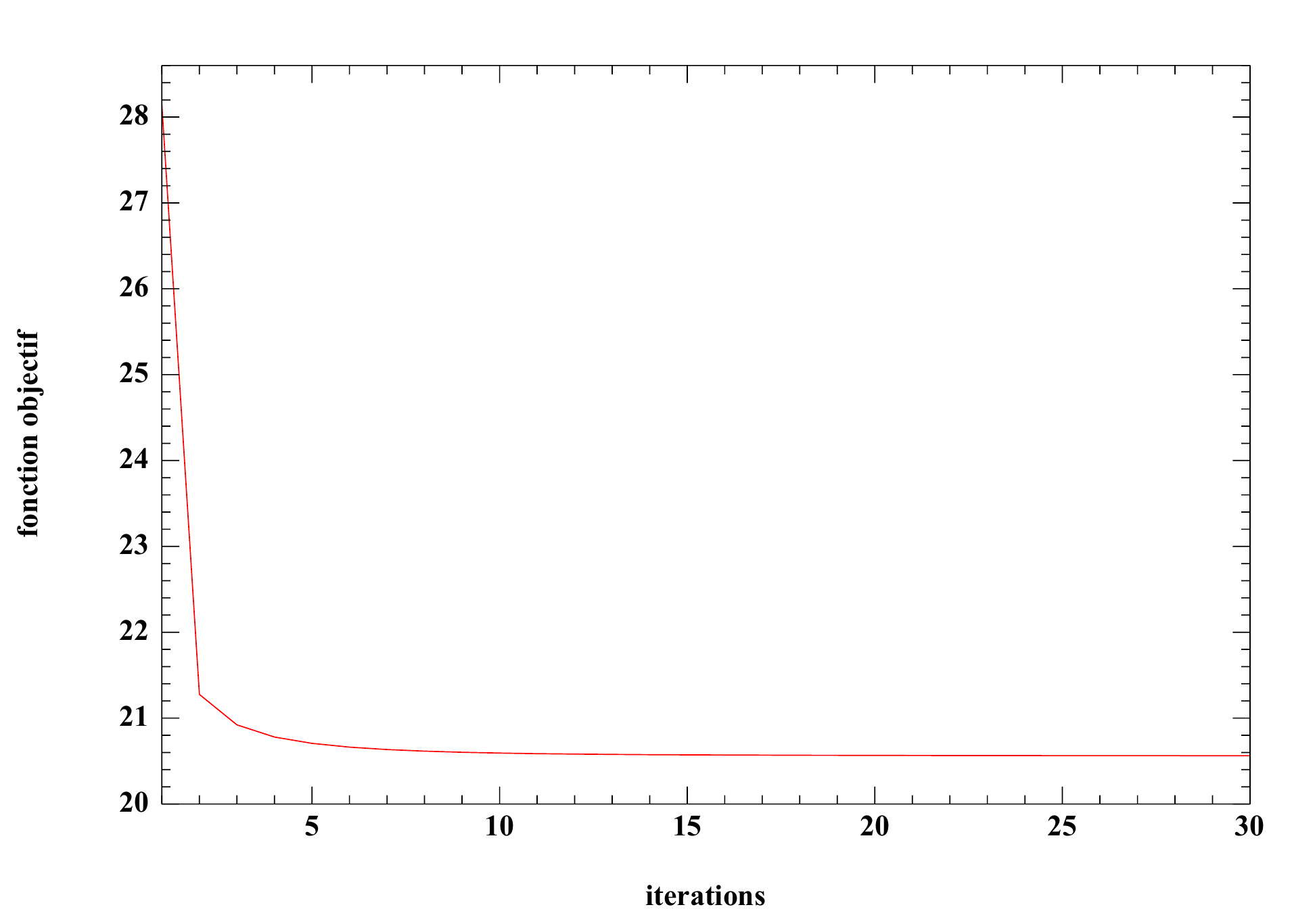}
    \end{minipage}
    \caption{Final deformed shapes and convergence history. }
\label{Final deformed shapes and convergence history}
\end{figure}
In Figure \ref{Thickness at iterations 1, 5, 10, 30 (uniform initialization)}, we used finite elements $P2$ for $u$ and $P0$ for $h$. However, numerical instabilities like checkerboards occur if we use finite elements $P1$ for $u$ and $P0$ for $h$ (see Figure \ref{Numerical instabilities (checkerboards)}). Therefore we consider a ``regularization" in order to avoid the instabilities. 

\subsection{Regularization}

In what follows, let us consider the ``regularized" framework to avoid numerical instabilities. The main idea, similar to that introduced in Remark \ref{change of scalar prod} is as follows. 
\begin{figure}[H]
    \begin{minipage}[h]{0.5\textwidth}
        \centering
        \includegraphics[width=\textwidth]{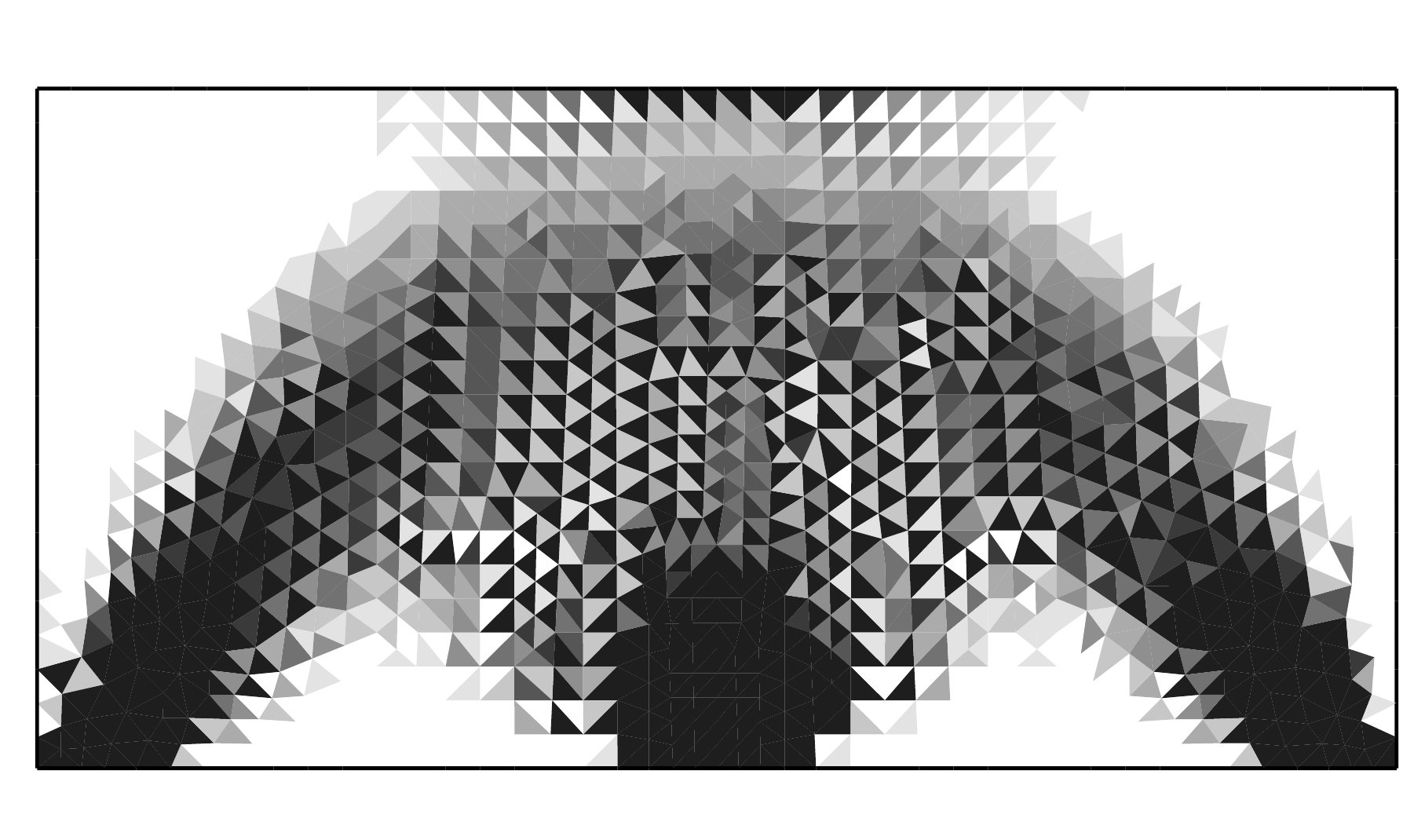}
    \end{minipage}
    \hfill
    \begin{minipage}[h]{0.5\textwidth}
        \centering
        \includegraphics[width=\textwidth]{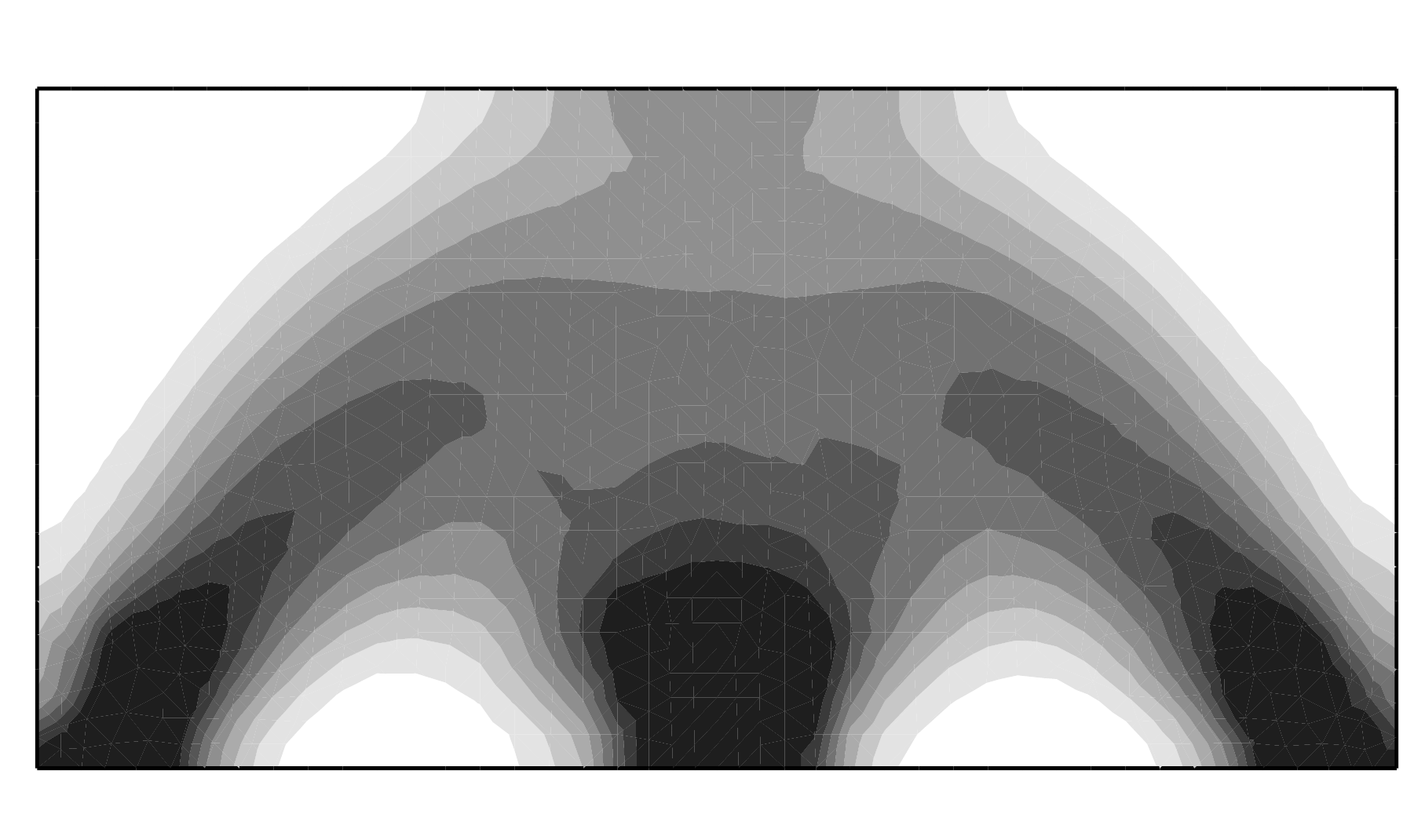}
    \end{minipage}
    \caption{Left: numerical instabilities (checkerboards), right: regularized optimal shape.}
    \label{Numerical instabilities (checkerboards)}
\end{figure}

We are going to replace the scalar product
\begin{equation*}
	\left\langle J'(h), k\right\rangle=\int_\Omega k\gr u\cdot\gr p\,dx, \quad k\in\mathcal{U}_{\rm ad}
\end{equation*}
with a different one.
Previously we identified $\mathcal{U}_{\rm ad}$ with a subspace of $L^2(\Omega)$, thus
\begin{equation*}
	\langle J'(h), k\rangle =\int_\Omega J'(h)k\,dx\Longrightarrow J'(h)=\gr u\cdot\gr p.
\end{equation*}
Now, we identify a ``regularized'' admissible set $\mathcal{U}^{\rm reg}_{\rm ad}$ to a subspace $H^1(\Omega)$, thus 
\begin{equation*}
	\langle J'(h), k\rangle=\int_\Omega \left( \e^2 \, \gr J'(h)\cdot\gr k+J'(h)k \right) \, dx,
\end{equation*}
where $\e > 0$ is a regularization parameter (which can be interpreted as a length scale). Therefore, we deduce a new formula for the gradient
\begin{equation}\label{regularized framework eq}
	\left\{
    \begin{aligned}
    	 -\e^2 \Delta J'(h)+J'(h) & = \gr u\cdot\gr p &&\text{ in } \Omega,\\
     	\dfrac{\partial J'(h)}{\partial n} & =0 &&\text{ on } \partial\Omega.
     \end{aligned}
   \right.
\end{equation}
Solving \eqref{regularized framework eq} and using a gradient algorithm such as projected gradient method, we obtain regularized optimal shape (see Figure \ref{Numerical instabilities (checkerboards)}).  

\section{Exercises}
\begin{problem}
Check the numerical instabilities (Figure \ref{Numerical instabilities (checkerboards)}, left) by using FreeFem++. 
\end{problem}
\begin{problem}
Solve \eqref{regularized framework eq} and see the regularized optimal shape (Figure \ref{Numerical instabilities (checkerboards)}, right) by using FreeFem++. 
\end{problem}


\chapter{Homogenization theory}\label{Homogenization theory}

In this section, we explain the homogenization method in order to apply shape optimization problems in Section \ref{Topo opti homo method}. 
Homogenization method is one of the averaging methods for partial differential equations. It is often concerned with the derivation of (macroscopic) equations whose solutions are defined as limits of solutions to (microscopic) equations with rapidly varying coefficients.
A particular case of homogenization is obtained when the coefficients of the partial differential equation are periodically and rapidly oscillating. 
Indeed, in many fields of science and technology one has to solve boundary value problems in periodic media.
In such a case, homogenization is simpler and can be achieved, at least formally, by using asymptotic expansions. 
This chapter is devoted to an elementary introduction of periodic
homogenization, without providing a fully rigorous justification.
Of course, homogenization methods using functional analysis method were considered for mathematical justification.
The interested reader is referred to the classical books \cite{BLP}, \cite{CD}, \cite{U}, \cite{JKO},  for further details. 

Note that, for applications in shape optimization, one should rely, in full rigor, on a more general homogenization method, called H-conver\-gence, introduced in \cite{MT} (see the textbook \cite{Allaire1} for more details). 
For simplicity, we restrict ourselves to the setting of periodic homogenization which is enough for a formal understanding. 


\begin{figure}[h]
\centering
\includegraphics[width=0.6\linewidth,center]{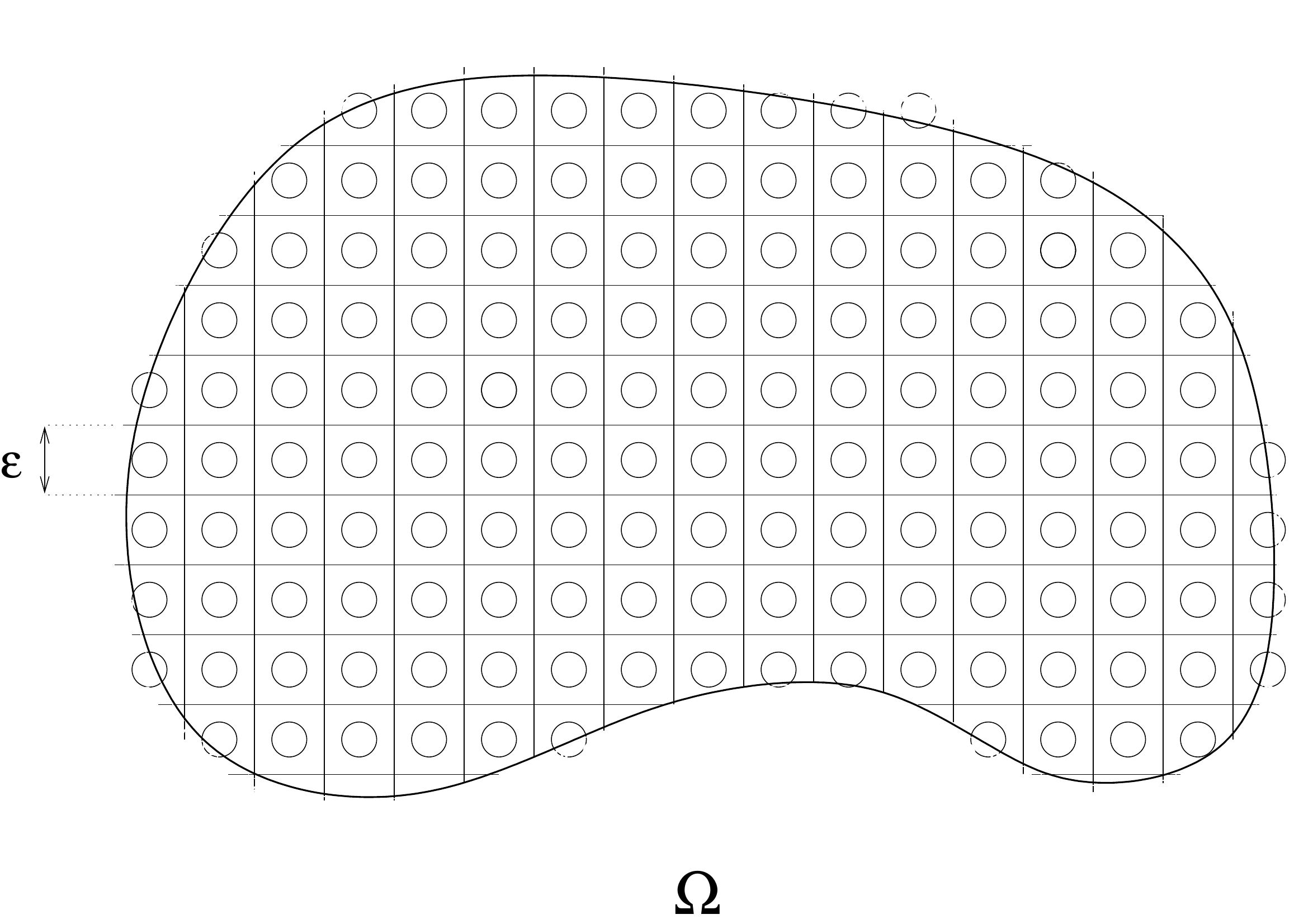}
\caption{A periodic heterogeneous medium.} 
\label{poreux}
\end{figure}

\section{Homogenization based on two-scale asymptotic expansions}


In what follows, we consider an elastic membrane made of a composite material with a fine periodic structure and apply the periodic homogenization method. We assume that ratio between the period and the characteristic size of the structure equals to $\e\ll 1$. We will find the ``true" problem by the limit problem obtained as $\e \to 0$. 

Let $\Omega$ be a bounded domain in $\rn$ ($N\ge 1$) and $f=f(x)$ be a load. Then we consider the displacement $u_{\e}$, which is defined as the solution of the following boundary value problem: 
\begin{equation}\label{micro eq}
\left\{
\begin{aligned}
-\dv\left(A\left(\frac{x}{\e}\right)\nabla u_{\e}\right)&=f &&\text{ in } \Omega, \\
u_{\e}&=0 &&\text{ on } \partial\Omega ,
\end{aligned}
\right.
\end{equation}
where the coefficient $A(y)$ satisfies the variable Hooke’s law, that is,  $A(y)$ is a Y-periodic function with $Y=(0,1)^N$. Thus for any i-th vector of the canonical basis $e_i$, the coefficient $A(y)$ satisfies 
\[
A(y+e_i)=A(y).
\]
If we replace $y$ by $x/\e$, then  we obtain that the map $x\mapsto A(x/\e)$
is a periodic of period $\e$ in all the coordinate directions $e_1,\cdots, e_N$. 
A direct computation of $u_{\e}$ can be very expensive (since the mesh size $h$ should satisfy $h\ll\e$), thus we seek only the averaged values of $u_{\e}$.
We assume that the solution $u_\e$ can be expanded as follows: 
\begin{equation}\label{expansion for sol}
\displaystyle
u_{\e}(x)=
\sum_{i=0}^{+\infty}
\e^i u_i \left(x,\frac{x}{\e}\right),
\end{equation}
with $u_i(x,y)$ function of the two variables $x$ and $y$, periodic in $y$, with periodicity cell given by $Y = (0, 1)^N$. 
Plugging the series \eqref{expansion for sol} in the equation \eqref{micro eq}, we use the derivation rule
\begin{equation}\label{derivation rule}
\nabla \left(u_i\left(x,\frac{x}{\e}\right)\right)=(\e^{-1}\nabla_y u_i+\nabla_x u_i)\left(x,\frac{x}{\e}\right). 
\end{equation}
Then we get
\begin{equation}\label{exp of du}
\nabla u_{\e}(x)=\e^{-1}\nabla_y u_{0}
\left(x,\frac{x}{\e}\right)+\sum_{i=0}^{\infty}
\e^i\left(\nabla_y u_{i+1}+\nabla_x u_i\right)\left(x,\frac{x}{\e}\right).
\end{equation}

\begin{figure}[h]
\centering
\includegraphics[width=0.7\linewidth,center]{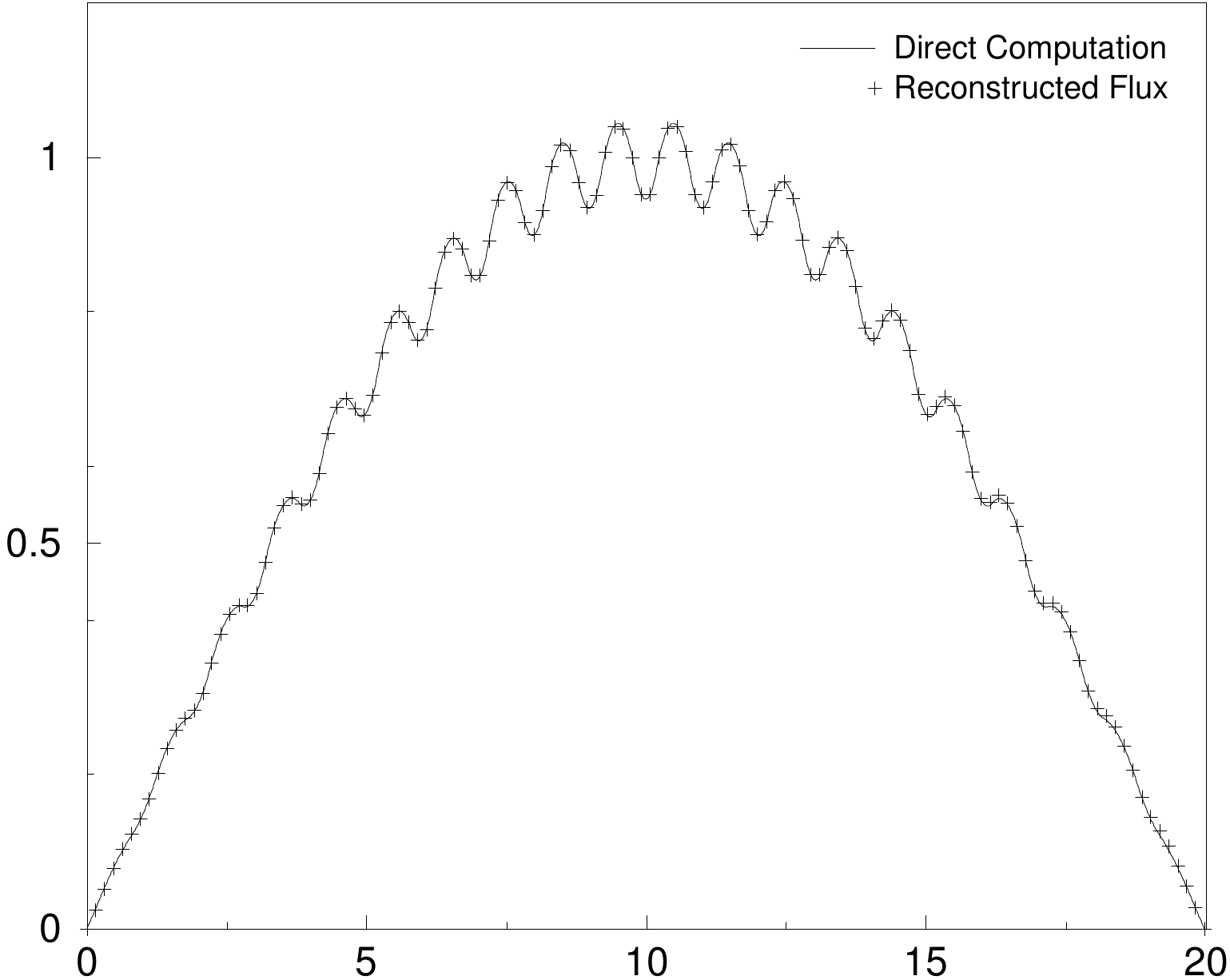}
\caption{Typical oscillating behavior of $x \mapsto u_i\left(x,\dfrac{x}{\e}\right)$.}
\label{wiggle}
\end{figure}

Substituting \eqref{exp of du} into \eqref{expansion for sol}, the equation becomes a series in $\e$
\begin{equation}\label{a series expansion}
\begin{aligned}
-\e^{-2}[\mathrm{div}_y(A(\nabla_yu_0))\left(x,\frac{x}{\e}\right) 
-\e^{-1}[\mathrm{div}_y(A(\nabla_x u_0+\nabla_yu_1))+\mathrm{div}_x(A\nabla_yu_0)]\left(x,\frac{x}{\e}\right) \\
-\displaystyle\sum_{i=0}^{+\infty}
\e^i[\dv_x(A(\nabla_x u_i+\nabla_yu_{i+1}))+\dv_y(A(\nabla_xu_{i+1}+\nabla_yu_{i+2}))]\left(x,\frac{x}{\e}\right)=f(x).
\end{aligned}
\end{equation}

In order to find the solution of the limit equation as $\e\to 0$, 
we identify each power of $\e$.
The most important terms are only the first three terms of the series.
We start by a technical lemma:

\begin{lemma}\label{fred}
Let $g\in L^2(Y)$ and suppose that $A(y)$ is a $Y$-periodic $N\times N$ matrices satisfying
\[
A(y)\xi\cdot\xi\ge \lambda|\xi|^2 \quad \forall\xi\in\RR^N, 
\]
for some $\lambda>0$.
Moreover, let 
$H^1_{\#}(Y)/\RR$ denote the quotient space, $H^1_{\#}(Y)$ up to an additive constant, equipped with the norm $\|\nabla\cdot\|_{L^2(Y)}$.
Then the problem 
\begin{equation*}
\left\{
\begin{aligned}
&-\dv_y (A(y)\nabla_y v(y))=g(y) &&\text{ in } Y, \\
&y\mapsto v(y)  && \ Y\text{-periodic}
\end{aligned}
\right.
\end{equation*}
admits a unique solution 
$v\in H^1_{\#}(Y)/\RR $
if and only if
\begin{equation*}
\displaystyle
\int_{Y}g(y) \, dy =0.
\end{equation*}
\end{lemma}
\begin{proof}
Let us check that $g$ being of zero mean over $Y$ is a necessary condition for existence. As a matter of fact, integrating the equation over $Y$, we get 
\begin{equation*}
\displaystyle
\int_{Y}\dv_y(A(y)\nabla_yv(y)) \, dy=\int_{\partial Y}A(y)\nabla_yv(y)\cdot n \, ds =0
\end{equation*}
because of the periodic boundary condition. Indeed 
$A(y)\nabla_y v(y)$ is periodic, but the normal
$n$ changes its sign on opposite faces of $Y$. 

The sufficient condition is obtained by applying Lax--Milgram theorem with respect to $H^1_{\#}(Y)/\RR$. 
Indeed, $\displaystyle a(u,v)=\int_{Y}A(y)\nabla u(y)\cdot\nabla v(y)\, dy$ is a coercive continuous bilinear form on $H^1_{\#}(Y)/\RR$ by uniform ellipticity.
 Furthermore, the map $F:H^1_{\#}(Y)/\RR\to \RR$, defined by
 $\displaystyle F(\phi)
= \int_{Y}g(y)\phi(y) \, dy, 
$
is a well defined bounded linear functional on $H^1_{\#}(Y)/\RR$ because  $g$ is a function of zero mean over $Y$.
Indeed, for all $\phi\in H^1_\#(Y)$, if we let $\displaystyle\bar{\phi}=\fint_Y \phi(y)\, dy$ denote the mean value of $\phi$ over $Y$, then, we get
\begin{equation*}
\begin{aligned}
\int_Y g(y)\phi(y) \, dy 
&= \int_Y g(y) \left( \phi(y)- \bar{\phi}\right) \, dy \\
&\le \norm{g}_{L^2(Y)} \norm{\phi-\bar{\phi}}_{L^2(Y)}\le \norm{g}_{L^2(Y)} \norm{\nabla \phi}_{L^2(Y)}, 
\end{aligned}
\end{equation*}
where we used the Poincar\'e--Wirtinger inequality in the last inequality.
This implies that the map $F:H^1_\#(Y)/\RR\to \RR $ defined above is bounded in the norm $\norm{\nabla \cdot}_{L^2(Y)}$ as claimed.
Hence, by Lax--Milgram's theorem, there exists a unique solution $v\in H_{\#}^1(Y)/\RR$ such that 
\[
\int_{Y}A(y)\nabla v(y)\cdot\nabla \phi(y) \, dy  
=
\int_{Y} g(y)\phi(y) \, dy \quad \forall \phi\in H_{\#}^1(Y)/\RR.
\]
\end{proof}

\begin{figure}[h]
\centering
\includegraphics[width=0.4\linewidth,center]{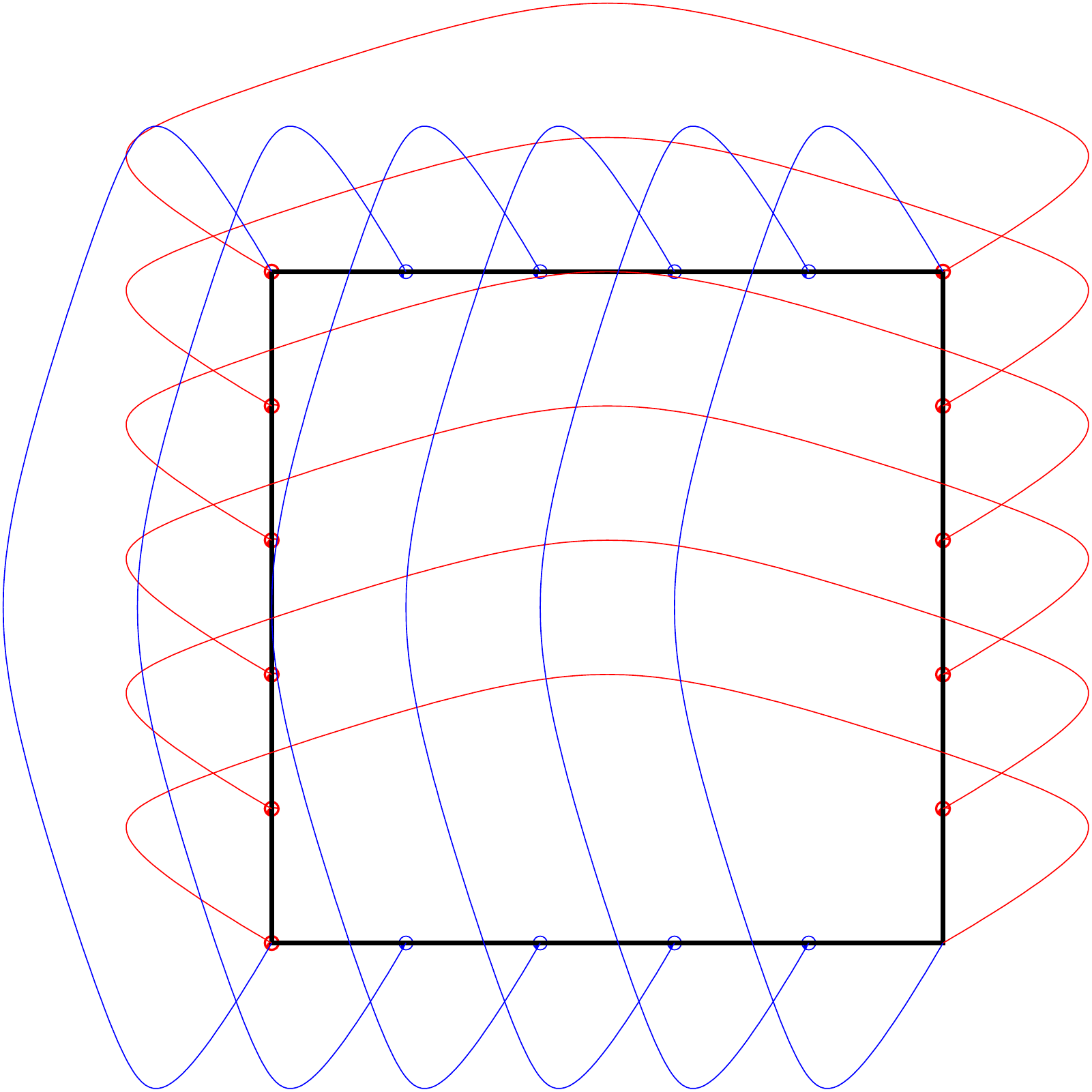}
\caption{Periodic boundary conditions in $H_{\#}^1(Y)$.} 
\label{Periodic boundary conditions}
\end{figure}
By using Lemma \ref{fred}, we can find the solution of the limit equation. 
Let us consider the equations that arise when we consider the first three terms of the series in \eqref{a series expansion}.
\begin{itemize}
\item[$\e^{-2}$:]
\begin{equation}\label{e-2}
\left\{
\begin{aligned}
&-\dv_y(A(y)\nabla_yu_0(x,y))=0 &&\text{ in } Y,\\
&y\mapsto u_0(x,y)&& \ Y\text{-periodic}.
\end{aligned}
\right. 
\end{equation}
It is a partial differential equation with respect to $y$ in $Y$ (here $x$ is just a parameter). By the uniqueness of the solution up to an additive constant, we deduce that 
\begin{equation}\label{u0==u}
u_0(x,y)\equiv u(x).
\end{equation}

\item[$\e^{-1}$:]
\begin{equation}\label{e-1}
\left\{
\begin{aligned}
&-\dv_y(a(y)\nabla_y u_1(x,y))=\dv_y(a(y)\nabla_xu_0(x,y))&&\text{ in } Y,\\
&y\mapsto u_1(x,y)&&\ Y\text{-periodic}.
\end{aligned}
\right.
\end{equation}
The necessary and sufficient condition of existence is satisfied. Thus, by \eqref{u0==u}, $u_1$ (seen as an element of $H^1_\#(Y)/\RR$) depends linearly on $\nabla_xu(x)$.
In particular, if we let $(e_i)_{1\le i\le N}$ denote the canonical basis of $\rn$, then it is easy to check that 
\begin{equation}\label{u1=sum}
u_1(x,y)=\displaystyle\sum_{i=1}^{N}\frac{\partial u}{\partial x_i}(x)w_i(y),
\end{equation}
where $w_i$ is the solutions of the following auxiliary problems (cell problems) for $i=1, \cdots, N$:
\begin{equation}\label{cell pb}
\left\{
\begin{aligned}
&-\dv_y(A(y)(\nabla_y w_i(y)+e_i))=0 &&\text{ in }Y,\\
&y\mapsto w_i(y)&& \ Y\text{-periodic}. 
\end{aligned}
\right.
\end{equation}
The functions $w_i$ are usually called the correctors. 

\item[$\e^{ 0}\ $:]
\begin{equation}\label{e-0}
\left\{
\begin{aligned}
&-\dv_y(A(y)\nabla_yu_2(x,y)) &&\!\!\!\!\!=f(x)+\dv_y(a(y)\nabla_xv_1) &&\\
& &&\!\!\!\!\!\quad +\dv_{x}(a(y)(\nabla_y v_1+\nabla_x u)) &&\text{ in }Y ,\\
&y\mapsto u_2(x,y) && && \ Y\text{-periodic}.
\end{aligned}
\right.
\end{equation}
By using Lemma \ref{fred}, the necessary and sufficient condition of existence of the solution $u_2$ is 
\begin{equation*}
\displaystyle\int_{Y}\left(\dv_y(A(y)\nabla_xu_1)+\dv_x(A(y)(\nabla_yu_1+\nabla_xu))+f(x) \right) dy=0. 
\end{equation*}
By employing the use of the representation formula \eqref{u1=sum}, we can rewrite $u_1$ in terms of $\nabla_xu(x)$:
\begin{equation*}
\displaystyle
\dv_x\int_{Y}A(y)\left(\sum_{i=1}^N\frac{\partial u}{\partial x_i}(x)\nabla_y w_i(y)+\nabla_x u(x)\right) dy+f(x)=0.
\end{equation*}
In other words, we have succeeded in identifying the the homogenized problem
\begin{equation}\label{hom eq}
\left\{
\begin{aligned}
-\dv_x(A^{\ast}\nabla_x u(x))&=f &&\text{ in }\Omega,\\
u&=0  &&\text{ on }\partial\Omega,
\end{aligned}
\right.
\end{equation}
where the homogenized tensor $A^*$ is defined by
\begin{equation}\label{A*=?}
\displaystyle A^{\ast}_{ji}=\int_{Y}A(y)(e_i+\nabla_y w_i)\cdot e_j \,dy, 
\end{equation}
or, integrating by parts
\begin{equation*}
\displaystyle A^{\ast}_{ji}=\int_{Y}A(y)(e_i+\nabla_yw_i(y))\cdot(e_j+\nabla_y w_j(y)) \,dy. 
\end{equation*}
Indeed, the cell problems \eqref{cell pb} yield
\begin{equation*}
\displaystyle\int_{Y}A(y)(e_i+\nabla_yw_i(y))\cdot\nabla_yw_j(y)\,dy=0.
\end{equation*}
\end{itemize}

\begin{remark}
The formula for $A^{\ast}$ is not fully explicit because cell problems \eqref{cell pb} must be solved. However $A^{\ast}$ does not depend on $\Omega$, nor $f$, nor the boundary conditions.
It only characterizes the microstructure. 
Later, we shall compute explicitly some examples of $A^{\ast}$.
\end{remark}

Under mild smoothness assumptions on the data, one can justify the expansion in $H^1(\Omega)$ \cite{BLP,JKO}. 
\begin{theorem}
Assume that the homogenized solution $u$ is smooth. Then the following expansion holds in $H^1(\Omega)$: 
\begin{equation*}
u_{\epsilon}(x)=u(x)+\e u_1\left(x,\frac{x}{\e}\right)+r_{\e}\ \text{with}\ \|r_{\e}\|_{\hi}\le C\e^{1/2}.
\end{equation*}
In particular
\begin{equation*}
\|u_{\e}-u\|_{L^2(\Omega)}\le C\e^{1/2}.
\end{equation*}
\end{theorem}

\begin{remark}[Rigorous justification]
 Employing a formal asymptotic expansion is a very useful method. 
 However we don't know a priori whether the solution of the microscopic equation can be expanded as \eqref{expansion for sol}. 
 We refer the interested reader to Tartar's method \cite{MT} and the two-scale convergence method\cite{Ng, Allaire6} for a rigorous mathematical justification. 
\end{remark}

\begin{remark}[Homogenized coefficients $A^{\ast}$]\label{a_hom}
In dimension $N=1$, the explicit formula for $A^{\ast}$ is the so-called harmonic mean.
In dimension $N\ge 2$, there is no explicit formula for $A^\ast$, which has to be computed
numerically.
Nevertheless, one can obtain explicit bounds on $A^{\ast}$.
\end{remark}

\begin{remark}
Homogenization works for non-periodic media too (H-convergence or G-convergence).
\end{remark}


\begin{remark}[Asymptotic expansions for the stress]
We assume that
\begin{equation*}
\displaystyle
u_{\e}(x)=\sum_{i=0}^{+\infty}\e^iu_i\left(x,\frac{x}{\e}\right),\ \sigma_{\e}(x)=A\left(\frac{x}{\e}\right)\nabla u_{\e}(x)=\sum_{i=0}^{+\infty}\e^{i}\sigma_i\left(x,\frac{x}{\e}\right), 
\end{equation*}
where $\sigma_i(x,y)$ is a function of the two variables $x$ and $y$, periodic in $y$ with period $Y=(0, 1)^N$. Plugging this series in the equation \eqref{micro eq}, we find
\begin{equation*}
-\dv_y\sigma_0 =0\ ,\ \dv_x\sigma_0-\dv_y\sigma_1=f.
\end{equation*}
On the other hand, 
\begin{equation*}
\sigma_0(x,y)=A(y)(\nabla_xu(x)+\nabla_yu_1(x,y))
\end{equation*}
and  
\begin{equation*}
\sigma_0(x,y)=A^{\ast}\nabla_xu(x)+\tau(x,y)\ \text{with}\ \int_{Y}\tau \,dy=0.
\end{equation*}
One can prove that $\tau$ is the solution of the dual cell problem. 
\end{remark}


\section{Composite materials}

Composite materials are ubiquitous in engineering, mechanics and physics and 
their effective properties can be understood through homogenization theory \cite{Allaire1,cherkaev,milton}.
In what follows, we identify a composite material by its homogenized tensor $A^{\ast}$. We restrict ourselves to two-phase composites.
We mix two isotropic constituents $A(y)=\alpha\chi(y)+\beta(1-\chi(y))$, where $\chi:Y\to \{0,1\}$ is a characteristic function. 
Let $\displaystyle\theta=\int_Y \chi(y)dy$ be the volume fraction of phase $\alpha$ and $(1-\theta)$ be that of phase $\beta$. 

We focus on the characterization of $G_{\theta}$ defined as follows: 
\begin{definition}[The set of all homogenized tensors $G_{\theta}$] \label{def Gtheta}
Let $G_{\theta}$ be the set of all homogenized tensors $A^{\ast}$ obtained
by homogenization of the two phases $\alpha$ and $\beta$ in proportions $\theta$ and $(1-\theta)$.
\end{definition}

\begin{remark}
Of course, we have $G_0=\{\beta {\rm Id}\}$ and $G_1=\{\alpha {\rm Id}\}$.
However, $G_{\theta}$ is usually a (very) large set of tensors (corresponding to different choices of $\chi(y)$).
\end{remark}

\subsection{Lamination for two phase composites}
For two phase composites, the density $\theta(x)$, as well as the homogenized tensor $A^{\ast}(x)$, depends on the position $x$. 
For two-phase mixtures, an explicit characterization of $G_{\theta}$ is possible by the variational principle of Hashin and Shtrikman \cite{HS}.
We make the following assumptions:
\begin{itemize}
\item[(i)]
Linear model of conduction or membrane stiffness (it is more delicate for linearized elasticity and very few results are known in the non-linear case).
\item[(ii)]
Perfect interfaces between the phases (continuity of both displacement and normal stress), no possible effects of delamination or debonding.
\end{itemize}
In dimension one, the cell problem \eqref{cell pb} reads:
\begin{equation*}
\left\{
\begin{aligned}
&-(A(y)(1+w'(y)))'=0\ &&\text{ in } [0,1),\\
&y\mapsto w(y)\ && \ 1\text{-periodic}.
\end{aligned}
\right.
\end{equation*}
The solution computed explicitly as follows:
\begin{equation*}
\displaystyle w(y)=-y+\int_{0}^{y}\frac{C_1}{A(t)}dt+ C_2\ \text{with}\ 
C_1=\left(\int_0^1\frac{1}{A(y)}dy\right)^{-1}.
\end{equation*}
By \eqref{A*=?}, we know that $\displaystyle A^{\ast}=\int_{0}^{1}A(y)(1+w'(y))^2\,dy$, 
which yields the harmonic mean of $A(y)$:
\begin{equation*}
A^{\ast}=\left(\int_0^1 \frac{1}{A(y)}dy\right)^{-1}.
\end{equation*}
Therefore, if we choose $A(y)=\alpha\chi(y)+\beta(1-\chi(y))$,
then homogenized tensor of any two-phase material is just
\begin{equation*}
A^{\ast}=\left(\frac{\theta}{\alpha}+\frac{1-\theta}{\beta}\right)^{-1}.
\end{equation*}
This formula tells us that, in one dimension, the homogenized tensor depends on the characteristic function $\chi$ by means of its volume fraction $\theta$ only. 
\begin{figure}[h]
\centering
\includegraphics[width=0.5\linewidth,center]{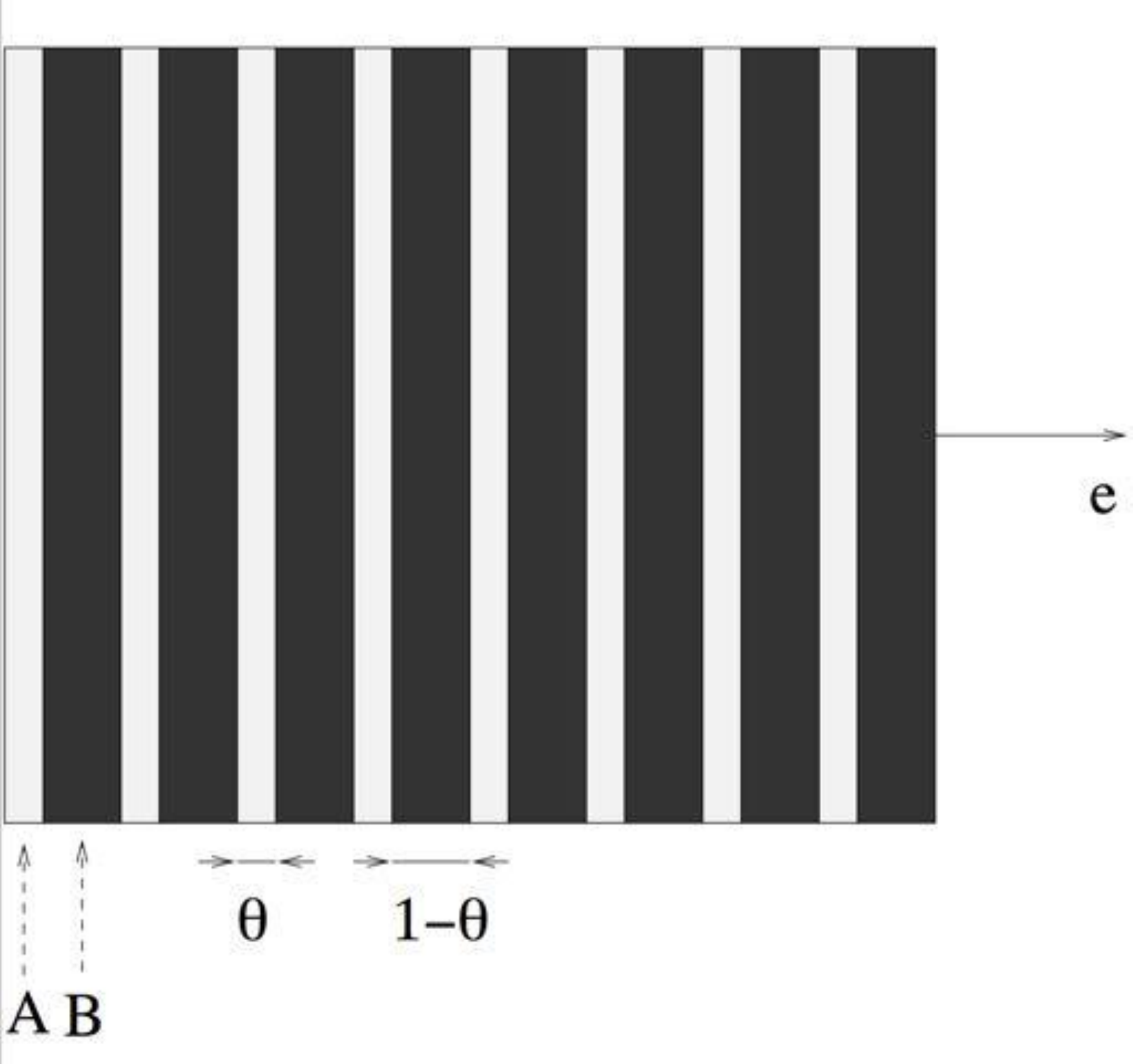}
\caption{Simple laminated composites.} 
\label{Simple laminated composites}
\end{figure}

In dimension $N \ge 2$, we cannot express $A^{\ast}$ explicitly in general
as mentioned in Remark \ref{a_hom}.
However 
it is possible under the following special case. 
We consider parallel layers of two isotropic phases $\alpha$ and $\beta$, orthogonal to the direction $e_1$.
Assume that $A^{\e}$ depends only on $y_1$. Let
\begin{equation*}
\chi(y_1)=
\left\{
\begin{aligned}
&1 &&\text{ if } 0<y_1<\theta \\
&0 &&\text{ if } \theta<y_1<1 \\
\end{aligned}
\right.
\ 
\text{ with }
\ 
\theta=\displaystyle\int_{Y}\chi\, dy.
\end{equation*}
We denote by $A^{\ast}$ the homogenized tensor of 
$A(y) = \left(\alpha\chi(y_1) + \beta(1-\chi(y_1))\right)I$.
Then we obtain the following lemma. This lemma is a simple case of the more general Lemma \ref{lem simple laminate}. 
\begin{lemma} \label{lem:1-laminate}
Define $\displaystyle\lambda_{\theta}^{-}=\left(\frac{\theta}{\alpha}+\frac{1-\theta}{\beta}\right)^{-1}$ and $\lambda_{\theta}^{+}=\theta\alpha+(1-\theta)\beta$. Then we have
\begin{equation}\label{layerling}
A^{*}=\displaystyle
\begin{pmatrix}
\lambda_{\theta}^{-}&&&&0 \\
&\lambda_{\theta}^{+}\\
&&\ddots&\\
0&&&&\lambda_{\theta}^{+}
\end{pmatrix}.
\end{equation}
\end{lemma}

\begin{remark}[Interpretation (resistance $=$ inverse of conductivity)]
In the context of electrical conductivity, the harmonic mean is the effective conductivity of a mixture of conductors placed in series (in the direction $e_1$), while the arithmetic mean is the effective conductivity of a mixture of conductors placed in parallel 
( in any direction orthogonal to $e_1$). 
\end{remark}

\begin{lemma}[Simple laminate of two non-isotropic phases]\label{lem simple laminate}
The homogenized tensor $A^{\ast}$ of a simple laminate made of $A$
and $B$ in proportions $\theta$ and $(1-\theta)$ in the direction $e_1$ is
\begin{equation}\label{slform}
A^{\ast}=\theta A+(1-\theta)B-\displaystyle\frac{\theta(1-\theta)(A-B)e_1\otimes(A-B)^{t}e_1}{(1-\theta)Ae_1\cdot e_1+\theta Be_1\cdot e_1}.
\end{equation}
Moreover, if we assume that $(A-B)$ is invertible, then this formula is equivalent to
\begin{equation}\label{seq laminate form}
\theta(A^{\ast}-B)^{-1}=(A-B)^{-1}+\frac{(1-\theta)}{Be_1\cdot e_1}e_1\otimes e_1.
\end{equation}
\end{lemma}

\begin{proof}
Recall that by definition \eqref{A*=?}
\begin{equation*}
\displaystyle A_{ji}^{\ast}=\int_Y A(y)(e_i+\nabla_y w_i)\cdot e_j \, dy
=
\int_{Y}A(y)(e_i+\nabla_yw_i(y))\cdot (e_j+\nabla_y w_j(y)) \, dy,
\end{equation*}
namely
\begin{equation*}
\displaystyle A^{\ast}e_i=\int_{Y}A(y)(e_i+\nabla_y w_i)\,dy.
\end{equation*}
Consequently, for any $\xi\in\RR^N$, we have
\begin{equation}\label{homo tensor with xi}
A^{\ast}\xi=\displaystyle\int_{Y}A(y)(\xi+\nabla_{y}w_{\xi})\,dy,
\end{equation}
where $\displaystyle w_{\xi}(y)=\sum_{i=1}^N\xi_iw_i(y)$ is the solution of 
\begin{equation*}
\left\{
\begin{aligned}
&-\dv_y \left( A(y)(\xi+\nabla w_{\xi}(y)) \right)=0\ &&\text{ in } Y,\\
&y\mapsto w_{\xi}(y)\ && \ Y\text{-periodic}.
\end{aligned}
\right.
\end{equation*}

Defining $u(y) = \xi\cdot y + w_{\xi}(y)$, we seek a solution $u$ such that the gradient of $u$ is constant in each phase, 
\begin{equation*}
\nabla u(y) = a\chi(y_1)+b\left( 1-\chi(y_1) \right). 
\end{equation*}
Thus, we have 
\begin{equation}\label{eq of u(y)}
u(y)=\chi(y_1) (c_a + a \cdot y) + (1-\chi(y_1))( c_b + b\cdot y), 
\end{equation}
where $c_a$ and $c_b$ are constant vectors. 

Let $\Gamma$ be the interface between the two phases. 
By continuity of \eqref{eq of u(y)} through the interface $\Gamma$, we have
\begin{equation}\label{const through gamma1}
c_a + a \cdot y = c_b + b\cdot y. 
\end{equation}
Since $c_a$ and $c_b$ are constant vectors, by \eqref{const through gamma1} we have 
\begin{equation*}
(a-b)\cdot x=(a-b)\cdot y\quad \forall x,y\in\Gamma.
\end{equation*}
Since $(x-y)$ is orthogonal to $ e_1$, there exists a real number $t\in\RR$ such that $b-a=te_1$. 

Moreover, by continuity of the flux $A(y) \nabla u \cdot n$ through the interface $\Gamma$, we have 
\begin{equation}\label{const through gamma2}
Aa\cdot e_1=Bb\cdot e_1. 
\end{equation}
In particular, it implies -$\dv (A(y)\nabla u)=0$ in the weak sense. 

Since $b-a=te_1$, \eqref{const through gamma2} yields the following value for $t$:
\[
t=\frac{(A-B)a\cdot e_1}{Be_1\cdot e_1}.
\]
Since $w_{\xi}$ is periodic, it satisfies $\displaystyle\int_{Y}\nabla w_{\xi}\ dy=0$, thus by the definition of $u$ we have 
\begin{equation*}
\displaystyle
\int_{Y}\nabla u\ dy=\theta a+(1-\theta)b=\xi.
\end{equation*}
On the other hand, by \eqref{homo tensor with xi} and the definition of $u$ we have
\begin{equation*}
\displaystyle
A^{\ast}\xi=\int_{Y}A(y)(\xi+\nabla w_{\xi}) \, dy =\int_{Y}A(y)\nabla u \, dy =\theta Aa+(1-\theta)Bb.
\end{equation*}
Thus we obtain 
\begin{equation*}
A^{\ast}(\theta a+(1-\theta)b)=\theta Aa+(1-\theta)Bb.
\end{equation*}
Since $b=a+te_1$ with $\displaystyle t = \frac{(A-B)a\cdot e_1}{Be_1\cdot e_1}$, we find
\begin{equation*}
\displaystyle
a=\xi-(1-\theta)\frac{(A-B)\xi\cdot e_1}{(1-\theta)Ae_1\cdot e_1+\theta Be_1\cdot e_1}e_1.
\end{equation*}
Then, a simple computation gives
\begin{equation*}
\displaystyle
A^{\ast}\xi=\theta A\xi+(1-\theta)B\xi-\frac{\theta(1-\theta)(A-B)\xi\cdot e_1}{(1-\theta)Ae_1\cdot e_1+\theta Be_1\cdot e_1}(A-B)e_1.
\end{equation*}
The other formula is a consequence of the following fact: if $M$ is invertible, then
\begin{equation*}
(M+c(Me)\otimes(M^te))^{-1}=M^{-1}-\displaystyle\frac{c}{1+c(Me\cdot e)}e\otimes e, 
\end{equation*}
where $c \in \RR$ and $e$ is a unit vector in $\RR^N$ which determines the direction of the lamination. 
\end{proof}

The composite $A^{\ast}$ is said to be a single lamination in the direction $e_1$ of the two phases $A$ and $B$ in proportions $\theta$ and $(1-\theta)$ (see Figure \ref{Simple laminated composites}). By varying the proportion $\theta$ and the direction $e_1$, we obtain a whole family of composite materials. This family can still be enlarged by laminating again these simple laminates.
Then we laminate again the preceding composite with always the same phase $B$. 

\begin{figure}[h]
\centering
\includegraphics[width=0.6\linewidth,center]{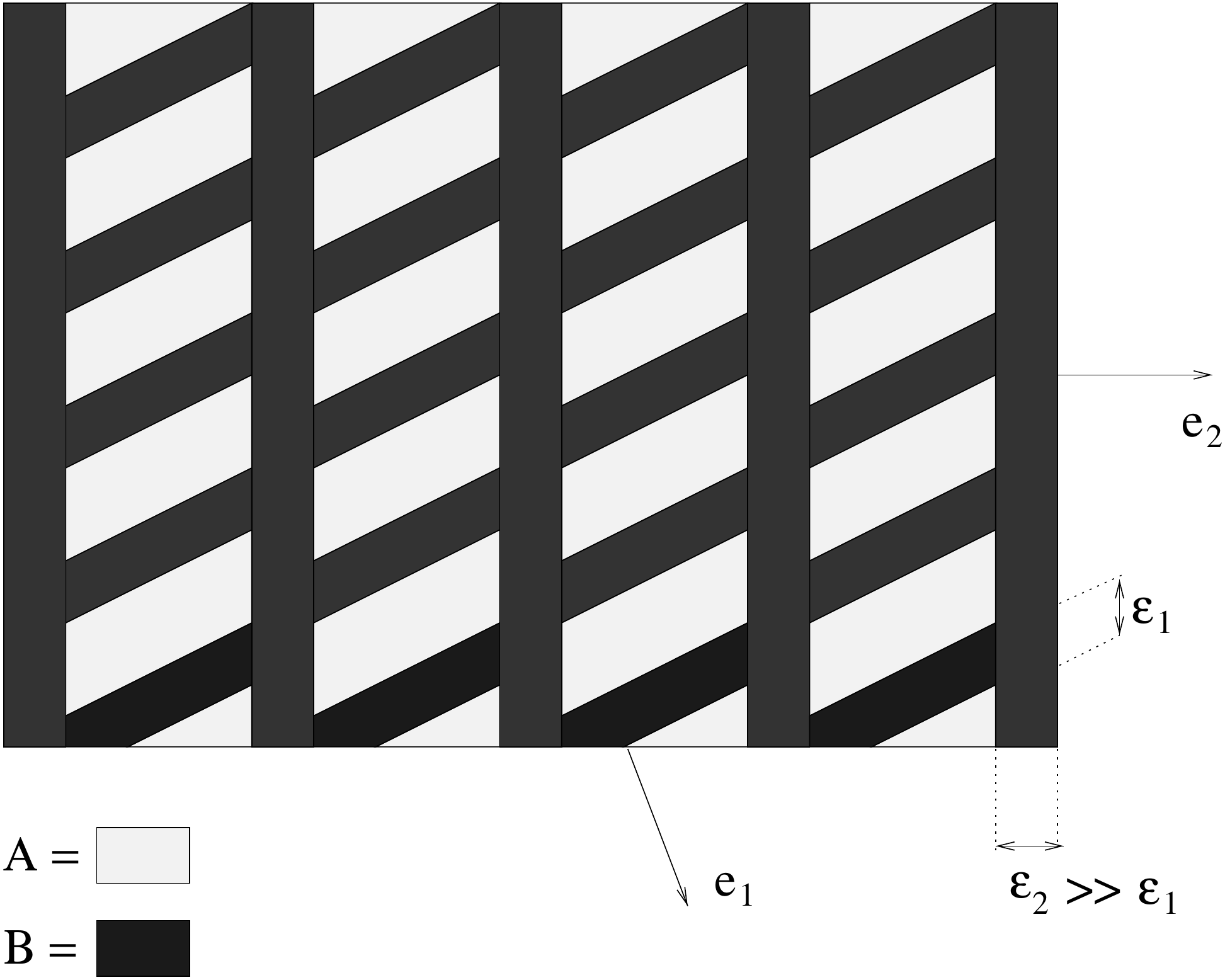}
\caption{A sequential laminate composite.} 
\label{lam2}
\end{figure}

A sequential laminate is obtained by an iterative process of lamination where the previous laminate is laminated again with a single pure phase (always the same one). By using the special form of \eqref{seq laminate form} (which does not deliver directly the value of $A^{\ast}$, contrary to \eqref{slform}), the iterative or sequential laminate can be explicitly characterized. 
Let $(e_{i})_{1\le i\le p}$ be a collection of unit vectors and $(\theta_{i})_{1\le i\le p}$ be proportions in $[0,1]$. By \eqref{seq laminate form} a simple laminate $A_{1}^{\ast}$ of $A$ and $B$ in proportions $\theta$, $(1-\theta)$ is 
\begin{equation*}
\theta_1(A^{\ast}_1-B)^{-1}=(A-B)^{-1}+\frac{(1-\theta_1)}{Be_1\cdot e_1}e_1\otimes e_1.
\end{equation*}

This simple laminate $A_{1}^{\ast}$ can again be laminated with phase $B$, in direction $e_2$ and in proportions $\theta_2$, $(1-\theta_2)$ respectively, to obtain a new laminate denoted by $A_2^{\ast}$. By induction, we obtain $A_p^{\ast}$ by lamination of $A_{p-1}^{\ast}$ and $B$, in direction $e_p$ and in proportions $\theta_p$, $(1-\theta_p)$, respectively. Then the homogenized tensor $A_{p}^{\ast}$ is 
\begin{equation}\label{formula for Ap}
\theta_p(A^{\ast}_p-B)^{-1}=(A_{p-1}-B)^{-1}+\frac{(1-\theta_p)}{Be_p\cdot e_p}e_p\otimes e_p.
\end{equation}
Replacing $(A_{p-1}^{\ast}-B)^{-1}$ in \eqref{formula for Ap} by the similar formula defining $(A_{p-2}^{\ast}-B)^{-1}$, and so on up to $A_{0}^{\ast}\equiv A$, we obtain a formula of the same type as \eqref{seq laminate form}, namely,
\begin{equation}\label{rank-p form}
\displaystyle
\left(
\prod_{j=1}^{p}\theta_{j}\right)
(A^{\ast}_p-B)^{-1}=(A-B)^{-1}+\sum_{i=1}^p\left(
(1-\theta_i)\prod_{j=1}^{i-1}\theta_j
\right)
\frac{e_i\otimes e_i}{Be_i\cdot e_i}.
\end{equation}

We remark that we always laminate an intermediate laminate with the same phase $B$. In other words, the other phase $A$ is coated by several layers of $B$. 
One can say that $B$ plays the role of a matrix phase, and $A$ plays the role of a core phase. Globally, $A^{\ast}$ can be seen as a mixture of $A$ and $B$ in different layers having a large separation of scales (see Figure \ref{lam2}).

Let us define rank-$p$ sequential laminate with matrix $B$ and inclusion $A$. 

\begin{lemma}[rank-$p$ sequential laminate]\label{lem seq lam}
If we laminate $p$ times with $B$, we obtain a rank-$p$ sequential laminate with matrix $B$ and inclusion $A$, in proportions $(1-\theta)$ and $\theta$, is defined by 
\begin{equation*}
\theta(A^{\ast}_p-B)^{-1}=(A-B)^{-1}+(1-\theta)\displaystyle\sum_{i=1}^{p}m_i\frac{e_i\otimes e_i}{Be_i\cdot e_i}
\end{equation*}
with $\displaystyle \sum_{i=1}^{p}m_i =1$ and $m_i\ge 0$, $1\le i\le p$.
\end{lemma}

\begin{proof}
By \eqref{rank-p form} we already have 
\begin{equation*}
\displaystyle
\left(
\prod_{j=1}^{p}\theta_{j}\right)
(A^{\ast}_p-B)^{-1}=(A-B)^{-1}+\sum_{i=1}^p\left(
(1-\theta_i)\prod_{j=1}^{i-1}\theta_j
\right)
\frac{e_i\otimes e_i}{Be_i\cdot e_i}. 
\end{equation*}
We make the change of variables
\begin{equation*}
\theta=\prod_{i=1}^p \theta_i \quad \text{and}\quad  (1-\theta)m_i=(1-\theta_i)\prod_{j=1}^{i-1}\theta_j, \ 1\le i\le p
\end{equation*}
which is indeed one-to-one with the constraints on the $m_{i}$'s and the $\theta_i$'s. 
\end{proof}

Of course the same can be done when exchanging the roles of $A$ and $B$. 
\begin{lemma}
A rank-$p$ sequential laminate with matrix $A$ and inclusion $B$, in proportions $\theta$ and $(1-\theta)$, is defined by
\begin{equation*}
(1-\theta)(A_{p}^{\ast}-A)^{-1}=(B-A)^{-1}+\theta\displaystyle\sum_{i=1}^{p}m_i\frac{e_i\otimes e_i}{A e_i\cdot e_i}
\end{equation*}
with $\displaystyle \sum_{i=1}^{p}m_i =1$ and $m_i\ge 0$, $1\le i\le p$.
\end{lemma}

\begin{remark}
Sequential laminates form a very rich and explicit class of composite materials which, as we shall see, completely describes the boundaries of the set $G_{\theta}$.
\end{remark}

\subsection{Characterization of $G_\theta$}
From now on, we assume that the microscopic tensor $A(y)$ is symmetric.
Then $A^{\ast}$ is also symmetric. Furthermore, $A^{\ast}$ is characterized by the following variational principle:
\begin{equation}\label{var pple}
\displaystyle
A^{\ast}\xi\cdot\xi=\min_{w\in H^1_{\#}(Y)/\RR}\int_{Y}A(y)(\xi+\nabla w)\cdot (\xi+\nabla w)\,dy \quad \forall\xi\in \rn.
\end{equation}
Indeed, if $w_{\xi}$ is a minimizer of \eqref{var pple}, then it satisfies the Euler optimality condition
\begin{equation*}
\left\{
\begin{aligned}
&-\dv\left( A(y)(\xi+\nabla w_{\xi}(y)) \right) = 0\ &&\text{ in } Y,\\
&y\mapsto w_{\xi}(y)\ && \ Y\text{-periodic}.
\end{aligned}
\right.
\end{equation*}
By linearity, we have $w_{\xi}=\sum_{i=1}^{N}\xi_{i}w_{i}$, where $w_i$ ($i=1,\cdots, N$) denotes the solution of \eqref{cell pb}, and thus, by \eqref{A*=?} we get
\begin{equation*}
\displaystyle
\int_{Y}A(y)(\xi+\nabla w_{\xi})\cdot (\xi+\nabla w_{\xi})\, dy =
\sum_{i,j=1}^{N}\xi_{i}\xi_{j}A_{ij}^{\ast}=A^{\ast}\xi\cdot\xi.
\end{equation*}

By using the variational principle of $A^{\ast}$ \eqref{var pple}, we can obtain arithmetric and harmonic mean bounds for $A^{\ast}$. 
\begin{lemma}[Arithmetic and harmonic mean bounds]\label{mean bounds}
Any homogenized tensor $A^{\ast}$ satisfies the arithmetic mean bound
\begin{equation*}
\displaystyle A^{\ast}\xi\cdot\xi\le\left(\int_{Y}A(y)\,dy\right)\xi\cdot\xi
\end{equation*}
and the harmonic mean bound
\begin{equation*}
\displaystyle
\left(\int_{Y}A^{-1}(y)\,dy\right)^{-1}\xi\cdot\xi\le A^{\ast}\xi\cdot\xi.
\end{equation*}
\end{lemma}

\begin{proof}
Taking $w=0$ in the variational principle \eqref{var pple}, we deduce the arithmetic mean bound. 
For the harmonic mean bound we enlarge the minimization space as follows. 
Indeed, since $\displaystyle\int_{Y}\nabla w \,dy = 0$, we replace $\nabla w$ with
any vector field $\zeta(y)$ with zero-average on $Y$
\begin{equation*}
\displaystyle A^{\ast}\xi\cdot\xi\ge\min_{\substack{\zeta\in L^2_{\#}(Y)^N,\\ \int_Y\zeta\, dy=0}}\int_{Y}A(y)(\xi+\zeta(y))\cdot(\xi+\zeta(y)) \, dy. 
\end{equation*}
The Euler equation for the minimizer $\zeta_{\xi}(y)$ of this convex problem is
\begin{equation*}
A(y)(\xi+\zeta_{\xi}(y))=\lambda, 
\end{equation*}
where $\lambda\in\RR$ is the Lagrange multiplier for the constraint $\displaystyle\int_{Y}\zeta \,dy=0$. Thus
\begin{equation*}
\displaystyle \xi=\left(\int_YA(y)^{-1}\,dy\right)\lambda
\end{equation*}
and
\begin{equation*}
\displaystyle\int_{Y}A(y)(\xi+\zeta_{\xi}(y))\cdot (\xi+\zeta_{\xi}(y))\, dy
=\left(\int_{Y}A(y)^{-1}\,dy\right)^{-1}\xi\cdot\xi.
\end{equation*}
\end{proof}
Lemma \ref{mean bounds} can be improved for two-phase composites.
Next, we consider two isotropic phases $A = \alpha {\rm Id}$ and 
$B=\beta {\rm Id}$ with $0 < \alpha < \beta$. 

\begin{theorem}[Hashin and Shtrikman bounds \cite{HS,TA}] \label{thm:Gtheta}
The set $G_{\theta}$ of all homogenized tensors obtained by mixing $\alpha$ and 
$\beta$ in proportions $\theta$ and $(1-\theta)$ is the set of all symmetric matrices $A^{\ast}$ with eigenvalues $\lambda_1,\cdots,\lambda_{N}$ such that
\begin{equation}\label{ineq for eigenvalues}
\displaystyle\left(
\frac{\theta}{\alpha}+\frac{1-\theta}{\beta}
\right)^{-1}=\lambda_{\theta}^{-}\le\lambda_{i}\le\lambda_{\theta}^{+}=\theta\alpha+(1-\theta)\beta,\ 1\le i\le N,
\end{equation}
\begin{equation}\label{lower bound for eigen}
\displaystyle\sum_{i=1}^N\frac{1}{\lambda_{i}-\alpha}\le\frac{1}{\lambda_{\theta}^{-}-\alpha}+\frac{N-1}{\lambda_{\theta}^{+}-\alpha},
\end{equation}
\begin{equation}\label{upper bound for eigen}
\sum_{i=1}^N\frac{1}{\beta-\lambda_{i}}\le\frac{1}{\beta-\lambda_{\theta}^{-}}+\frac{N-1}{\beta-\lambda_{\theta}^{+}}.
\end{equation}
Furthermore, these so-called Hashin and Shtrikman bounds are optimal and attained by rank-$N$ sequential laminates.
\end{theorem}


\begin{figure}[htbp]
    \begin{minipage}[t]{.45\textwidth}
        \centering
        \includegraphics[width=\textwidth]{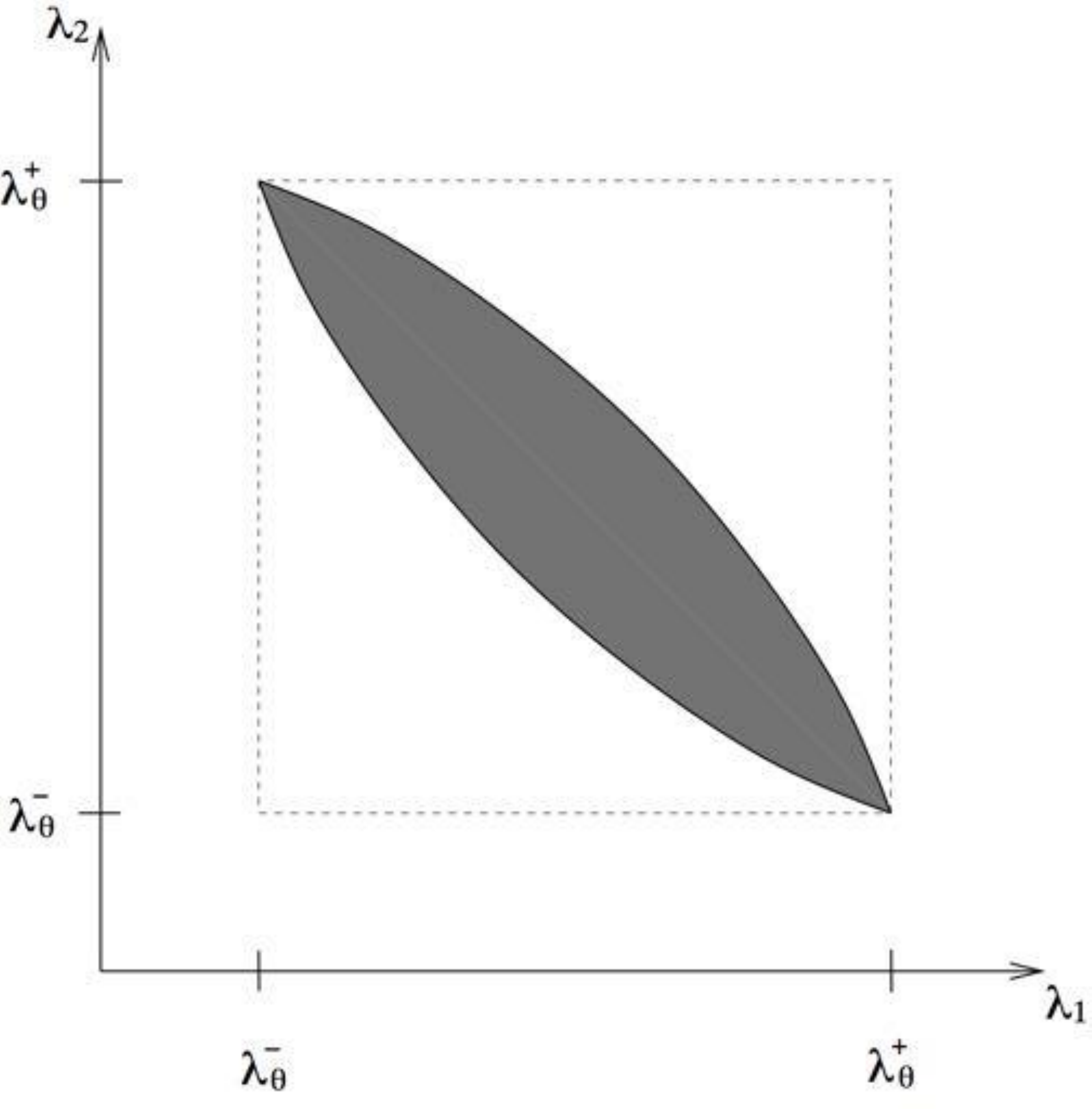}
    \end{minipage}
    \hfill
    \begin{minipage}[t]{.60\textwidth}
        \centering
        \includegraphics[width=\textwidth]{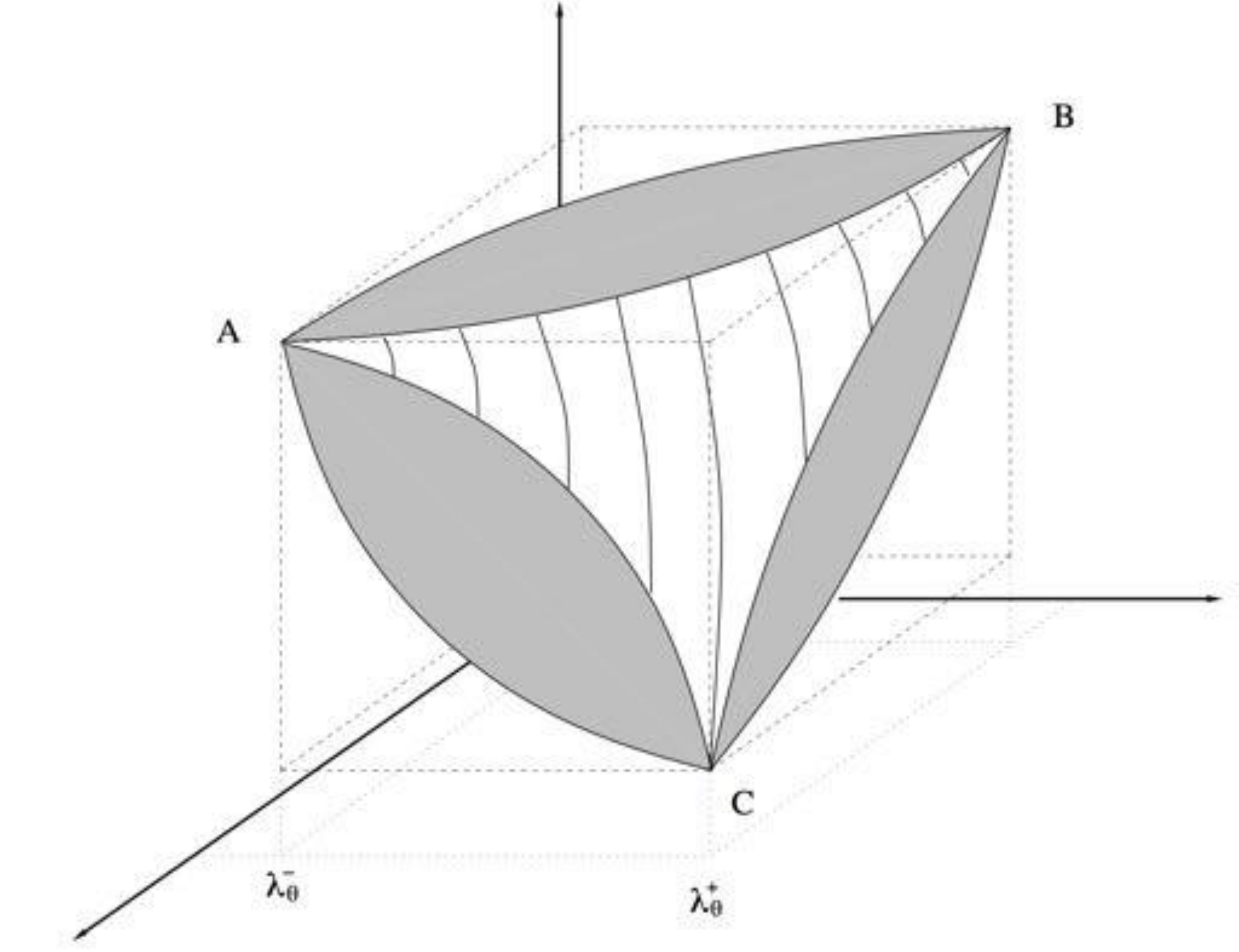}
    \end{minipage}
    \hfill
    \label{fig:1-2-3}
    \caption{The set $G_\theta$ in dimension $N=2$ and $3$ respectively.}
\end{figure}


\begin{proof}
We first show that all matrices satisfying these inequalities (Hashin-Shtrikman bounds) belong to $G_{\theta}$. 
Let us start by showing that the upper bound \eqref{upper bound for eigen} is attained by sequential laminates. Take a matrix $A^\ast$ such that
\begin{equation*}
\displaystyle\sum_{i=1}^{N}\frac{1}{\beta-\lambda_{i}}=\frac{1}{\beta-\lambda_{\theta}^{-}}+\frac{N-1}{\beta-\lambda_{\theta}^{+}}.
\end{equation*}
Define a rank-$N$ sequential laminate $A_{L}^{\ast}$ of matrix $\beta$ and inclusion $\alpha$, with lamination directions being the (orthogonal) eigenvectors of $A^{\ast}$. By Lemma \ref{lem seq lam} we have
\begin{equation*}
\displaystyle
\theta(A^{\ast}_{L}-\beta \mathrm{Id})^{-1}=\frac{1}{\alpha-\beta}\mathrm{Id} + (1-\theta)\sum_{i=1}^{N}m_i\frac{e_i\otimes e_i}{\beta} \ \text{with} \  m_i\ge 0, \,\, \sum_{i=1}^{N}m_i=1.
\end{equation*}
We obtain $A^{\ast} = A^{\ast}_L$ if we can choose the $m_i$'s such that
\begin{equation*}
\displaystyle
\frac{\theta}{\lambda_i-\beta}=\frac{1}{\alpha-\beta}+\frac{m_i(1-\theta)}{\beta}, 
\end{equation*}
that is, 
\begin{equation*}
m_i=\frac{\beta(\lambda_{\theta}^{+}-\lambda_i)}{(1-\theta)(\beta-\alpha)(\beta-\lambda_i)}. 
\end{equation*}
We check that $0<m_i<1$ is equivalent to $\lambda_{\theta}^-<\lambda <\lambda_{\theta}^+$ and that
\begin{equation*}
\displaystyle
\sum_{N=1}^N m_i=1\Longleftrightarrow \sum_{i=1}^N\frac{1}{\beta-\lambda_i}=\frac{1}{\beta-\lambda_{\theta}^-}+\frac{N-1}{\beta-\lambda_{\theta}^+}. 
\end{equation*}
Thus any matrix on the upper bound \eqref{upper bound for eigen} is a rank-$N$ sequential laminate with matrix $\beta$ and inclusion $\alpha$.
The same proof works for the lower bound \eqref{lower bound for eigen} upon exchanging the role of $\alpha$ (now the matrix) and $\beta$ (now the inclusion). 

A simple but lengthy computation shows that all the matrices satisfying the inequalities \eqref{ineq for eigenvalues}, \eqref{lower bound for eigen} and \eqref{upper bound for eigen} can be obtained as a rank-$N$ sequential laminate of two suitable matrices, one realizing the equality in the upper bound \eqref{upper bound for eigen} and the other realizing the equality in the lower bound \eqref{lower bound for eigen} (see the full proof of \cite[Theorem 2.2.13]{Allaire1} for the details). 
It remains to prove that the lower and upper Hashin--Shtrikman bounds hold true. To establish the lower bound \eqref{lower bound for eigen} we introduce the so-called Hashin and Shtrikman variational principle. 
Main idea is to use Fourier analysis and Plancherel theorem. 

By definition of $A^{\ast}$, for $\xi\in\RR^N$, we have 
\begin{equation*}
A^{\ast}\xi\cdot\xi=\displaystyle\min_{w\in H^1_{\#}(Y)}\int_{Y}(\chi(y)\alpha+(1-\chi(y))\beta)(\xi+\nabla w)\cdot (\xi+\nabla w)\, dy. 
\end{equation*}
Subtracting a reference material $\alpha$,
\begin{equation*}
\int_{Y}(\chi\alpha+(1-\chi)\beta)|\xi+\nabla w|^2\,dy =\int_{Y}(1-\chi)(\beta-\alpha)|\xi+\nabla w|^2\,dy+\int_{Y}\alpha|\xi+\nabla w|^2 \,dy.
\end{equation*}

We use convex duality (or Legendre transform): for any symmetric positive definite matrix $K$, the following holds
\begin{equation}\label{Legendre transform}
K\zeta\cdot\zeta=\displaystyle\max_{\eta\in\RR^N}(2\zeta\cdot\eta-K^{-1}\eta\cdot\eta) \quad\forall\zeta\in\RR^N.
\end{equation}
Since $0 < \alpha < \beta$, we apply the formula \eqref{Legendre transform} at each point in $Y$. Then we get 
\begin{equation*}
\begin{aligned}
&\int_{Y}(1-\chi)(\beta-\alpha)|\xi+\nabla w|^2 \, dy \\
&\quad= \displaystyle\max_{\eta\in L^2_{\#}(Y)^N} \int_{Y}(1-\chi) \left( 2(\xi+\nabla w)\cdot \eta-(\beta-\alpha)^{-1}|\eta|^2 \right) \, dy, 
\end{aligned}
\end{equation*}
which becomes an inequality if we restrict the minimization to constant $\eta$ in $Y$
\begin{equation*}
\begin{aligned}
\int_{Y}(1-\chi)(\beta-\alpha)|\xi+\nabla w|^2 \, dy
&\ge\displaystyle\max_{\eta\in\rn}\int_{Y}(1-\chi)(2(\xi+\nabla w)\cdot\eta-(\beta-\alpha)^{-1}|\eta|^2) \, dy\\
&\ge\left( 2\xi\cdot\eta-(\beta-\alpha)^{-1}|\eta|^2 \right) - 2\int_{Y}\chi\nabla w\cdot\eta \, dy.
\end{aligned}
\end{equation*}
On the other hand, because of periodicity, $\int_{Y}\nabla w\,dy=0$ which implies
\begin{equation*}
\int_{Y}\alpha|\xi+\nabla w|^2\,dy=\alpha|\xi|^2+\int_{Y}\alpha|\nabla w|^2\,dy.
\end{equation*}
Overall, we obtain that, for any $\eta\in\RR^N$,
\begin{equation}\label{shiki1}
A^{\ast}\xi\cdot\xi\ge\alpha|\xi|^2+(1-\theta)\left( 2\xi\cdot\eta-(\beta-\alpha)^{-1}|\eta|^2 \right) - g(\chi,\eta),
\end{equation}
where $g(\chi,\eta)$ is a so-called non-local term, defined by
\begin{equation*}
g(\chi,\eta)=\displaystyle\min_{w\in H^1_{\#}(Y)}\int_{Y}(\alpha|\nabla w|^2-2\chi\nabla w\cdot \eta)\,dy.
\end{equation*}

We can now use Fourier analysis to compute $g(\chi,\eta)$.
By periodicity, both $\chi$ and the test function $w$ can be written as Fourier series:
\begin{equation*}
\chi(y)=\displaystyle\sum_{k\in\ZZ^N}\hat{\chi}(k)e^{2i\pi k\cdot y}, \quad 
w(y)=\sum_{k\in\ZZ^{N}}\hat{w}(k)e^{2i\pi k\cdot y}.
\end{equation*}
Since $\chi$ and $w$ are real-valued, their Fourier coefficients satisfy
\begin{equation*}
\overline{\hat{\chi}(k)}=\hat{\chi}(-k) \,\, \text{and} \,\,  \overline{\hat{w}(k)}=\hat{w}(-k)\quad\forall k\in \ZZ^N.
\end{equation*}
The gradient of $w$ at $y\in Y$ is given by
\begin{equation*}
\nabla w(y)=\displaystyle\sum_{k\in\ZZ^N}2i\pi e^{2i\pi k\cdot y}\hat{w}(k)k.
\end{equation*}
Then, Plancherel formula yields
\begin{equation*}
\begin{aligned}
\int_{Y}(\alpha|\nabla w|^2-2\chi\nabla w\cdot\eta)\, dy
&=\sum_{k\in\ZZ^N}(4\pi^2\alpha|\hat{w}(k)k|^2-4i\pi\overline{\hat{\chi}(k)}\hat{w}(k)k\cdot\eta)\\
&=\sum_{k\in\ZZ^N}\left(4\pi^2\alpha|k|^2|\hat{w}(k)|^2+4\pi\mathcal{I}m \left( \overline{\hat{\chi}(k)}\hat{w}(k)\right)\eta\cdot k \right).
\end{aligned}
\end{equation*}
Notice that minimizing in $w(y)\in H^1_{\#}(Y)$ is equivalent to minimizing in $\hat{w}(k)\in\CC$.
For $k\neq 0$ the minimum is achieved by
\begin{equation*}
\hat{w}(k)= -\frac{i\hat{\chi}(k)}{2\pi\alpha|k|^2}\eta\cdot k,
\end{equation*}
and we deduce that 
\begin{equation}\label{shiki2}
g(\chi,\eta)=\displaystyle\left(
\alpha^{-1}\sum_{k\in\ZZ^N, k\neq 0}|\hat{\chi}(k)|^2\frac{k}{|k|}\otimes\frac{k}{|k|}
\right)\eta\cdot\eta=\alpha^{-1}\theta(1-\theta)M\eta\cdot\eta,
\end{equation}
where $M$ is a symmetric non-negative matrix defined by 
\begin{equation*}
M = \dfrac{1}{\theta(1-\theta)} \sum_{k\in\ZZ^N, k\neq 0}|\hat{\chi}(k)|^2\frac{k}{|k|}\otimes\frac{k}{|k|}. 
\end{equation*}
Since, by Plancherel theorem,
we have
\begin{equation*}
\displaystyle\sum_{k\in\ZZ^N, k\neq 0}|\hat{\chi}(k)|^2=\int_{Y}|\chi(y)-\theta|^2\,dy=\theta(1-\theta),
\end{equation*}
we deduce that the trace of $M$ is equal to $1$. 

Substituting \eqref{shiki2} to \eqref{shiki1}, for any $\xi, \eta \in \RR^N$, 
\begin{equation}\label{shiki3}
A^{\ast}\xi\cdot\xi\ge\alpha|\xi|^2+(1-\theta)(2\xi\cdot\eta-(\beta-\alpha)^{-1}|\eta|^2)-\alpha^{-1}\theta(1-\theta)M\eta\cdot\eta.
\end{equation}
The minimum (in $\xi$) of the inequality \eqref{shiki3} is obtained when
\begin{equation*}
\xi=(1-\theta)(A^{\ast}-\alpha)^{-1}\eta. 
\end{equation*}
Then we deduce
\begin{equation*}
(1-\theta)(A^{\ast}-\alpha)^{-1}\eta\cdot\eta
\le (\beta-\alpha)^{-1}|\eta|^2+\alpha^{-1}\theta M\eta\cdot\eta \quad  \forall\eta\in\RR^N. 
\end{equation*}
Thus, we have 
\begin{equation}\label{matrix inequality}
(1-\theta)(A^{\ast}-\alpha)^{-1}\le(\beta-\alpha)^{-1}I+\alpha^{-1}\theta M. 
\end{equation}
Taking the trace of this matrix inequality \eqref{matrix inequality}, and recalling that $\tr M = 1$, we obtain the lower Hashin--Shtrikman bound.
The proof of the upper bound is similar.
\end{proof}

\section{The elasticity setting} \label{sec:the elasticity setting}
In what follows, let us consider the elasticity setting. 
The homogenization method can be generalized to the elasticity setting. However, an explicit characterization of $G_{\theta}$
is still lacking in the elasticity setting. 

We set
\begin{equation} \label{eq:Hooke's law}
\begin{aligned}
&A\xi=2\mu_A\xi+\lambda_A(\mathrm{tr}\xi)I_2,\\
&B\xi=2\mu_B\xi+\lambda_B(\mathrm{tr}\xi)I_2,
\end{aligned}
\end{equation}
with the identity matrix $I_2$, and $\kappa_{A,B} = \lambda_{A,B} + 2\mu_{A,B}/N$. We assume $B$ to be
weaker than $A$:
\begin{equation*}
0 \leq \mu_B<\mu_A, \quad \quad 0\leq \kappa_B<\kappa_A.
\end{equation*}
We work with stresses rather than strains, thus we use inverse elasticity tensors. The similar results of the two-phase composites in the elasticity setting as follows (in details, see \cite[Section 2.3]{Allaire1}): 
\begin{lemma}[Sequential laminates in elasticity]
The Hooke's law of a simple laminate of $A$ and $B$, in proportions $\theta$ and $(1-\theta)$, respectively, in the direction $e$, is
\begin{equation*}
(1-\theta)(A^{\ast^{-1}}-A^{-1})^{-1}=(B^{-1}-A^{-1})^{-1}+\theta f_A^c(e),
\end{equation*}
where $f^c_A(e)$ is the tensor, defined, for any symmetric matrix $\xi$, by
\begin{equation*}
f^c_A(e_i)\xi\cdot\xi=A\xi\cdot\xi-\frac{1}{\mu_A}|A\xi e_i|^2+\frac{\mu_A+\lambda_A}{\mu_A(2\mu_A+\lambda_A)}((A\xi)e_i\cdot e_i)^2.
\end{equation*}
\end{lemma}

\begin{proposition}[Reiterated lamination formula] \label{prop:lamination formula}
A rank-$p$ sequential laminate with matrix $A$ and inclusions $B$, in proportions $\theta$ and $(1-\theta)$, respectively, in the directions
$(e_i)_{1\le i\le p}$ with parameter $(m_i)_{1\le i\le p}$ such that $0\le m_i\le 1$
and $\displaystyle\sum_{i=1}^pm_i=1$, is given by
\begin{equation*}
(1-\theta)(A^{\ast^{-1}}-A^{-1})^{-1}=(B^{-1}-A^{-1})^{-1}+\theta\displaystyle\sum_{i=1}^pm_{i}f^{c}_A(e_i). 
\end{equation*}
\end{proposition}


\begin{theorem}[Hashin--Shtrikman bounds in elasticity]\label{HS bounds elasticity}
Let $A^{\ast}$ be a homogenized elasticity tensor in $G_{\theta}$ which is assumed isotropic
\begin{equation*}
A^{\ast}=2\mu_{\ast}I_4+\left(\kappa_{\ast}-\frac{2\mu_{\ast}}{N}\right)I_2\otimes I_2.
\end{equation*}
Its bulk $\kappa_{\ast}$ and shear $\mu_{\ast}$ moduli satisfy
\begin{equation*}
\frac{1-\theta}{\kappa_A-\kappa_{\ast}}\le
\frac{1}{\kappa_A-\kappa_B}+\frac{\theta}{2\mu_A+\lambda_A}\ \text{ and }\ 
\frac{\theta}{\kappa_{\ast}-\kappa_{B}}\le \frac{1}{\kappa_A-\kappa_B}+\frac{1-\theta}{2\mu_{B}+\lambda_B},
\end{equation*}
\begin{equation*}
\frac{1-\theta}{2(\mu_A-\mu_{\ast})}\le\frac{1}{2(\mu_A-\mu_B)}+\frac{\theta(N-1)(\kappa_A+2\mu_A)}{(N^2+N-2)\mu_{A}(2\mu_A+\lambda_A)},
\end{equation*}
\begin{equation*}
\frac{\theta}{2(\mu_{\ast}-\mu_B)}\le\frac{1}{2(\mu_A-\mu_B)}-\frac{(1-\theta)(N-1)(\kappa_B+2\mu_B)}{(N^2+N-2)\mu_{B}(2\mu_B+\lambda_B)}.
\end{equation*}
Furthermore, the two lower bounds, as well as the two upper bounds are
simultaneously attained by a rank-$p$ sequential laminate with
$p=3$ if $N=2$, and $p=6$ if $N=3$.
\end{theorem}
\begin{proof}
We refer to \cite[Theorem 2.3.13]{Allaire1}
\end{proof}
\begin{remark}
These bounds do not characterize all possible isotropic homogenized tensors $A^{\ast}$ in $G_{\theta}$. In other words, there exist isotropic elasticity tensors with moduli satisfying these bounds that are not composite materials obtained by mixing phases $A$ and $B$ in proportions $\theta$, $(1-\theta)$, respectively. 
\end{remark}

\begin{proposition}[Hashin--Shtrikman optimal energy bound] \label{prop:rank-N laminate}
Let $G_{\theta}$ be the set of all homogenized elasticity tensors obtained by
mixing the two phases $A$ and $B$ in proportions $\theta$ and $(1-\theta)$. 
Let $L_{\theta}$ be the subset of $G_{\theta}$ made of sequential laminated composites. 
For any stress $\sigma$,
\begin{equation*}
HS(\sigma)=\displaystyle\min_{A^{\ast}\in G_{\theta}}A^{\ast^{-1}}\sigma\cdot\sigma=\min_{A^{\ast}\in L_{\theta}}A^{\ast^{-1}}\sigma\cdot\sigma.
\end{equation*}
Furthermore, the minimum is attained by a rank-$N$ sequential laminate with lamination directions given by the eigendirections of $\sigma$.
\end{proposition}


\begin{remark}
An optimal tensor $A^{\ast}$ can be interpreted as the most rigid composite material in $G_{\theta}$ able to sustain the stress $\sigma$.
$HS(\sigma)$ is called Hashin--Shtrikman optimal energy bound. 
In practical conclusion, $G_{\theta}$ can be replaced by $L_{\theta}$ for compliance minimization.
\end{remark}

\section{Numerical applications}
Let us consider the case of parametrized periodicity cells.
For example, the square cell with a rectangular hole (as used in the seminal work of Bends\o e and Kikuchi \cite{BK}), parametrized by $m_1$, $m_2$, and denoted by $Y(m)$ (see Figure \ref{rectangular hole}). 
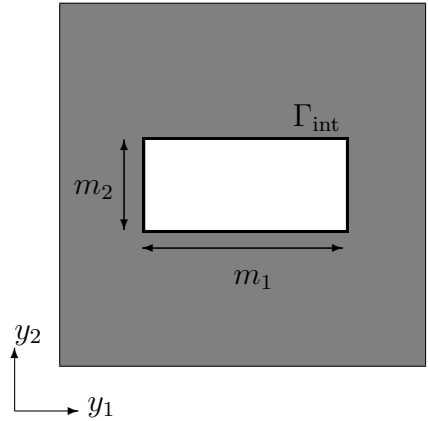
\begin{figure}[h]
\begin{center}
\setlength{\unitlength}{1.2cm}
	\begin{picture}(4.5,4.5)
	\setlength\fboxsep{0cm}
	\put(0.5,0.5){\colorbox{gray}{\framebox(4,4) {}}}
	\setlength{\fboxrule}{1.1pt}
	\put(1.4,2){\fcolorbox{black}{white}{\framebox(2.2,1) {}}}

	\put(1.2,2){\vector(0,1){1}}
	\put(1.2,3){\vector(0,-1){1}}
	\put(0.65,2.4){$m_2$}

	\put(1.4,1.8){\vector(1,0){2.2}}
	\put(3.6,1.8){\vector(-1,0){2.2}}
	\put(2.4,1.4){$m_1$}

	\put(3.05,3.15){$\Gamma_{\rm int}$}
	
	\put(0,0){\vector(1,0){0.7}}
	\put(0.8,0){$y_1$}
	\put(0,0){\vector(0,1){0.7}}
	\put(0,0.8) {$y_2$}
	\end{picture}
\end{center}
\caption{Square cell with a rectangular hole.}
\label{rectangular hole}
\end{figure}

We compute the so-called correctors or cell solutions:
\begin{equation*}
\left\{
\begin{aligned}
& \dv(A(e_{ij}+e(w_{ij})))&&\!\!\!\!\!=0 &&\text{ in } Y(m),\\
& \quad \!A(e_{ij}+e(w_{ij}))\cdot n&&\!\!\!\!\!=0 &&\text{ on } \Gamma_{\rm int},\\
&y\mapsto w_{ij}(y) && &&\ (0,1)^2\text{-periodic,}\\
\end{aligned}
\right.
\end{equation*}
where $e_{ij}=(e_{i}\otimes e_{j}+e_{j}\otimes e_{i})/2$ 
is a basis of the symmetric tensors of order $2$,
and $n$ is the normal to the hole’s boundary $\Gamma_{\rm int}$ in $Y(m)$.
Hence we find a unique solution (up to an additive translation) $w_{ij}\in H^1_{\#}(Y(m),\RR^2)$ to
the variational formulation:
\begin{equation} \label{eq:variational formulation wij}
\displaystyle
\int_{Y(m)}Ae(w_{ij})\cdot e(\phi)\,dy+\int_{Y(m)}Ae_{ij}\cdot e(\phi)\,dy=0 \quad \forall \phi\in H^1_{\#}(Y(m),\RR^2).
\end{equation}
 The tensor $A^{\ast}$ is then given by
\begin{equation} \label{eq:def tensor A*}
\displaystyle
A^{\ast}_{ijkl}=\int_{Y(m)}A(e_{ij}+e(w_{ij}))\cdot(e_{kl}+e(w_{kl})) \, dy, \quad i,j,k,l\in\{1,2\}.
\end{equation}


\section{Exercises}
\begin{problem}
Compute numerically $A^{\ast}$ for various parameters $m_1$, $m_2$.
\end{problem}

\begin{problem}
Orthotropic composite: check numerically that $A^{\ast}_{1112} = A^{\ast}_{2212} = 0$, then prove it theoretically. 
\end{problem}

\begin{problem}
Check that if $m_1\rightarrow 1$, then $A^{\ast}$ is close to the formula of a rank-$1$ laminate.
\end{problem}

\begin{problem}
What happens if $m_1\rightarrow 0$? Is $A^{\ast}$ close to $A$?
\end{problem}

\begin{problem}
If $m_1=m_2$, is $A^{\ast}$ isotropic ?
\end{problem}

\begin{problem}
Check numerically that $A^{\ast}$ is isotropic for a honeycomb structure with hexagonal holes.
\end{problem}

\chapter{Topology optimization by the homogenization method}\label{Topo opti homo method}
\section{Why topology optimization?} \label{topology optimization}

Shape optimization consists in ``shape tracking'' algorithms, hence the method cannot change the topology, such as the number of holes in the case of $2$-dimension.
On the other hand, topology optimization consists in ``shape capturing'' algorithms, that allow us to consider the optimization in a wider class which includes the different topological properties.
There are several methods of topology optimization but we focus on just one, called the homogenization method (see \cite{Allaire1, BS} and references therein). 
In the following of this chapter, we introduce a model problem for topology optimization with a constraint on the volume of holes and study it using the method of homogenization as mentioned in Chapter~\ref{Homogenization theory}.


\section{Homogenization method in the conductivity setting}
\label{sec:hom con setting}

In this section, we apply the homogenization method in the conductivity setting.
As we mentioned in Section~\ref{exist thm}, there is no minimizer for the corresponding minimizing problem (we will explain the detail below) in general.
To solve the problem, we introduce a set of generalized shapes as the limit of classical shapes.
More precisely, goals of the homogenization method for topology optimization are following:
\begin{itemize}
\item To introduce the notion of generalized shapes made of composite material,
\item To show that those generalized shapes are limits of sequences of classical shapes (in the sense of homogenization),
\item To compute the generalized objective function and its gradient,
\item To prove an existence theorem for optimal generalized shapes,
\item To deduce new numerical algorithms for topology optimization. 
\end{itemize}
In order to consider the limit of classical shapes, we recall one of the main results of homogenization theory. 
For the details of the proof, see \cite[Theorem 1.2.16 and 2.1.2]{Allaire1}.

\begin{theorem} \label{thm homogenization}
Let $\Omega$ be a bounded domain in $\RR^N$ and $\left(\chi_{\e}(x)\right)_{\e>0}$ be a sequence of characteristic functions in $\Omega$. 
Set $A_{\e}(x) = \alpha \chi_{\e}(x) + \beta (1-\chi_{\e}(x))$ for $x\in \Omega$.
Then there exists a subsequence, still denoted by $\chi_\e$, a density $0 \le \theta(x) \le 1$ and a homogenized tensor $A^{*}$ such that 
\begin{equation*}
\chi_\e\rightharpoonup\theta \quad {\rm weakly} \ \text{{\rm *}} \ {\rm in} \ L^\infty(\Omega; [0, 1])
\end{equation*}
and $A_{\e}$ converges in the sense of homogenization to $A^{*}$, i.e., for all $f \in L^2(\Omega)$, the solution $u_{\e}$ of the problem 
\begin{equation*}
\left\{
\begin{aligned}
- \dv \left( A_{\e}(x) \nabla u_{\e} \right) &= f &&\text{ in } \Omega, \\
u_{\e} &= 0 &&\text{ on } \pa \Omega
\end{aligned}
\right.
\end{equation*}
converges strongly in $L^2(\Omega)$ to the solution $u$ of the homogenized problem 
\begin{equation} \label{eq hom}
\left\{
\begin{aligned}
- \dv \left( A^{*}(x) \nabla u \right) &= f &&\text{ in } \Omega, \\
u &= 0 &&\text{ on } \pa \Omega.
\end{aligned}
\right.
\end{equation}
Furthermore, for almost all $x \in \Omega$, $A^{*}(x)$ belongs to the set $G_{\theta(x)}$ defined in Definition~ \ref{def Gtheta}.
\end{theorem}

\subsection{Relaxed problem for the conductivity model}
\label{sec:app hom method}
In this subsection, we introduce the conductivity model and derive generalized shapes as the limits of classical shapes by the homogenization method.
We impose a simplifying assumption: the ``holes" (with a Neumann, free boundary condition) are filled with a weak (``ersatz") material $\alpha$, while the other space is filled with a material $\beta$, that is, $\alpha<\beta$. 
We consider a membrane we introduced in Chapter~\ref{Parametric optimal design} with two thicknesses $\alpha$ and $\beta$, whose distribution is give by $h_{\chi} = \alpha \chi + \beta (1 - \chi)$, where $\chi$ is a characteristic function which denotes the position of the holes.
If $f \in L^{2}(\Omega)$ is the applied load, the displacement $u_\chi$ satisfies
\begin{equation*}
\left\{
\begin{aligned}
-\dv (h_{\chi} \nabla u_{\chi}) &= f &&\text{ in } \Omega, \\ 
u_{\chi} &= 0 \quad &&\text{ on } \pa \Omega. 
\end{aligned}
\right.
\end{equation*}
We will optimize the membrane's shape amounts by considering the minimizing problem
\begin{equation} \label{eq minimizer}
\inf_{\chi \in \mathcal{U}_{\rm ad}} J(\chi), 
\end{equation}
with 
\begin{equation*}
	\mathcal{U}_{\rm ad} = \left\{ \chi \in L^{\infty}(\Omega; \{0,1\})\;:\; \displaystyle  \int_{\Omega} \chi(x)\,dx = V_{\alpha} \right\} \quad {\rm and} \quad J(\chi) = \int_{\Omega} j(u_\chi)\,dx,
\end{equation*}
where $V_\alpha$ is a given positive constant, denoting the volume of the holes and 
\begin{equation} \label{eq j example}
j(u)=fu, \,\, {\rm or} \,\,  |u-u_0|^2.
\end{equation}
Since there is no minimizer of \eqref{eq minimizer} in general, we apply Theorem~\ref{thm homogenization} to introduce a set of generalized shapes.
Let $\chi_{\e}$ be a sequence (minimizing sequence of \eqref{eq minimizer} or not) of characteristic functions. 
Applying Theorem~\ref{thm homogenization}, we see that there exist $\theta$ and $A^*$ such that
\begin{equation*}
\begin{aligned}
&\chi_\e\rightharpoonup\theta & &\,\, {\rm weakly} \ \text{*} \ {\rm in} \ L^\infty(\Omega; [0, 1]),\\
&A_\e\to A^* & &\,\, {\rm in \ the \ sense \ of \ homogenization}
\end{aligned}
\end{equation*}
and
\begin{equation*}
J(\chi_{\e}) = \int_{\Omega} j(u_{\e}) \,dx \to \int_{\Omega} j(u) \,dx =:J(\theta,A^{*}),
\end{equation*}
where $u$ is the solution of \eqref{eq hom} and $j$ is defined by \eqref{eq j example}.
From this convergence result, we define the set of admissible homogenized shapes
\begin{equation*}
\mathcal{U}^{*}_{\rm ad} := \left\{ (\theta, A^{*})  \in L^{\infty}\left( \Omega; [0,1] \times \mathbb{R}^{N^{2}} \right)\;:\;A^{*}(x) \in G_{\theta(x)} \,\, \text{a.e. in} \,\, \Omega, \,\, \displaystyle  \int_{\Omega} \theta(x) \,dx = V_{\alpha} \right\}
\end{equation*}
 and consider the following relaxed or homogenized optimization problem:
\begin{equation} \label{eq hom op pro}
\inf_{(\theta, A^{*}) \in \mathcal{U}^{*}_{\rm ad}} J(\theta, A^{*}). 
\end{equation}
We easily check that $\mathcal{U}_{\rm ad} \subset \mathcal{U}^{*}_{\rm ad}$ if we identify $\chi\in\mathcal{U}_{\rm ad}$ with the pair $(\chi, \alpha\chi {\rm Id}+\beta(1-\chi){\rm Id})\in\mathcal{U}^*_{\rm ad}$.
The inclusion implies that we have enlarged the set of admissible shapes.
Moreover, one can prove that the relaxed problem \eqref{eq hom op pro} always admits an optimal solution, and the homogenized formulation is a relaxation of the original topology optimization problem as follows.

\begin{theorem} \label{thm cla hom}
The homogenized formulation is a relaxation of the original topology optimization problem in the sense that: 
\begin{itemize}
\item there exists, at least, one optimal composite shape $(\theta, A^{*})\in\mathcal{U}_{\rm ad}^*$, i.e. a minimizer of \eqref{eq hom op pro},
\item for any minimizing sequence $\left(\chi_n\right)_{n\in\NN}$ of \eqref{eq minimizer}, there exists a minimizer $(\theta, A^*)\in \mathcal{U}^*_{\rm ad}$ of \eqref{eq hom op pro} such that, up to subsequence,  $\chi_n$ converges weakly {\rm *} in $L^\infty(\Omega; [0, 1])$ to $\theta$ and $A_n:=\alpha\chi_n+\beta(1-\chi_n)$ converges to $A^*$ in the sense of homogenization,
\item any composite optimal solution $(\theta, A^{*})\in\mathcal{U}_{\rm ad}^*$ of \eqref{eq hom op pro} is the limit of a minimizing sequence of \eqref{eq minimizer}.
\end{itemize}
Moreover, the infima of the original and homogenized objective functions coincide 
\begin{equation*}
\inf_{\chi \in \mathcal{U}_{\rm ad}} J(\chi) = \min_{(\theta, A^{*}) \in \mathcal{U}^{*}_{\rm ad}} J(\theta, A^{*}). 
\end{equation*}
\end{theorem}

For the proof of Theorem~\ref{thm cla hom}, see \cite[Theorem 3.2.1]{Allaire1}.
Theorem~\ref{thm cla hom} means that the topology optimization problem is not changed by relaxation.
Moreover, close to any optimal composite shape, we are sure to find a quasi-optimal classical shape.
This theorem is at the root of new numerical algorithms.

\begin{remark}
The homogenized formulation is similar to a parametric or sizing optimization problem.
This is the main reason why the homogenization method is computationally cheap and works like a shape capturing algorithm.
Moreover, computing gradients or optimality conditions are thus very simple.
On the other hand, the design parameters $(\theta, A^{\ast})$ are quite complicated.
Another further (drastic and unjustified) simplification is to suppress the parameter $A^{\ast}$ and to keep only the material density $\theta$.
This is the main idea of the SIMP method \cite{BS}, we will mention in Section~\ref{sec:SIMP}.
\end{remark}

\subsection{Optimality conditions} \label{sec:optimality conditions}
In Section~\ref{sec:app hom method}, we introduced the relaxed formulation \eqref{eq hom op pro} of the original optimization problem \eqref{eq minimizer}.
One of the advantages of the relaxed problem is that we always have the existence of a minimizer of \eqref{eq hom op pro}.
There is also another advantage: we can get the optimality condition for the relaxed problem since it is possible to perform variations of composite designs.
In this subsection, we will consider the optimality condition for problem \eqref{eq hom op pro}.

We now compute the gradient of the following objective function
\begin{equation*}
J(\theta,A^{\ast})=\int_{\Omega}|u-u_0|^2\,dx,
\end{equation*}
where $u$ is the solution to \eqref{eq hom} and $u_{0}$ is a given function in $L^2(\Omega)$. 
We introduce the adjoint state $p$ of $u$ as the unique solution in $H^1_{0}(\Omega)$ of 
\begin{equation}\label{eq hom adj}
\left\{
\begin{aligned}
-\dv (A^{\ast}\nabla p)&=-2(u-u_0) &&\text{ in }\Omega, \\
p&=0 &&\text{ on }\partial\Omega.
\end{aligned}
\right.
\end{equation}
By the use of the adjoint state $p$, we can obtain the derivative of the functional $J$.

\begin{proposition} \label{prop deriv J(the, A)}
Let $\alpha>0$ and $\mathcal{M}_\alpha$ be the set of symmetric positive definite matrices $M$ such that $M\ge\alpha {\rm Id}$.
The functional $J$ is differentiable with respect to $A^*$ in $L^\infty(\Omega; \mathcal{M}_\alpha)$, and its derivative is
\begin{equation*}
\nabla_{A^*}J(\theta, A^*) =\nabla u \otimes\nabla p,
\end{equation*}
i.e.,
\begin{equation*}
\left\langle\nabla_{A^*}J(\theta, A^*), B^*\right\rangle =\int_\Omega B^*\nabla u\cdot\nabla p\,dx,\quad B^*\in L^\infty(\Omega; \mathcal{M}_\alpha)
\end{equation*}
where $u\in H^1_0(\Omega)$ is the unique solution of \eqref{eq hom} and $p$ is the adjoint state \eqref{eq hom adj} of $u$.
\end{proposition}

\begin{proof}
We can prove the proposition by the same strategy that we used in Lemma~\ref{Lemma:der u}, \ref{Lemma:J'} and Theorem~\ref{Theorem:adjoint}.
Thus we omit the proof.
\end{proof}

Remark that the partial derivative with respect to $\theta$ vanishes because $\theta$ appears only in the constraint of $A^*$.
Moreover, we can also consider the Lagrangian
\begin{equation*}
\mathcal{L}(A^{\ast}, v, q)=\int_{\Omega}|v-v_0|^2\,dx +\int_{\Omega}A^{\ast}\nabla v\cdot\nabla q\,dx-\int_{\Omega}fq \,dx,
\end{equation*}
where $(A^*, v, q)\in L^\infty(\Omega; \mathcal{M}_\alpha)\times H^1_0(\Omega)\times H^1_0(\Omega)$.
The partial derivatives of $\mathcal{L}$ with respect to $q$ and $v$ yield the state and adjoint state respectively. 
Furthermore, the functional $J$ is also differentiable with respect to $A^*$ with derivative 
\begin{equation*}
\nabla_{A^{\ast}}J(\theta,A^{\ast})=\frac{\partial \mathcal{L}}{\partial A^{\ast}}(A^{\ast},u, p)=\nabla u\otimes\nabla p.
\end{equation*}

The essential consequence of this section is the following optimality condition.

\begin{theorem} \label{thm:simplification}
Let $(\theta,A^{\ast})$ be a global minimizer of $J$ in $\mathcal{U}_{\rm ad}^{\ast}$ which admits $u$ and $p$ as state and adjoint. 
Then there exists $(\tilde{\theta},\tilde{A}^{\ast})$, another global minimizer of $J$ in $\mathcal{U}_{\rm ad}^{\ast}$, which admits the same state and adjoint $u$ and $p$, and such that $\tilde{A}^{\ast}$ is a rank-$1$ simple laminate.
\end{theorem}

\begin{proof}
We fix $\theta\in L^\infty(\Omega; [0, 1])$ with
\begin{equation*}
\int_\Omega\theta(x)\,dx=V_\alpha.
\end{equation*}
We remark that  by Theorem~\ref{thm:Gtheta}, the set
\begin{equation*}
\mathcal{G}_\theta:=\left\{A^0\in L^\infty\left(\Omega; \RR^{N^2}\right) \;:\; A^0(x)\in G_{\theta(x)}\, {\rm a.e.} \, x\in\Omega\right\}
\end{equation*}
is a convex set.
Then, by Theorem~\ref{thm Euler ineq} and Proposition~\ref{prop deriv J(the, A)}, we have
\begin{equation*}
\int_{\Omega}(A^0-A^\ast)\nabla u\cdot\nabla p\,dx\ge 0 \quad  \forall A^0\in\mathcal{G}_\theta.
\end{equation*}
We easily check that the above  inequality is equivalent to the point-wise constraint
\begin{equation} \label{eq:pointwise constraint}
A^{\ast}(x)\nabla u(x)\cdot\nabla p(x)=\displaystyle\min_{A^{0}\in G_{\theta(x)}}(A^{0}\nabla u(x)\cdot\nabla p(x))
\end{equation}
for almost all $x\in\Omega$.
Fix $x\in\Omega$ which satisfies \eqref{eq:pointwise constraint}.
If $\nabla u(x)$ or $\nabla p(x)$ vanishes, then any $A^{\ast}\in G_{\theta(x)}$ is optimal in \eqref{eq:pointwise constraint} (the fact that $u$ and $p$ does not change at these points is more delicate to establish and we refer to \cite{Allaire1,TA} for details).
Otherwise, we define two unit vectors
\begin{equation*}
e=\frac{\nabla u}{|\nabla u|}\ \text{ and }\ e'=\frac{\nabla p}{|\nabla p|}.
\end{equation*}
Then \eqref{eq:pointwise constraint} is equivalent to finding a minimizer $A^*(x)$ of
\begin{equation*}
4A^0e\cdot e'=A^0(e+e')\cdot(e+e')-A^0(e-e')\cdot(e-e').
\end{equation*}
A lower bound is easily seen to be
\begin{equation} \label{eq:lower bound}
\begin{aligned}
\displaystyle\min_{A^{0}\in G_{\theta}}4A^0e\cdot e'
&\ge\min_{A^0\in G_{\theta}}A^0(e+e')\cdot(e+e')-\max_{A^0\in G_{\theta}}A^0(e-e')\cdot(e-e')\\
&=\lambda_{\theta}^-|e+e'|^2-\lambda_{\theta}^+|e-e'|^2,
\end{aligned}
\end{equation}
where $\lambda^\pm_\theta$ is defined in Theorem~\ref{thm:Gtheta}.
We can see that the lower bound is attained by a matrix $A^1$ corresponding to a rank-1 laminate (see \cite[Remark 2.2.14]{Allaire1}).
More precisely, $A^1$ satisfies
\begin{equation} \label{eq:con A1}
A^1(e+e')=\lambda^-_\theta(e+e'), \quad A^1(e-e')=\lambda^+_\theta(e-e').
\end{equation}
Thus we obtain
\begin{equation*}
\displaystyle\min_{A^0\in G_{\theta}}4A^0e\cdot e'=\lambda_\theta^-|e+e'|^2-\lambda_\theta^+|e-e'|^2.
\end{equation*}
Moreover, if $A^{\ast}$ is any optimal tensor, then, $A^{\ast}$ must also satisfy
\begin{equation} \label{eq:op con A*}
A^{\ast}(e+e')=\lambda_{\theta}^{-}(e+e')\ \text{ and }\ A^{\ast}(e-e')=\lambda_{\theta}^{+}(e-e').
\end{equation}
Indeed, if \eqref{eq:op con A*} does not hold true, the bounds on eigenvalues in Lemma~\ref{thm:Gtheta} imply the strict inequality 
\begin{equation*}
4A^{\ast}e\cdot e'=A^{\ast}(e+e')\cdot (e+e')-A^{\ast}(e-e')\cdot (e-e')
>\lambda_{\theta}^{-}|e+e'|^2-\lambda_{\theta}^{+}|e-e'|^2,
\end{equation*}
which is a contradiction with \eqref{eq:pointwise constraint} and \eqref{eq:lower bound}.
By \eqref{eq:con A1} and \eqref{eq:op con A*}, we see that
\begin{equation*}
\begin{aligned}
&2A^{\ast}\nabla u=(\lambda_{\theta}^++\lambda_{\theta}^-)\nabla u+(\lambda_{\theta}^+-\lambda_{\theta}^-) \frac{|\nabla u|}{|\nabla p|}\nabla p=2A^1\nabla u,\\
&2A^{\ast}\nabla p=(\lambda_{\theta}^++\lambda_{\theta}^-)\nabla p+(\lambda_{\theta}^+-\lambda_{\theta}^-) \frac{|\nabla p|}{|\nabla u|}\nabla u=2A^1\nabla p
\end{aligned}
\end{equation*}
for almost all $x\in\Omega$.
Therefore any optimal tensor $A^*$ can be replaced by the rank-1 simple laminate $A^1$ without changing $u$ and $p$.
\end{proof}

\begin{remark} \label{rem:rank-1 laminates}
Theorem~{\rm \ref{thm:simplification}} implies that in the definition of $\mathcal{U}_{\rm ad}^{\ast}$, the set $G_{\theta}$ can be replaced by its simpler subset of rank-$1$ simple laminates.
We actually use this simplification in the numerical algorithms.
We remark that this simplification holds true for other objective functions as well. However, it does not hold for multiple loads optimization in general.
For the details concerning the conditions that the objective function $J$ must satisfy in order to obtain the optimality conditions, see {\rm \cite[Chapter 3.2.2]{Allaire1}}.
\end{remark}

As stated in Remark~\ref{rem:rank-1 laminates}, we can simplify the admissible set of the minimization problem \eqref{eq hom op pro}.
We consider the parametrization of rank-1 laminates.
For simplicity, we consider the case $N=2$.
A rank-1 laminate is defined by
\begin{equation*}
A^{\ast}(\theta,\phi)=
\begin{pmatrix}
\cos\phi & \sin\phi\\
-\sin\phi&\cos\phi
\end{pmatrix}
\begin{pmatrix}
\lambda_{\theta}^+&0\\
0&\lambda_{\theta}^-
\end{pmatrix}
\begin{pmatrix}
\cos\phi & -\sin\phi\\
\sin\phi&\cos\phi
\end{pmatrix}
,
\end{equation*}
where the angle $\phi\in [0,\pi]$ determines the orientation of the unit cell.
Hence the admissible set is rewritten as
\begin{equation*}
\mathcal{U}_{\rm ad}^{L}:=\left\{
(\theta,\phi)\in L^{\infty}(\Omega;[0,1]\times[0,\pi]) \;:\; \int_{\Omega}\theta(x)\,dx=V_{\alpha}
\right\}.  
\end{equation*}
If we set $J(\theta, \phi):=J(\theta, A^*(\theta, \phi))$, then the derivative of the objective function $J(\theta, \phi)$  
follows by Proposition \ref{prop deriv J(the, A)} immediately.

\begin{proposition} \label{prop:derivative J(theta,phi)}
The objective function $J(\theta,\phi)$ is differentiable with respect to $(\theta,\phi)$ in $\mathcal{U}_{\rm ad}^{L}$, and its partial derivatives are
\begin{equation*}
\nabla_{\phi}J(\theta,\phi)=\frac{\partial A^{\ast}}{\partial \phi}\nabla u\cdot \nabla p\ \text{ and }\ \nabla_{\theta}J(\theta,\phi)=\frac{\partial A^{\ast}}{\partial\theta}\nabla u\cdot\nabla p.
\end{equation*}
\end{proposition}

\subsection{Numerical algorithm}
\label{subsec:numerical algorithm conductivity}
In this subsection, we show the numerical algorithm to seek the optimal shape $\theta$ of \eqref{eq hom op pro}.
As stated in Section~\ref{sec:optimality conditions}, we can treat $J(\theta, \phi)$ instead of $J(\theta, A^*)$ if $N=2$.
We explain the projected gradient algorithm for the minimization of $J(\theta, \phi)$ by the use of Proposition~\ref{prop:derivative J(theta,phi)}.

\begin{algorithm}[H]
\caption{Projected gradient algorithm for \eqref{eq hom op pro}} 
\begin{itemize}
	\item[{\rm 1}.]	
   We initialize the design parameters $\theta_0$ and $\phi_0$ (for example, equal to constants).
   \item[{\rm 2}.]
   Until convergence, for $k\ge0$ we iterate by computing the state $u_k$ and adjoint $p_k$, solutions of \eqref{eq hom} and \eqref{eq hom adj} respectively with respect to the previous design parameters $(\theta_k, \phi_k)$, then we update these parameters by
   \begin{equation*}
   		\begin{aligned}
   			 \theta_{k+1} &= \max\left(
          0, \min\left(
          1, \theta_k-t_k\left(
          \ell_k+\dfrac{\partial A^*}{\partial\theta}(\theta_k, \phi_k)\gr u_k\cdot\gr p_k
          \right)
          \right)
          \right), \\
          \phi_{k+1} &= \phi_k-t_k\dfrac{\partial A^*}{\partial\theta}(\theta_k, \phi_k)\gr u_k\cdot\gr p_k,
   		\end{aligned}
   \end{equation*}
   where $\ell_k$ a Lagrange multiplier for the volume constraint, and $t_k>0$ a descent step such that $J(\theta_{k+1}, \phi_{k+1})<J(\theta_k, \phi_k)$.
\end{itemize}
\end{algorithm}

For the details about the multiplier $\ell_k$, we refer to Section~\ref{Numerical algorithm for optimal thickness}. 
In the following, we will consider two simpler self-adjoint cases. 
Finally, we will show the algorithm to obtain the optimal shape which is close to the classical shapes.

\subsubsection{First example of a self-adjoint case} \label{subsec:first ex}
A first example is the minimization of the torsional rigidity (maximization of compliance)
\begin{equation} \label{eq:torsional rigidity}
	\min_{(\theta, A^*)\in\mathcal{U}^L_{\rm ad}}\left\{
   J(\theta, A^*)=-\int_\Omega u(x)\,dx
   \right\},
\end{equation}
where $u$ is the solution of
\begin{equation} \label{eq:ex self adjoint}
	\left\{
   \begin{aligned}
   		-\dv(A^*\gr u)&=1 && \text{ in } \Omega, \\
      u&=0 &&\text{ on } \partial \Omega.
   \end{aligned}
   \right.
\end{equation}
In this case, the adjoint state $p$ is the solution of
\begin{equation*}
	\left\{
   \begin{aligned}
   		-\dv(A^*\gr p)&=1 &&\text{ in } \Omega, \\
      p&=0 &&\text{ on } \partial \Omega,
   \end{aligned}
   \right.
\end{equation*}
i.e., the adjoint state is just $p=u$.

Before applying a numerical algorithm, we simplify the minimization problem by using the argument in Section~\ref{sec:optimality conditions}.
By the similar argument as in Proposition~\ref{prop deriv J(the, A)}, we get
\begin{equation*}
\left\langle\nabla_{A^{\ast}}J(\theta,A^{\ast}), B^*\right\rangle=\int_\Omega B^*\nabla u\cdot\nabla u\,dx\ge 0 ,\quad B^*\in L^\infty(\Omega; \mathcal{M}_\alpha).
\end{equation*}
Hence we have to decrease $A^{\ast}$ to minimize $J$.
Indeed, by \eqref{eq:op con A*} any minimizer $(\theta, A^{\ast})$ satisfies
\begin{equation*}
A^{\ast}\nabla u=\lambda_{\theta}^{-}\nabla u \quad {\rm a.e. \ in}\ \Omega,
\end{equation*}
i.e., the optimal composite is the worst possible conductor.
This condition allows us to eliminate the angle $\phi$ and it remains to optimize with respect to $\theta$ only, i.e., instead of \eqref{eq:ex self adjoint} we have to consider the state of the problem
\begin{equation}\label{lam- gr u}
	\left\{
   \begin{aligned}
   		-\dv(\lambda^-_\theta\gr u)&=1 &&\text{ in } \Omega, \\
      u&=0 && \text{ on } \partial \Omega.
   \end{aligned}
   \right.
\end{equation}
Moreover, we recall that the state $u$, which is the solution of \eqref{lam- gr u}, is characterized as the minimizer of the corresponding energy, i.e.,
\begin{equation*}
\int_\Omega \lambda^-_\theta|\nabla u|^2\,dx-2\int_\Omega u\,dx=\min_{v\in H^1_0(\Omega)}\left\{\int_\Omega \lambda^-_\theta|\nabla v|^2\,dx-2\int_\Omega v\,dx\right\}.
\end{equation*}
On the other hand, by the weak form of the state $u$, we have
\begin{equation*}
\int_\Omega \lambda^-_\theta|\nabla u|^2\,dx=\int_\Omega u\,dx
\end{equation*}
and thus
\begin{equation*}
-\int_\Omega u\,dx=\min_{v\in H^1_0(\Omega)}\left\{\int_\Omega \lambda^-_\theta|\nabla v|^2\,dx-2\int_\Omega v\,dx\right\}.
\end{equation*}
Therefore, we can rewrite the minimizing problem \eqref{eq:torsional rigidity} as
\begin{equation*}
\min_{(\theta, v)}\left\{\int_\Omega\lambda^-_\theta|\nabla v|^2\,dx-2\int_\Omega v\,dx\right\},
\end{equation*}
where the minimum in the right hand side of the above equation is taken over the set
\begin{equation*}
\left\{(\theta, v)\in L^\infty(\Omega; [0, 1])\times H^1_0(\Omega) \;:\; \int_\Omega\theta\,dx=V_\alpha\right\}.
\end{equation*}
Furthermore, since the function $(\theta,v)\mapsto\lambda_{\theta}^-|\nabla v|^2$ is convex, there are only global minima by Proposition~\ref{prop:global minimizer}.

Numerically, we use an algorithm based on alternate direction minimization (see Section~\ref{Numerical algorithm for optimal thickness}).
We solve in the domain $\Omega=(0,1)^2$ with phases $\alpha=1$ and $\beta=2$ respectively. 
We work with a volume constraint $50\%$ of phase $\alpha$.
We initialize with a constant value of $\theta=0.5$ and a constant zero lamination angle $\phi=0$. 
We perform $30$ iterations.
We show the numerical result how the objective function $J$ convergent to some value (Figure~\ref{fig:history}) and the volume fraction $\theta$ at some iteration numbers (Figure~\ref{fig:volume theta}).
Here, the horizontal axis in Figure~\ref{fig:history} means the iteration number.

\begin{figure}[H]
\centering
\includegraphics[width=.7
\linewidth 
,center]{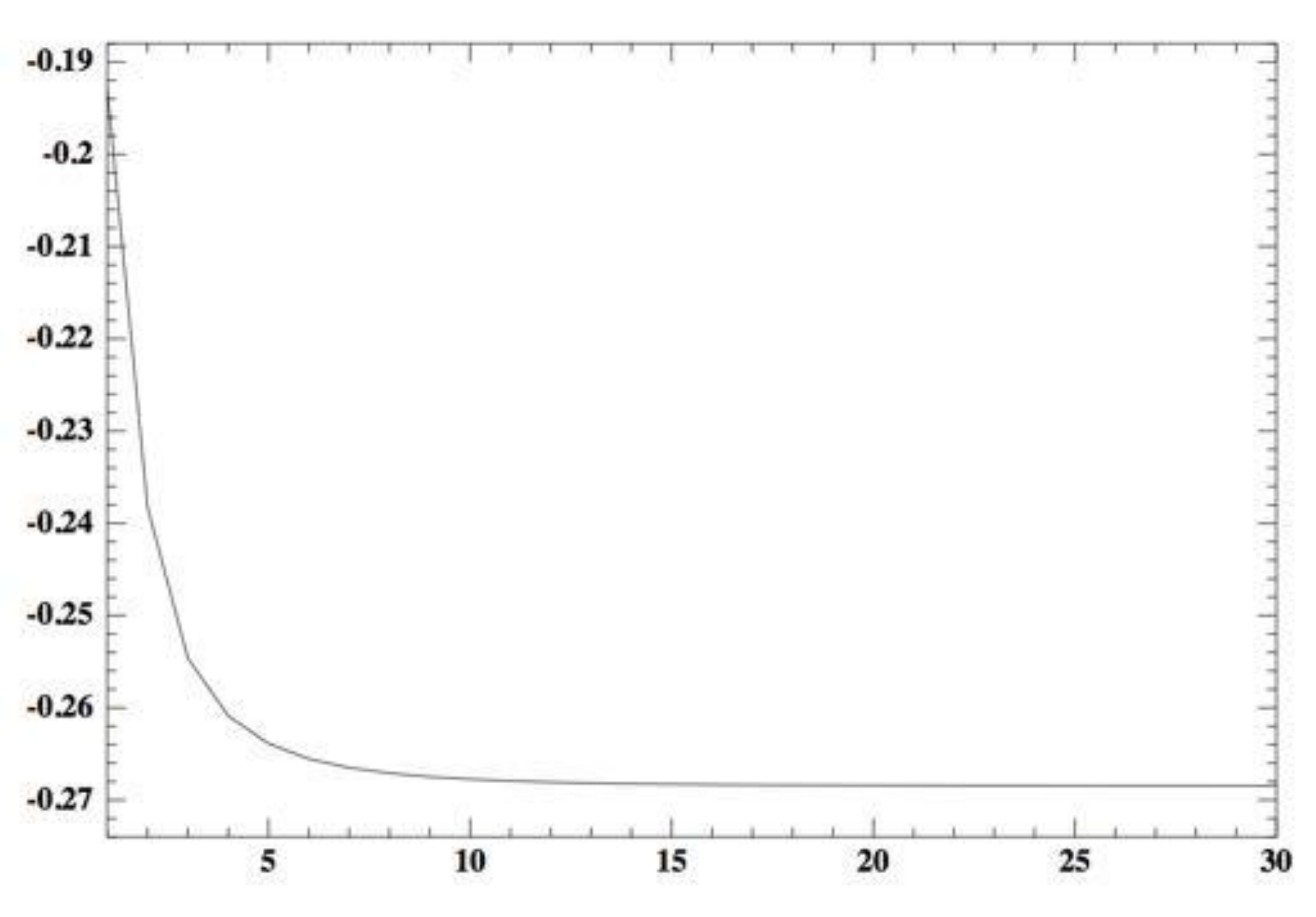}
\caption{Convergence history of $J$} 
\label{fig:history}
\end{figure}

\begin{figure}[H]
\centering
\begin{tabular}{c}
\begin{minipage}{0.3\hsize}
\begin{center}
\includegraphics[width=1.0\linewidth,center]{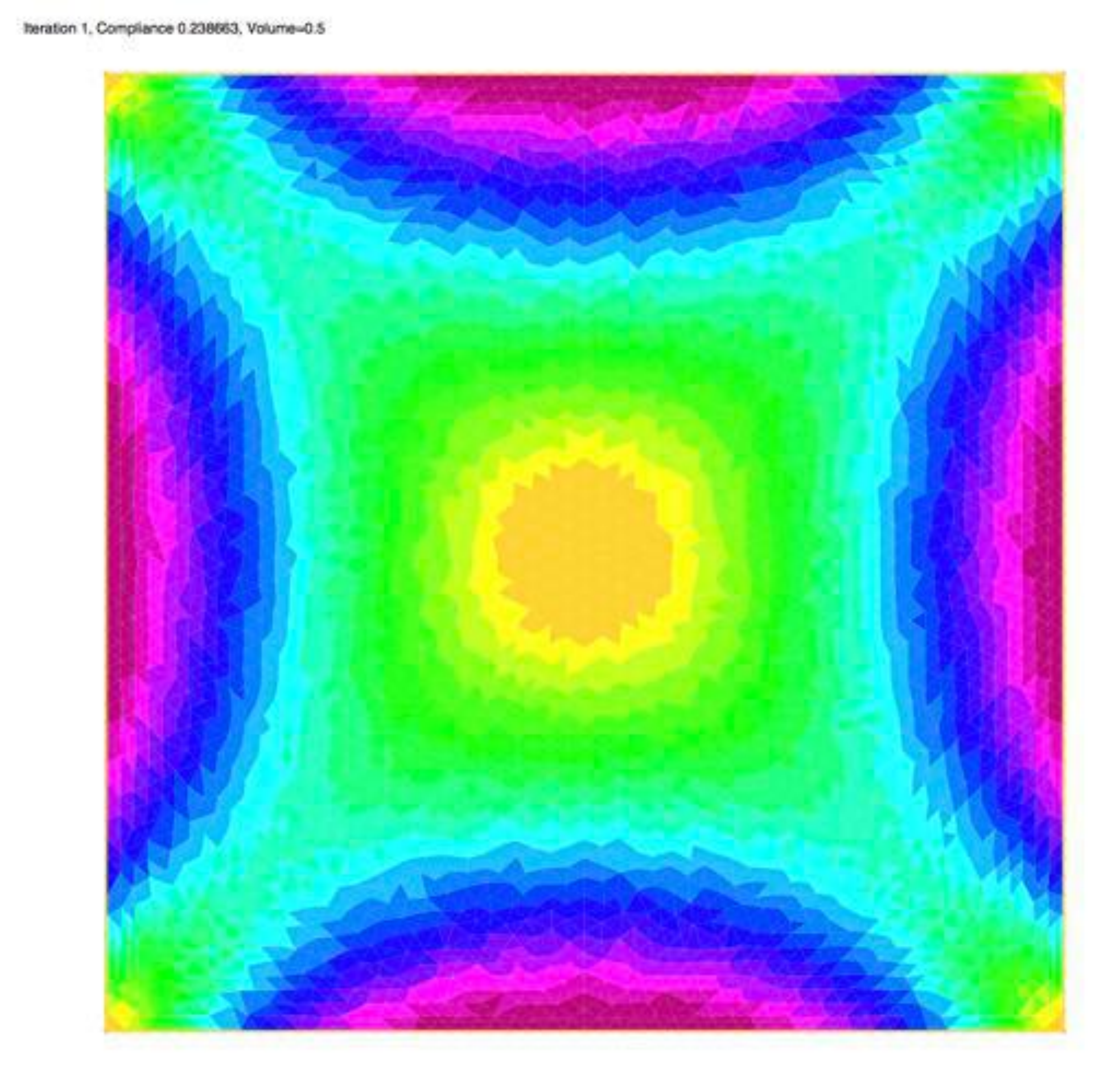}
\end{center}
\end{minipage}
\begin{minipage}{0.3\hsize}
\begin{center}
\includegraphics[width=1.0\linewidth,center]{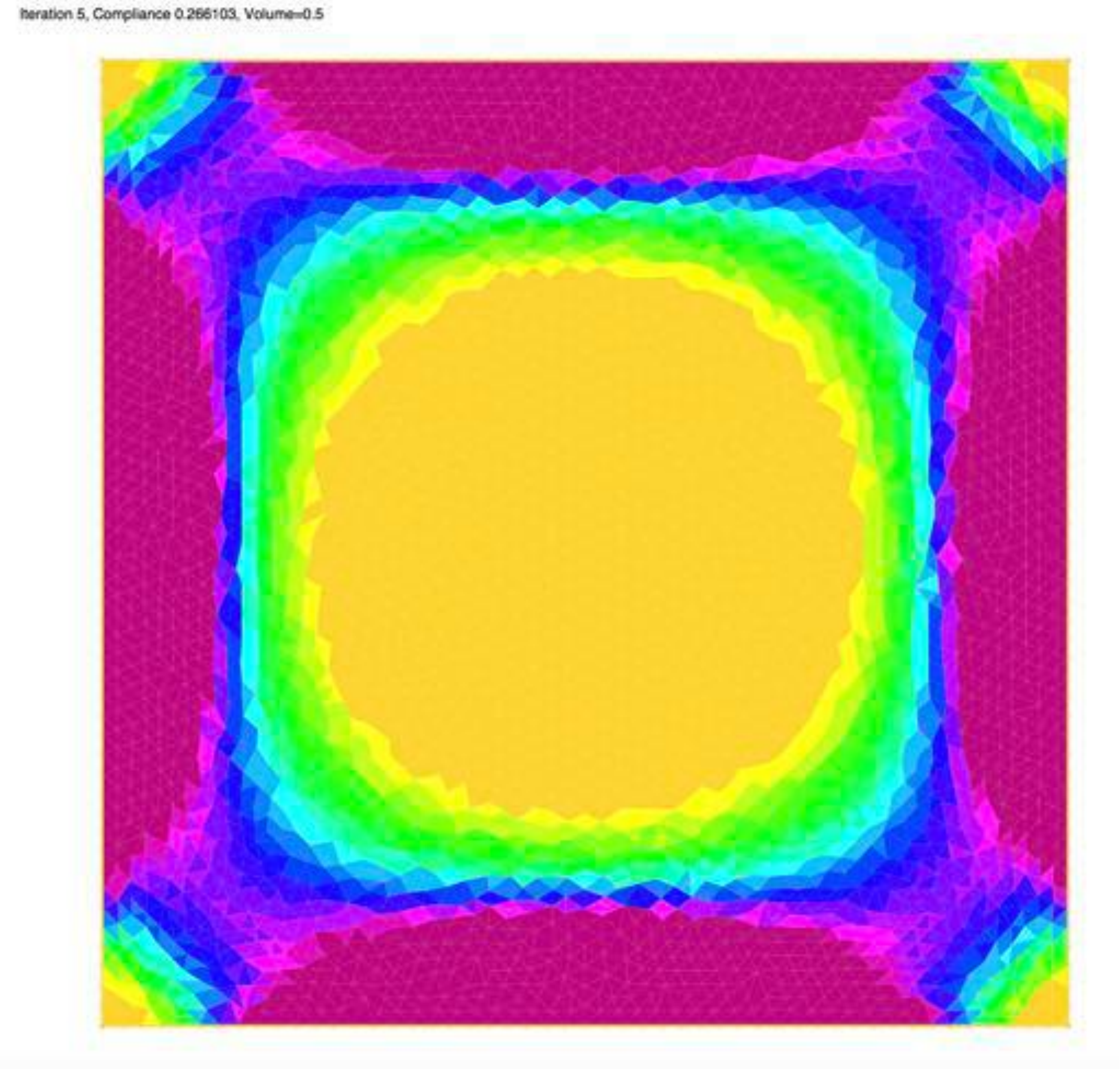}
\end{center}
\end{minipage}
\begin{minipage}{0.3\hsize}
\begin{center}
\includegraphics[width=1.0\linewidth,center]{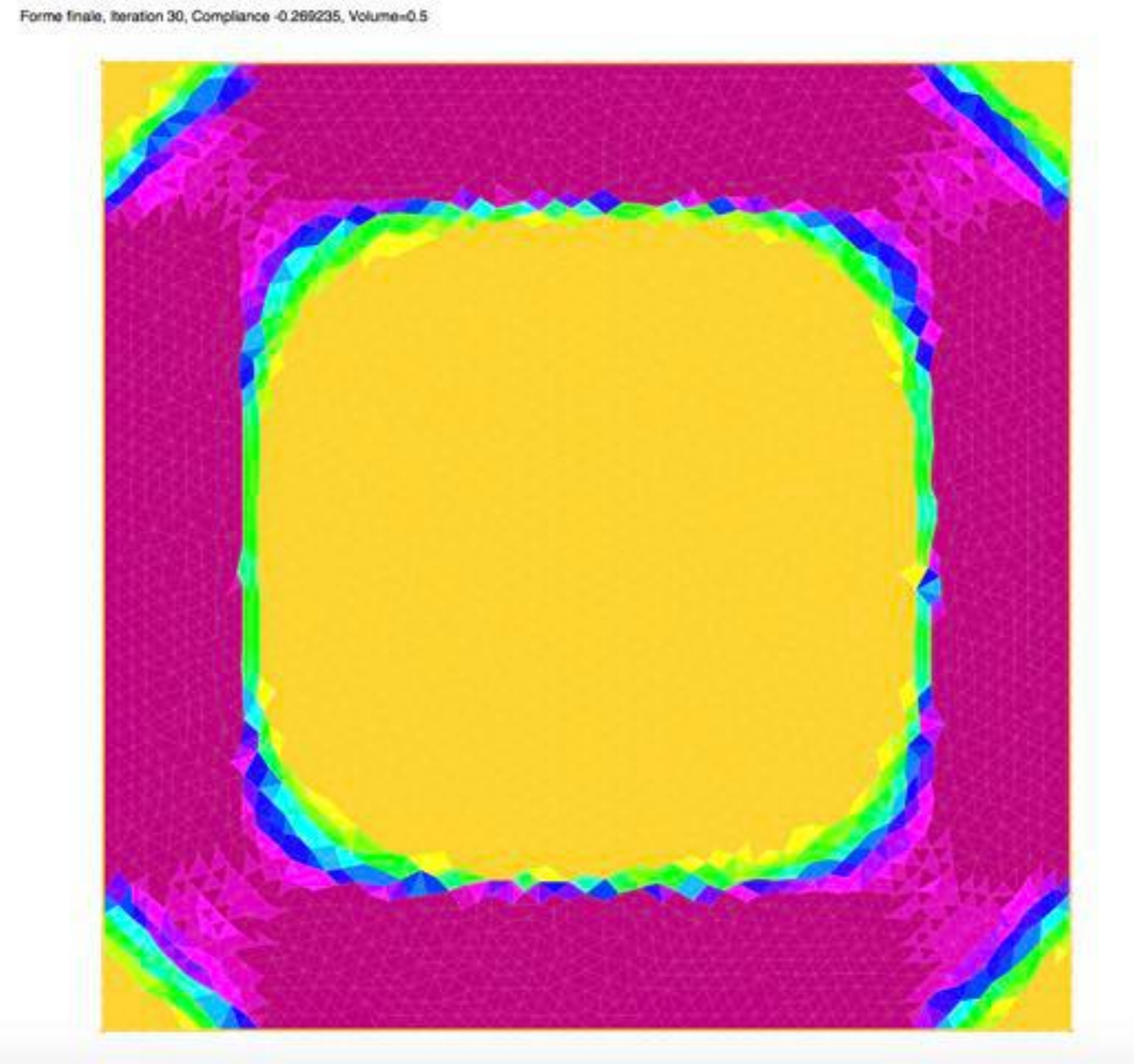}
\end{center}
\end{minipage}
\end{tabular}
\caption{Volume fraction $\theta$ (iteration number 1, 5 and 30 respectively) under the following color convention: red = 1, yellow = 0.} 
\label{fig:volume theta}
\end{figure}


\subsubsection{Second example of a self-adjoint case} \label{subsec:second ex}

A second self-adjoint example is a compliance minimization
\begin{equation}\label{eq:compliance minimization}
\displaystyle\min_{(\theta,A^{\ast})\in\mathcal{U}_{\rm ad}^L}\left\{J(\theta,A^{\ast})=\int_{\Omega}u(x)\,dx\right\},
\end{equation}
where $u$ is the solution of \eqref{eq:ex self adjoint}.
In this case, the adjoint state is $p=-u$.

In order to apply the numerical algorithm, we will simplify the minimizing problem as we did in the first example.
By the same argument as in the first example, we see that
\begin{equation*}
\left\langle\nabla_{A^{\ast}}J(\theta, A^{\ast}), B^*\right\rangle=-\int_\Omega B^*\nabla u\cdot\nabla u\,dx\le 0, \quad B^*\in L^\infty(\Omega; \mathcal{M}_\alpha)
\end{equation*}
and
\begin{equation*}
A^{\ast}\nabla u=\lambda_{\theta}^{+}\nabla u
\end{equation*}
if $(\theta, A^*)$ is a minimizer of \eqref{eq:compliance minimization}.
Hence the optimal composite is the best possible conductor and we have only to consider the following problem instead of \eqref{eq:ex self adjoint}:
\begin{equation*}
	\left\{
   \begin{aligned}
   		-\dv(\lambda^+_\theta\gr u)&=1 && \text{ in } \Omega, \\
      u&=0 && \text{ on } \partial \Omega.
   \end{aligned}
   \right.
\end{equation*}
Therefore, as in the previous section, we can eliminate the dependency on the angle $\phi$ and then optimize with respect to $\theta$ only.
We rewrite the optimization problem thanks to the dual energy
\begin{equation*}
\displaystyle\int_{\Omega}u \,dx=
\min_{\substack{\tau\in L^2(\Omega)^N, \\ -\dv\, \tau=1 \ {\rm in} \ \Omega}}
\int_{\Omega}(\lambda_{\theta}^+)^{-1}|\tau|^2dx.
\end{equation*}
We can obtain the dual energy by the similar calculation in Example~\ref{ex:dual energy}.
Thus instead of \eqref{eq:compliance minimization} we obtain a double minimization
\begin{equation*}
\min_{(\theta,\tau)}\int_{\Omega}(\lambda_{\theta}^+)^{-1}|\tau|^2 dx,
\end{equation*}
where the minimum is taken over the set
\begin{equation*}
\left\{(\theta, \tau)\in L^\infty(\Omega; [0, 1])\times (L^2(\Omega))^N \;:\; \int_\Omega\theta\,dx=V_\alpha, \ -\dv\,\tau=1 \ {\rm in} \ \Omega\right\}.
\end{equation*}
Furthermore, since the function $(\theta, \tau)\mapsto(\lambda^+_\theta)^{-1}|\tau|^2$ is convex, there are only global minima by Proposition~\ref{prop:global minimizer}.

We apply the same algorithm as in the first example.
The setting of the problem is also same as in the first example.
Figure~\ref{fig:volume theta2} shows the numerical result of the volume fraction $\theta$ at some iteration numbers.

\begin{figure}[h]
\centering
\begin{tabular}{c}
\begin{minipage}{0.3\hsize}
\begin{center}
\includegraphics[width=1.0\linewidth,center]{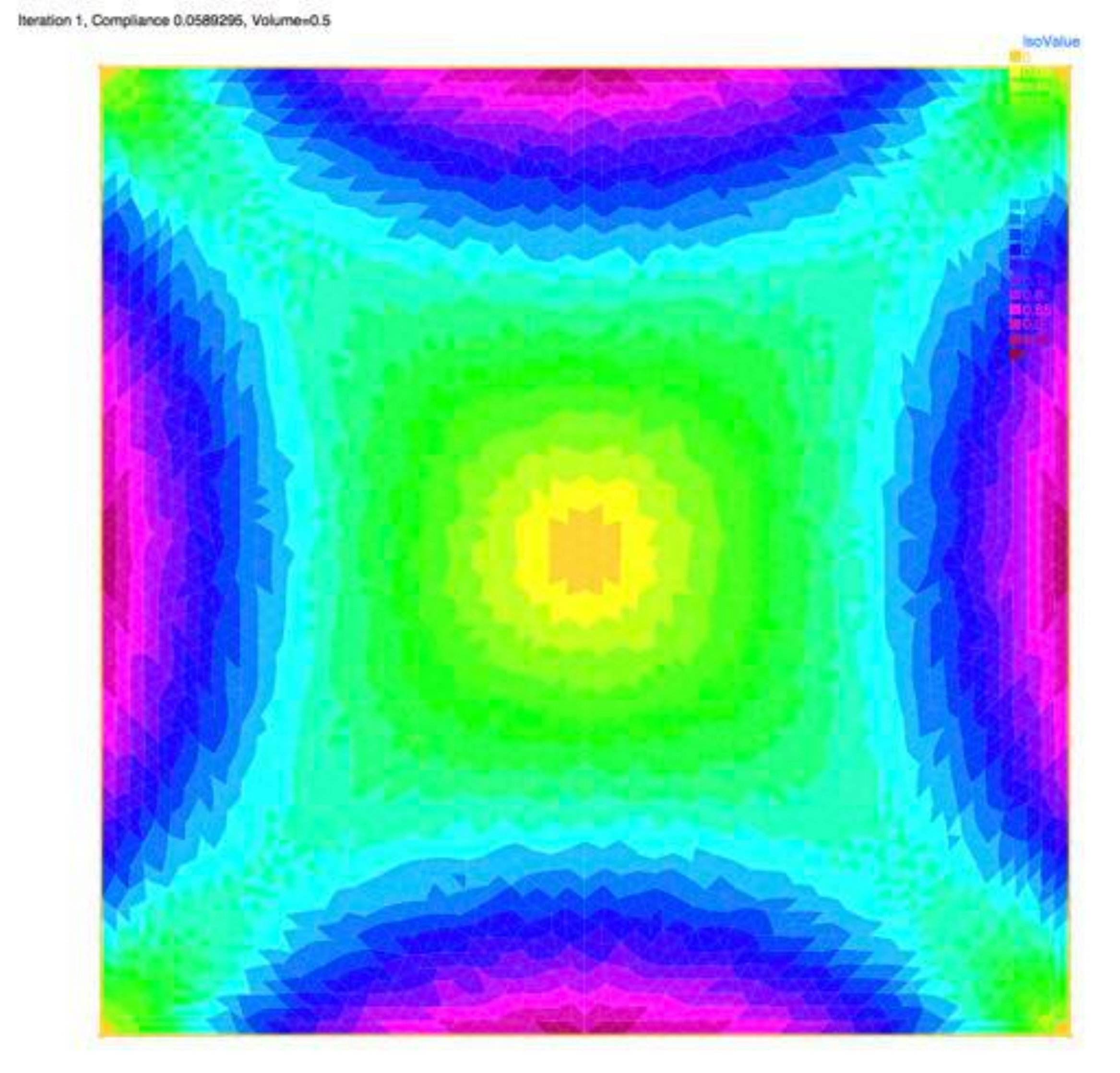}
\end{center}
\end{minipage}
\begin{minipage}{0.3\hsize}
\begin{center}
\includegraphics[width=1.0\linewidth,center]{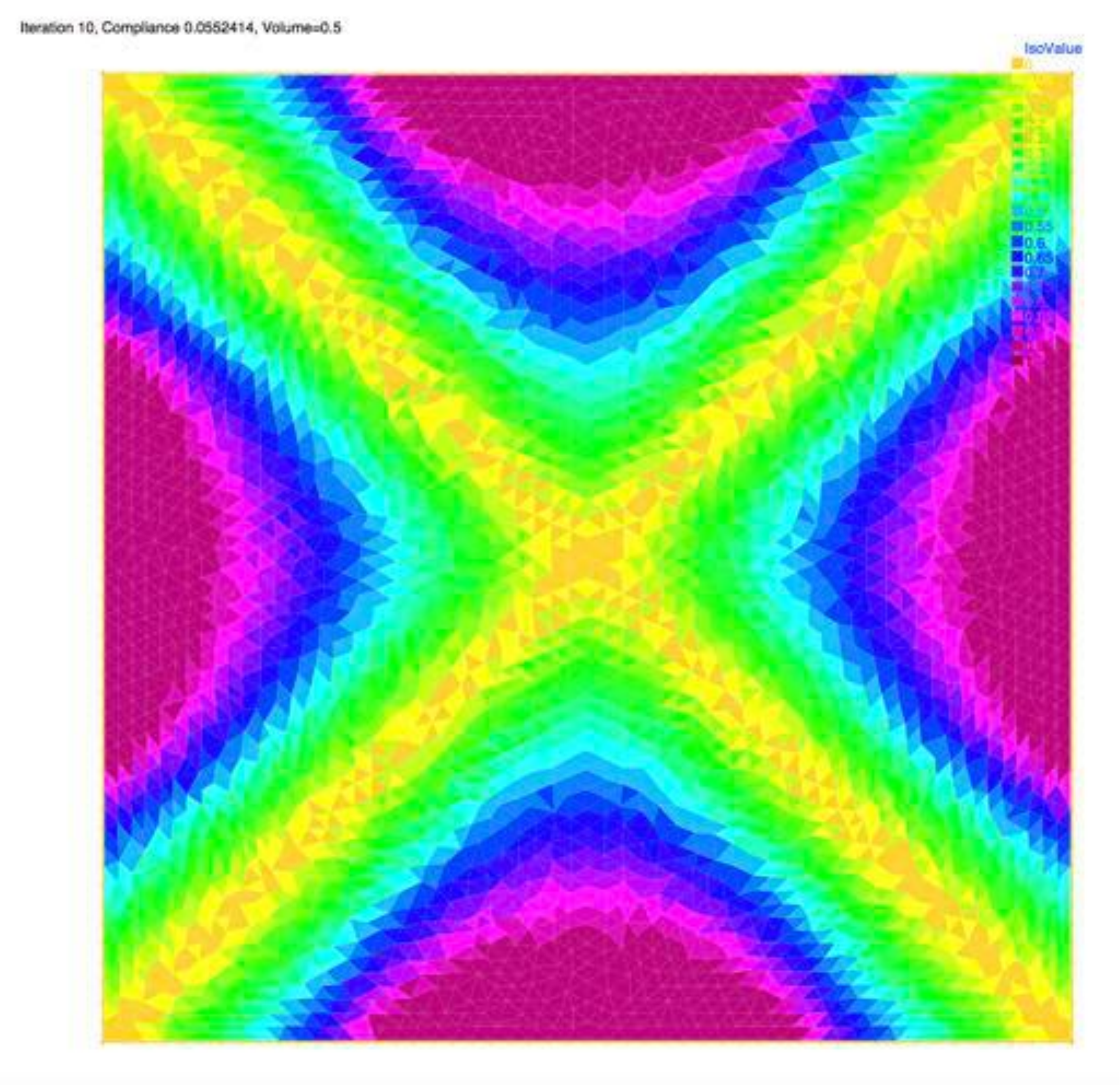}
\end{center}
\end{minipage}
\begin{minipage}{0.3\hsize}
\begin{center}
\includegraphics[width=1.0\linewidth,center]{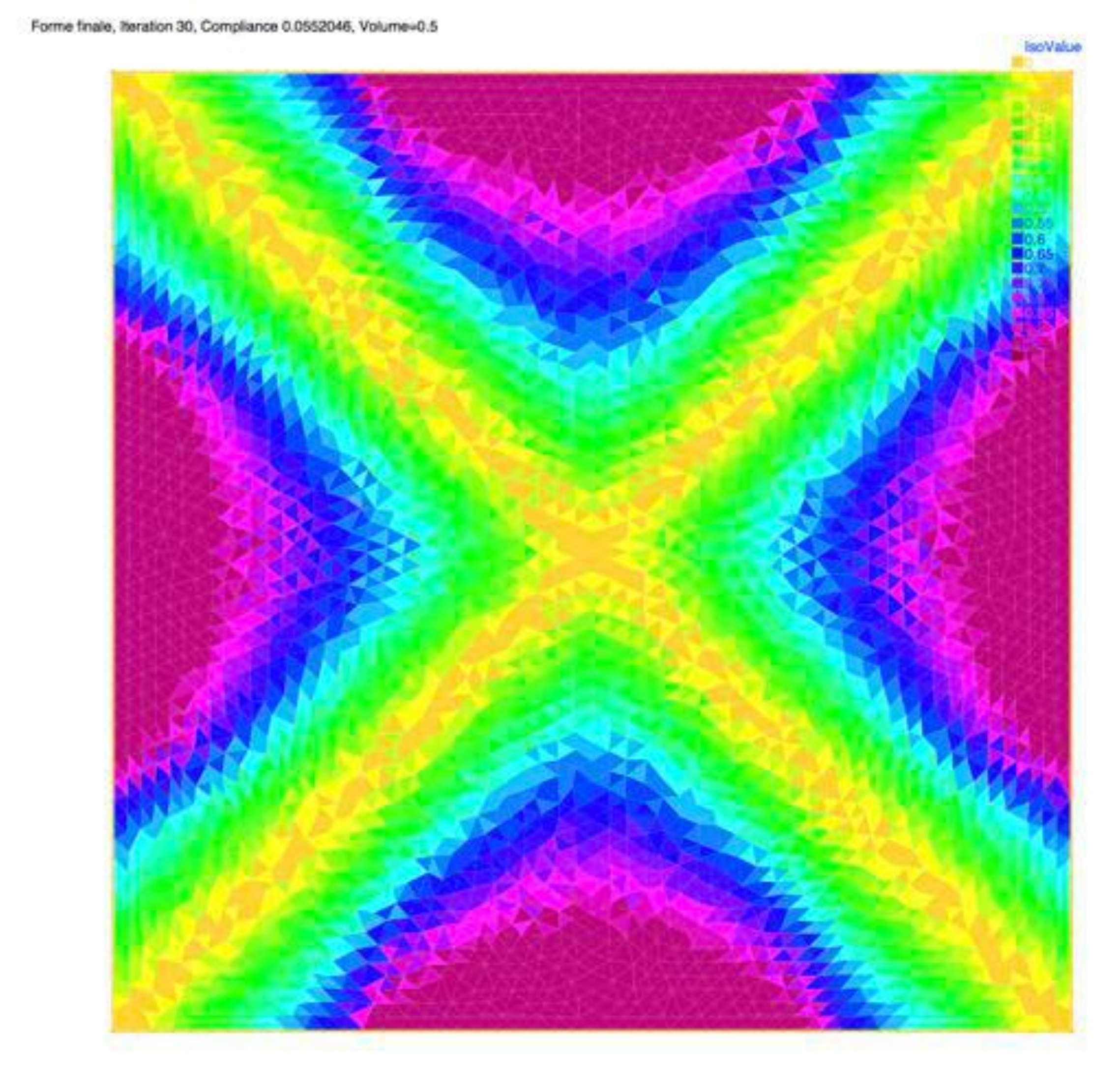}
\end{center}
\end{minipage}
\end{tabular}
\caption{Volume fraction $\theta$ (iteration number 1, 10 and 30 respectively) under the following color convention: red = 1, yellow = 0.} 
\label{fig:volume theta2}
\end{figure}

\begin{remark}
Thanks to the convexity properties of the functionals, the convergence to a global minimum is guaranteed. In practice, it can be checked by numerical experiments with various initializations converging to the same solution.

If one is interested by shape optimization rather than two-phase optimization, 
then, in numerical practice, holes can be mimicked by a very weak phase $\alpha$, such as $10^{-3}\beta$.
Mathematically, when $\alpha\to 0$ we obtain Neumann boundary conditions on the holes boundaries.
\end{remark}

\subsubsection{Penalization} \label{subsec:penalization}
By the algorithm stated in the two examples, we obtain optimal shapes in the wider class of composite shapes. 
Since, in practice, we are rather interested in classical shapes, we choose to use a penalization process to force the density to take values close to $0$ or $1$.

\begin{algorithm}[H]
\caption{Penalization process for \eqref{eq hom op pro}}
Apply either of the following algorithms after convergence to a composite shape by the previous algorithm.

\begin{enumerate}
\item 
We add a penalization term to the objective function
\begin{equation*}
J(\theta,A^{\ast})+c_{\rm pen}\int_{\Omega}\theta(1-\theta)\, dx,
\end{equation*}
where $c_{\rm pen}$ is a constant for penalization.
\item
We continue the previous algorithm with a modified “penalized” density
\begin{equation*}
\displaystyle\theta_{\rm pen}:=\frac{1-\cos(\pi\theta_{\rm opt})}{2},
\end{equation*}
where $\theta_{\rm opt}$ is the optimal density obtained by the previous algorithm.
\end{enumerate}
\end{algorithm}

In the second algorithm, we note that if $0<\theta_{\rm opt}<1/2$, then $\theta_{\rm pen}<\theta_{\rm opt}$, while $1/2<\theta_{\rm opt}<1$, then $\theta_{\rm pen}>\theta_{\rm opt}$.
Hence we see that the density $\theta$ goes to $0$ or $1$ when we apply the algorithm.

\begin{example}
We consider the optimal radiator (Figure~\ref{fig:radiator}-(a))
\begin{equation*}
\left\{
\begin{aligned}
-\dv (A^{\ast}\nabla u)&=0 &&\text{ in } \Omega ,\\
A^{\ast}\nabla u\cdot n&=1 &&\text{ on } \Gamma_N ,\\
A^{\ast}\nabla u\cdot n&=0 &&\text{ on } \Gamma ,\\
u&=0 &&\text{ on } \Gamma_{D}.
\end{aligned}
\right.
\end{equation*}
We minimize the temperature where heating takes place
\begin{equation*}
\displaystyle\min_{(\theta,A^{\ast})\in\mathcal{U}_{\rm ad}^L}\left\{
J(\theta,A^{\ast})=\int_{\Gamma_{N}}u \, ds
\right\}.
\end{equation*}
This is another case of compliance minimization.
Thus, the problem is self-adjoint and $p = -u$.
We solve in the domain $\Omega = (0, 1)^2$ with phases $\alpha = 0.01$ and $\beta = 1$.
We work with a volume constraint $50\%$ of phase $\alpha$.
We initialize with a value $\theta$ as Figure~{\rm \ref{fig:radiator}-(b)}.
Figure~{\rm \ref{fig:radiator}-(c)}, {\rm (d)} and {\rm (e)} show the numerical results of the volume density at some iteration numbers under the above penalization algorithm.
\end{example}

\begin{figure}[h]
\centering
\begin{tabular}{c}
\begin{minipage}{0.3\hsize}
\begin{center}
\includegraphics[width=1.0\linewidth,center]{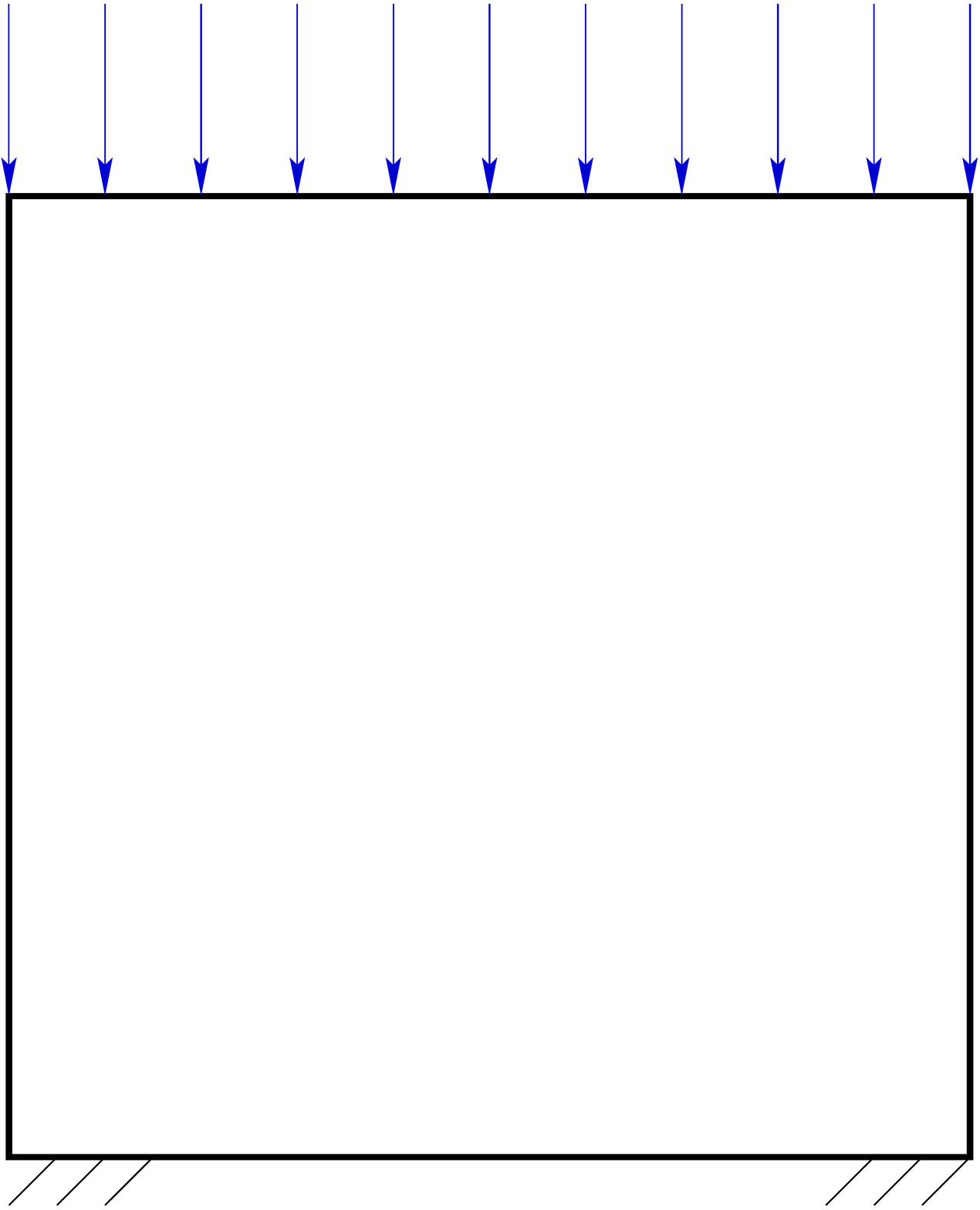}
\hspace{1.6cm} \ref{fig:radiator}-(a)
\end{center}
\end{minipage}
\begin{minipage}{0.4\hsize}
\begin{center}
\includegraphics[width=1.0\linewidth,center]{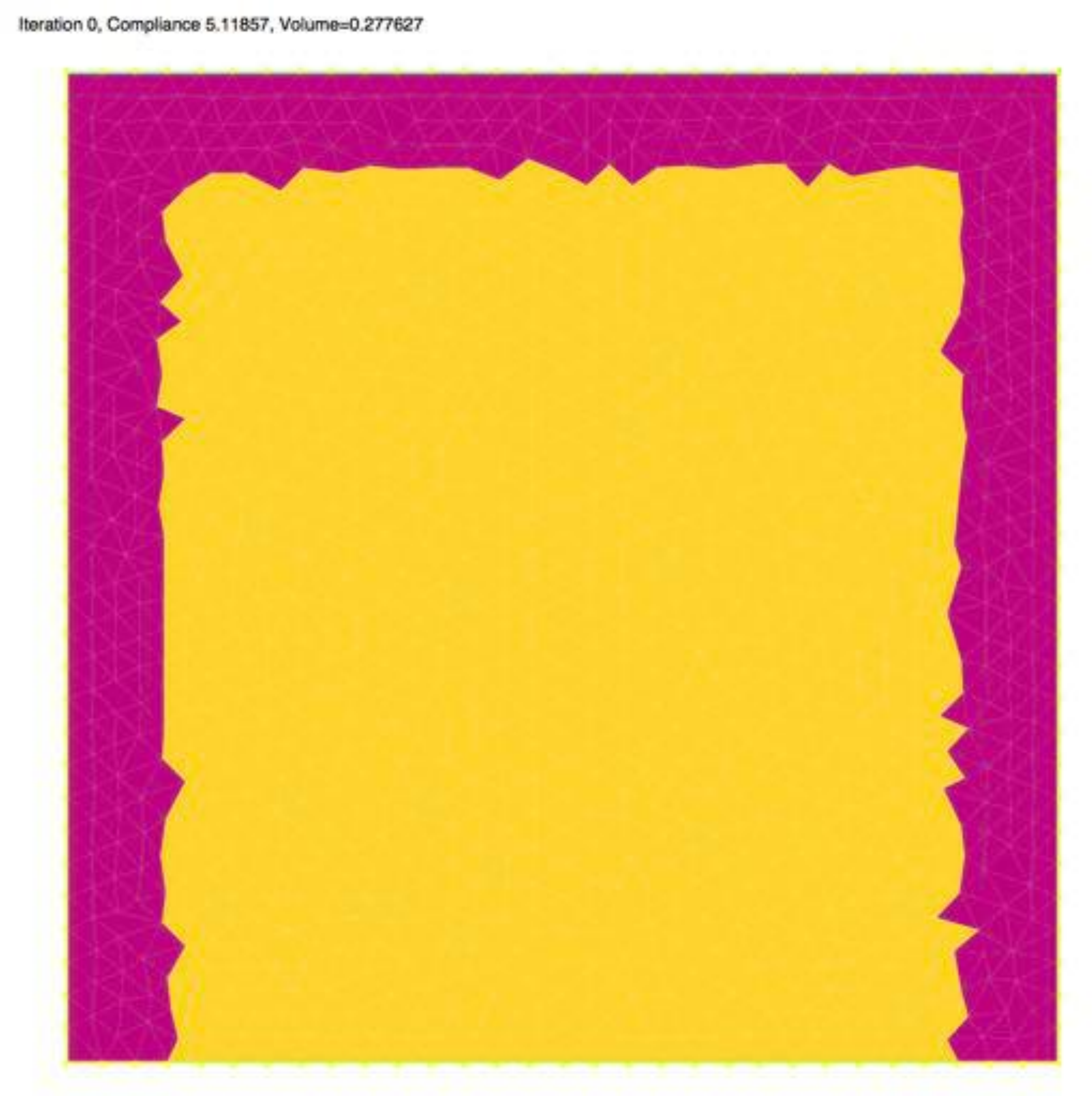}
\hspace{1.6cm} \ref{fig:radiator}-(b)
\end{center}
\end{minipage}\\
\begin{minipage}{0.3\hsize}
\begin{center}
\includegraphics[width=1.0\linewidth,center]{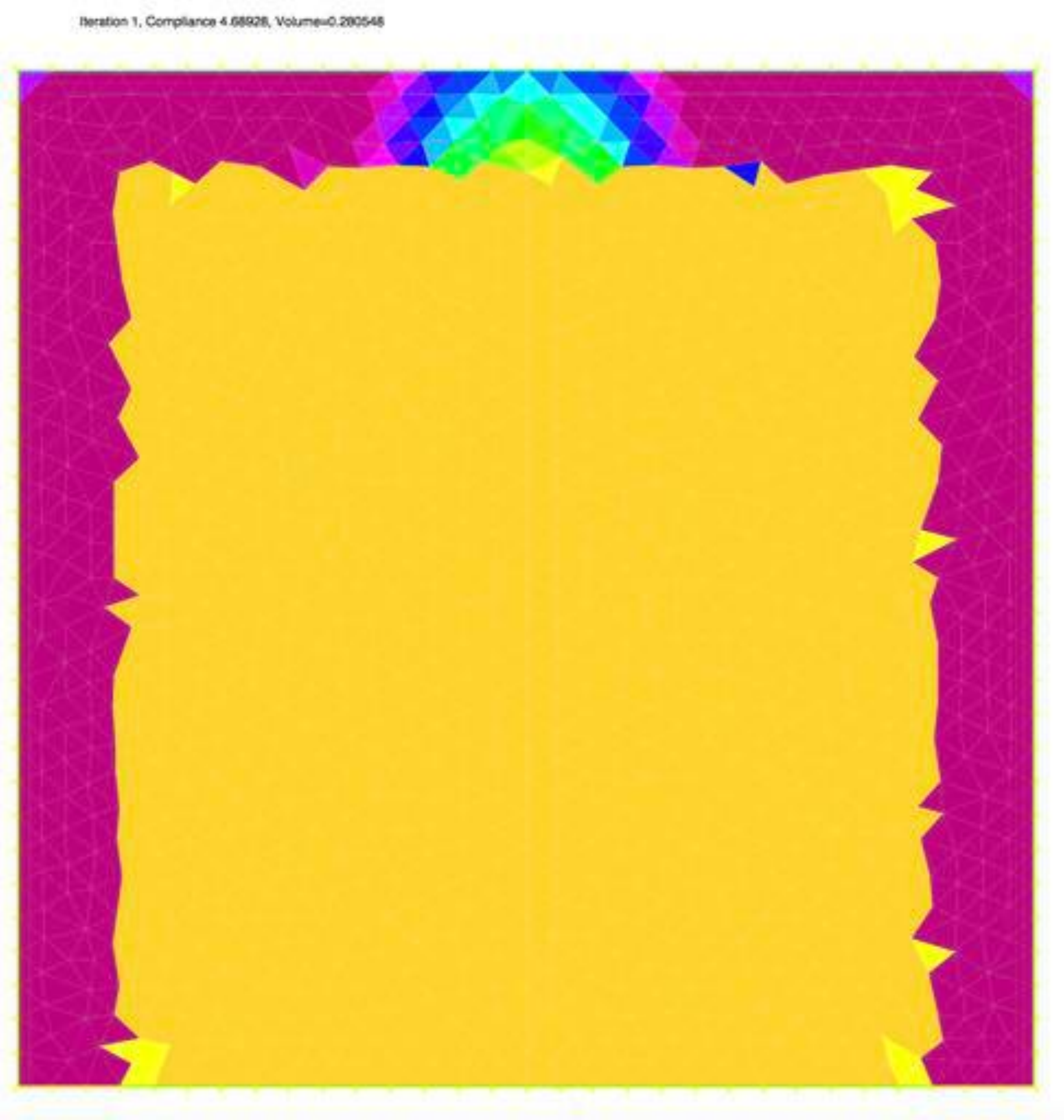}
\hspace{1.6cm} \ref{fig:radiator}-(c)
\end{center}
\end{minipage}
\begin{minipage}{0.3\hsize}
\begin{center}
\includegraphics[width=1.0\linewidth,center]{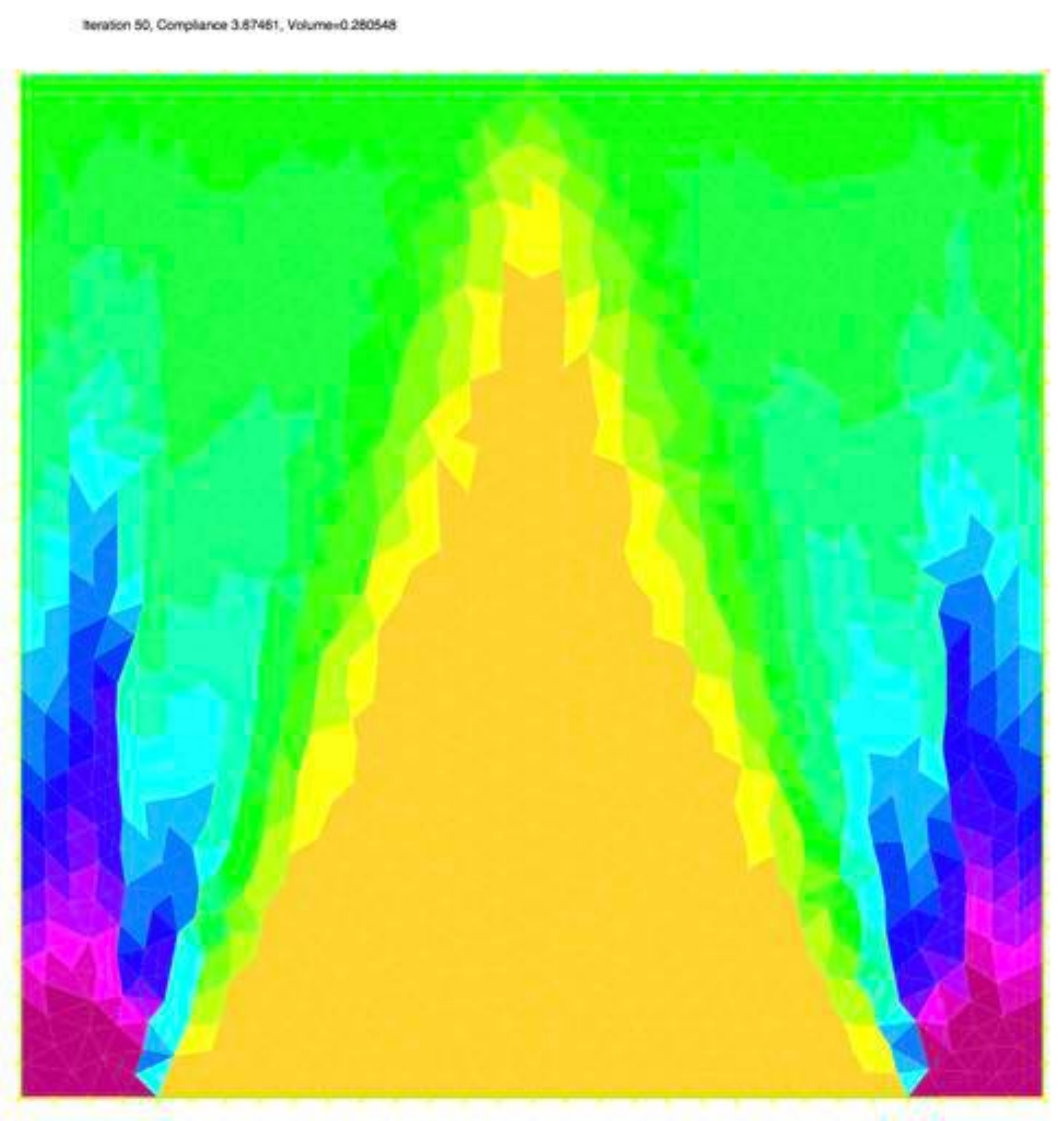}
\hspace{1.6cm} \ref{fig:radiator}-(d)
\end{center}
\end{minipage}
\begin{minipage}{0.3\hsize}
\begin{center}
\includegraphics[width=1.0\linewidth,center]{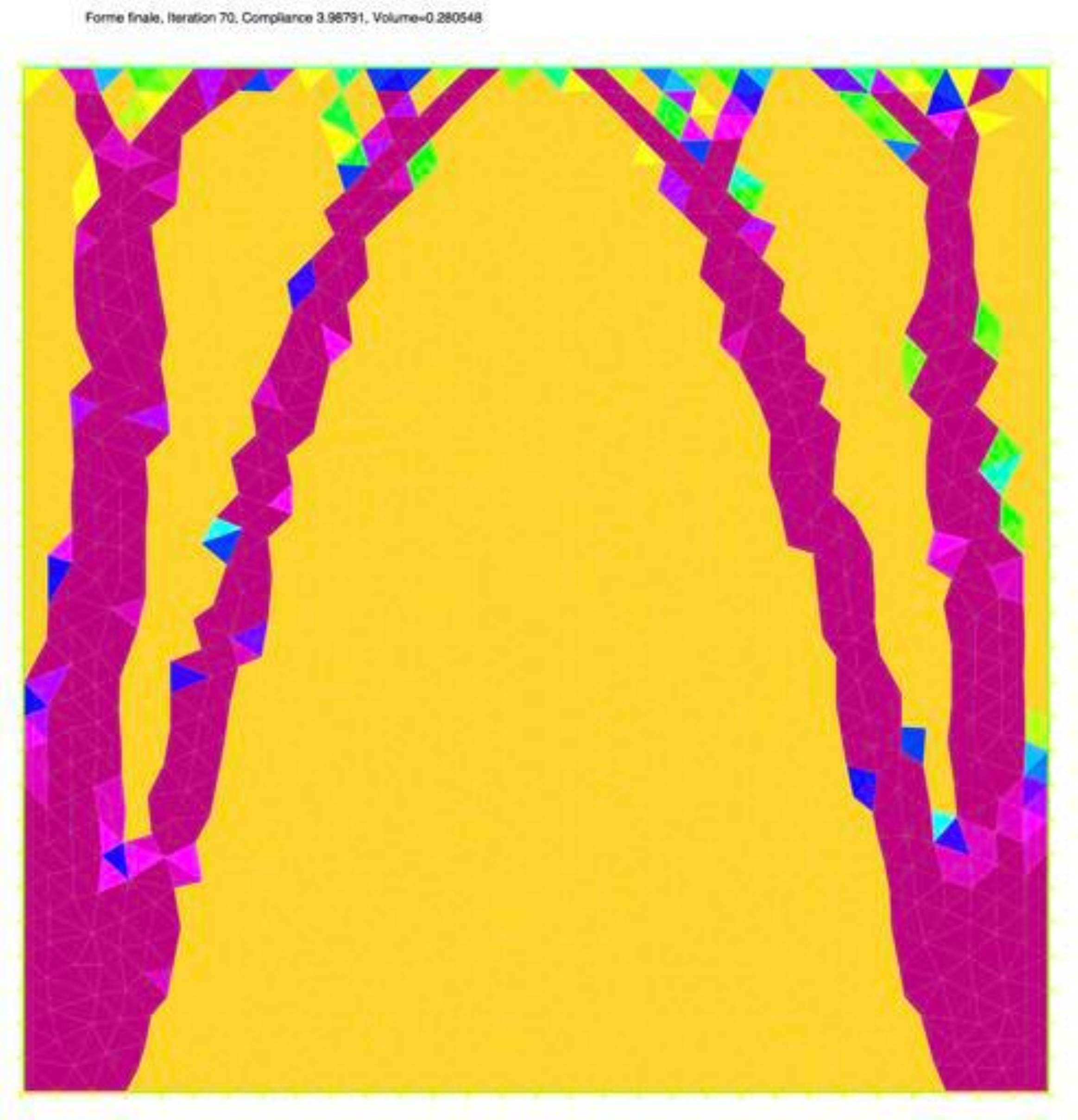}
\hspace{1.6cm} \ref{fig:radiator}-(e)
\end{center}
\end{minipage}
\end{tabular}
\caption{Optimal radiator: (a) setting of the problem, (b) initial condition, (c)-(e) volume fraction $\theta$ (iteration number 1, 50 and 70 respectively). Here we used the following color convention: red = 1, yellow = 0.} 
\label{fig:radiator}
\end{figure}


\section{Homogenization method in the elasticity setting}

In this section, we will apply the homogenization method in the elasticity setting. 
We remark that it is very similar to the conductivity setting but there are some additional hurdles.
We shall review the results without proofs, however, the basic ingredients of the homogenization method which we will consider in this section are the sames$\colon$
\begin{itemize}
	\item
   Introduction of composite designs characterized by $(\theta, A^*)$,
   \item
   Hashin-Shtrikman bounds for composites,
   \item
   Sequential laminates are optimal microstructures for compliance minimization, which we will consider the following.
\end{itemize}
We remark that, unfortunately, the full set of composites $G_\theta$ is unknown as stated in Section~\ref{sec:the elasticity setting}, unlike the case we considered in the Section~\ref{sec:hom con setting}.

\subsection{Introduction of the model of the elasticity and relaxed problem} \label{subsec:intro ela setting}

\begin{figure}[h]
\centering
\includegraphics[width=0.7\linewidth,center]{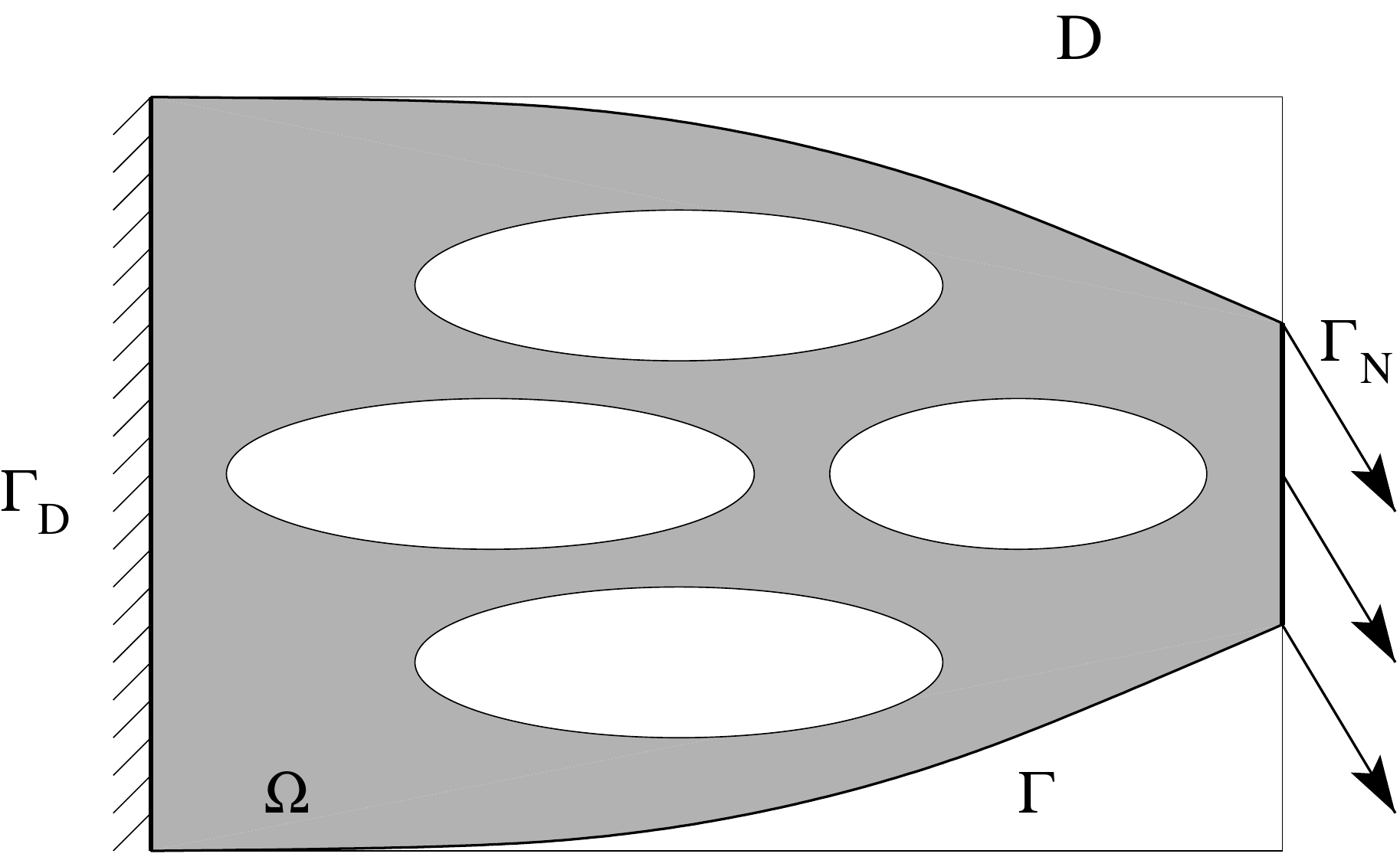}
\caption{Setting of the problem \eqref{eq:original compliance}}
\label{fig:compliance minimization}
\end{figure}

We introduce the model compliance minimization problem (Figure~\ref{fig:compliance minimization}).
Let $N=2$ or $3$ and $D\subset\RR^N$ be a bounded domain.
If the Hooke's law $A$ is isotropic, with positive bulk and shear moduli $\kappa$ and $\mu$, we have
\begin{equation*}
A=\left(\kappa-\dfrac{2\mu}{N}\right)I_2\otimes I_2+2\mu I_4.
\end{equation*}
Let $\Gamma_D\subset \partial D$ be the Dirichlet part and $\Gamma_N\subset \partial D$ be the Neumann part loaded by $g$.
For any domain $\Omega\subset D$ with $\Gamma_D$, $\Gamma_N\subset \partial \Omega$, the displacement vector field $u\colon\Omega\to\RR^N$ is defined as the solution of the problem
\begin{equation} \label{eq:original compliance}
	\left\{
    \begin{aligned}
    	\dv \, \sigma &=0 && \text{ in } \Omega,\\
      \sigma&=2\mu e(u)+\lambda {\rm tr}(e(u)){\rm Id} &&\text{ in } \Omega,\\
      u &=0 && \text{ on } \Gamma_D, \\
      \sigma\cdot n&=g && \text{ on } \Gamma_N, \\
      \sigma\cdot n&=0 && \text{ on }  \Gamma,
    \end{aligned}
	\right.
\end{equation}
where $\Gamma:=\partial\Omega\setminus(\Gamma_D\cup\Gamma_N)$, $e(u):=(\nabla u+(\nabla u)^t)/2$ and $\lambda:=\kappa-2\mu/N$.
We consider the following minimization problem to obtain the optimal shape such that the weight is minimized and the rigidity is maximized:
\begin{equation} \label{eq:ela original prob}
	\inf_{\Omega\subset D}
    \left\{
    	J(\Omega)=\int_{\Gamma_N}g\cdot u\,ds+\ell\int_\Omega\,dx
    \right\},
\end{equation}
where $\ell$ is a Lagrange multiplier and the infimum is taken over the all subset $\Omega$ of $D$ with $\Gamma_D$, $\Gamma_N\subset\partial\Omega$.

The shape optimization problem \eqref{eq:ela original prob} can be approximated by a two-phase optimization problem:
the original material $A$ and the holes of rigidity $B\approx0$.
Then the Hooke's law of the mixture in $D$ is rewritten as
\begin{equation*}
	\chi_{\Omega}(x)A+(1-\chi_\Omega(x))B\approx\chi_\Omega(x)A, \quad x\in D.
\end{equation*}
Hence the admissible set becomes
\begin{equation*}
	\mathcal{U}_{\rm ad}=\left\{\chi\in L^\infty(D; \{0, 1\})\right\}.
\end{equation*}
As in conductivity (membrane) case, we can apply the relaxation approach based on homogenization theory.

We introduce composite structures characterized by a local volume fraction $\theta(x)$ of the phase $A$ and a homogenized tensor $A^*(x)$, corresponding to its microstructure.
The set of admissible homogenized designs is
\begin{equation*}
	\mathcal{U}^*_{\rm ad}=
    \left\{
    	(\theta, A^*)\in L^\infty\left(D; [0, 1]\times\RR^{N^4}\right) \;:\; A^*(x)\in G_{\theta(x)} \ {\rm for \ a.e.} \ x\in D
    \right\},
\end{equation*}
where, for fixed $x$ in $D$, $G_\theta(x)$ denotes the set of all possible two-phase composite materials at fixed volume fraction $\theta(x)$.
In this case, the homogenized state equation is 
\begin{equation} \label{eq:linear ela}
	\left\{
    \begin{aligned}
    	\dv\,\sigma&=0 && \text{ in } D,\\
      \sigma&=A^*e(u) &&\text{ in } D,\\
      u&=0 && \text{ on } \Gamma_D, \\
      \sigma\cdot n&=g && \text{ on } \Gamma_N, \\
      \sigma\cdot n&=0 && \text{ on } \partial D\setminus(\Gamma_D\cup\Gamma_N)
    \end{aligned}
	\right.
\end{equation}
and the homogenized compliance is defined by
\begin{equation*}
	c(\theta, A^*)=\int_{\Gamma_N}g\cdot u\,ds.
\end{equation*}
By the above setting, the relaxed or homogenized optimization problem is derived as follows:
\begin{equation} \label{eq:ela hom prob}
	\min_{(\theta, A^*)\in\mathcal{U}^*_{\rm ad}}
   \left\{
   J(\theta, A^*)=c(\theta, A^*)+\ell\int_D \theta(x)\,dx
   \right\}.
\end{equation}

In the elasticity setting, an explicit characterization of $G_\theta$ is still lacking, it is a major inconvenience of the problem \eqref{eq:ela hom prob}.
Fortunately, for compliance one can replace $G_\theta$ by its explicit subset of laminated composites.
The key argument to avoid the knowledge of $G_\theta$ is that, thanks to the complementary energy minimization, the compliance can be rewritten as
\begin{equation*}
	c(\theta, A^*)=\int_{\Gamma_N}g\cdot u\,ds=\min_{\substack{
    \dv\,\sigma=0 \ {\rm in} \ D, \\
    \sigma\cdot n=g \ {\rm on} \ \Gamma_N,\\
    \sigma\cdot n=0 \ {\rm on} \ \partial D\setminus(\Gamma_N\cup\Gamma_D)
    }}
   \int_D (A^*)^{-1}\sigma\cdot\sigma\,dx.
\end{equation*}
The complementary energy is followed by a similar argument in Example~\ref{ex:dual energy}.
The shape optimization problem \eqref{eq:ela hom prob} thus becomes a double minimization problem
\begin{equation} \label{eq:ela double min}
	\min_{(\theta, A^*)\in \mathcal{U}^*_{\rm ad}}
    \left\{
    	\min_{\substack{
    \dv\,\sigma=0 \ {\rm in} \ D, \\
    \sigma\cdot n=g \ {\rm on} \ \Gamma_N,\\
    \sigma\cdot n=0 \ {\rm on} \ \partial D\setminus(\Gamma_N\cup\Gamma_D)
    }
    }
   \int_D (A^*)^{-1}\sigma\cdot\sigma\,dx
   +\ell\int_D \theta(x)\,dx
    \right\}.
\end{equation}
We will exchange the order of the minimization \eqref{eq:ela double min}.
Since the order of minimization is irrelevant, \eqref{eq:ela double min} can be rewritten as
\begin{equation*}
	\min_{\substack{
    \dv\,\sigma=0 \ {\rm in} \ D, \\
    \sigma\cdot n=g \ {\rm on} \ \Gamma_N,\\
    \sigma\cdot n=0 \ {\rm on} \ \partial D\setminus(\Gamma_N\cup\Gamma_D)
    }}
   \min_{(\theta, A^*)\in \mathcal{U}^*_{\rm ad}}
   \left\{
   \int_D (A^*)^{-1}\sigma\cdot\sigma\,dx
   +\ell\int_D \theta(x)\,dx
	\right\}.
\end{equation*}
The minimization with respect to the design parameters $(\theta, A^*)$ is local.
Hence the above minimization becomes
\begin{equation} \label{eq:ela double min2}
	\min_{\substack{
    \dv\,\sigma=0 \ {\rm in} \ D, \\
    \sigma\cdot n=g \ {\rm on} \ \Gamma_N,\\
    \sigma\cdot n=0 \ {\rm on} \ \partial D\setminus(\Gamma_N\cup\Gamma_D)
    }}
    \int_D \min_{\substack{
    0\le\theta\le1, \\
    A^*\in G_\theta
    }}
    \left(
    (A^*)^{-1}\sigma\cdot\sigma+\ell\theta
    \right)\,dx.
\end{equation}
For a given stress tensor $\sigma$, the minimization of complementary energy
\begin{equation*}
	\min_{A^*\in G_\theta}(A^*)^{-1}\sigma\cdot\sigma
\end{equation*}
is a classical problem in homogenization of finding optimal bounds on the effective properties of composite materials.
It turns out that we can restrict ourselves to sequential laminates which form an explicit subset $L_\theta$ of $G_\theta$ by Proposition~\ref{prop:rank-N laminate}.
Recall that in the conductivity setting, it was enough to consider only the case of rank-$1$ laminates.
On the other hand, Proposition~\ref{prop:rank-N laminate} tells us that, in the elasticity setting, rank-$1$ laminates are not enough and $G_\theta$ has to be replaced by the set of rank-$N$ sequential laminates instead.

As in the case of the conductivity setting in Section~\ref{sec:app hom method}, one can prove that the problem \eqref{eq:ela hom prob} is a relaxation of the original shape optimization in the following sense (for the details of the proof, see \cite[Theorem 4.1.12]{Allaire1}).
\begin{theorem}
	The homogenized formulation \eqref{eq:ela hom prob} is the relaxation of the original problem \eqref{eq:ela original prob} in the sense where
    \begin{itemize}
    	\item[{\rm 1}.]
       there exists, at least, one optimal composite $(\theta, A^*)$ minimizing \eqref{eq:ela hom prob},
       \item[{\rm 2}.]
       any minimizing sequence of classical shapes $\Omega$ for \eqref{eq:ela original prob} converges, in the sense of homogenization, to a minimizer $(\theta, A^*)$ of \eqref{eq:ela hom prob},
       \item[{\rm 3}.]
       the minimal values of the original and homogenized objective functions coincide.
    \end{itemize}
\end{theorem}




\subsection{An explicit optimal bound $HS(\sigma)$} \label{subsec:explicit formula HS}

As we stated in Section~\ref{subsec:intro ela setting}, we can restrict the set $G_\theta$ to the sequential laminates $L_\theta$.
In this subsection, we will show the explicit computation of $HS(\sigma)$ for a special case.
We will consider the case $B=0$ for the simplification of algebra.
Note that the case $B=0$ is natural since the weak material is actually degenerate.

In two dimension case, we can obtain an explicit formula for the bound (see \cite[Theorem 2.3.35]{Allaire1}):

\begin{theorem}
    Assume that $N=2$, $B=0$ and $\theta\ne0$.
    Then for any stress tensor $\sigma$, the optimal bound $HS(\sigma)$ is rewritten as 
    \begin{equation} \label{eq:HS explicit}
	HS(\sigma)=A^{-1}\sigma\cdot\sigma+\dfrac{1-\theta}{\theta}g^*(\sigma),
\end{equation}
    where
\begin{equation}\label{g* N=2}
	g^*(\sigma)=\dfrac{\kappa+\mu}{4\mu\kappa}(|\sigma_1|+|\sigma_2|)^2
\end{equation}
and $\sigma_1$, $\sigma_2$ are the eigenvalues of $\sigma$.
Furthermore, an optimal rank-$2$ sequential laminate is given by
\begin{equation*}
	m_1=\dfrac{|\sigma_2|}{|\sigma_1|+|\sigma_2|}, \quad m_2=\dfrac{|\sigma_1|}{|\sigma_1|+|\sigma_2|},
\end{equation*}
where $m_i$ is the parameters appeared in Lemma~\ref{lem seq lam}.
\end{theorem}

In the three dimension case, we can also obtain the explicit formula for the bound (see \cite[Theorem 2.3.36]{Allaire1}), however, it is more complicated than the two dimension case.
Hence we restrict ourselves to the simple case of zero Poisson ratio, i.e. $\kappa=2\mu/3$.

\begin{theorem}
    Assume that $N=3$, $B=0$ and $\theta\ne0$.
    We also assume that the constants $\kappa$ and $\mu$ satisfies $\kappa=2\mu/3$.
    For any $\sigma$, we label the eigenvalues of $\sigma$ as $|\sigma_1|\le|\sigma_2|\le|\sigma_3|$.
    Then, we obtain \eqref{eq:HS explicit} with
    \begin{equation*}
        g^*(\sigma)=
    \left\{
    \begin{aligned}
    	& \dfrac{1}{4\mu}(|\sigma_1|+|\sigma_2|+|\sigma_3|)^2 && \text{ if } |\sigma_3|\le|\sigma_1|+|\sigma_2|, \\
       & \dfrac{1}{2\mu}\left((|\sigma_1|+|\sigma_2|)^2+|\sigma_3|^2\right) &&\text{ if }  |\sigma_3|>|\sigma_1|+|\sigma_2|.
    \end{aligned}
    \right.
    \end{equation*}
    Furthermore, in the first regime, an optimal rank-$3$ sequential laminate is given by
\begin{equation*}
	m_1=\dfrac{|\sigma_3|+|\sigma_2|-|\sigma_1|}{|\sigma_1|+|\sigma_2|+|\sigma_3|}, \quad m_2=\dfrac{|\sigma_1|-|\sigma_2|+|\sigma_3|}{|\sigma_1|+|\sigma_2|+|\sigma_3|}, \quad m_3=\dfrac{|\sigma_1|+|\sigma_2|-|\sigma_3|}{|\sigma_1|+|\sigma_2|+|\sigma_3|},
\end{equation*}
and in the second regime, an optimal rank-$2$ sequential laminate is 
\begin{equation*}
	m_1=\dfrac{|\sigma_2|}{|\sigma_1|+|\sigma_2|}, \quad m_2=\dfrac{|\sigma_1|}{|\sigma_1|+|\sigma_2|}, \quad m_3=0,
\end{equation*}
where $m_i$ is the parameters appeared in Lemma~\ref{lem seq lam}.
\end{theorem}

\subsection{Optimality conditions}

We consider the optimality condition for the minimization problem \eqref{eq:ela double min2}.
In this subsection, we assume the same condition as in Section~\ref{subsec:explicit formula HS}, i.e., $B=0$ and $\theta\ne0$.
If $(\theta, A^*, \sigma)$ is a minimizer, then by Proposition~\ref{prop:rank-N laminate}, $A^*$ is a rank-$N$ sequential laminate aligned with $\sigma$.
Moreover, Proposition~\ref{prop:lamination formula} leads the explicit proportions
\begin{equation*}
	(A^*)^{-1}=A^{-1}+\dfrac{1-\theta}{\theta}
    \left(
    \sum^N_{i=1}m_i f^c_A(e_i)
    \right)^{-1}.
\end{equation*}
If we consider the case $N=2$, then, by \eqref{g* N=2}, we can rewrite the minimization appearing in the integrand \eqref{eq:ela double min2} as
\begin{equation*}
A^{-1}\sigma\cdot\sigma+\min_{0<\theta\le1}\left(\dfrac{(\kappa+\mu)(1-\theta)}{4\mu\kappa\theta}(|\sigma_1|+|\sigma_2|)^2+\ell\theta\right).
\end{equation*}
Hence we obtain the explicit optimality formula for $\theta$ as follows:
\begin{equation}\label{theta opt}
	\theta_{\rm opt}=\min\left(
   1, \sqrt{\dfrac{\kappa+\mu}{4\mu\kappa \ell}}(|\sigma_1|+|\sigma_2|)
   \right).
\end{equation}
The explicit formula for $\theta$ in the case $N=3$ follows by a similar argument.





\subsection{Numerical algorithm}

In this subsection, we will introduce a numerical algorithm for the minimization problem \eqref{eq:ela hom prob}.
We use the following double ``alternating'' minimization in $\sigma$ and in $(\theta, A^*)$:

\begin{algorithm}[H]
\caption{Double alternating minimization for \eqref{eq:ela hom prob}}
\begin{itemize}
	\item[1.] Initialization of the shape $(\theta_0, A^*_0)$ by a finite element method.
   \item[2.] Iterations until convergence, for $n\ge1$:
   \begin{itemize}
   		\item[{\rm -}]
        Given a shape $(\theta_{n-1}, A^*_{n-1})$, we compute the stress $\sigma_n$ by solving the linear elasticity problem \eqref{eq:linear ela},
       \item[{\rm -}]
       Given the stress field $\sigma_n$, we update the new design parameters $(\theta_n, A^*_n)$ by the explicit optimality formula \eqref{theta opt} in terms of $\sigma_n$. 
   \end{itemize}
\end{itemize}
\end{algorithm}

Since the problem is self-adjoint, we can exchange the problem which we can consider the explicit optimality formula and the fact allows us to consider the above numerical algorithm.
The algorithm uses a local microstructure $A^*$ and a global density $\theta$.
Such algorithm is called micro-macro method.

\begin{remark}
The objective function always decreases.
Indeed, since $(\theta_{n}, A^*_{n})$ minimizes the compliance under the stress $\sigma_{n}$, we have
\begin{equation*}
\int_D(A^*_{n-1})^{-1}\sigma_n\cdot\sigma_n\,dx+\ell\int_D\theta_{n-1}\,dx\ge\int_D(A^*_{n})^{-1}\sigma_n\cdot\sigma_n\,dx+\ell\int_D\theta_{n}\,dx.
\end{equation*}
On the other hand, $\sigma_{n+1}$ minimizes the elastic complementary energy corresponding to the Hooke's law $A^*_n$, we see that
\begin{equation*}
\int_D(A^*_{n})^{-1}\sigma_n\cdot\sigma_n\,dx\ge\int_D(A^*_{n})^{-1}\sigma_{n+1}\cdot\sigma_{n+1}\,dx.
\end{equation*}
Combining the above inequalities, we obtain the claim
\begin{equation*}
J(\theta_{n-1}, A^*_{n-1})\ge J(\theta_n, A^*_n).
\end{equation*}
\end{remark}

We show the numerical results for the case of the cantilever (Figure~\ref{fig:cantilever}).
The optimal shape of the short cantilever, i.e., the domain size 10$\times$20, is displayed on the left of Figure~\ref{fig:short cantilever}.
Moreover, the left figure of Figure~\ref{fig:short cantilever convergence history} shows the convergence history of the objective function $J$.
Here, the horizontal axis means the iteration number.
The right figure of Figure~\ref{fig:short cantilever convergence history} shows the transition of the quantity
\begin{equation*}
\max\left(\max_{i}|\theta^{k+1}_i-\theta^k_i|, 1-\dfrac{\displaystyle\int_\Omega(A^*_{k+1})^{-1}\sigma^{k}\cdot\sigma^{k}\,dx+l\int_\Omega\theta^{k+1}\,dx}{\displaystyle\int_\Omega(A^*_k)^{-1}\sigma^{k-1}\cdot\sigma^{k-1}\,dx+l\int_\Omega\theta^k\,dx}\right),
\end{equation*}
where the index $i$ refers to the cell number.
We also show the optimal shape of the medium cantilever, i.e., the domain size is 20$\times$10 in the left of Figure~\ref{fig:medium cantilever}.

\begin{figure}[H]
\centering
\includegraphics[width=0.6\linewidth,center]{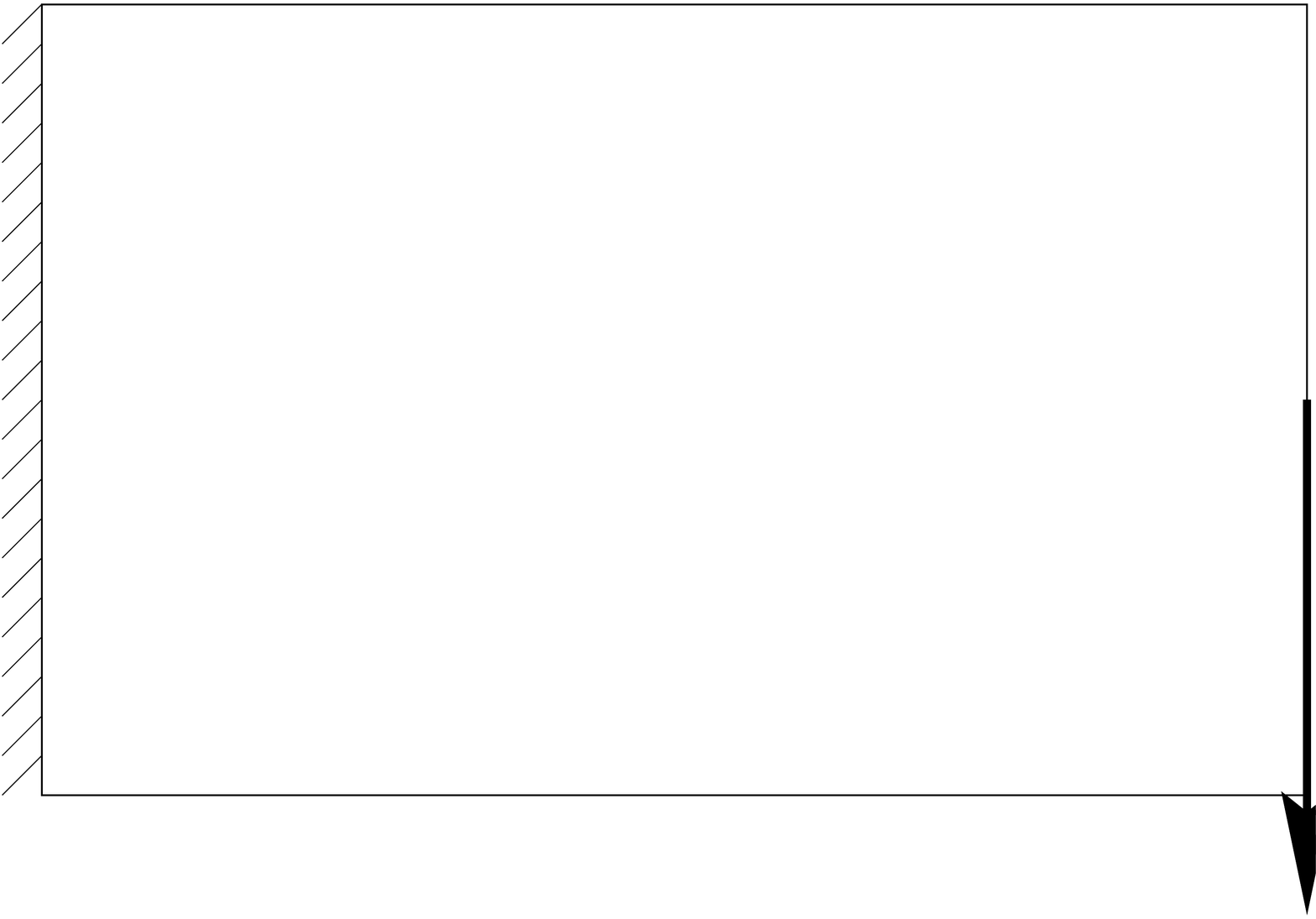}
\caption{Boundary conditions for the cantilever problem} 
\label{fig:cantilever}
\end{figure}

\begin{figure}[H]
\centering
\begin{tabular}{c}
\begin{minipage}{0.3\hsize}
\begin{center}
\includegraphics[width=\linewidth,center]{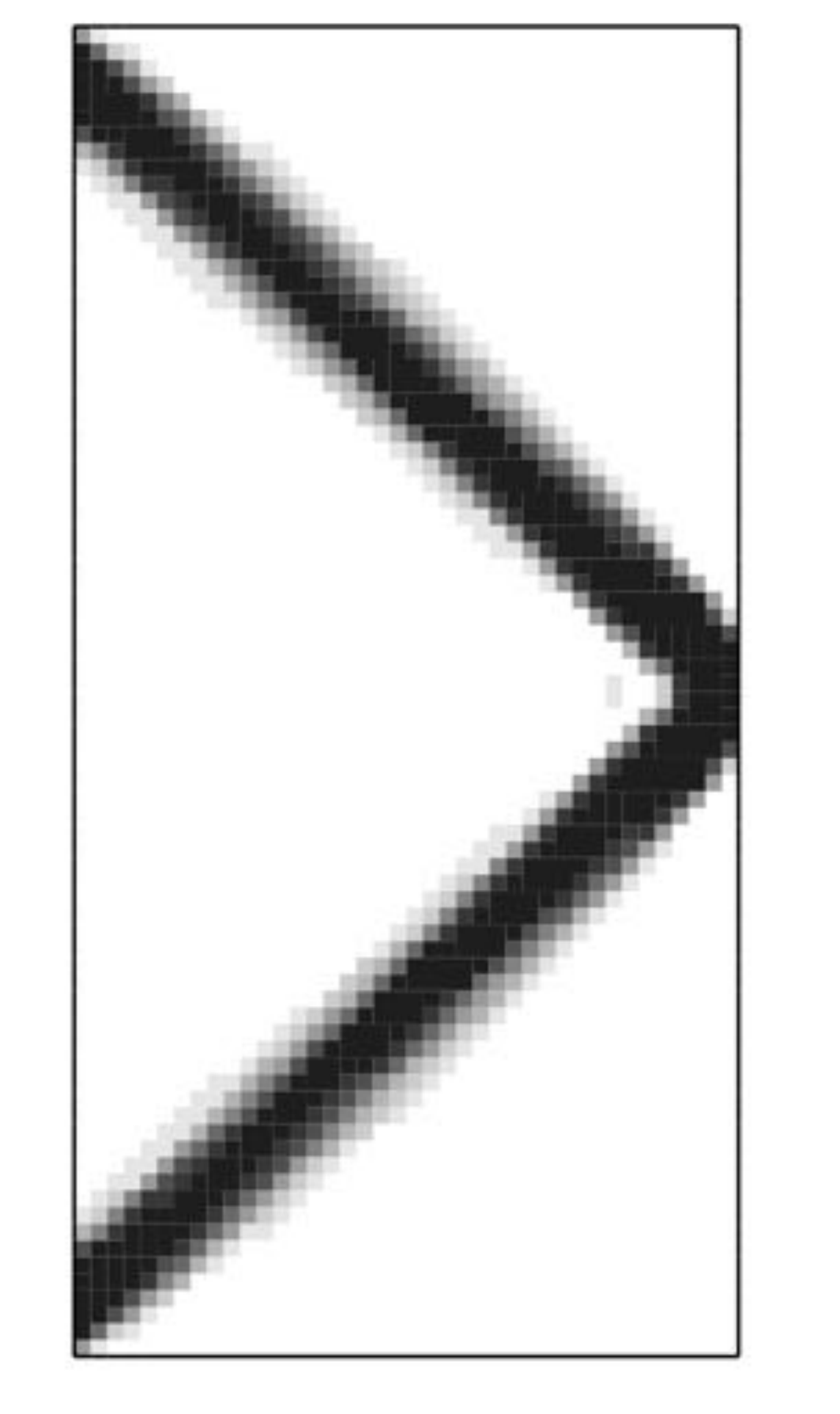}
\end{center}
\end{minipage}
\begin{minipage}{0.3\hsize}
\begin{center}
\includegraphics[width=1.0\linewidth,center]{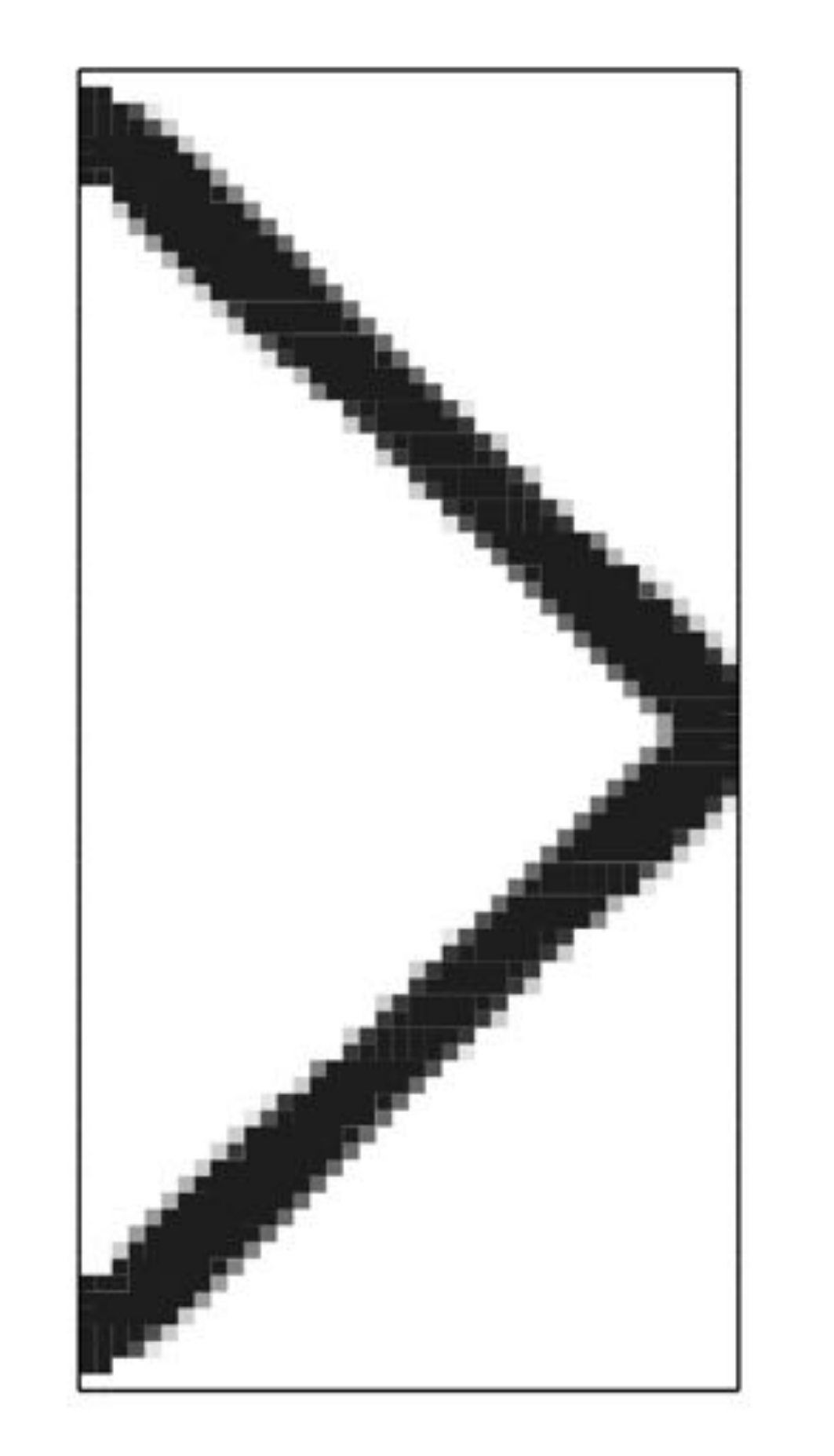}
\end{center}
\end{minipage}
\end{tabular}
\caption{Optimal shape of the short cantilever (left:composite, right:penalized)}
\label{fig:short cantilever}
\end{figure}

\begin{figure}[H]
\centering
\begin{tabular}{c}
\begin{minipage}{0.35\vsize}
\begin{center}
\includegraphics[width=1.0\linewidth,center]{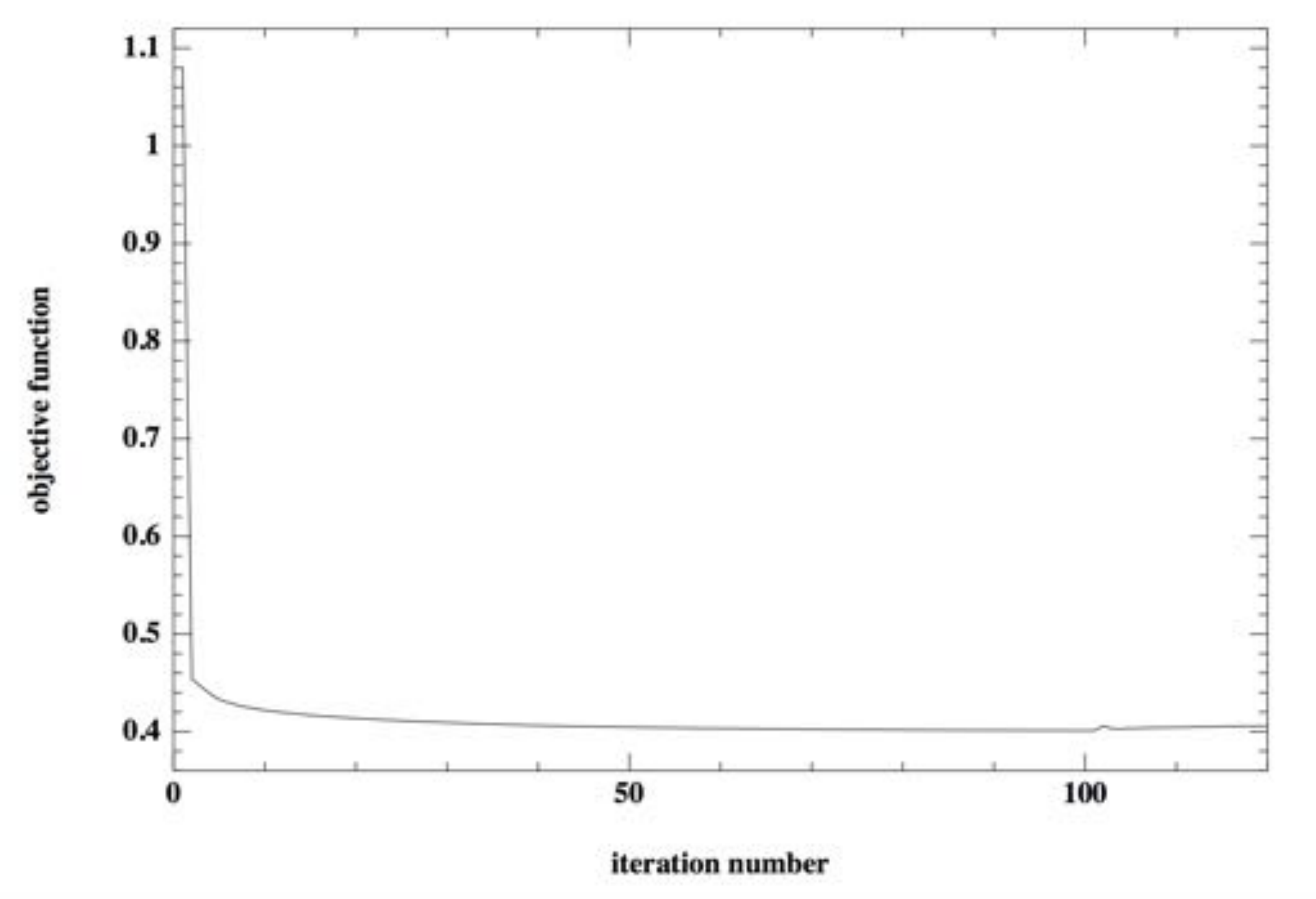}
\end{center}
\end{minipage}
\begin{minipage}{0.35\vsize}
\begin{center}
\includegraphics[width=1.0\linewidth,center]{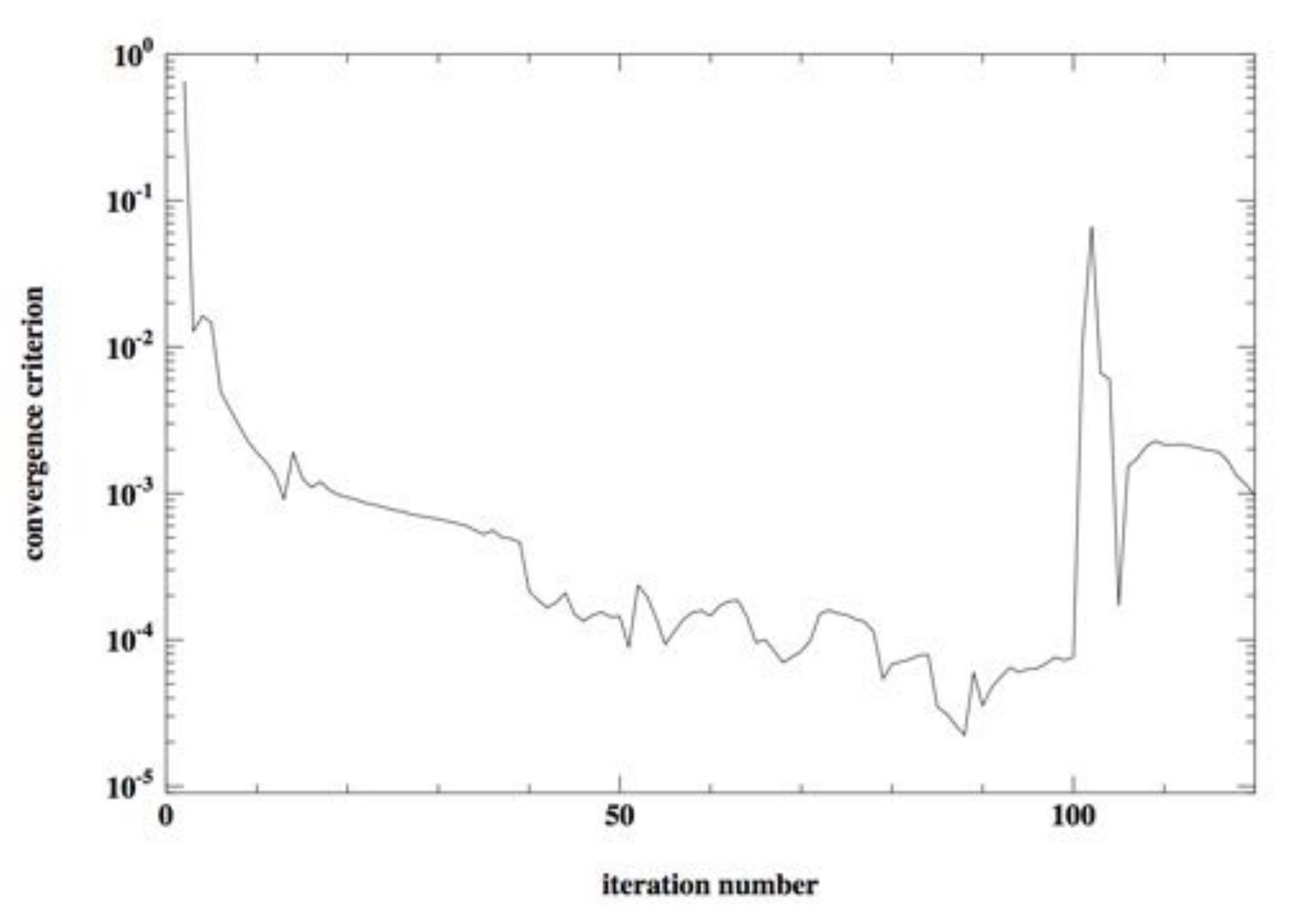}
\end{center}
\end{minipage}
\end{tabular}
\caption{Convergence history of the short cantilever (left:objective function, right:convergence criterion)} 
\label{fig:short cantilever convergence history}
\end{figure}
\begin{figure}[H]
\centering
\begin{tabular}{c}
\begin{minipage}{0.35\vsize}
\begin{center}
\includegraphics[width=1.0\linewidth,center]{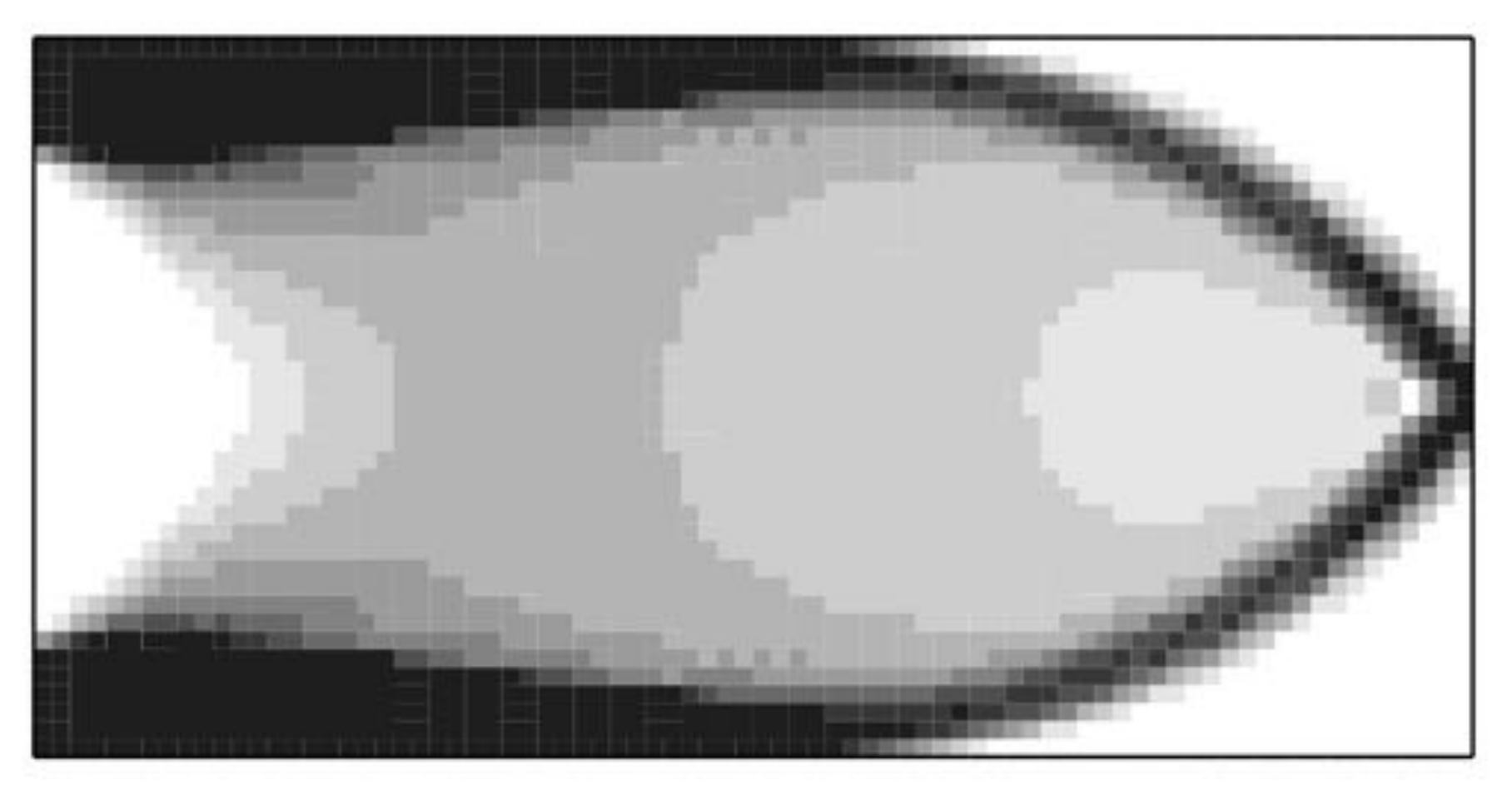}
\end{center}
\end{minipage}
\begin{minipage}{0.35\vsize}
\begin{center}
\includegraphics[width=1.0\linewidth,center]{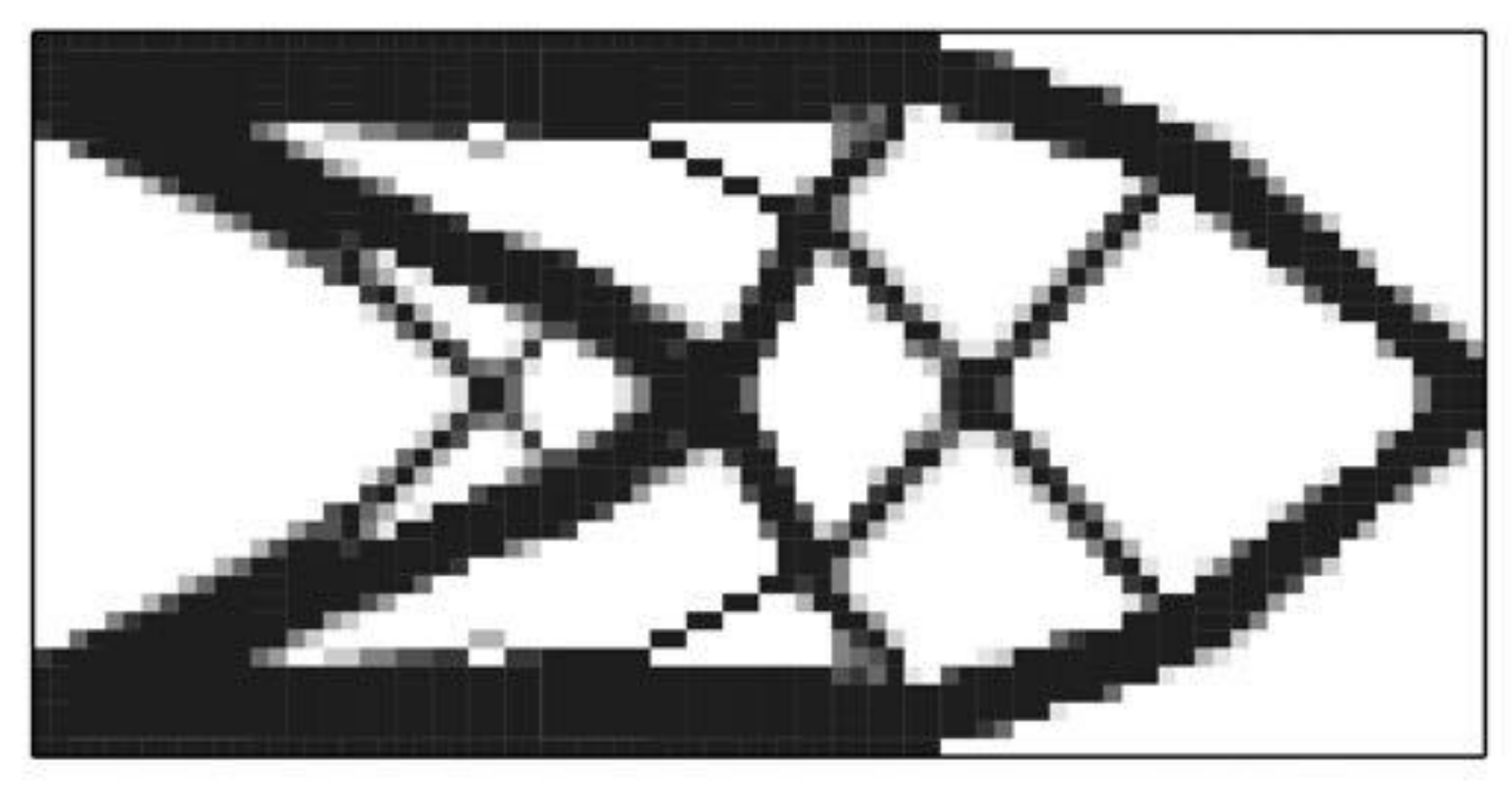}
\end{center}
\end{minipage}
\end{tabular}
\caption{Optimal shape of the medium cantilever (left:composite, right:penalized)} 
\label{fig:medium cantilever}
\end{figure}



\subsubsection{Penalization}

The algorithm, we considered in this subsection, compute composite shapes instead of classical shapes.
We thus use a penalization technique to force the density to take values close to $0$ or $1$ as in Section~\ref{subsec:numerical algorithm conductivity}.
The algorithm is the following$\colon$

\begin{algorithm}[H]
\caption{Penalization process for \eqref{eq:ela hom prob}}
After convergence to a composite shape, we perform a few more iterations with a penalized density
\begin{equation*}
	\theta_{\rm pen}=\dfrac{1-{\rm cos}(\pi\theta_{\rm opt})}{2}.
\end{equation*}
\end{algorithm}

Note that if $0<\theta_{\rm opt}<1/2$, then $\theta_{\rm pen}<\theta_{\rm opt}$, while, if $1/2<\theta_{\rm opt}<1$, then $\theta_{\rm pen}>\theta_{\rm opt}$.
By using this algorithm, we can obtain the penalized resulting shape of the short cantilever and the medium cantilever (see the right figure of Figure~\ref{fig:short cantilever} and \ref{fig:medium cantilever} respectively).
We also show the numerical result for the bridge problem (Figure~\ref{fig:bridge} and Figure~\ref{fig:bridge optimal shapes}).

\begin{figure}[h]
\centering
\includegraphics[width=0.5\linewidth,center]{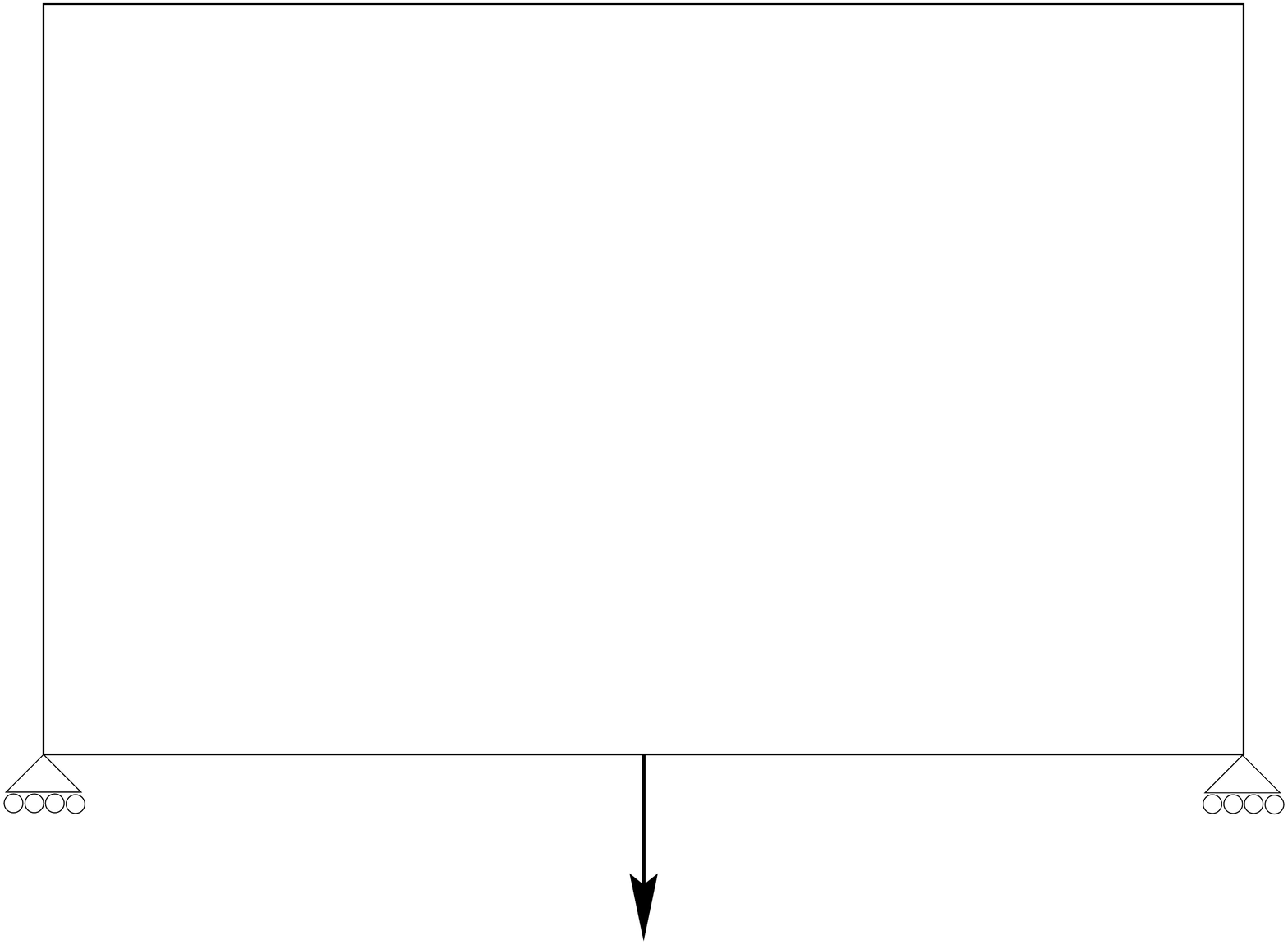}
\caption{Boundary conditions for the bridge problem} 
\label{fig:bridge}
\end{figure}

\begin{figure}[h]
\centering
\begin{tabular}{c}
\begin{minipage}{0.3\vsize}
\begin{center}
\includegraphics[width=1.0\linewidth,center]{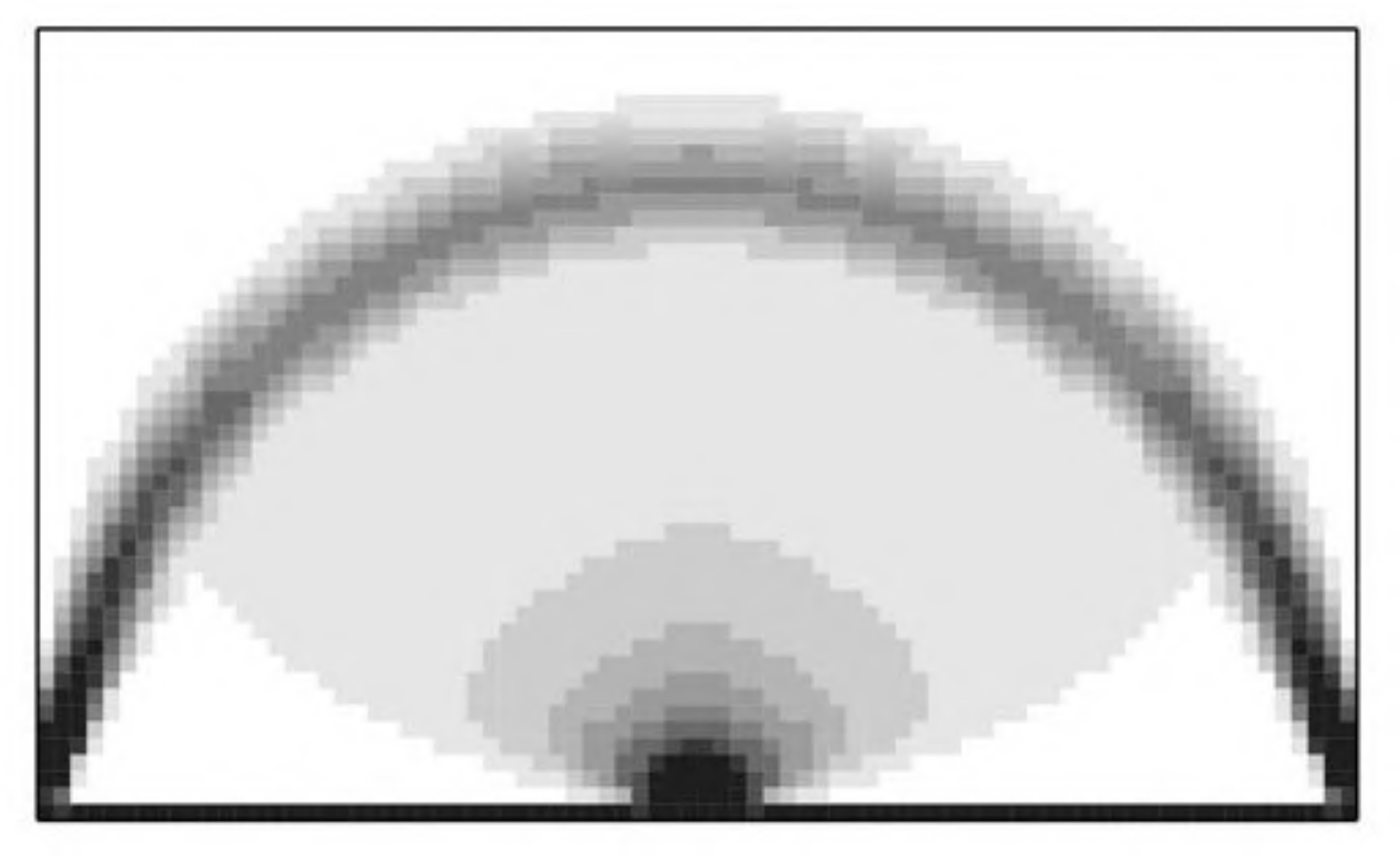}
\end{center}
\end{minipage}
\begin{minipage}{0.3\vsize}
\begin{center}
\includegraphics[width=1.0\linewidth,center]{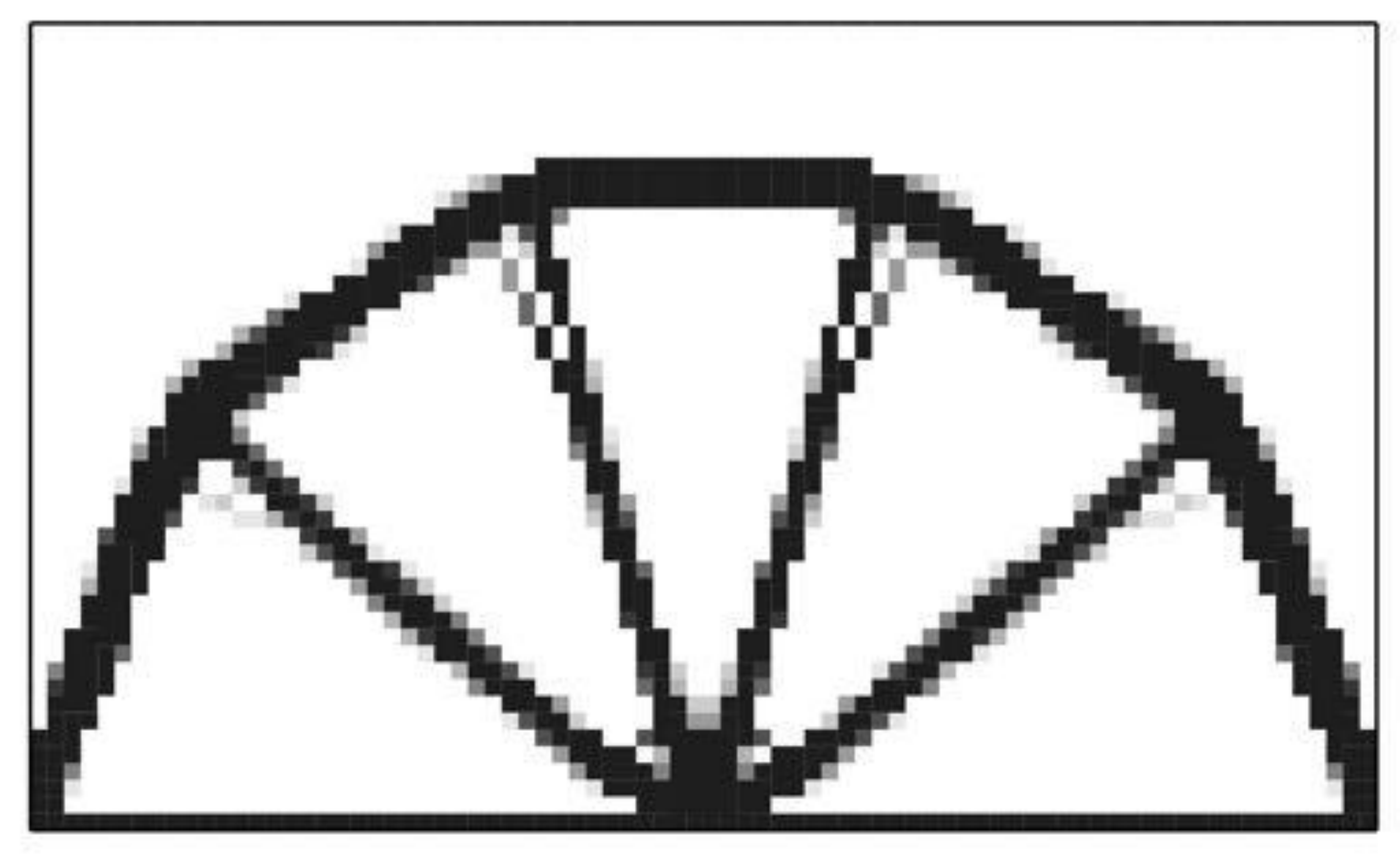}
\end{center}
\end{minipage}
\end{tabular}
\caption{Optimal shape of the bridge (left: composite, right: penalized)} 
\label{fig:bridge optimal shapes}
\end{figure}

\section{Convexification, ``fictitious materials" and SIMP}
\label{sec:SIMP}
In the homogenization method, composite materials are introduced, however, discarded at the end by penalization.
In this section, we will consider whether we can simplify the approach by introducing merely a density $\theta$.
We will use the following "convexification" approach.
A classical shape is parametrized by characteristic functions $\chi(x)$.
If we convexify this admissible set, we obtain $\theta(x)\in[0, 1]$.
Replacing the admissible set by the convexified set, the Hooke's law, which was $\chi(x)A$, becomes $\theta(x)A$.
We will call these $\theta(x)A$ ``fictitious materials", because one can not realize them by a true homogenization process in general.
Combined with a penalization scheme, this method is called SIMP method.

We consider the elasticity setting.
For $\theta\in L^\infty(D; [0, 1])$, the convexified formulation for the problem reads as follows:
\begin{equation*}
	\left\{
    \begin{aligned}
    	\dv\,\sigma &=0 && \text{ in } D,\\
      \sigma&=\theta Ae(u) &&\text{ in } D,\\
       u &=0 && \text{ on } \Gamma_D, \\
      \sigma\cdot n&=g && \text{ on } \Gamma_N, \\
      \sigma\cdot n&=0 && \text{ on } \partial D\setminus(\Gamma_D\cup\Gamma_N),
    \end{aligned}
	\right.
\end{equation*}
where $e(u):=(\nabla u+(\nabla u)^t)/2$.
Moreover, the compliance minimization becomes
\begin{equation} \label{eq:ela convexify min}
	\min_{\theta\in L^\infty(D; [0, 1])}\left\{
    c(\theta)+\ell\int_D \theta(x)\,dx
    \right\}
\end{equation}
with
\begin{equation*}
	c(\theta)=\int_{\Gamma_N}g\cdot u\,ds=\int_D(\theta(x)A)^{-1}\sigma\cdot\sigma\,dx=\min_{\substack{
    -\dv\, \tau=0 \ {\rm in} \ D, \\
    \tau\cdot n=g \ {\rm on} \ \Gamma_N,\\
    \tau\cdot n=0 \ {\rm on} \ \partial D\setminus(\Gamma_N\cup\Gamma_D)
    }}
   \int_D (\theta(x)A)^{-1}\tau\cdot\tau\,dx.
\end{equation*}
There is only one single design parameter, the material density $\theta$. In other words, any information concerning the microstructure $A^*$ has disappeared.

\subsection{Existence of solutions}

In this section, we will show the existence of the minimizer of \eqref{eq:ela convexify min}. 

\begin{theorem}
	The convexified formulation
   \begin{equation*}
   	\min_{\theta\in L^\infty(D; [0, 1])}\min_{\substack{
    -\dv \,\tau=0 \text{ in } D, \\
    \tau\cdot n=g \text{ on } \Gamma_N,\\
    \tau\cdot n=0 \text{ on } \partial D\setminus(\Gamma_N\cup\Gamma_D)
    }}
   \left(
   \int_D (\theta(x)A)^{-1}\tau\cdot\tau\,dx+\ell\int_D \theta(x)\,dx
   \right)
   \end{equation*}
   admits at least one solution.
\end{theorem}

\begin{proof}
    Let $\mathcal{M}^s_N$ be the set of symmetric squared matrices of order $N$.
	The function
    \begin{equation*}
    	\phi(a, \sigma)=a^{-1}A^{-1}\sigma\cdot\sigma,\quad (a, \sigma)\in\RR_{\ge0}\times\mathcal{M}^s_N
    \end{equation*}
    is convex because
    \begin{equation}
    	\begin{aligned}
    	\phi(a, \sigma)
        &=\phi(a_0, \sigma_0)+D\phi(a_0, \sigma_0)\cdot(a-a_0, \sigma-\sigma_0)+\phi(a, \sigma-aa^{-1}_0\sigma_0)\\
        &\geq\phi(a_0, \sigma_0)+D\phi(a_0, \sigma_0)\cdot(a-a_0, \sigma-\sigma_0),
        \end{aligned}
    \end{equation}
    where the derivative $D\phi$ is given by
    \begin{equation*}
    	D\phi(a_0, \sigma_0)\cdot(b, \tau)=-\dfrac{b}{a_0^2}A^{-1}\sigma_0\cdot\sigma_0+2a^{-1}_0A^{-1}\sigma_0\cdot\tau.
    \end{equation*}
    Then, by Theorem~\ref{existence of min in infinite dim convex space} we can see that there exists a minimizer of \eqref{eq:ela convexify min}.
\end{proof}

\subsection{Optimality condition}
If we exchange the minimizations in $\tau$ and in $\theta$, we can compute the optimal $\theta$ which is
\begin{equation} \label{eq:explicit theta}
	\theta(x)=
    \left\{
    \begin{aligned}
    	& 1 && \quad {\rm if} \ A^{-1}\tau\cdot\tau\ge \ell,\\
      & \sqrt{\ell^{-1}A^{-1}\tau\cdot\tau} && \quad {\rm if} \ A^{-1}\tau\cdot\tau< \ell.
    \end{aligned}
    \right.
\end{equation}
By using this explicit optimality formula, we can use again an ``alternating'' double minimization algorithm which will be shown in the following subsection.

\subsection{Numerical algorithm}

By the use of the explicit optimality formula \eqref{eq:explicit theta} for $\theta$, we can apply the following double minimization algorithm.

\begin{algorithm}[H]
\caption{Double minimization algorithm and penalization for \eqref{eq:ela convexify min}}
\begin{itemize}
	\item[1.]
    Initialization of the shape $\theta_0$,
   \item[2.]
   Iterations $k\ge1$ until convergence
   \begin{itemize}
   		\item[{\rm -}]
      given a shape $\theta_{k-1}$, we compute the stress $\tau_k$ by solving an elasticity problem by a finite element method,
      \item[{\rm -}]
      given a stress field $\tau_k$, we update the new material density $\theta_k$ with the explicit optimality formula in terms of $\tau_k$.
   \end{itemize}
\end{itemize}
As a penalization, we use the penalized density
\begin{equation*}
	\theta_{\rm pen}=\dfrac{1-{\rm cos}(\pi\theta_{\rm opt})}{2}
\quad\text{ or }\quad
	\theta_{\rm pen}=\theta^p,
\end{equation*}
where $p>1$ if we consider the SIMP method.
\end{algorithm}

In practice, it is extremely simple, however, the numerical results are not as good. This can be explained as follows: since the SIMP method uses very little information on the given composites (in particular, it neglects its microstructure all together), we have a lack of a relaxation theorem. In other words, applying the SIMP method changes the problem, and, as a consequence, we have no guarantee that the optimal solution obtained is really an approximation of the optimal solution of the original problem. 
Moreover, we have to be careful, as it could be very delicate to monitor the penalization.
We show the numerical result for the case of the bridge problem (Figure~\ref{fig:bridge}).
In Figure~\ref{fig:bridge convergence history}, we show the convergence history of the objective function.
Here, the horizontal axis means the iteration number.
In Figure~\ref{fig:bridge convexify}, we show the numerical result of the optimal shape of the bridge by the algorithm which is mentioned in this subsection.
\begin{figure}[H]
\centering
\begin{center}
\includegraphics[width=0.5\linewidth,center]{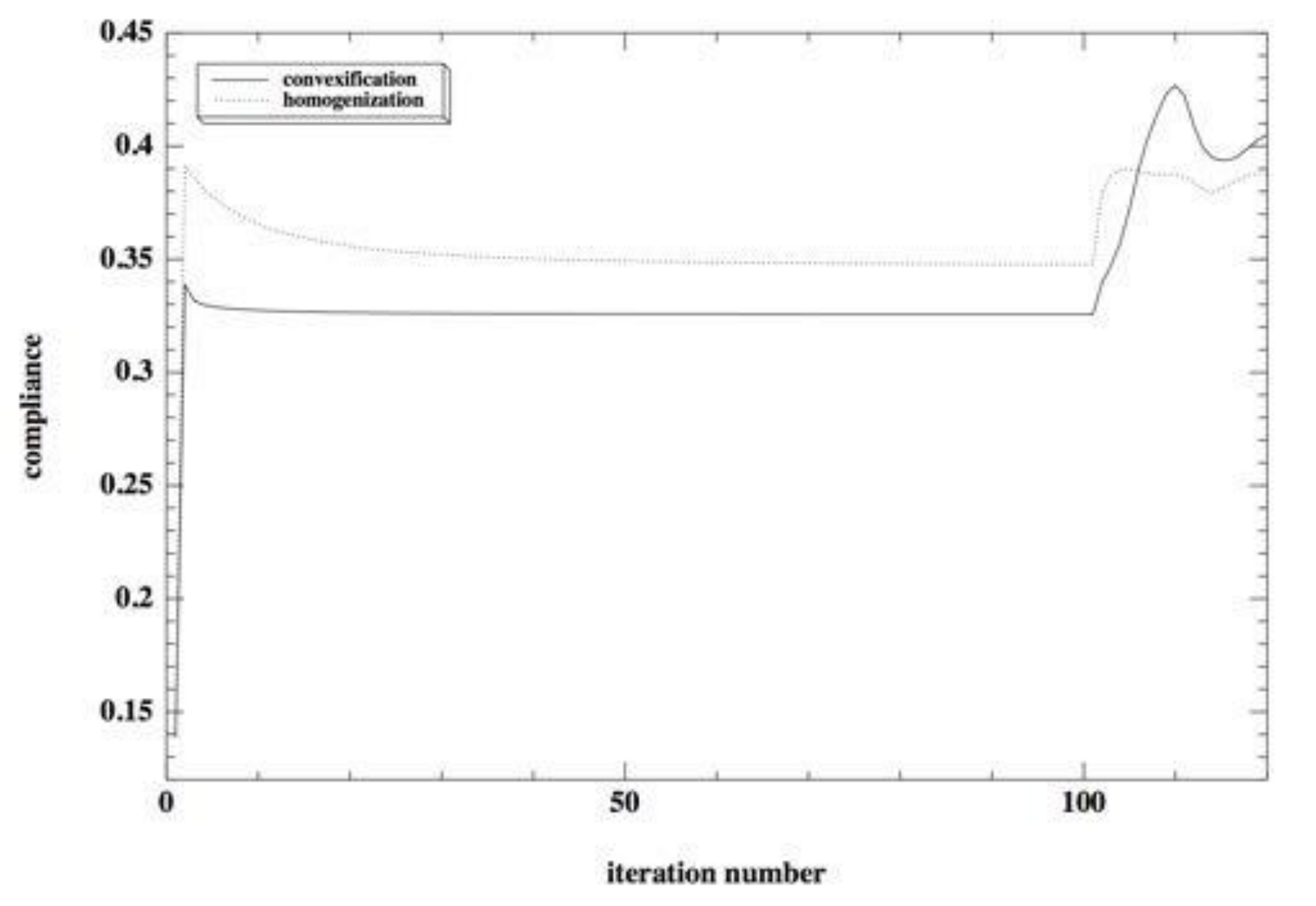}
\end{center}
\caption{Convergence history of the objective function for the bridge problem} 
\label{fig:bridge convergence history}
\end{figure}
\begin{figure}[H]
\centering
\begin{tabular}{c}
\begin{minipage}{0.35\vsize}
\begin{center}
\includegraphics[width=1.0\linewidth,center]{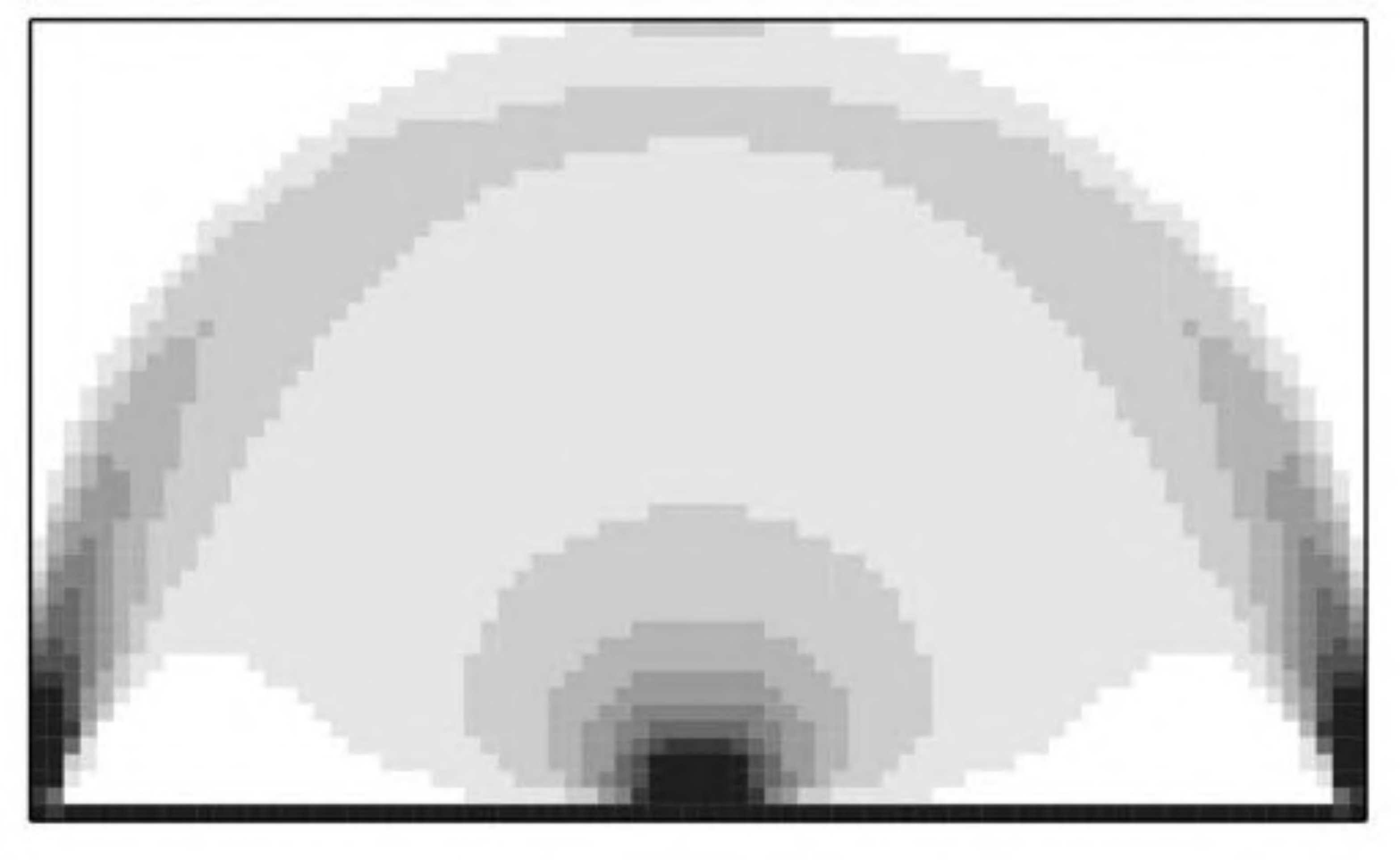}
\end{center}
\end{minipage}
\begin{minipage}{0.35\vsize}
\begin{center}
\includegraphics[width=1.0\linewidth,center]{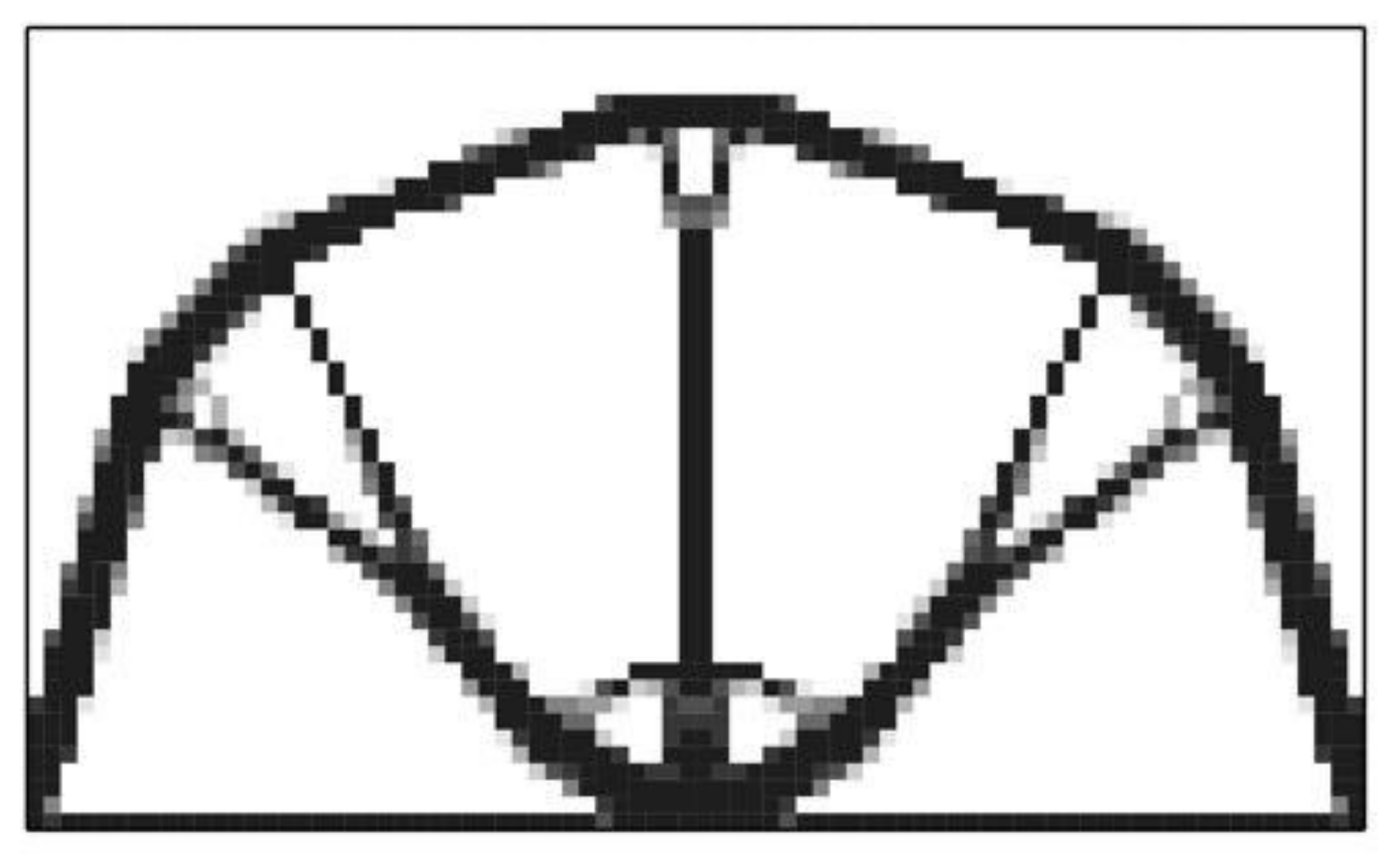}
\end{center}
\end{minipage}
\end{tabular}
\caption{Optimal shape of the bridge (left: convexified, right: penalized)} 
\label{fig:bridge convexify}
\end{figure}

\subsection{Concluding remarks on the SIMP method}

SIMP (or convexification, or ``fictitious materials'') is very simple and very popular.
Actually, many commercial codes are using the method. 
However, since SIMP uses very little information on composites, its simplicity comes at a cost. In particular, contrary to the homogenization method, SIMP is a convexification method and not a relaxation method, i.e., it changes the problem.
Hence there is a gap between the true minimal value of the objective function and that of SIMP.

\section{Generalizations of the homogenization method}
As possible generalizations of the homogenization method, we give the following three examples:
\begin{itemize}
	\item
   Multiple loads,
   \item
   Vibration eigenfrequency,
   \item
   General criterion of the least-square type.
\end{itemize}
The first two cases are self-adjoint and we have a complete understanding and justification of the relaxation process. The third case is, however, not self-adjoint and only a partial relaxation is known.
For the detail of these cases, see \cite{Allaire1}.




\section{Exercises}
\begin{problem}
   Implement the optimal radiator test case (with penalization).
\end{problem}
\begin{problem}
   Implement the SIMP method (start with exponent $p=1$ and increase to $p=3$) for compliance minimization$\colon$ cantilever, bridge, MBB beam, L-beam.
\end{problem}
\begin{problem}
	Implement SIMP method with an adjoint for the objective function
	\begin{equation*}
   		J(\theta)=\int_D \theta|u|^2\,dx.
   \end{equation*}
\end{problem}
\begin{problem}
	Minimize compliance for the cantilever problem with a parametrized cell (rectangular hole in a square cell).
\end{problem}

\chapter{Resurrection of the homogenization method: lattice materials in additive manufacturing}\label{chap of additive manufacturing}

\section{Introduction}

As we mentioned in the previous chapters, the homogenization method uses true composite materials, possibly anisotropic, and this makes it complicated to implement.
Therefore, in practice, the method was replaced by its much simplified version, the so-called SIMP method, which uses only fictitious isotropic materials. 
(For the detail of the SIMP method, see \cite{BS}.)
Since intermediate densities (between full material and void) are penalized in the end, there is indeed no need to have a detailed knowledge and optimization of microstructures.

Nevertheless, the recent progress of additive manufacturing techniques has revived the interest for the use of graded or microstructured materials since they are now manufacturable (see Figure \ref{lattice struc}, page \pageref{lattice struc}).
Since the homogenization method is the right technique to deal with microstructured materials, where anisotropy plays a key role (a feature which is absent from SIMP), we could well see a resurrection of the homogenization method for such applications.
There is however one final hurdle to overcome, once an optimal composite structure has been obtained, that is the projection of the optimal microstructure at a chosen finite lengthscale to get a global and detailed picture of the optimal microstructure. 
This is the most delicate part of this homogenization approach and the one where this chapter is most contributing.

Often (but not always) lattice materials are periodic structures, with macroscopically varying parameters.
Hence we will restrict to periodic homogenization and macroscopically modulated periodic structures, i.e., the material parameters are of the type
\begin{equation*}
A\left(x,\frac{x}{\e}\right),
\end{equation*}
where $y \mapsto A(x,y)$ is $Y$-periodic and $x \mapsto A(x,y)$ describes the macroscopic variations.
(We discuss how to choose the period and the holes in Section~\ref{sec:setting of the problem}.)
The orientation of the microstructures is rarely taken into account and optimized, although it is well-known that their orientation is a crucial and determining parameter in topology optimization (see \cite{Allaire1} and \cite{Pe}).
Actually, even if optimizing the microstructure orientation is not difficult, reconstructing the oriented periodic structure is a challenging issue. 
We propose a method to settle the difficulty by projecting the optimal microstructure on a fine mesh of the overall structure in a smoothly varying way. This idea was first introduced by \cite{PT}.

In this chapter, we consider the post-treatment of 2-d compliance minimization, which is based on \cite{Allaire5}.
We will improve the pioneer work \cite{PT} in several aspects, which we will note below.
See also \cite{GS} for another homogenization method in the spirit of \cite{PT}.
Note that this methods can be extended for the case of 3-d compliance minimization as well \cite{GAP}.
One of the difficulties of this extension is to treat the rotation of the materials.
Indeed, compared to the 2-d case, where rotations are parametrized by a single angle (see Section~\ref{sec:1st step: pre-computing the homogenized properties}), the 3-d case is more involved and requires new ingredients.
\begin{figure}[H]
\centering
   \begin{subfigure}[b]{0.8\textwidth}
   \includegraphics[width=\textwidth]{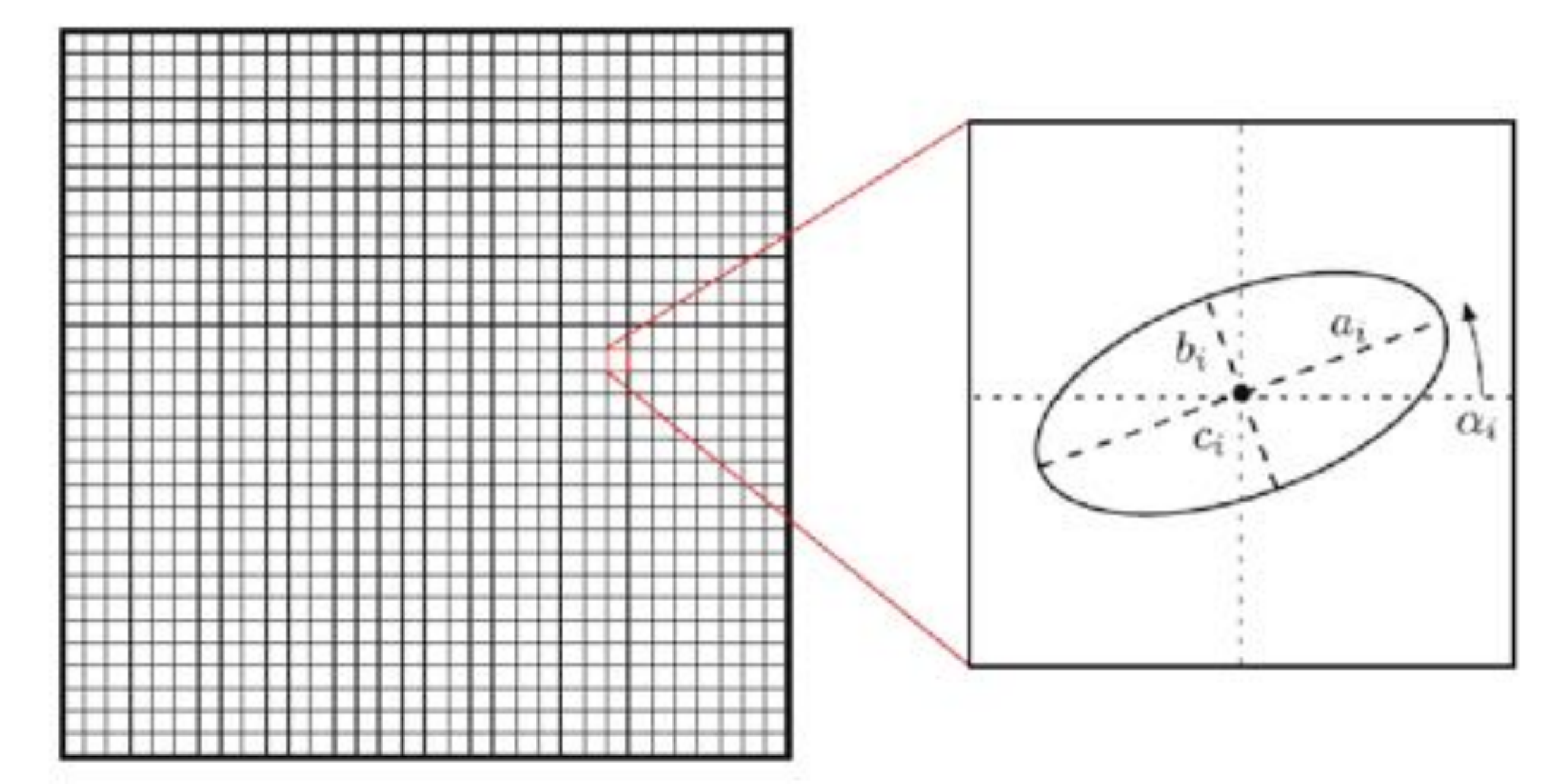}
   \label{fig:Ng1} 
\end{subfigure}
\\

\begin{subfigure}[b]{0.8\textwidth}
   \includegraphics[width=\textwidth]{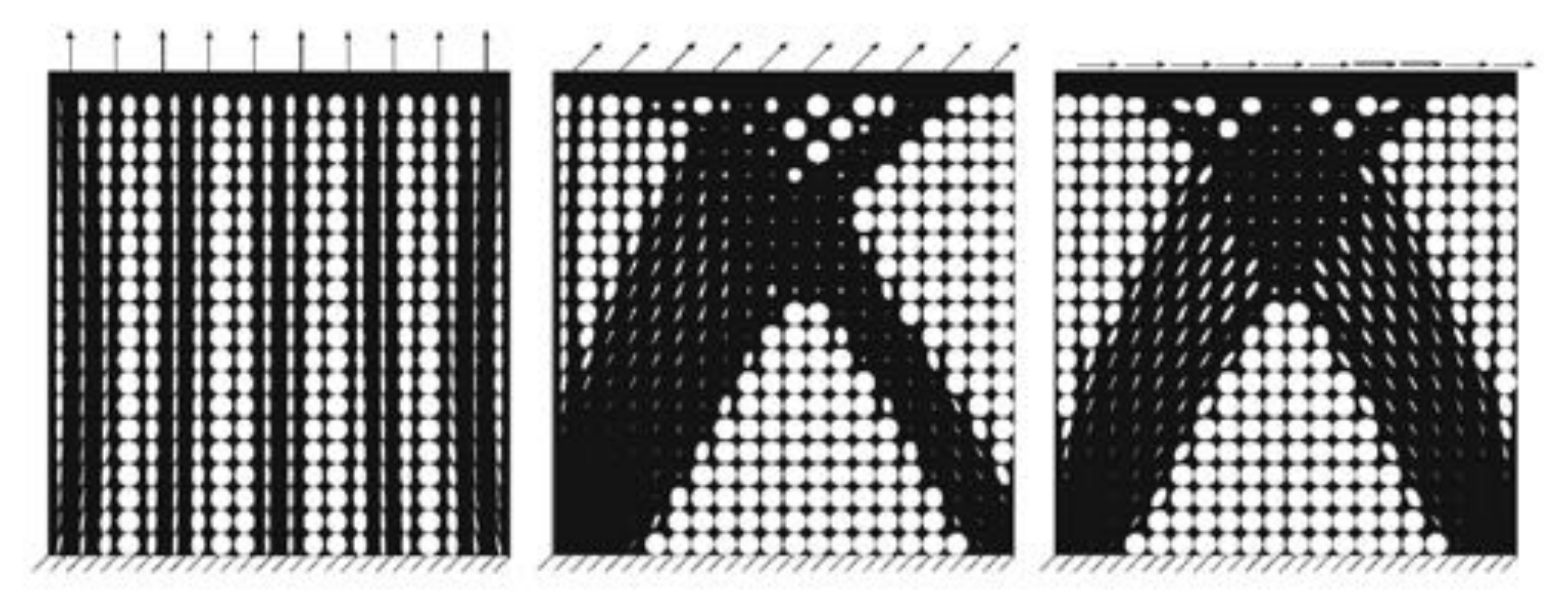}
   \label{fig:Ng2}
\end{subfigure}
\caption{Example of a macroscopically varying microstructure (extracted from \cite{GLRS})}
\end{figure}



\section{Setting of the problem} \label{sec:setting of the problem}
As in the previous chapters, we will consider compliance minimization problems in $N = 2$. Let $D\subset\RR^N$ be a smooth bounded domain and $\Omega \subset D$ be the reference configuration of a homogeneous isotropic linear elastic body whose Hooke's law $A$ is defined by \eqref{eq:Hooke's law}, with Lam\'{e} coefficients $\lambda$ and $\mu$. We assume that $\Omega$ is clamped on $\Gamma_D \subset \pa \Omega$ and subject to surface loads $g$ on $\Gamma_N \subset \pa \Omega$. Also, for simplicity, these parts $\Gamma_D$ and $\Gamma_N$ of the boundary $\partial\Om$ are assumed to be fixed and subsets of $\pa D$. The displacement vector field $u$ and the stress tensor $\sigma$ are given by the following system
\begin{equation*}
	\left\{
    \begin{aligned}
    	\dv \, \sigma &=0 && \text{ in } \Omega,\\
      \sigma&= A e(u) && \text{ in } \Omega,\\
      u &=0 && \text{ on } \Gamma_D, \\
      \sigma \cdot n&=g && \text{ on } \Gamma_N, \\
      \sigma \cdot n&=0 && \text{ on } \Gamma = \pa \Omega \setminus (\Gamma_D\cup\Gamma_N),
    \end{aligned}
	\right.
\end{equation*}
where $e(u):=(\nabla u+(\nabla u)^t)/2$ is the strain tensor. Now, let us consider the following compliance minimization problem:
\begin{equation}\label{origofcomp}
\min_{\substack{|\Omega| \le V, \\ 
\Gamma_D \cup \Gamma_N \subset \pa \Omega}} J(\Omega), 
\end{equation}
where $V>0$ is the maximum admissible volume and the objective function $J$ is the compliance 
\begin{equation*}
J(\Omega) = \int_{\Gamma_N} g \cdot u \, ds. 
\end{equation*}
As we have already seen in Chapter \ref{Topo opti homo method}, the compliance minimization problem \eqref{origofcomp} does not admit a classical solution. This is why we consider the homogenized problem. We introduce composite structures characterized by the local volume density $\theta(x)$ of the material and a homogenized elasticity tensor $A^*(x)$, corresponding to its microstructure. Then, the homogenized or macroscopic displacement $u^{*}$ is the solution of the system 
\begin{equation} \label{eq:ch6 ela hom pro}
	\left\{
    \begin{aligned}
    	\dv\,\sigma&=0 && \text{ in } D,\\
        \sigma&=A^*e(u^{*}) && \text{ in } D,\\
      u^{*}&=0 && \text{ on } \Gamma_D, \\
      \sigma \cdot n&=g && \text{ on } \Gamma_N, \\
      \sigma \cdot n&=0 && \text{ on } \partial D\setminus(\Gamma_D\cup\Gamma_N). 
    \end{aligned}
	\right.
\end{equation}  
By the above setting, the relaxed or homogenized optimization problem is obtained as follows: 
\begin{equation}\label{homogofcomp}
\min_{\substack{\int_D \theta(x)\, dx\leq V,\\ A^{*}(x) \in P_{\theta}(x)}} \left\{J^{*}(\theta, A^{*})=\int_{\Gamma_N}g\cdot u^*\,ds\right\}, 
\end{equation}
where $u^*$ is the solution of \eqref{eq:ch6 ela hom pro} and $P_{\theta}(x)$ is a given subset of effective or homogenized Hooke's laws for some well-chosen microstructures of density $\theta(x)$. 

Our aim in Chapter \ref{chap of additive manufacturing} is to propose a specific subset $P_\theta$ of periodic composites and to construct a minimizing sequence for \eqref{homogofcomp}. A typical example is that of a rectangular hole in a square cell (Figure \ref{rectangular hole}, page \pageref{rectangular hole}), where the cell parameters are the lengths $m_{1}, m_{2} > 0$ and the rotation angle $\alpha$ (acting either on the hole or the whole cell). 
Also, we let the homogenized tensors be denoted by $A^{*}(m_{1},m_{2},\alpha)$. We note that the same ideas in the following are applicable to other geometries as well.

\section{A three steps approach}
To achieve our purpose, we take a three steps approach for the optimization. 
The first step is to pre-compute the homogenized properties $A^{*}(m_{1},m_{2},\alpha)$ for all values of the parameters. 
The second step is to apply a simple parametric optimization process to the homogenized problem.
Compared to Chapter~\ref{Parametric optimal design}, we replace the thickness field $h$ by new parameter fields $m_{1}$, $m_{2}$ and $\alpha$ which vary in space. 
The third step is to choose a length scale $\e$ and reconstruct a periodic domain $A(x,\frac{x}{\e})$ approximating the optimal $A^{*}$. We remark that the third step is the trickiest part of the whole process. 

The most delicate point is the combined problem of the orientation of the microstructure and the reconstruction of a macroscopically varying periodic lattice. 
In order to avoid these difficulties, there are two possible approaches. 
The first one is a ``naive" approach, that is, we assume that the periodic grid is never deformed and the holes are simply rotated. 
The main advantage of the ``naive'' approach is that the reconstruction of the periodic perforated structure is very easy (see e.g. \cite{GLRS}). 
We remark that this approach is naive because the ``skeleton" of the reconstructed structure is fixed and thus it does not follow the supported stresses or forces. 

The second one is a deeper approach initiated by Pantz and Trabelsi \cite{PT}. 
The main advantage of this new approach is that the reconstructed structure adapts its geometry to the supported stresses or forces (in some sense, it looks like Michell trusses \cite{rozvany} in the case of dimension $2$). 

The main difficulty is then to reconstruct such a macroscopically deformed periodic perforated structure. There may be issues on the regularity of the orientation field for stresses or forces. This is the approach that we follow in the sequel.

\begin{figure}
\centering
\begin{tabular}{c}
\begin{minipage}{0.33\linewidth}
\begin{center}
\includegraphics[width=0.9\linewidth,center]{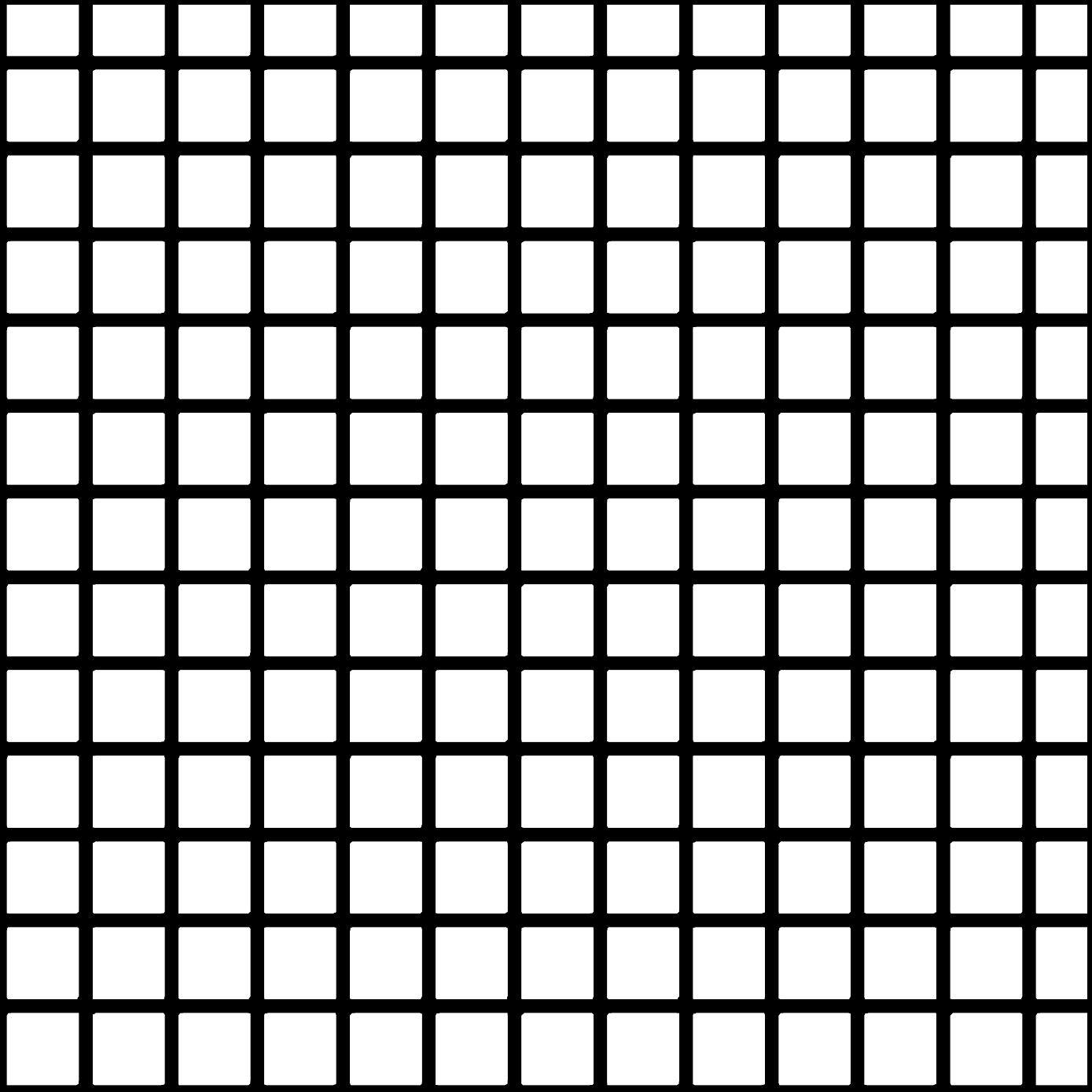}
\end{center}
\end{minipage}
\begin{minipage}{0.33\linewidth}
\begin{center}
\includegraphics[width=0.9\linewidth,center]{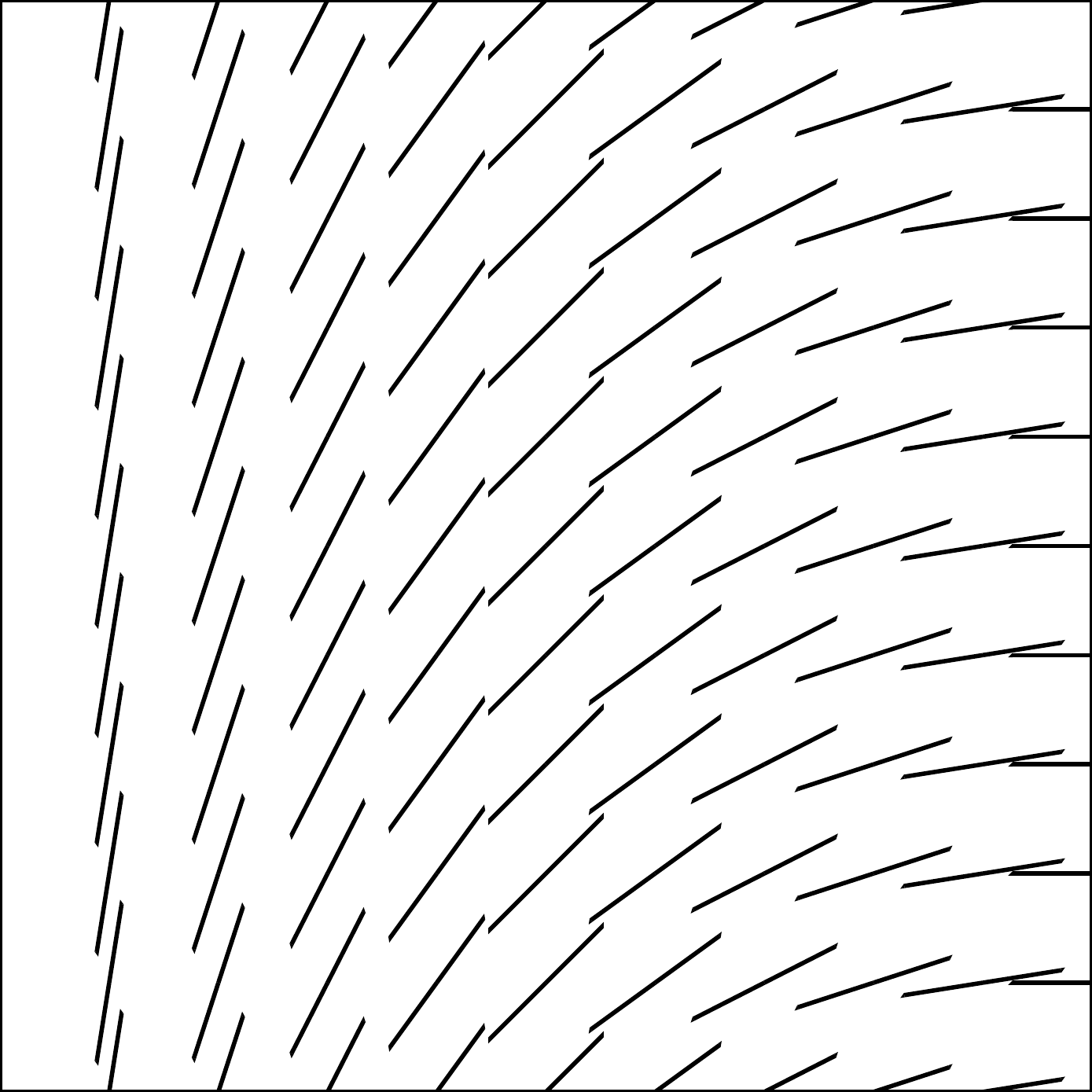}
\end{center}
\end{minipage}
\begin{minipage}{0.33\linewidth}
\begin{center}
\includegraphics[width=.9\linewidth,center]{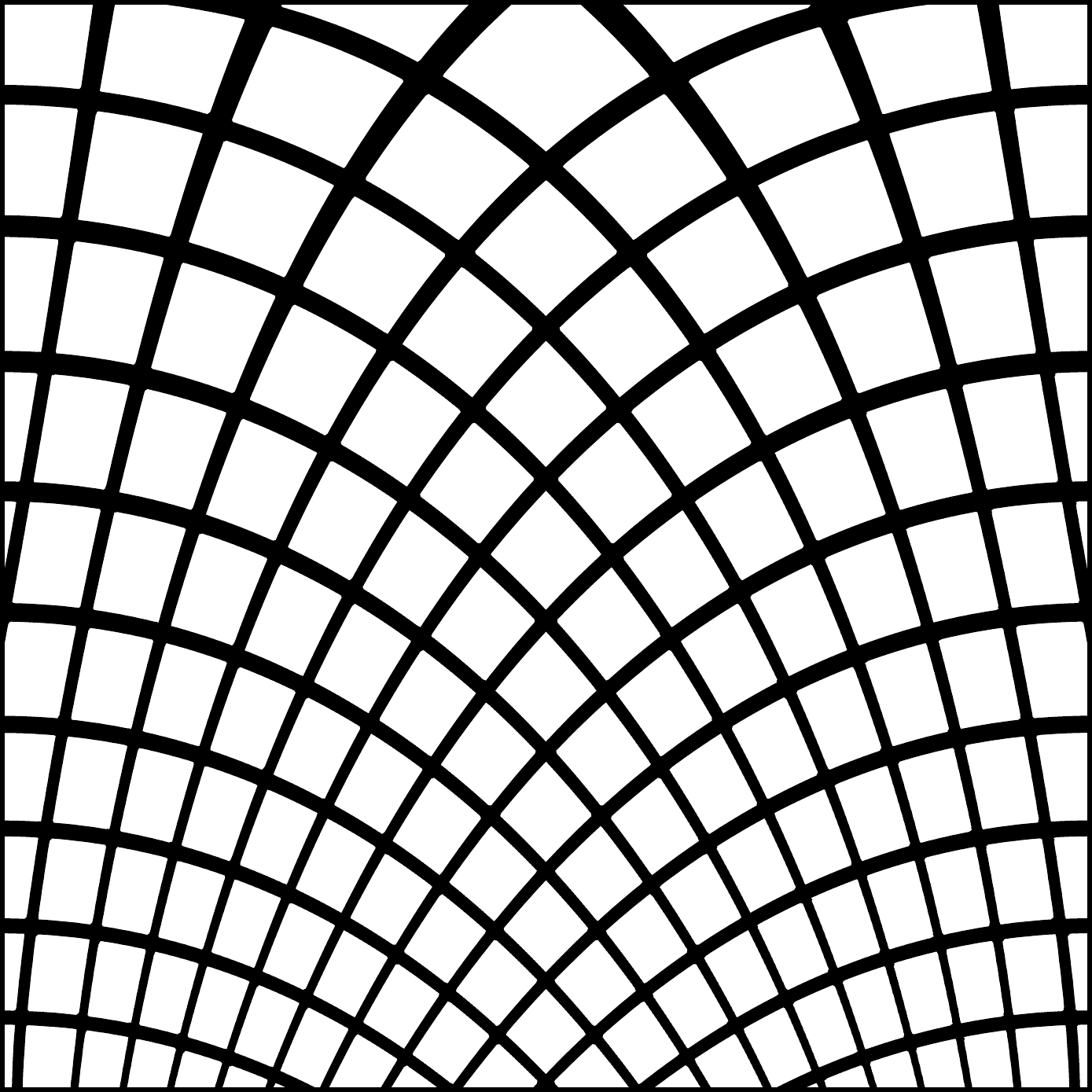}
\end{center}
\end{minipage}
\end{tabular}
\caption{A regular grid (left) is associated to an orientation field (middle), giving the
local orientation of each cell: it yields a distorted grid (right).} 
\end{figure}

\section{1st step: pre-computing the homogenized properties} \label{sec:1st step: pre-computing the homogenized properties}
For the square cell with a rectangular hole, with a fixed orientation, we compute the homogenized properties $A^{*}(m_{1},m_{2})$ for a discrete sampling of $0 \leq m_{1},m_{2} \leq 1$. 
We also compute the derivatives of $A^{*}(m_{1},m_{2})$ with respect to $(m_{1},m_{2})$ by using a shape derivative and an adjoint approach in the same manner in Section~\ref{sec:Computation of a continuous gradient} (see \cite[Section $3.3$ and $3.4$]{Allaire5} for the details and \cite{HP, SK} for an introduction on shape derivatives). We note that we can also compute the derivatives of $A^{*}(m_{1},m_{2})$ with respect to $(m_{1},m_{2})$ numerically by means of a finite difference method. 

In the following, for simplicity, we will only consider the case of dimension $N=2$. 
Assume that the cell is rotated by an angle $\alpha$.
We define $Y_\alpha(m)$ as the periodic cells with hole $m=(m_1, m_2)$, together with the orientation $\alpha$ of the cell.
Then the dependency of the homogenized properties $A^{*}(m_1, m_2, \alpha)$ with respect to the angle $\alpha$ is given by
\begin{equation}\label{rotation!}
A^{*}(m_{1},m_{2},\alpha) = R(\alpha)^{T} A^{*}(m_{1},m_{2},0)R(\alpha), 
\end{equation}
where $R(\alpha)$ is the fourth-order tensor defined by 
\begin{equation*}
\forall \xi \in \mathcal{M}^{s}_{2}, \quad R(\alpha) \xi = Q(\alpha)^{T} \xi Q(\alpha), 
\end{equation*}
and $Q(\alpha)$ is the rotation matrix of angle $\alpha$. 
Indeed, the variational formulation \eqref{eq:variational formulation wij} and the characterization of $A^*$ \eqref{eq:def tensor A*} imply that the quadratic form $A^*(m_1, m_2, \alpha)\xi:\xi$ satisfies
\begin{equation*}
	A^*(m_1, m_2, \alpha)\xi\cdot\xi=\min_{w\in H^1_\#(Y_\alpha(m))^N}\int_{H^1_\#(Y_\alpha(m))^N}A(\xi+e(w))\cdot(\xi+e(w))\,dy
\end{equation*}
for any $\xi\in\mathcal{M}^s_2$.
On the other hand, we see that, for $\xi\in\mathcal{M}^s_2$,
\begin{equation*}
	\begin{aligned}
		&\left[R(\alpha)^{T} A^{*}(m_{1},m_{2},0)R(\alpha)\right] \xi\cdot\xi
		=A^{*}(m_{1},m_{2},0)(R(\alpha)\xi)\cdot(R(\alpha)\xi)\\
		&\qquad=\min_{w\in H^1_\#(Y_0(m))^N}\int_{H^1_\#(Y_0(m))^N}A(R(\alpha)\xi+e(w))\cdot(R(\alpha)\xi+e(w))\,dy\\
		&\qquad=\min_{w\in H^1_\#(Y_0(m))^N}\int_{H^1_\#(Y_0(m))^N}A(\xi+Q(\alpha)e(w)Q(\alpha)^T)\cdot(\xi+Q(\alpha)e(w)Q(\alpha)^T)\,dy\\
		&\qquad=\min_{w\in H^1_\#(Y_\alpha(m))^N}\int_{H^1_\#(Y_\alpha(m))^N}A(\xi+e(w))\cdot(\xi+e(w))\,dy.
	\end{aligned}
\end{equation*}
Thus, the numerical computation of the homogenized properties $A^{*}(m_{1},m_{2},\alpha)$ can be restricted to the case $\alpha = 0$. 
Note that \eqref{rotation!} implies that a rotation of the cell by an angle $\pi$ does not change its Hooke's law as $R(\pi) = -{\rm Id}$. 
Hence the optimal orientation can only be defined modulo $\pi$. 

\begin{figure*}
        \centering
        \begin{subfigure}[b]{0.475\textwidth}
            \centering
            \includegraphics[width=\textwidth]{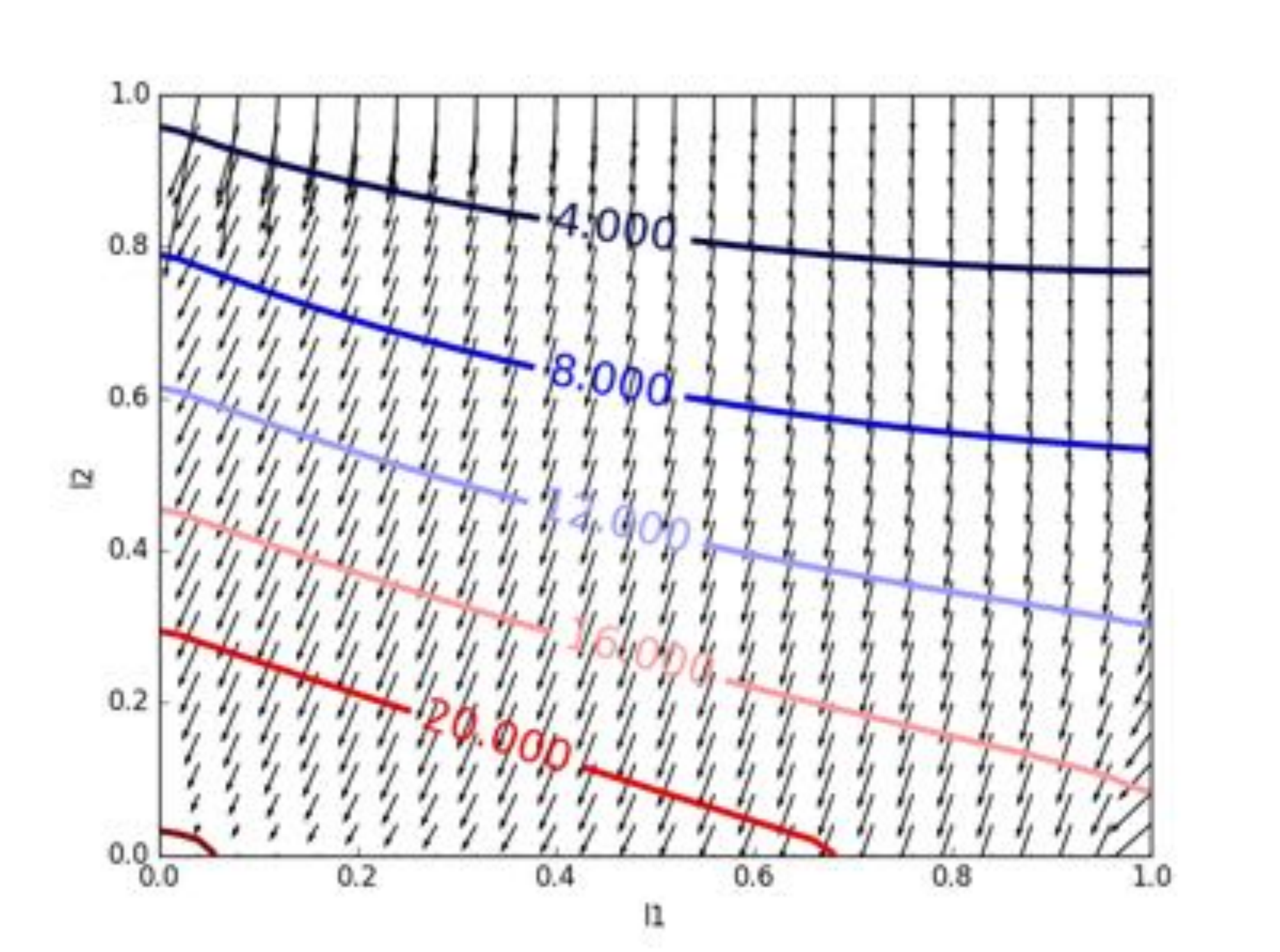}
            \caption[]%
            {{\small $\left(A_0^*(m)\right)_{1111}$}}    
        \end{subfigure}
        \hfill
        \begin{subfigure}[b]{0.475\textwidth}  
            \centering 
            \includegraphics[width=\textwidth]{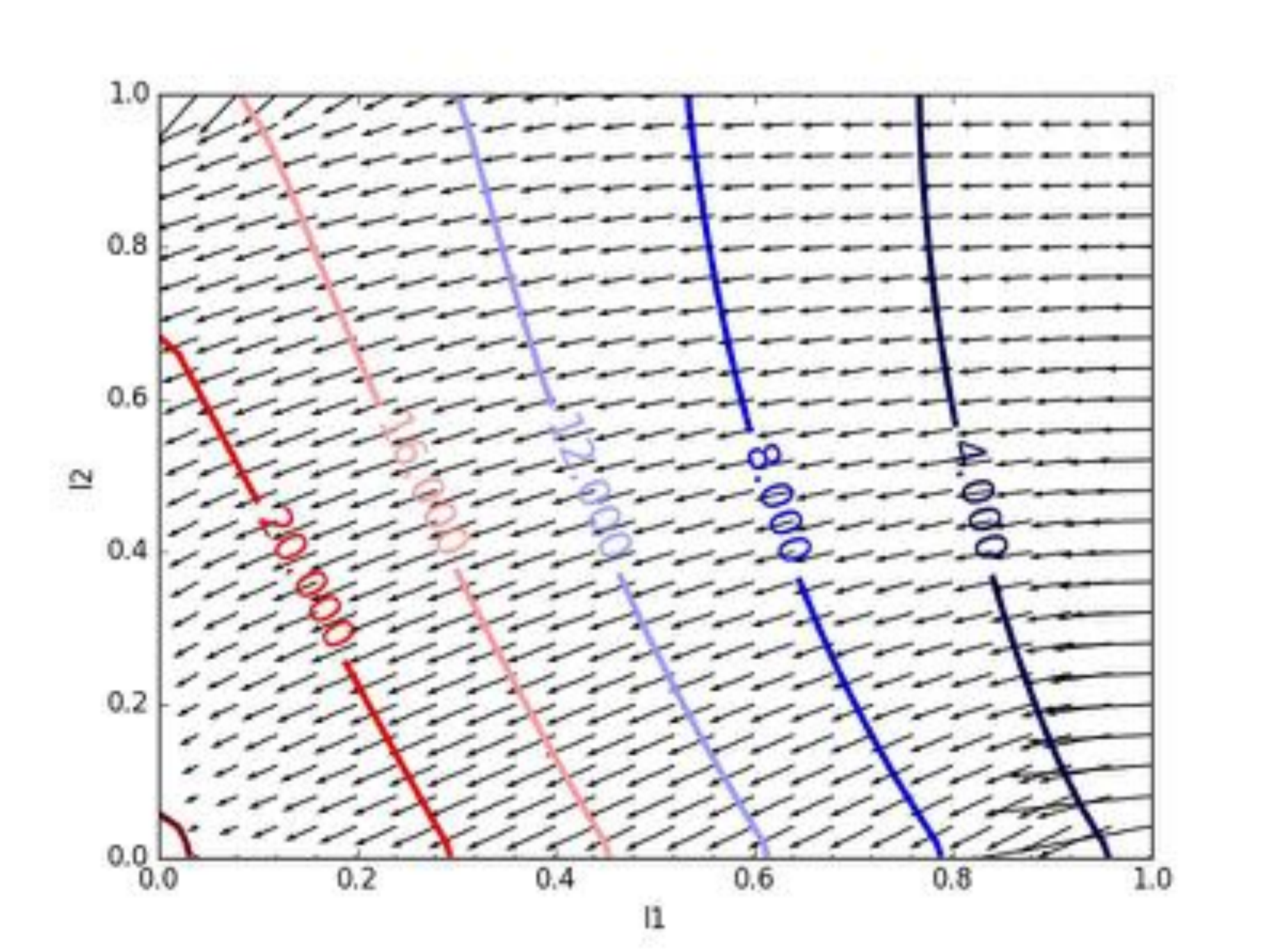}
            \caption[]%
            {{\small $\left(A_0^*(m)\right)_{2222}$}}    
        \end{subfigure}
        \vskip\baselineskip
        \begin{subfigure}[b]{0.475\textwidth}   
            \centering 
            \includegraphics[width=\textwidth]{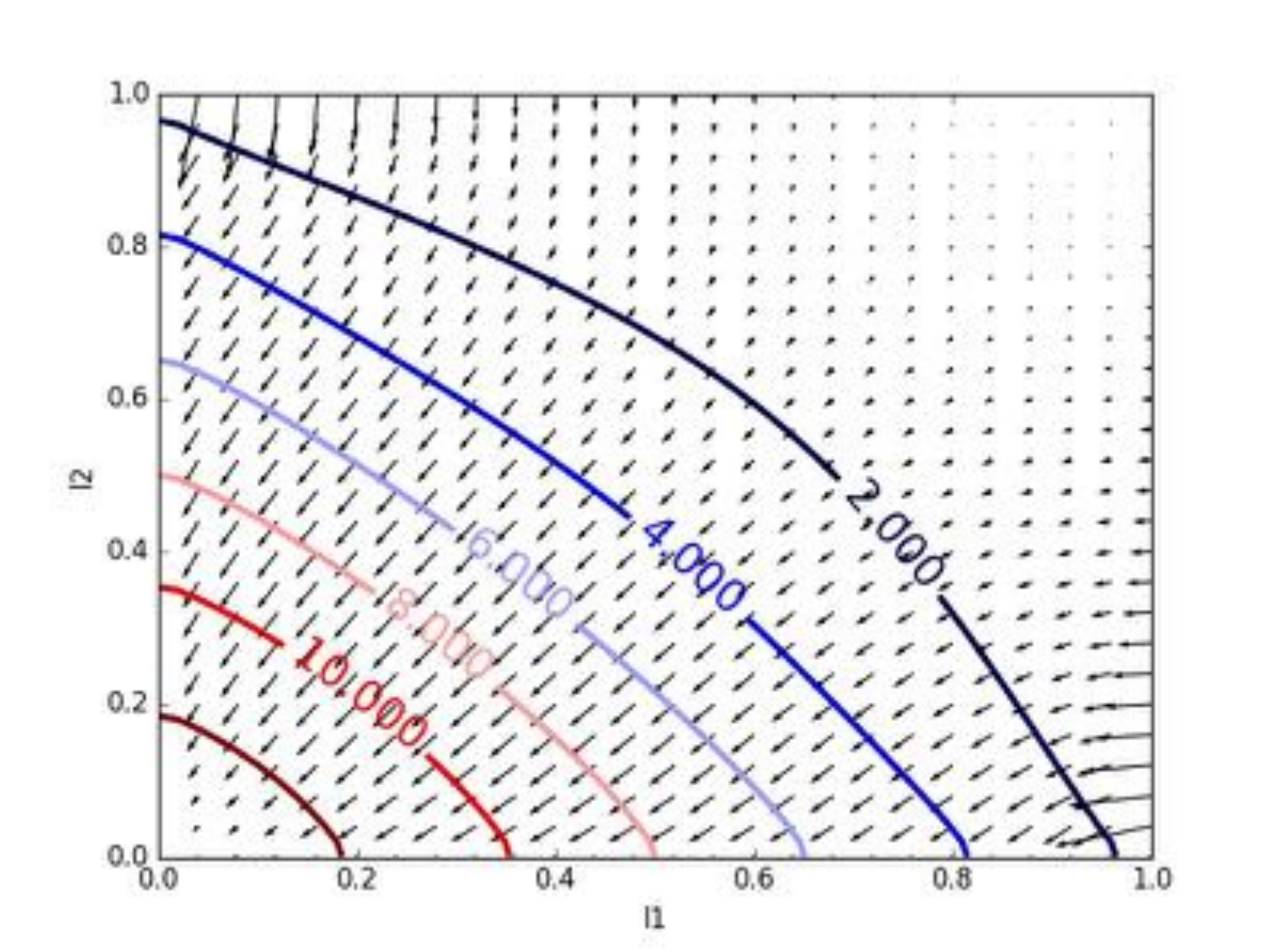}
            \caption[]%
            {{\small $\left(A_0^*(m)\right)_{1122}$}}    
        \end{subfigure}
        \quad
        \begin{subfigure}[b]{0.475\textwidth}   
            \centering 
            \includegraphics[width=\textwidth]{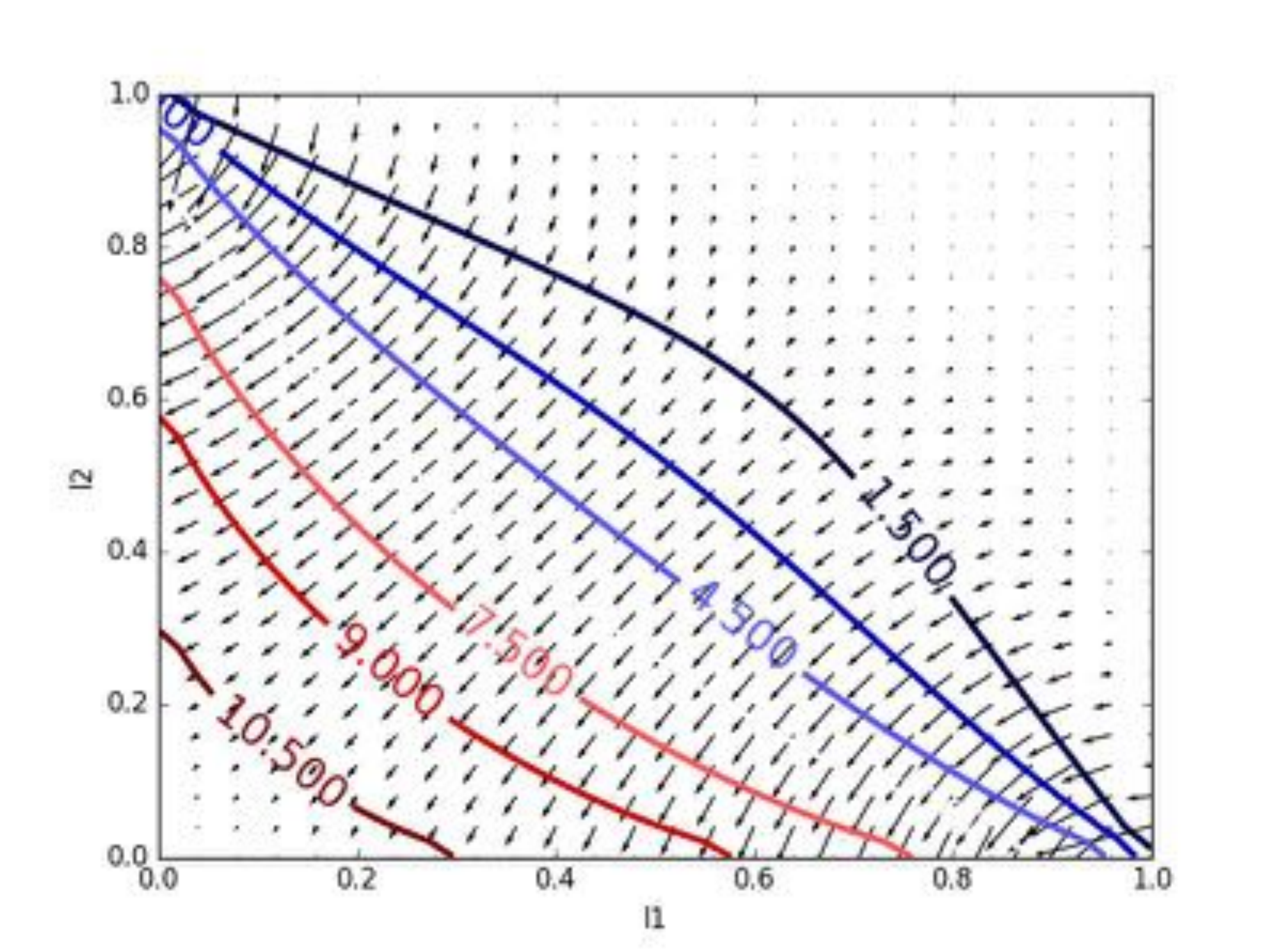}
            \caption[]%
         {{\small $\left(A_0^*(m)\right)_{1212}$}}    
        \end{subfigure}
        \caption[]
        {{\small Isolines of the entries of the homogenized tensor $A^{\ast}(m_1, m_2, 0)$ and their gradient (small arrows) according to the parameters $m_1$ ($x$-axis) and $m_2$ ($y$-axis).}}
      
        \label{c5p16}
    \end{figure*}


We show the numerical results for the entires of the homogenized tensor $A^*(m_1, m_2, 0)$ and their derivatives as functions of $m$ (see Figure~\ref{c5p16}).
When $m=0$, then the homogenized tensor $A^*(m_1, m_2, 0)$ is equal to $A$.
On the other hand, if $m$ is close to $(1, 1)$, then the homogenized tensor is converging to the null tensor.
Moreover, one can see easily check, that the entries of $A^*(m_1, m_2, 0)$ decrease, when $m_1$ is fixed and $m_2$ is increasing (and vice versa).
In other words, the cell is globally weaker when its hole is widening in one direction or the other.
However, the sensitivity of the component $A^*(m_1, m_2, 0)_{1111}$ to the parameter $m_2$ is greater than the one to the parameter $m_1$ (see Figure~\ref{c5p16}).
Figure~\ref{c5p16} shows the numerical results in the case $A_{1111} = A_{2222} = 24.07$, $A_{1122} = 12.96$ and $A_{1212} = 11.11$.
That is explained by the fact that, along the $y_1$ axis, the strength of the cell is mainly ensured by the material in the areas above and below the hole, whose sizes depend on $m_2$.
As one could expect, the homogenized elasticity tensor is quite smooth with respect to the parameter $m$, so it is amenable to a gradient based optimization method.

\section{2nd step: parametric optimization of the homogenized problem}
Let us recall that the homogenized equation in a box $D$ is 
\begin{equation*}
	\left\{
    \begin{aligned}
    	\dv \, \sigma &= 0 && \text{ in } D,\\
      \sigma &= A^{*}e(u) && \text{ in } D,\\
      u &=0 && \text{ on } \Gamma_D, \\
      \sigma\cdot n &= g && \text{ on } \Gamma_N, \\
      \sigma\cdot n &= 0 && \text{ on } \Gamma = \pa D \setminus (\Gamma_{D} \cup \Gamma_{N}) 
    \end{aligned}
	\right.
\end{equation*}
and the compliance minimization problem is 
\begin{equation*}
\min_{m_{1},m_{2},\alpha} \left\{J(A^{*}) = \int_{\Gamma_{N}} g \cdot u \, ds\right\}. 
\end{equation*}
Here, the minimum is taken over all functions $m_1$, $m_2\in L^\infty(D; [0, 1])$ and $\alpha\in L^\infty(D)$.
One can rewrite the compliance as the minimum of complementary energy (see Section \ref{sec:dual energy}) as follows 
\begin{equation*}
J(A^{*}) = \min_{\tau \in H_{0}} \int_{D} (A^{*})^{-1} \, \tau \cdot \tau \, dx, 
\end{equation*}
where 
\begin{equation*}
H_{0} = \left \{
\tau \in L^2(D; \mathcal{M}^{s}_{2}) 
\;:\;
\begin{aligned}
\dv \, \tau &= 0 &&\text{ in } D, \\
\tau \cdot n &= g &&\text{ on }  \Gamma_{N}, \\
\tau \cdot n &= 0 &&\text{ on }  \Gamma 
\end{aligned}
\right \}. 
\end{equation*}
This is interesting for algorithmic purposes because  we can apply the optimality criteria or alternate minimization algorithm of Chapter~\ref{Parametric optimal design}. 
We note that the orientation optimization with respect to $\alpha$ is very simple, due to a result of Pedersen \cite{Pe}. Pedersen proved that the optimal orientation of an orthotropic cell for a given displacement field is the one where the cell is aligned with the principal eigen-direction of the strain tensor. We can easily show a similar result for stress field in the same way. 

We compute the volume fraction of material in a single unit cell as
\begin{equation*}
\theta(x) = 1 - m_{1}(x)m_{2}(x),  
\end{equation*}
also, the total volume of the lattice structure is thus 
\begin{equation*}
\text{Vol} = \int_{D} \left( 1 - m_{1}(x) m_{2}(x) \right) \, dx. 
\end{equation*}
To implement a volume constraint, we rely on a Lagrangian algorithm 
\begin{equation*}
\mathcal{L}(m,\alpha,\sigma,\ell) = \int_{D} (A^{*})^{-1}(m,\alpha) \, \sigma \cdot \sigma \, dx + \ell \left( \int_{D} ( 1 - m_{1}(x) m_{2}(x)) \, dx - V_{0} \right),  
\end{equation*}
where $\ell$ is the Lagrange multiplier associated to the volume constraint $\vol=V_0$. 

Now let us consider a numerical algorithm. 
To minimize with respect to the microstructure $m$, we use the following algorithm of alternate minimization (or optimality criteria).
Moreover, to minimize with respect to the orientation $\alpha$, we could use the same method as for the minimization with respect to the microstructure $m$, but Pedersen's result \cite{Pe} is a better (more efficient) algorithm than the gradient descent method to compute the optimal orientation because it is a global minimization method, proving an optimal orientation at each iteration. However, this method can usually not be generalized to other objective functions. 

\begin{algorithm}[H]
\caption{Algorithm of alternate minimization (or optimality criteria)} 
\begin{enumerate}
\item Initialization of the cell parameters $m^{0}_{1}, m^{0}_{2}, \alpha^{0}$. For example, we take $m_1 = m_2$, constant satisfying the volume constraint, and $\alpha = 0$. 
\item Iterations until convergence, for $n \geq 0$: 
\begin{enumerate}
\item Computation of the stress tensor $\sigma^{n}$, unique solution of the (dual) elasticity equations with $A^*_{\alpha^{n}}(m^{n})$. 
\item Update of the parameters: 
\begin{itemize}
\item perform one iteration of the projected gradient algorithm for hole parameters 
\begin{equation}\label{projection algorithm for $m$}
m^{n+1}_{i} = \mathcal{P} \left( m^{n}_{i} - \mu_m \dfrac{\pa \mathcal{L}}{\pa m_{i}}(m^n, \alpha^n, \sigma^n, \ell^n) \right), 
\end{equation}
where $\mu_m > 0$ is the step size and $\mathcal{P}$ is the projection operator to satisfy the constraints.
\item by Pedersen's result \cite{Pe}, for a given stress tensor $\sigma^{n}$, the optimal orientation angle $\alpha^{n}$ is the one where the cell is aligned with the principal eigen-directions of the strain tensor.
\end{itemize}
\end{enumerate}
\end{enumerate}
\end{algorithm}
Let us focus on the technical details of the algorithm. 
The partial derivative of the Lagrangian $\mathcal{L}$ with respect to the parameter $m_i$ ($i=1,2$), is given by
\begin{equation*}
\frac{\pa \mathcal{L}}{\pa m_{i}} = - \frac{\pa A^{*}}{\pa m_{i}}(m,\alpha) (A^{*})^{-1}(m,\alpha) \, \sigma \, \cdot \, (A^{*})^{-1}(m,\alpha) \sigma - \ell m_{3-i}.
\end{equation*}
We remark that the derivative of $A^*$ with respect to $m_i$ can be obtained by employing the use of shape derivatives.
For the details of the computation, see \cite{Allaire5}.

Finally, $\mathcal{P}$ is the projection operator onto the interval $[0,1]$.
In the process of this projection, we have to update the Lagrange multiplier $\ell$, which is constant in $D$, by a dichotomy process designed to respect the volume constraint. 

\begin{figure}[H]
\centering
\includegraphics[width=.7\linewidth,center]{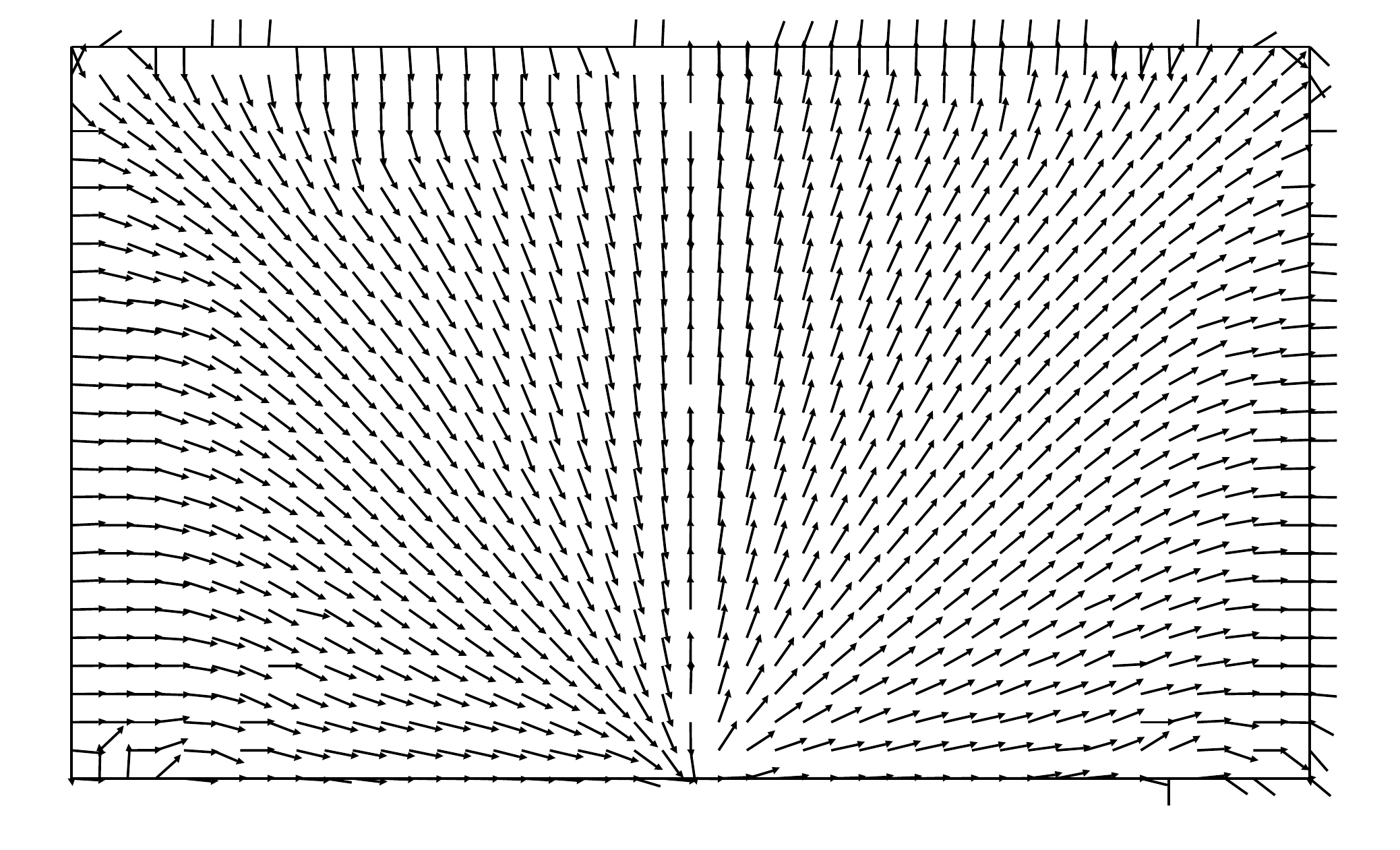}
\caption{Regularity issues for the computed optimal orientation $\alpha$ (bridge case).} 
\label{orientation}
\end{figure}

Note that except when $\sigma$ is proportional to the identity, the optimal orientation angle $\alpha$ is unique up to the addition of a multiple of $\pi$. As shown in Figure~\ref{orientation}, this non-uniqueness creates a regularity issue for $\alpha$. 
For example, 
\begin{itemize}
\item $\alpha$ or $\alpha + \pi$ correspond to the same orientation,
\item where the material density is close to $0$ or $1$, the orientation does not play any role (cf. the corners in Figure~\ref{orientation}), 
\item there are real singularities of the orientation, like a fan (cf. the bottom middle in Figure~\ref{orientation}),
\item if the values of $m_{1}$ and $m_{2}$ are exchanged, then the optimal orientation switches from $\alpha$ to $\alpha + \pi/2$, but it does not seem to appear in our results. 
\end{itemize}
These issues create some numerical difficulties, that we will explain in the next section. 

\section{3rd step: reconstruction of an optimal periodic structure}\label{sec:3rd step}
In the previous sections we computed an optimal homogenized design (with an underlying modulated periodic structure). In this section, let us enter a post-processing step. We choose a length scale $\e$ for this projection step and reconstruct a periodic shape with length scale $\e$, approximating the optimal one. 

\subsection{Projection in the simple case without varying orientation, $\alpha \equiv 0$} 
First, we consider the case where there is no varying orientation, that is, $\alpha \equiv 0$.  The unit cells (rectangular hole in a square) are defined by 
\begin{equation*}
Y(m) = \left\{ y \in [0,1]^{2} \;:\; 
\begin{aligned}
\cos(2\pi y_{1}) &\geq \cos(\pi (1-m_{1})), \\
&\text{or} \\
\cos(2\pi y_{2}) &\geq \cos(\pi (1-m_{2}))
\end{aligned}
\right\}.
\end{equation*}
The domain $D$ is paved with cells $\e Y(m)$. Since the hole size $m(x)$ is varying in $D$, the periodicity cell is macroscopically modulated and we define a projected lattice shape $\Omega_{\e}(m)$ as 
\begin{equation*}
\Omega_{\e}(m) := \left\{ x \in D \;:\; 
\begin{aligned}
\cos \left( \frac{2\pi x_{1}}{\e} \right) &\geq \cos(\pi (1-m_{1}(x))), \\
&\text{or} \\
\cos \left( \frac{2\pi x_{2}}{\e} \right) &\geq \cos(\pi (1-m_{2}(x)))
\end{aligned}
\right\},
\end{equation*}
where $m_{1}(x), m_{2}(x)$ are functions defined on $D$ with values in $[0,1]$. 

\begin{remark}
The values of $m_{1}$ $m_{2}$ are not necessarily constant in each cell of the structure. Hence, the holes in the cellular structure $\Omega_{\e}(m)$ are not exactly rectangles. But, when $\e$ goes to $0$, the sequence of cellular structures converges to the composite with local Hooke's law equal to $A^{*}_{0}(m)$. 
\end{remark}

The cellular structures can be defined using level-sets. We introduce two functions $f^{m}_{\e,i}$, one for each direction 
\begin{equation*}
f^{m}_{\e,i}(x) := - \cos \left( \frac{2\pi x_{i}}{\e} \right) + \cos (\pi (1 - m_{i}(x))), 
\end{equation*}
and the level-set function 
\begin{equation*}
F^{m}_{\e} := \min(f^{m}_{\e,1}, f^{m}_{\e,2}). 
\end{equation*}
The final structure $\Omega_{\e}(m)$ is then defined by 
\begin{equation*}
\Omega_{\e}(m) = \left\{ x \in D \,\, : \,\, F^{m}_{\e}(x) \leq 0 \right\}. 
\end{equation*}
The construction of a minimizing sequence of shapes is immediate: we just have to update the size $\e$ in the previous level-set function. 

\subsection{Projection in the general case with orientation, $\alpha \neq 0$} 
This section is based on the paper of Pantz and Trabelsi \cite{PT}. The main idea of Pantz and Trabelsi is to find a map $\varphi = (\varphi_{1}, \varphi_{2})$ from $D$ into $\mathbb{R}^{2}$ which distorts the regular square grid in order to orientate each square at the optimal angle $\alpha$ (or $\alpha + \pi$). Once $\varphi$ is found, we can proceed as before. 

The final shape, now denoted $\Omega_{\e}(\varphi,m)$, is still defined by a level-set function: 
\begin{equation*}
\Omega_{\e}(\varphi, m) = \left\{ x \in D \;:\; F^{\varphi, m}_{\e}(x) \leq 0 \right\}
\end{equation*}
with $F^{\varphi, m}_{\e} = \min(f^{\varphi,m}_{\e,1}, f^{\varphi,m}_{\e,2})$ and 
\begin{equation*}
f^{\varphi,m}_{\e,i}(x) = - \cos \left( \frac{2\pi \varphi_{i}(x)}{\e} \right) + \cos(\pi(1 - m_{i}(x))).  
\end{equation*}

Geometrically (in $2$-d), the Jacobian $\nabla \varphi$ should be proportional to the rotation matrix defined by 
\begin{equation}\label{rotation matrix}
Q(\alpha) = \begin{pmatrix}
\cos \alpha &-\sin \alpha \\
\sin \alpha  &\cos \alpha
\end{pmatrix}. 
\end{equation}
One possibility (Pantz and Trabelsi) is to find $\varphi$ which (roughly) minimizes 
\begin{equation*}
\int_{D} \abs{\nabla \varphi - Q(\alpha)}^{2} \, dx. 
\end{equation*}
However, it is not straightforward because $Q(\alpha)$ is not smooth (the orientation is not coherent a priori). 
Pantz and Trabelsi proposed a complicated trick to avoid this coherent orientation issue. We also mention that Groen and Sigmund \cite{GS} suggested another trick from image processing to obtain a coherent orientation. 

\subsection{The new approach of Allaire-Geoffroy-Pantz}
In this section, we propose a new approach, based on the paper of Allaire-Geoffroy-Pantz \cite{Allaire5}. 

Let us recall that geometrically (in $2$-d), at every point $x\in D$, the Jacobian $\nabla \varphi$ should be proportional to the rotation matrix $Q(\alpha)$ defined by \eqref{rotation matrix}. 
Moreover, the proportions of the cell have to be preserved in order to converge to a true square and not simply to a rectangle. For this purpose, we impose $|\nabla \varphi_1| = |\nabla \varphi_2| = e^{r}$, where $r \in H^1(D)$ is a scalar dilation field. 
Then the Jacobian $\nabla \varphi$ should be 
\begin{equation}\label{equations for the map varphi}
\nabla \varphi = e^{r} Q(\alpha) \quad \text{in} \,\, D. 
\end{equation}

This equation can be satisfied only if $\alpha$ satisfies the following conformality condition. 
\begin{lemma} \label{lem:ori con}
Let $\alpha$ be a regular orientation field and $D$ be a simply connected domain. There exists a mapping function $\varphi$ and a dilation field $r$ satisfying $\nabla \varphi = e^r Q(\alpha)$ if and only if 
\begin{equation}\label{conformality condition}
\Delta \alpha = 0 \quad \text{in} \,\, D. 
\end{equation}
\end{lemma}

We recall that, for a vector field $u = (u_{1}, u_{2})$, its curl is defined as $\text{curl} \, u = \nabla \wedge u = \dfrac{\pa u_{2}}{\pa x_{1}} - \dfrac{\pa u_{1}}{\pa x_{2}}$, where $\wedge$ is the $2$-d cross product of vectors. Of course, $\text{curl} \, \nabla \varphi = \left(\text{curl} \, \nabla \varphi_1, \text{curl} \, \nabla \varphi_2 \right)= 0$. 
\begin{proof}[Proof of Lemma~\ref{lem:ori con}]
Since the domain $D\subset \RR^2$ is simply connected, By Poincar\'e's lemma, the map $\varphi$ exists if and only if the right hand side of \eqref{equations for the map varphi} is curl-free. That is, 
\begin{equation*}
{\rm curl} \left( e^r Q(\alpha)  \right)=0.
\end{equation*}
Let $a_1, a_2$ be the columns of $Q(\alpha)$. Then
\begin{equation}\label{curl iff}
\textrm{curl}\left(e^r Q(\alpha)\right)=0 \iff \gr r \land a_i = -\gr \land a_i, \quad i=1,2.
\end{equation}
Since, for fixed $\alpha$, $(a_1,a_2)$ is an orthonormal basis of $\RR^2$, by \eqref{curl iff} we see that the vector $\gr r$ can be decomposed as
\begin{equation*}
\gr r = (-\gr \land a_2) a_1 + (\gr \land a_1)a_2.
\end{equation*}
We compute
\begin{equation*}
\gr\land a_1= \frac{\pa\alpha}{\pa x_1} \cos(\alpha)+ \frac{\pa\al}{\pa x_2} \sin(\al)\quad \textrm{ and }\quad \gr\land a_2 = - \frac{\pa \al }{\pa x_1} \sin(\al) + \frac{\pa \al}{\pa x_2} \cos(\al).
\end{equation*}
It leads to 
\begin{equation*}
\gr r=\left(-\frac{\pa\al}{\pa x_2}, \frac{\pa \al}{\pa x_1}\right)^T.
\end{equation*}
Thus, again by Poincar\'e's lemma, the dilation factor $r$ exists if and only if the right hand side in the above is curl-free, which leads to the harmonic condition on $\al$.
\end{proof}
The following proposition is a very useful property of conformal orientations. 
\begin{proposition}
If there exists a map $\varphi=(\varphi_1,\varphi_2)$ from $D\subset\RR^2$ to $\RR^2$ such that  
\begin{equation*}
\gr \varphi= e^r Q(\al) \quad \text{ in }D,
\end{equation*}
then all angles are preserved by the map $\varphi$. In particular, small square cells are deformed into almost square cells locally.
\end{proposition}
\begin{proof}[Sketch of the proof]
Let $x=(x_1,x_2)$ be the origin of a small square $S$ of side length $\e>0$ and edges $\e e_1$, $\e e_2$ (in other words, the vertices of $S$ are $x, x+\e e_1, x+\e e_2$  and $x+\e e_1 + \e e_2$). 
The map $\varphi$ then transforms $S$ into an ``almost'' square $\varphi(S)$ of vertices $\varphi(x), \varphi(x+\e e_1), \varphi(x+\e e_2)$ and $\varphi(x+\e e_1+ \e e_2)$.
By a Taylor expansion, we see that $\varphi(S)$ is, up to terms of order $\e^2$, equal to a parallelogram of origin $\phi(x)$ and edges $\displaystyle\e\frac{\pa \varphi}{\pa x_1}$ and $\displaystyle\e\frac{\pa \varphi}{\pa x_2}$.
Since $\gr \varphi = e^r Q(\alpha)$, the two edges are orthogonal, so $\varphi(S)$ is ``almost'' a square with side length $\e e^r(x)$ (possibly rotated with respect to the initial square $S$).
\begin{figure}[h]
\centering
\includegraphics[width=0.8\linewidth,center]{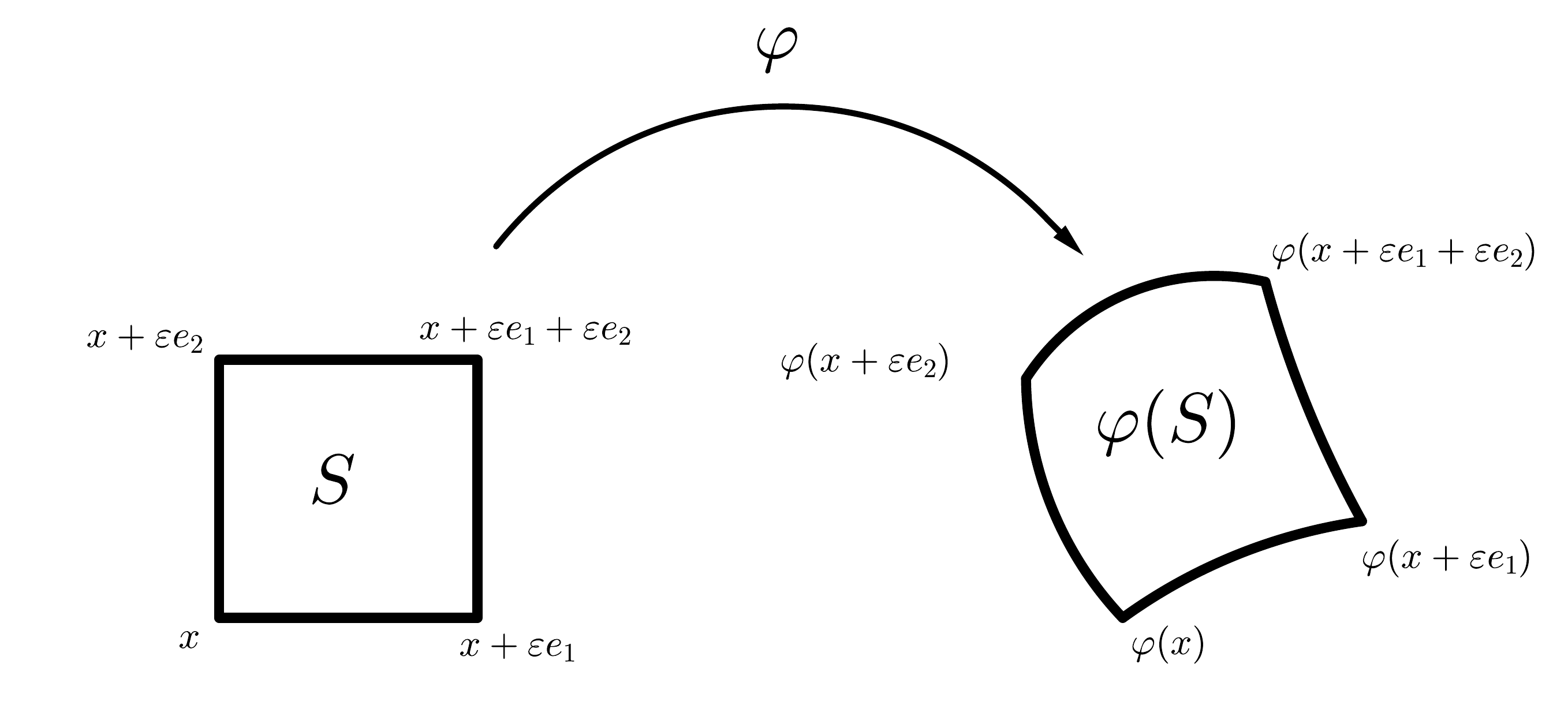}
\caption{A small square $S$ and the distorted square $\varphi(S)$.
} 
\label{distorted square}
\end{figure}
\end{proof}
Nevertheless, in the applications we face a problem. Since $\al$ is a stress eigen-direction, it has no reason to be a harmonic function in general. Even worse, $\alpha$ might not even be smooth at some places (for example, at corners or at the junction point of different boundary conditions, but at other places as well). 
A more profound reason lies in the following observation: both $\alpha$ and $\alpha+\pi$ give rise to the same orientation. By Pedersen's result, we can show that the rotated Hooke's law $A^*(m_1,m_2,\alpha)= R(\alpha)^T A^*(m_1,m_2,0) R(\alpha)$ depends only on the double angle $2\al$. 

From now on, we are going to be working with the double angle $\beta=2\alpha$, thus, removing the indeterminate additive constant $\pi$. 
In what follows, we shall regularize the double angle $\beta=2\alpha$ and make it harmonic.

\subsection{Regularization of the double angle $\beta=2\alpha$}

As we mentioned before, the orientation $\alpha$ given by the optimization does not necessarily satisfy the conformality condition $\Delta \alpha=0$. Thus, at each iteration of the algorithm, instead of minimizing locally (by using Pedersen's result) the following quantity
\begin{equation*}
A^*(m_1,m_2,\beta)^{-1}\sigma\cdot\sigma,
\end{equation*}
let us consider to minimize globally
\begin{equation*}
\int_D \left(A^* (m_1,m_2,\beta)^{-1} \sigma\cdot\sigma +\eta^2 |\gr \beta|^2  \right)\,dx
\end{equation*}
for a small parameter $\eta>0$, under the harmonic constraint
\begin{equation*}
\int_D \gr \beta\cdot\gr q\, dx = 0 \quad \forall q\in H_0^1(D).
\end{equation*}
This is a non-linear (and non-quadratic) constrained optimization problem.
It turns out that working with the angle $\beta$ is not so easy since the Hooke's law $A^*(m_1,m_2,\beta)$ is highly nonlinear in terms of $\beta$. It is however quadratic with respect to the vector $b$, defined by
\begin{equation*}
b=(\cos\beta,\sin\beta) \quad\textrm{ and }\quad A^*(m_1,m_2,\beta)=S(b)^T A(m_1,m_2,0) S(b),
\end{equation*}
where the matrix $S(b)$ is defined by $S(b) = R(\alpha)$. It can be shown that $S(b)$ is affine with respect to $b$ by a careful computation based on Pedersen's result. Therefore, from now on we shall work with $b$ as the main unknown. 
We remark that $\gr\beta=b\land\gr b$. Indeed, 
\begin{equation*}
b\land \gr b = (\cos \beta, \sin \beta)\land \left( - \sin\beta \, \gr \beta, \cos\beta\, \gr \beta \right)= \gr \beta.     
\end{equation*}
Therefore, the objective function becomes
\begin{equation}\label{problem with b}
\int_D \left(A^* (m_1,m_2,b)^{-1} \sigma\cdot\sigma+\eta^2 |b\land \gr b|^2 \right)\, dx,
\end{equation}
which in turn has to be minimized under the harmonic constraint
\begin{equation}\label{constraint of problem with b}
\int_D (b\land \gr b)\cdot \gr q\,dx=0 \quad \forall q\in H_0^1(D).
\end{equation}
We iteratively solve the non-linear problem \eqref{problem with b} under the minimization constraint \eqref{constraint of problem with b} by a Newton-type approximation with an increment $\delta b$ as follows. 

Find a step $\delta b^n\in H^1(D;\RR^2)$ and a Lagrange multiplier $p^{n+1}\in H_0^1(D)$ such that, for any $\delta c \in H^1(D;\RR^2)$ and $q\in H_0^1(D)$,
\begin{equation}\label{linearized minimization problem}
\begin{aligned}
\int_D A^*(m)^{-1} S(b^n+\delta b^n) \sigma\cdot S'(\delta c) \sigma \, dx + \eta^2 \int_D (b^n\land \gr (b^n+\delta b^n))\cdot (b^n\land \gr \delta c)\, dx \\
+ \int_D (b^n\land \gr \delta c)\cdot \gr p^{n+1}\, dx=0
\end{aligned}
\end{equation}
and 
\begin{equation}\label{constraint for harmonic}
\displaystyle \int_D (b^n\land \gr (b^n+\delta b^n))\cdot \gr q \,dx=0,
\end{equation}
where $S'(\delta c)$ is the directional derivative of $S(b)$ in the direction $\delta c$. Notice that, since $S$ is an affine function, then $S'(\delta c)= S(c+\delta c)-S(c)$.
At each iteration, we update the vector field $b$ as follows
\begin{equation}\label{update of the orientation}
b^{n+1}=\frac{b^n+\delta b^n}{|b^n+\delta b^n|}.
\end{equation}
We now apply this regularization process together with the alternate minimization algorithm. 

The above mentioned algorithm is structured as follows: 
\begin{algorithm}[H]
\caption{Regularization algorithm} 
\begin{enumerate}
\item Initialization of the design parameters $m^{0}_{1}, m^{0}_{2}, b$ with the results of the optimization without the harmonic constraint. 
\item Iterations until convergence, for $n \geq 0$: 
\begin{enumerate}
\item Computation of the stress tensor $\sigma^{n}$ through a problem of linear elasticity. 
\item Updating of the hole parameters $m^n$ by using the projected gradient algorithm \eqref{projection algorithm for $m$} with the orientation $b^n$. 
\item Computation of the increment $\delta b^n$ by solving \eqref{linearized minimization problem} and \eqref{constraint for harmonic}. 
\item Updating of the orientation with \eqref{update of the orientation}. 
\end{enumerate}
\end{enumerate}
\end{algorithm}

\begin{remark}
In numerical practice, starting from the optimal (but not necessarily smooth) orientation, a few tens of iterations of this regularization process are enough.
\end{remark}

\begin{remark}
One advantage of this $b$-formulation is that it is insensitive to the $2\pi$-modulo of $\beta$.
\end{remark}

\begin{figure}
\centering
\begin{tabular}{c}
\begin{minipage}{0.4\linewidth}
\begin{center}
\includegraphics[width=1.0\linewidth,center]{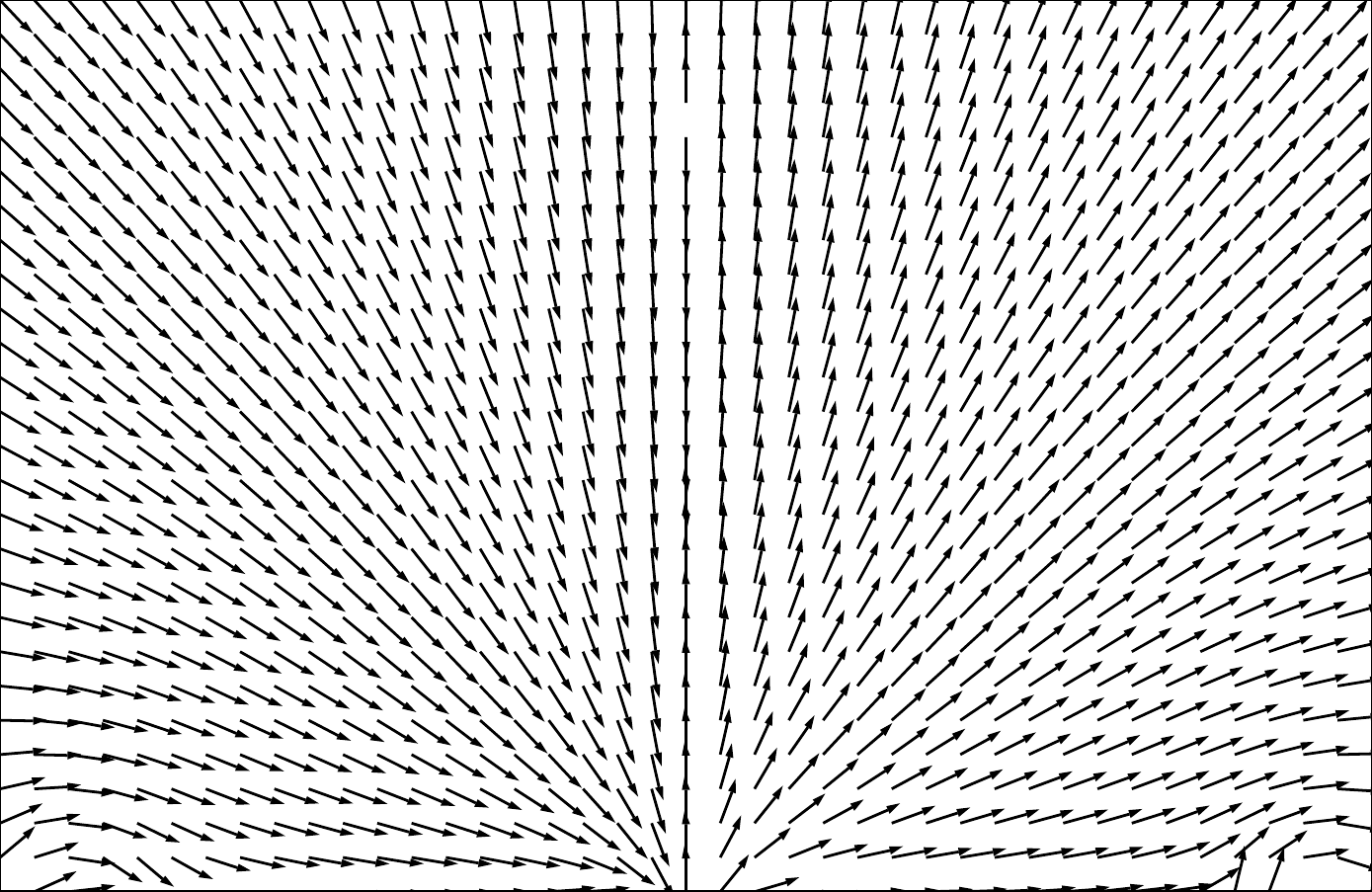}
\end{center}
\end{minipage}
\begin{minipage}{0.5\linewidth}
\begin{center}
\includegraphics[width=1.0\linewidth,center]{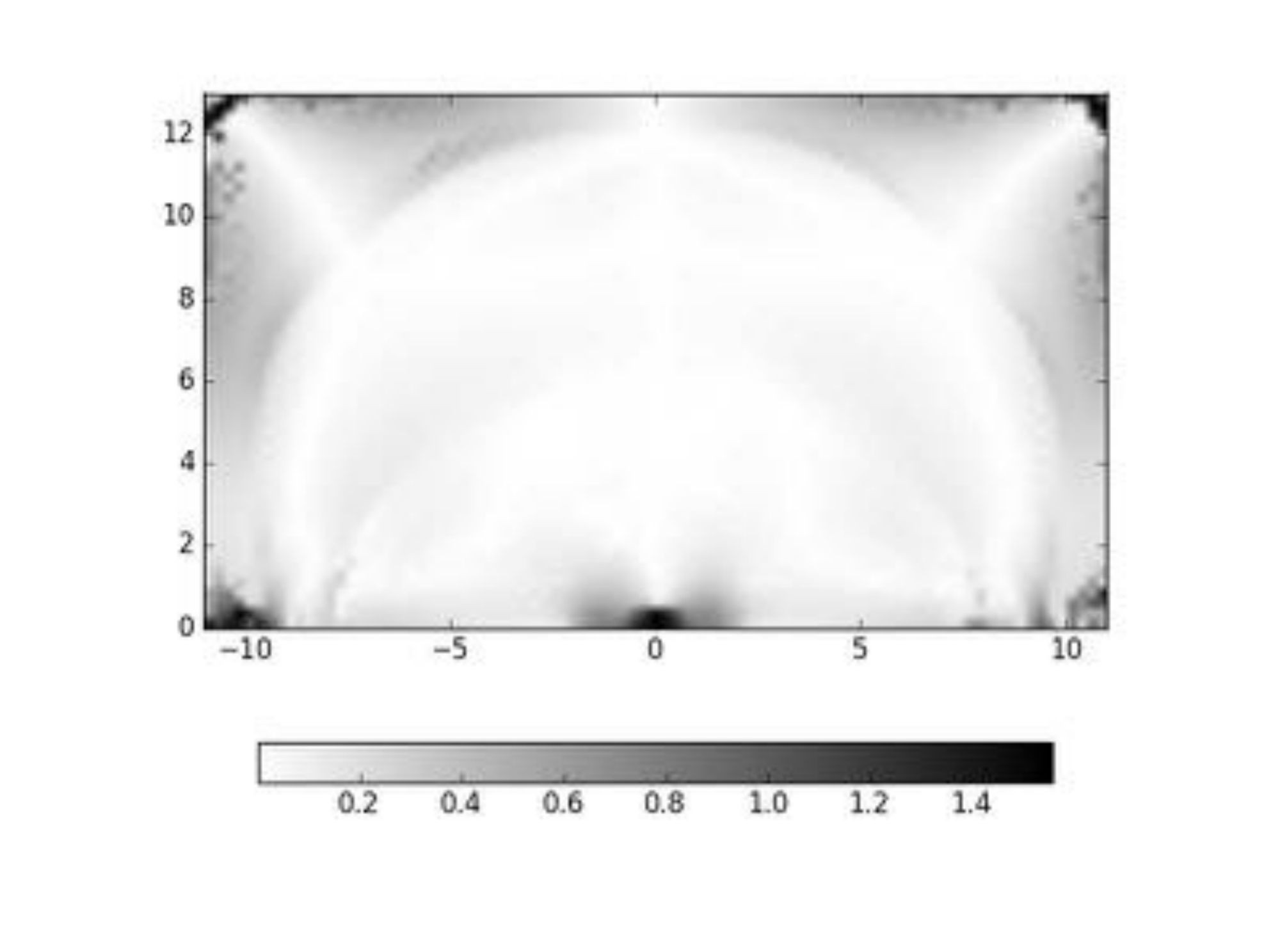}
\end{center}
\end{minipage}
\end{tabular}
\caption{Regularized orientation for the bridge case (left) and angle difference between optimized and regularized orientation.} 
\label{p43-44}
\end{figure}

As we can see by Figure \ref{p43-44}, the regularization occurs mainly in areas where the density is close to $0$ or $1$, i.e. where the homogenized material is almost isotropic and the orientation has no significant impact.

\begin{remark}
Unfortunately this process does not always work in general. In other words, it does not always yield a smooth harmonic angle $\beta$. This is because of ``true'' singularities, i.e. singularities of the orientation that remain and thus do not allow the angle to be harmonic. 
The vector field is not coherently orientable in these cases, as the vector rotates by an angle of $\pm \pi$ along circles which are enclosing the singularities (see Figure \ref{fig:singularity}. 
Such cases might be handled by means of other regularization processes (e.g. minimizing a Ginzburg-Landau energy \cite{geoffroy}).
\end{remark}


\begin{figure}
\centering
\begin{tabular}{c}
\begin{minipage}{0.45\textwidth}
\begin{center}
\includegraphics[width=.7\linewidth,center]{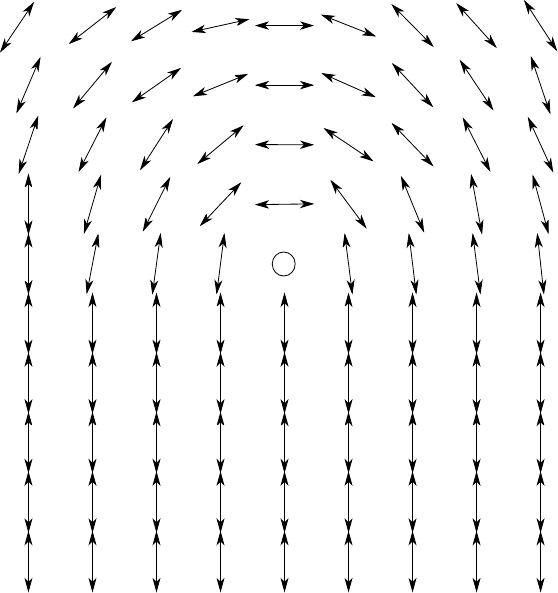}
\end{center}
\end{minipage}
\begin{minipage}{0.45\textwidth}
\begin{center}
\includegraphics[width=.7\linewidth,center]{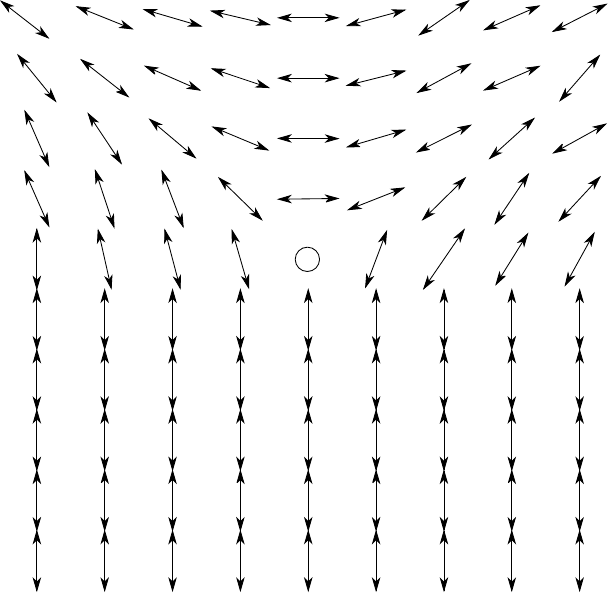}
\end{center}
\end{minipage}
\end{tabular}
\caption{Positive singularity (left) and negative singularity (right).} 
\label{fig:singularity}
\end{figure}

\subsection{Computation of the map $\varphi$}
Once a harmonic angle $\alpha=\beta/2$ has been found, one needs to compute $r$ and $\varphi$ such that
\begin{equation*}
\gr \varphi = e^r Q(\alpha) \quad\textrm{ in }D.
\end{equation*}
Since the dilation field $r$ satisfies 
\begin{equation*}
\gr r=(-\gr\land a_2)a_1+ (\gr\land a_1)a_2 \quad \text{ with } (a_1,a_2)= Q(\alpha),
\end{equation*}
one computes $r$ as the minimizer in $H^1(D)$ of
\begin{equation*}
\int_D |\gr r + (\gr \land a_2)a_1- (\gr \land a_1)a_2   |^2\, dx.
\end{equation*}
Once $r$ has been computed, a naive idea would be to compute $\varphi$ as a minimizer in $H^1(D;\RR^2)$ of 
\begin{equation*}
\int_D |\gr \varphi - e^r Q(\al)|^2\, dx.
\end{equation*}
However, we know that, even if $\beta$ is smooth, $\alpha$ may have jumps of the type $\pm \pi$ and thus $Q(\al)$ may have jumps of its sign (recall that $Q(\al+\pi)=-Q(\al)$).

To compute $\varphi$ there are two possibilities.
\begin{itemize}
	\item[{\rm 1}.]
    Find a coherent orientation of $\alpha$ (i.e. choose between $\alpha$ and $\alpha+\pi$ at every point)$\colon$ this is possible only if there are no singularities (this is the approach of Groen and Sigmund \cite{GS}).
    \item[{\rm 2}.]
    Leave the angle $\alpha$ as it is and extend $\varphi$ to be defined in an abstract manifold.
    This is the approach of Allaire--Geoffroy--Pantz \cite{Allaire5} and it works also in the presence of singularities.
\end{itemize}

\subsection{An abstract manifold setting}
Let us introduce the cover space of $D$. 
\begin{definition}
	Denote by $T$ a rotation matrix field which is a candidate for being $Q(\alpha)$. Then we define 
   \begin{equation*}
   		\mathcal{D}=\{(x, T)\in D\times {\rm SO}(2) \; \, \text{such that} \, \; T^2=Q(\beta)\},
   \end{equation*}
	where ${\rm SO}(2)$ is the set of rotations in $\RR^2$.
\end{definition}
We note that at every point $x\in D$ the rotation satisfies $T(x)^2=Q(\beta)(x)$. If the angle $\alpha$ is globally orientable, then $T(x)=Q(\alpha)(x)$ or $T(x)=-Q(\alpha)(x)$, and that $\mathcal{D}$ is simply the union of two copies of $D$, consisting of the two possible signs of $Q(\alpha)$. We assume the simple case where $\alpha$ could be globally oriented (no singularity) but extend it to the singular cases. 

\begin{figure}
\centering
\begin{tabular}{c}
\begin{minipage}{0.6\linewidth}
\begin{center}
\includegraphics[width=0.9\linewidth,center]{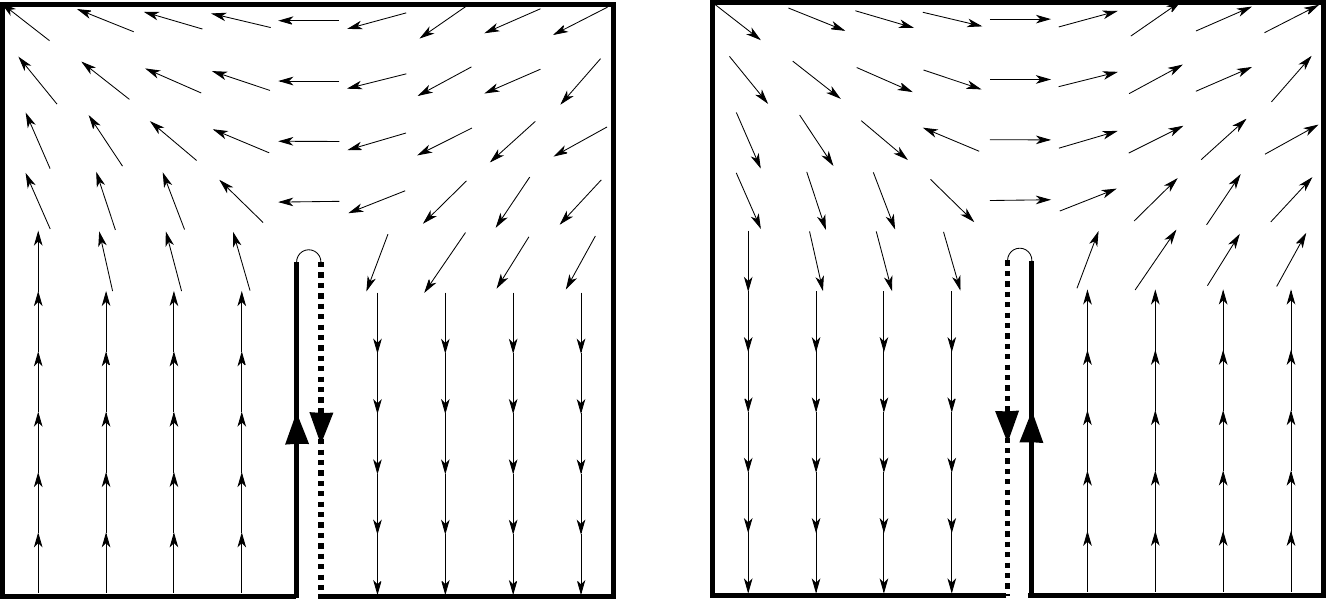}
\end{center}
\end{minipage}
\begin{minipage}{0.4\linewidth}
\begin{center}
\includegraphics[width=.9\linewidth,center]{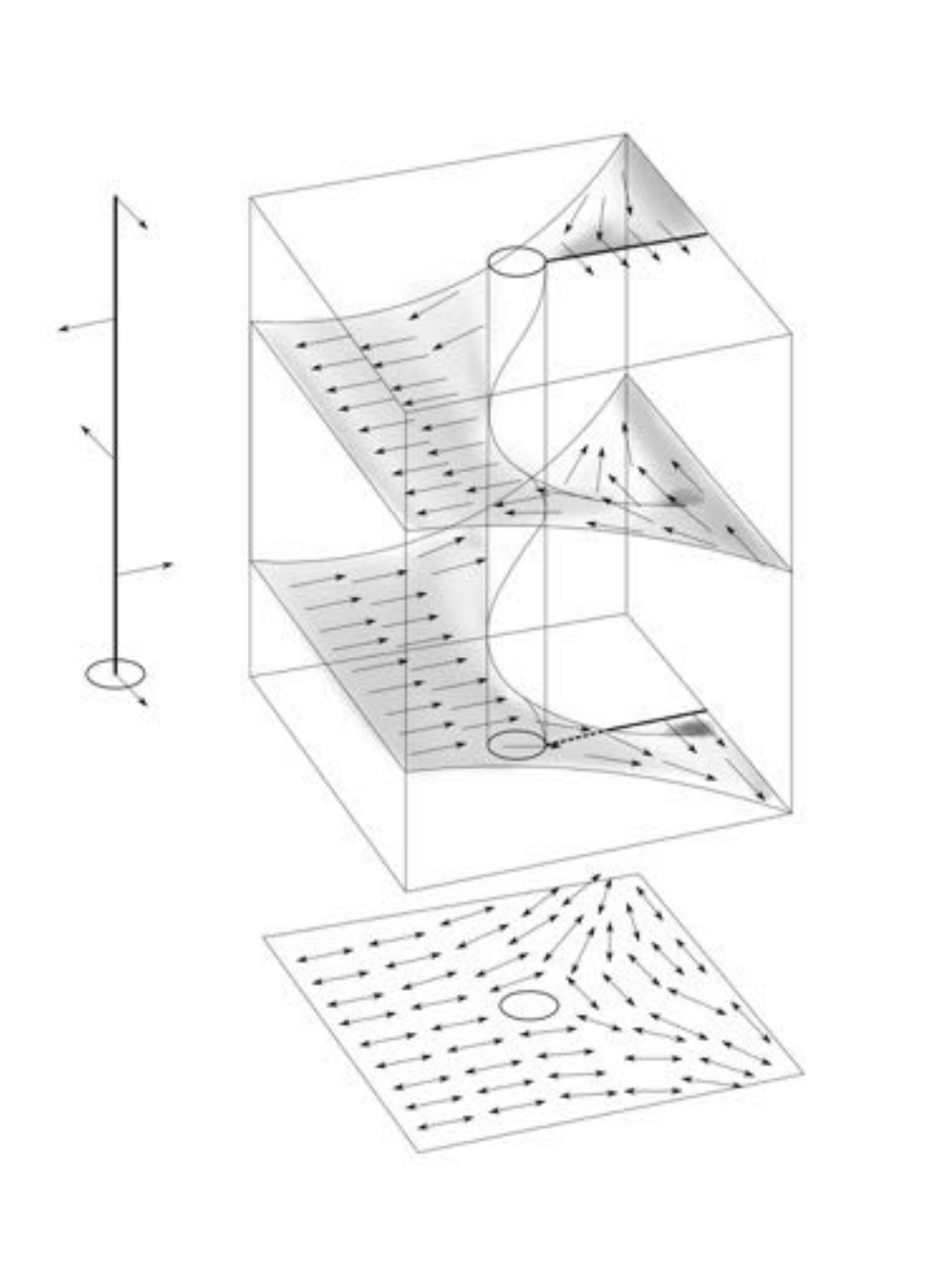}
\end{center}
\end{minipage}
\end{tabular}
\caption{Two possible orientation (left) and the manifold obtained by gluing them together.} 
\label{fig: cut and manifold}
\end{figure}
We change our working space from $D$ to $\mathcal{D}$. The map $\varphi(x, T)$ is now defined on the manifold $\mathcal{D}$ by
\begin{equation}\label{map of manifold}
	\gr \varphi=e^r T,
\end{equation}
and the gradient operator in \eqref{map of manifold} defined by
\begin{equation*}
	\gr\varphi(x, T)=\gr \varphi_{U}(x),
\end{equation*}
where $U$ is an orientable open subset of $D$, 
\begin{equation*}
	\varphi_{U}(x)=\varphi\circ g_{U}(x),
\end{equation*}
and $g_{U}$ is one of the charts
\begin{equation*}
	\begin{aligned}
		g^{\pm}_U \, \colon \, & U\longrightarrow \mathcal{D}\\
      & x\longmapsto(x, \pm T_U(x))
	\end{aligned}
\end{equation*}
with $T_U^2=Q(\beta)$ and $T_U\in C(U, {\rm SO}(2))$.
Moreover, without loss of generality, we can assume the antisymmetric property
\begin{equation*}
	\varphi(x, -T)=-\varphi(x, T).
\end{equation*}
Indeed, if $\varphi$ satisfies $\gr\varphi=e^r T$, then the map $(\varphi(x, T)-\varphi(x, -T))/2$ still satisfies \eqref{map of manifold} and is antisymmetric. Thus, if the orientation $\alpha$ satisfies the conformality condition \eqref{conformality condition}, the map $\varphi$ can be defined as a minimizer of
\begin{equation*}
	\min_{\varphi\in \mathcal{V}}\int_\mathcal{D}|\gr\varphi-e^r T|^2\,dx,
\end{equation*}
over all maps $\varphi$ in
\begin{equation*}
	\mathcal{V}:=\{\varphi\in H^1(D, \RR^2) \;:\;\varphi(x, -T)=-\varphi(x, T) \ {\rm for} \ {\rm all} \ (x, T)\in\mathcal{D}\}.
\end{equation*}
In practice we face the problem of making the actual computations on the abstract manifold $\mathcal{D}$.
In order to solve this problem we use a new idea, namely, non-conformal finite elements on $D$.

\begin{figure}
\centering
\begin{tabular}{c}
\begin{minipage}{0.4\linewidth}
\begin{center}
\includegraphics[width=0.9\linewidth,center]{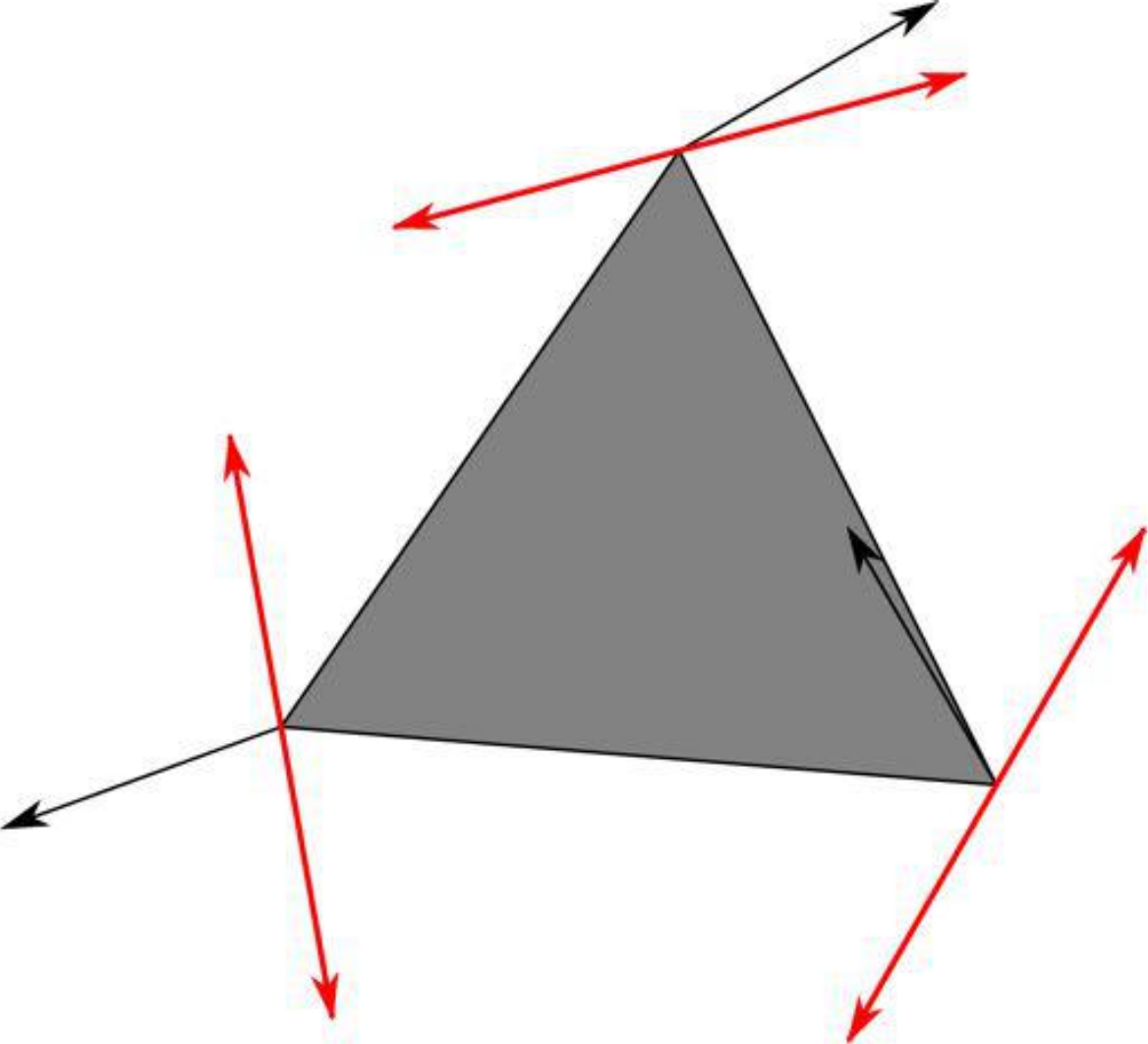}
\end{center}
\end{minipage}
\begin{minipage}{0.6\linewidth}
\begin{center}
\includegraphics[width=.9\linewidth,center]{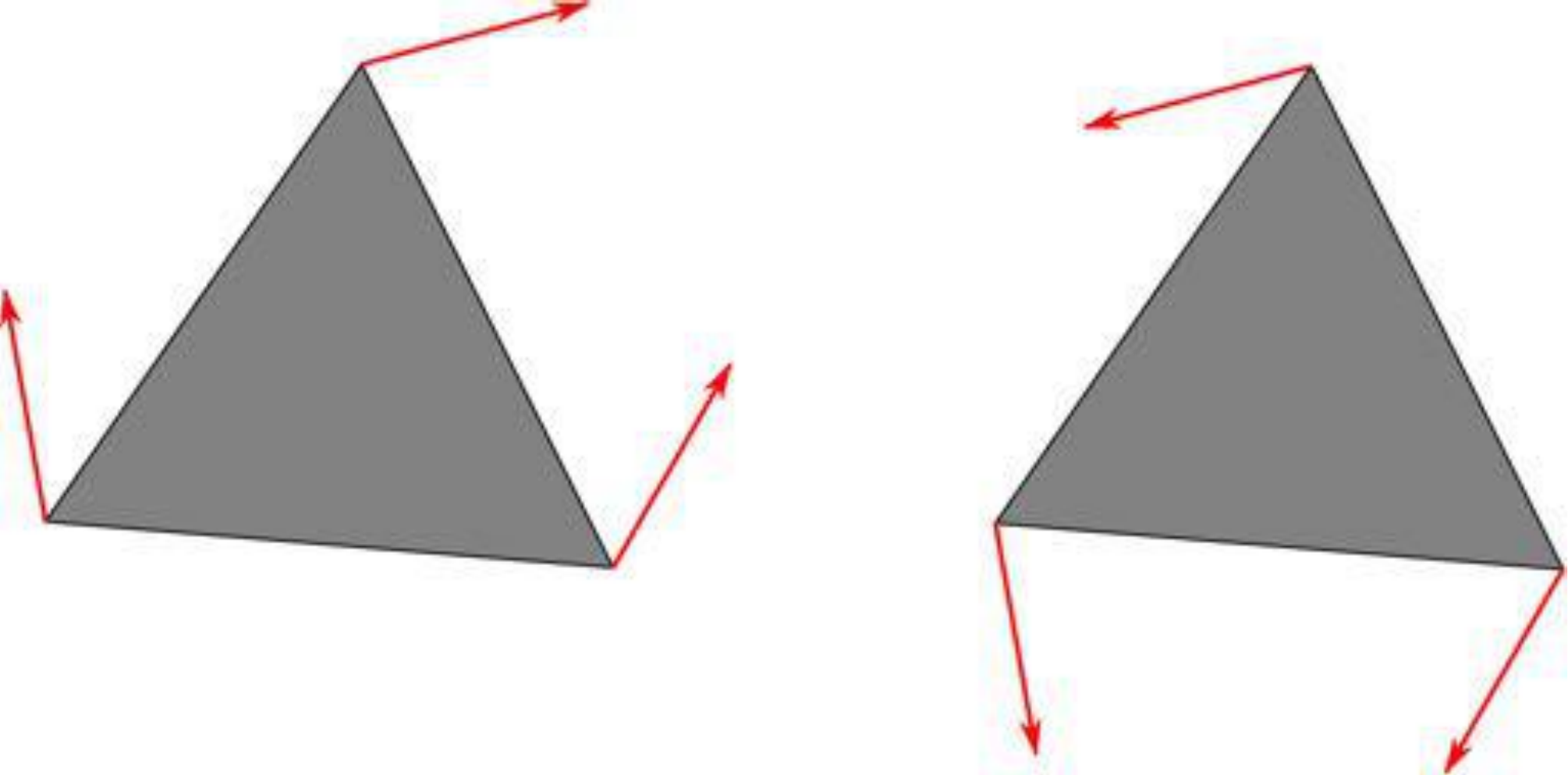}
\end{center}
\end{minipage}
\end{tabular}
\caption{Left: orientation of $\beta$ (black arrows) and of $\alpha$ (red arrows). Right: two
possible coherent orientations of $\alpha$.} 
\label{p55}
\end{figure}

On each triangle $K$ of the mesh we compute one continuous orientation $T_K$ such that $T_K^2=Q(\beta)$.
Then, we glue together (with $P_1$ discontinuous finite elements) these orientations. We compute
\begin{align*}
	\int_\mathcal{D}|\gr\varphi-e^rT|^2\,dx &= \sum_K\int_{g^+_K(K)\cup g^-_K(K)}|\gr \varphi-e^rT|^2\,dx \\
	&= \sum_K \int_{K} \abs{\nabla (\varphi \circ g^+_K) - e^r T_K(x)}^2 \, dx \\
	&\quad + \sum_K \int_{K} \abs{\nabla (\varphi \circ g^-_K) - e^r T_K(x)}^2 \, dx
\end{align*}
with $g^{\pm}_K = \mathrm{Id} \times (\pm T_K)$. By the antisymmetry of $\varphi$, we obtain 
\begin{equation*}
	\int_\mathcal{D}|\gr\varphi-e^rT|^2\,dx=2\sum_K\int_K|\gr (\varphi\circ g^+_K)-e^{r(x)}T_K(x)|^2\,dx. 
\end{equation*}
Then, we minimize with respect to $\varphi$ in the space of $P_1$ discontinuous finite elements.

\begin{figure*}
        \centering
        \begin{subfigure}[b]{0.475\textwidth}
            \centering
            \includegraphics[width=\textwidth]{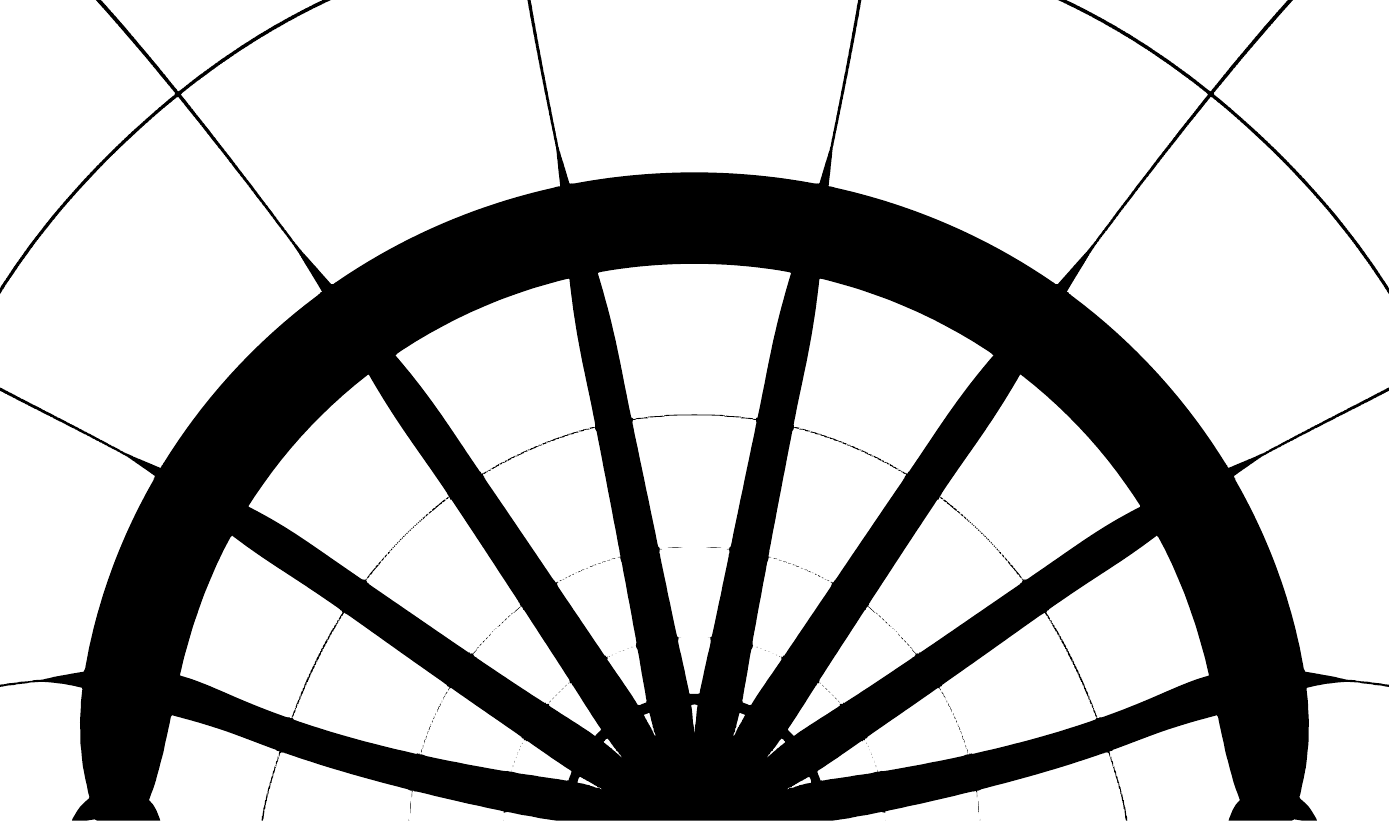}
            \caption[]%
            {{\small $\e=0.4$}}    
        \end{subfigure}
        \hfill
        \begin{subfigure}[b]{0.475\textwidth}  
            \centering 
            \includegraphics[width=\textwidth]{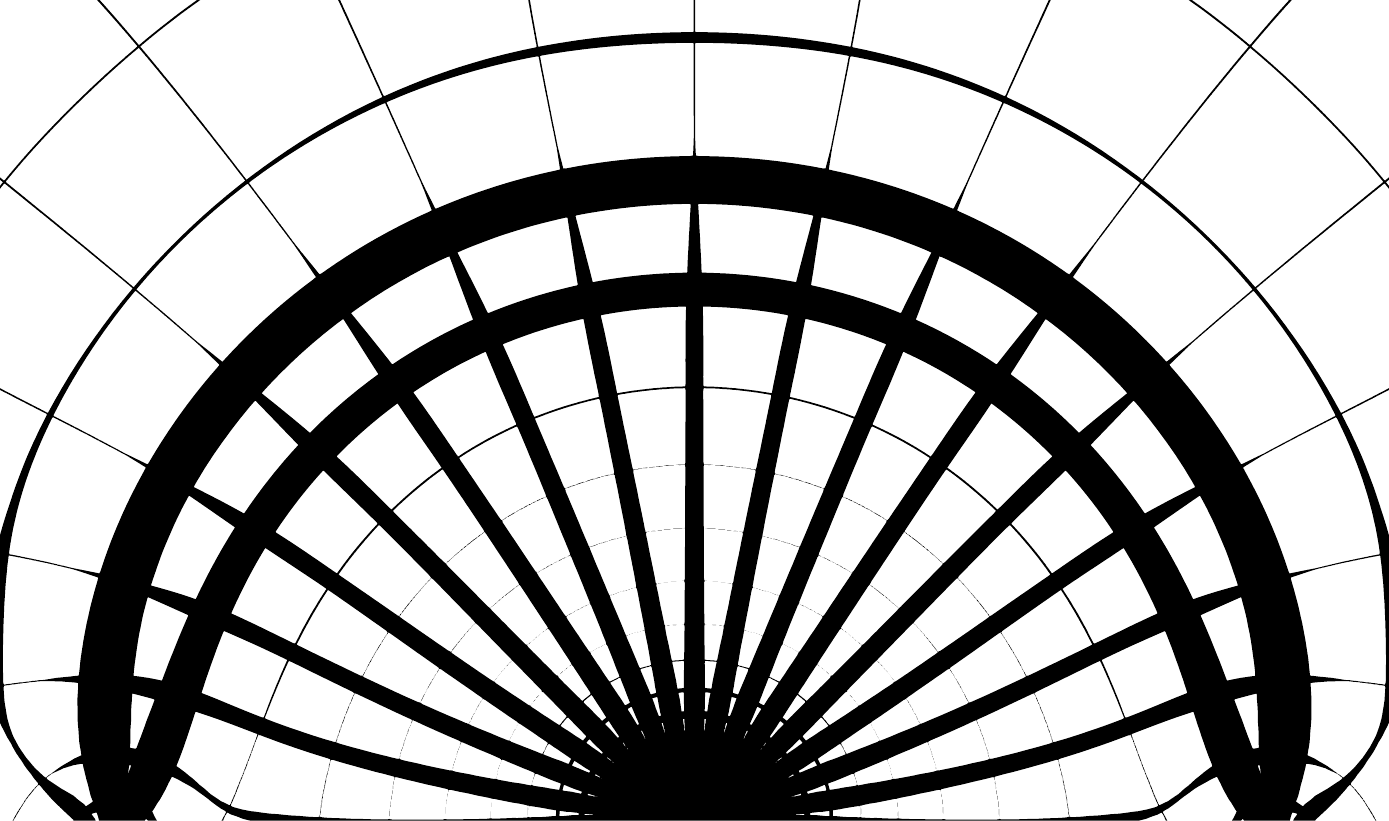}
            \caption[]%
            {{\small $\e=0.2$}}    
        \end{subfigure}
        \vskip\baselineskip
        \begin{subfigure}[b]{0.475\textwidth}   
            \centering 
            \includegraphics[width=\textwidth]{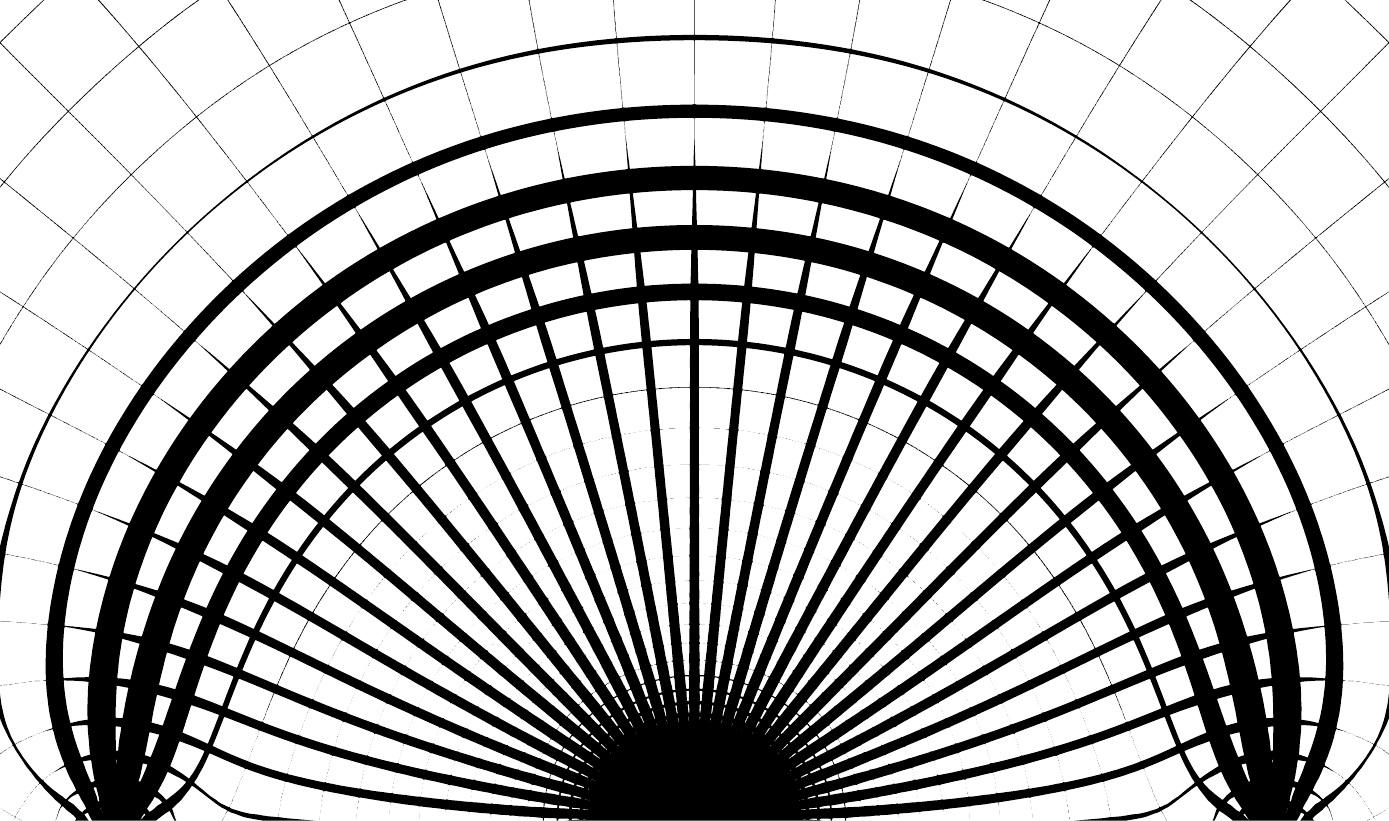}
            \caption[]%
            {{\small $\e=0.1$}}    
        \end{subfigure}
        \quad
        \begin{subfigure}[b]{0.475\textwidth}   
            \centering 
            \includegraphics[width=\textwidth]{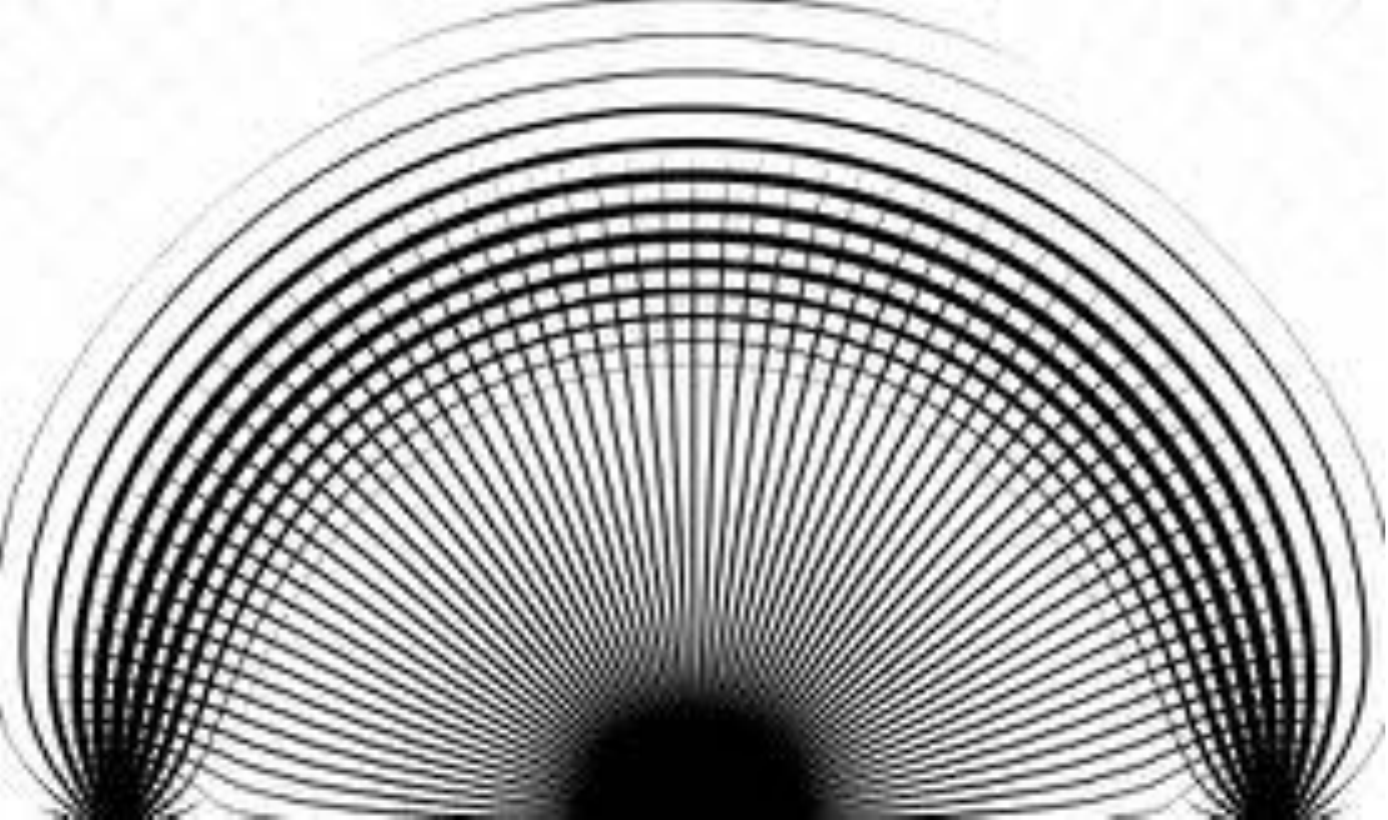}
            \caption[]%
         {{\small $\e=0.05$}}    
        \end{subfigure}
        \caption[]
        {{\small $\Omega_\e (\varphi,m)$ for several $\e$ in the case of the bridge.}}
      
    \end{figure*}

\begin{figure}
\centering
\begin{tabular}{c}
\begin{minipage}{0.33\linewidth}
\begin{center}
\includegraphics[width=\linewidth,center]{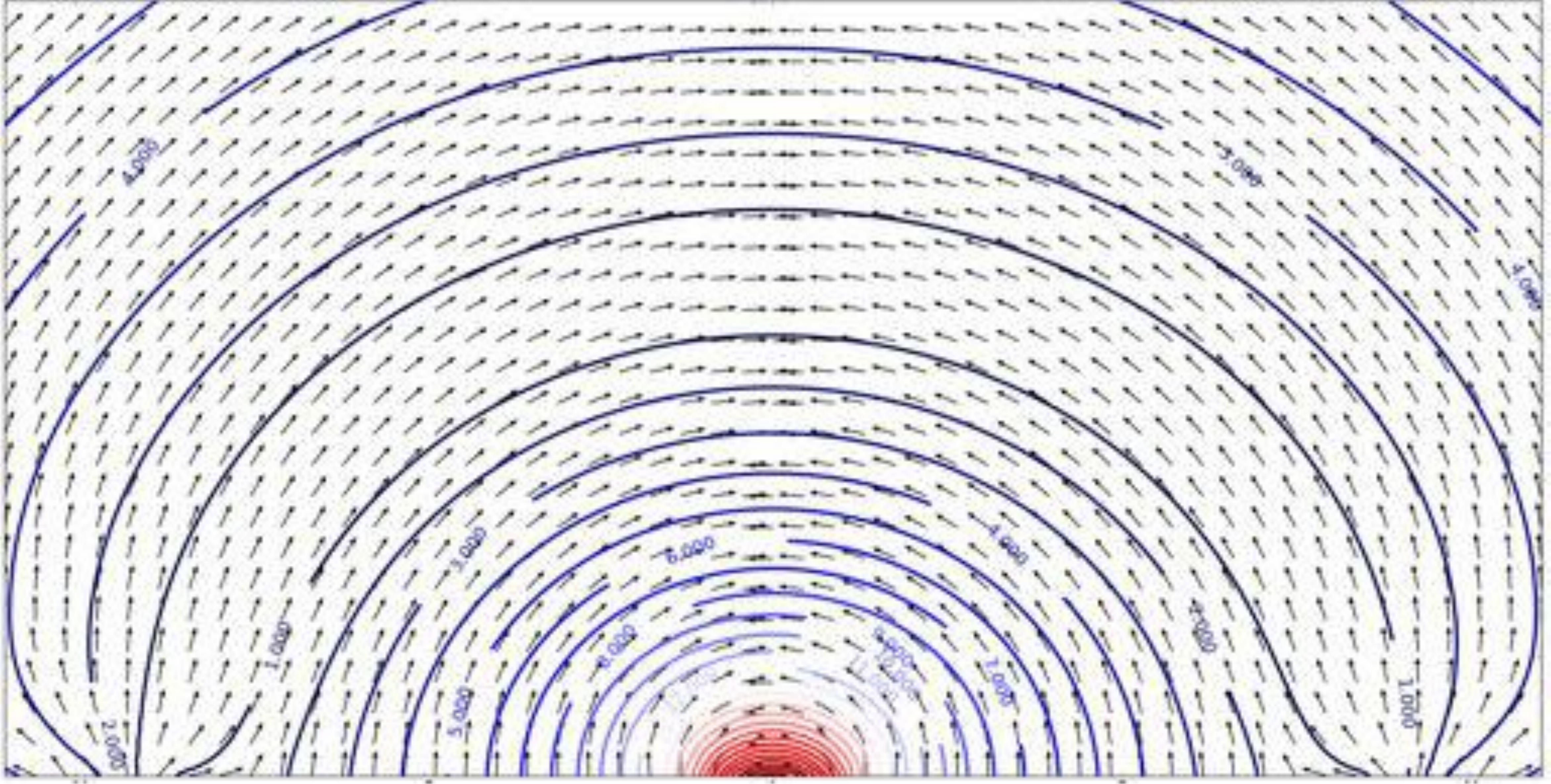}
\end{center}
\end{minipage}
\begin{minipage}{0.33\linewidth}
\begin{center}
\includegraphics[width=\linewidth,center]{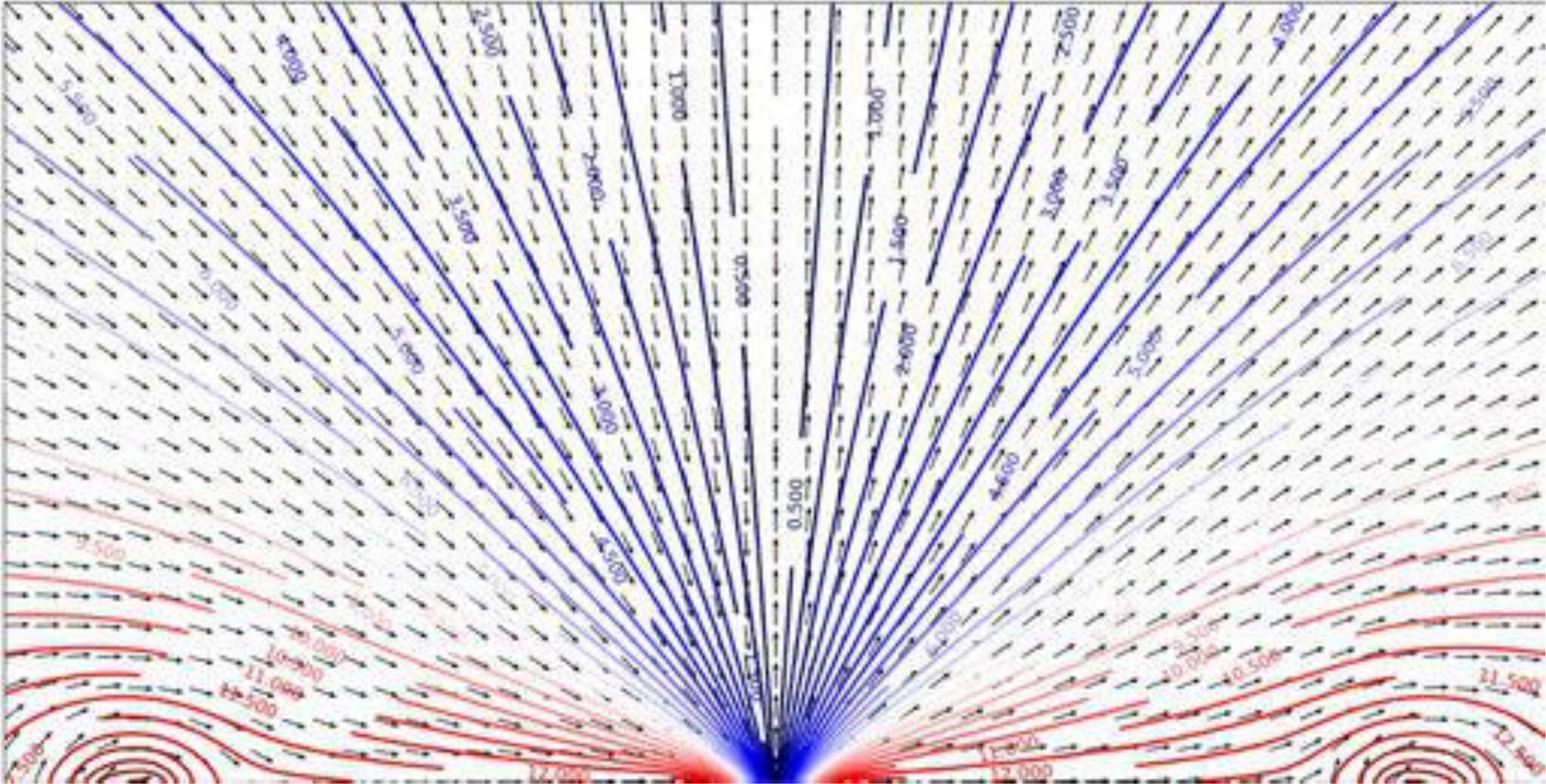}
\end{center}
\end{minipage}
\begin{minipage}{0.33\linewidth}
\begin{center}
\includegraphics[width=\linewidth,center]{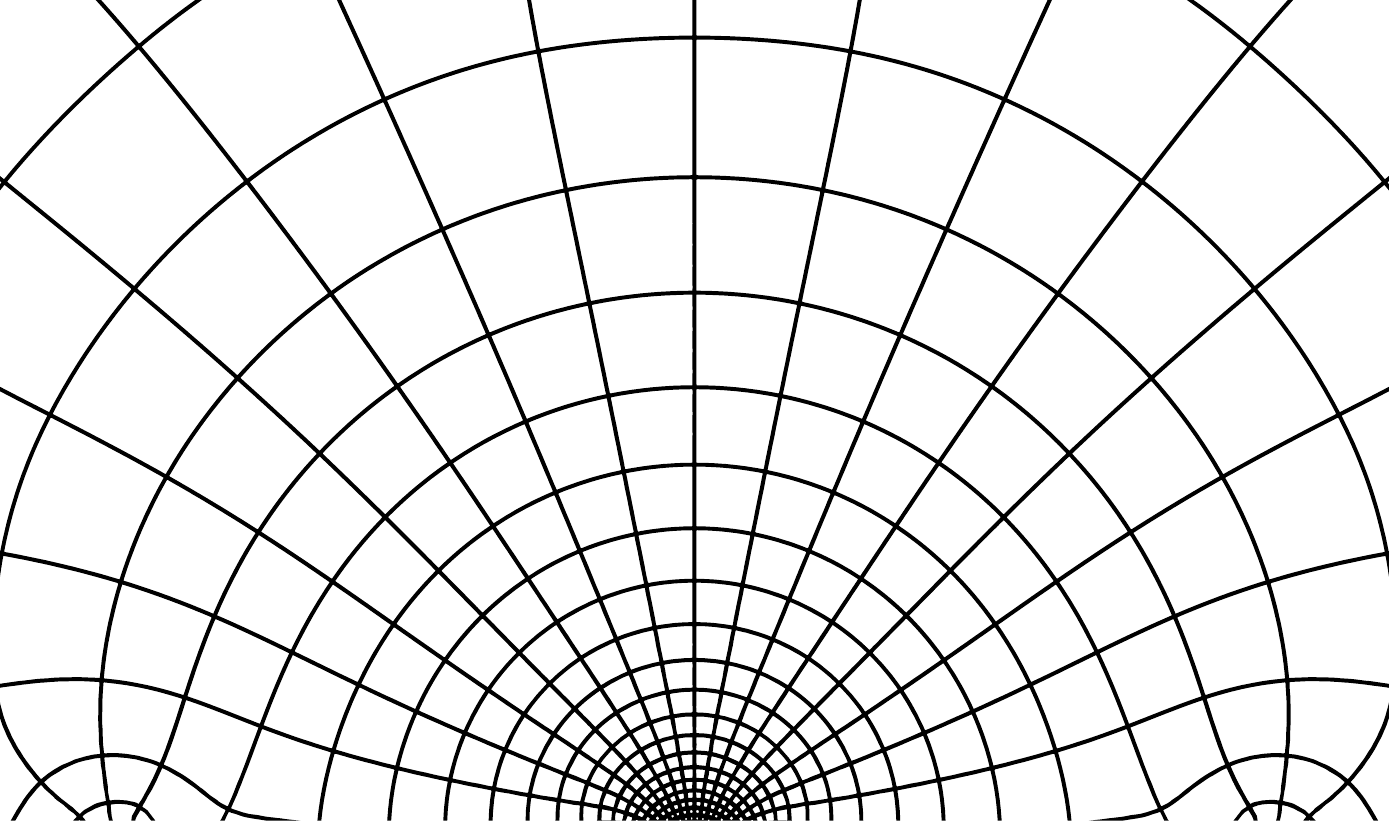}
\end{center}
\end{minipage}
\end{tabular}
\caption{The map $|\varphi_i|$ (isolines) and the vectors $a_i$ (arrows) for $i=1$ (left) and $i=2$ (middle). On the right we have the projection of a regular grid by the map $\varphi$.} 
\label{fig:p57-58}
\end{figure}

\section{A final post-processing/cleaning of the lattice reconstruction}

The shapes we obtained in Section~\ref{sec:3rd step} are not straightforwardly manufacturable.
Indeed, there are disconnected components of the lattice structure and/or too thin members that should be removed.
A final post-processing is made to cure these defects.
Note that there is room for improvement in the process.


Let $h_{\rm min}$ be the minimal manufacturable lengthscale or feature size, meaning the smallest possible width of bars and diameter of holes which can be effectively built.
Recall that $\e$ is our choice of a global size of cells.
After deformation, the cell size is $h_c(x)=\e e^{-r(x)}$.
Hence the local widths of the bars and holes are respectively given by $(1-m_i(x))h_c(x)$ and $m_i(x)h_c(x)$.

In the following, we distinguish two regimes, depending of the local size of the cell $h_c(x)$.
First, if the cell size is too small, a hole and a bar of minimal width cannot coexist and then we have to choose a completely full or void cell.
Hence, if $h_c<2h_{\rm min}$, a thresholding is applied separately to each field $m_i$: it is assigned the value $0$ if $m_i<0.5$ and $1$ otherwise.

Second, when $h_c\ge 2h_{\rm min}$, our post-processing criterion is satisfied if
\begin{equation*}
	\dfrac{h_{\rm min}}{h_c}\le m_i\le1-\dfrac{h_{\rm min}}{h_c}, \quad i=1, 2.
\end{equation*}
Otherwise, we simply threshold the values of $m_1$ and $m_2$, according to Figure~\ref{thresholding}, in order to reach void or full materials.
The thresholded $m$ is then denoted by $\tilde{m}$.

\begin{figure}[htb]
\centering
\includegraphics[width=0.7\linewidth,center]{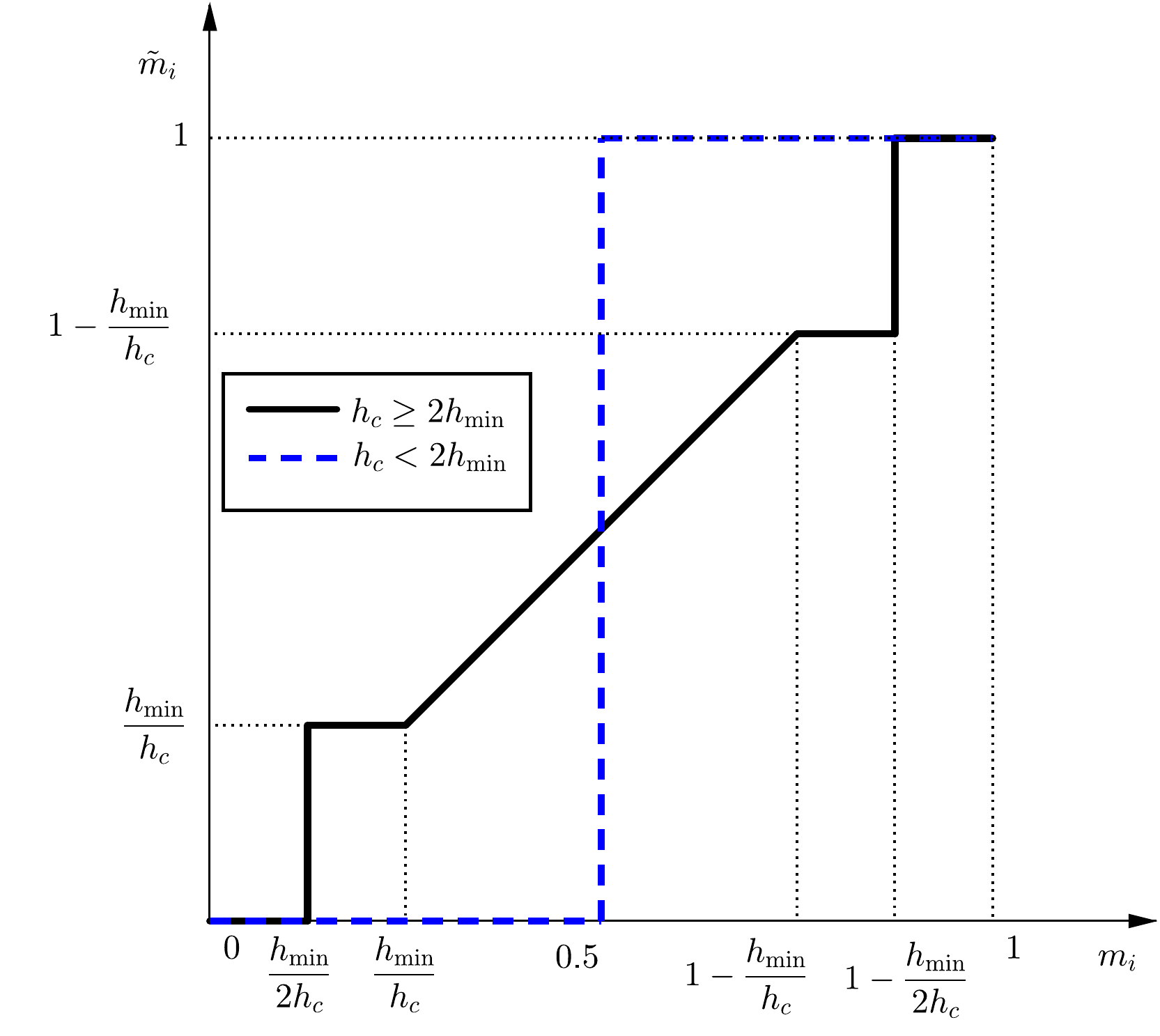}
\caption{Thresholding criteria.} 
\label{thresholding}
\end{figure}

Let $O_\e(\varphi, \tilde{m})$ be the shape obtained from $\Omega_\e(\varphi, \tilde{m})$ by filling its closed holes.
Numerically, the complement of $\Omega_\e(\varphi, \tilde{m})$ is computed step by step, by evaluating the sign of $F^{\varphi, \tilde{m}}_\e$.
If it is positive, the current vertex belongs to the complement $\Omega^c_\e(\varphi, \tilde{m})$ and then its neighbors, which are not already visited, are added to the list of vertices which should be tested.
Otherwise, the current vertex does not belong to $\Omega^c_\e(\varphi, \tilde{m})$.

We will regularize the subset $O_\e(\varphi, \tilde{m})$ in order to remove the disconnected bars or the bars that have one free endpoint.
Numerically, we explore all the vertices of the complement as follows.
For any given vertex, we check each other vertex not further away than a distance $h_{\rm min}$: 
if this vertex belongs to the complement too, all vertices between them are added to the complement.
In this way, we suppress all disconnected bars and all bars of $O_\ve(\varphi, \tilde{m})$ that have one free endpoint, which are not too wide.
This new subset is denoted by $\tilde{O}_\ve(\varphi, \tilde{m})$.

Finally, the post-processed structure is given by the intersection $\tilde{\Omega}_\ve(\varphi, \tilde{m}):=\Omega_\ve(\varphi, \tilde{m})\cap\tilde{O}_\ve(\varphi, \tilde{m})$.
Several post-processed structures $\tilde{\Omega}_\ve(\varphi, \tilde{m})$ for the bridge case are displayed in Figure~\ref{c5p63}.

\begin{figure*}[h]
        \centering
        \begin{subfigure}[b]{0.475\textwidth}
            \centering
            \includegraphics[width=\textwidth]{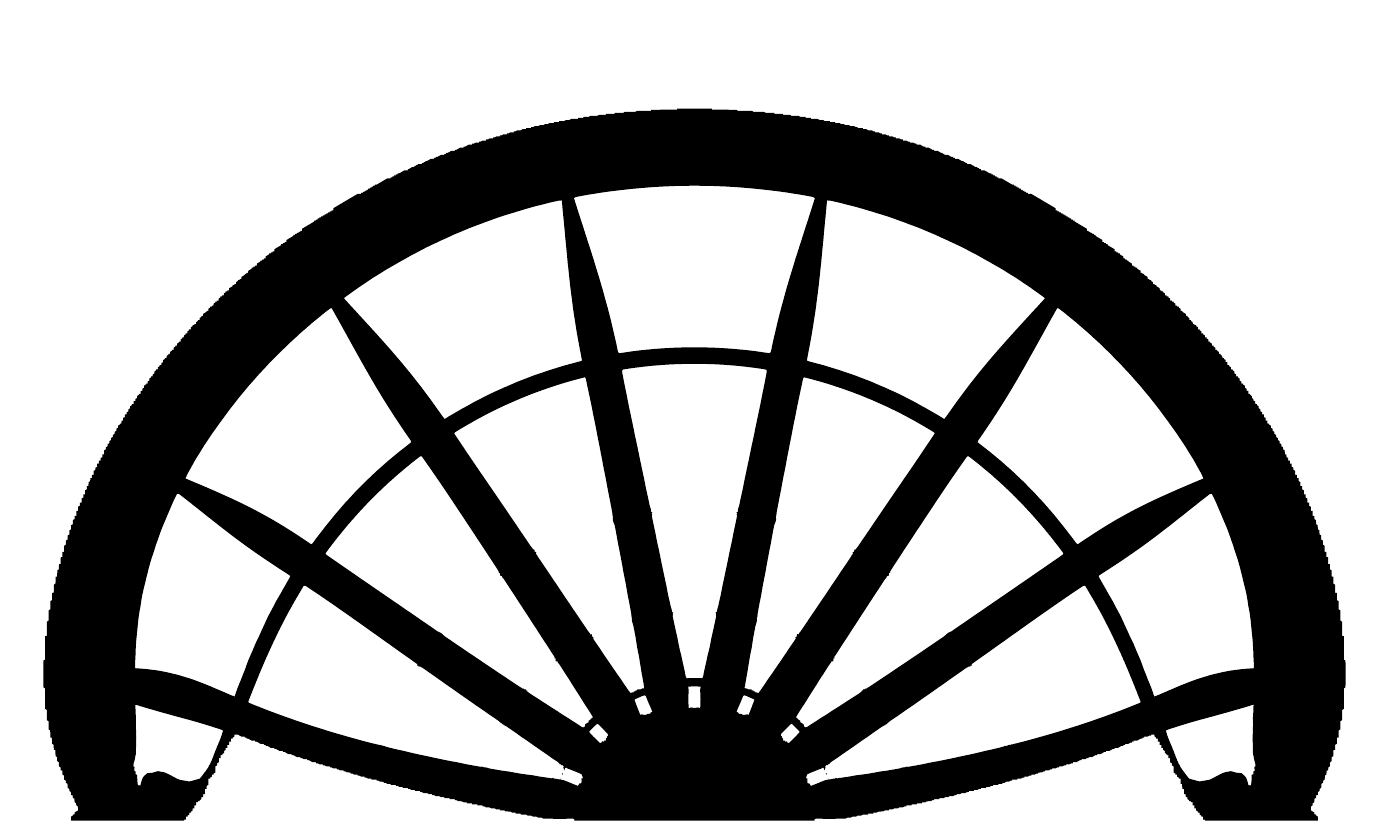}
            \caption[]%
            {{\small $\e=0.4$}}    
        \end{subfigure}
        \hfill
        \begin{subfigure}[b]{0.475\textwidth}  
            \centering 
            \includegraphics[width=\textwidth]{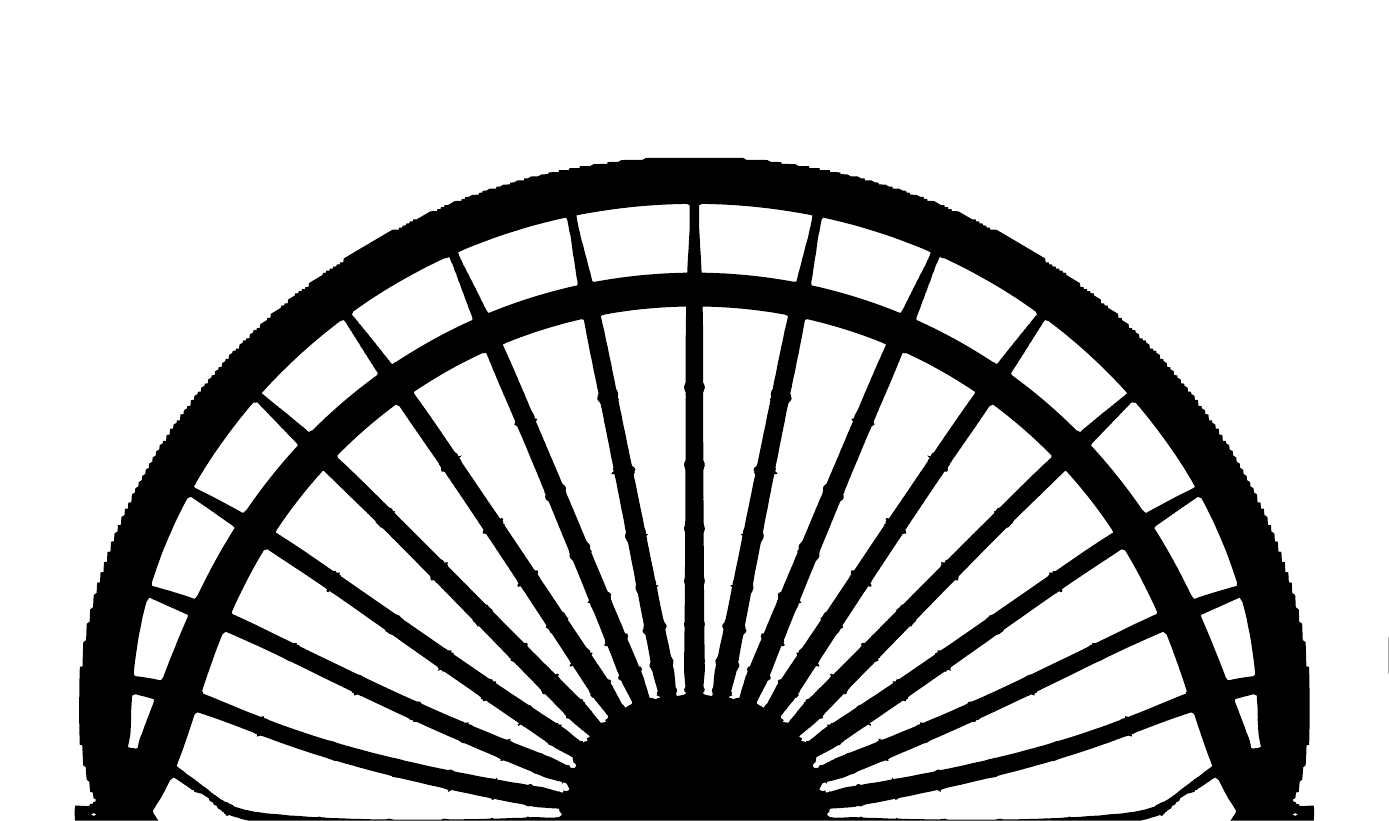}
            \caption[]%
            {{\small $\e=0.2$}}    
        \end{subfigure}
        \vskip\baselineskip
        \begin{subfigure}[b]{0.475\textwidth}   
            \centering 
            \includegraphics[width=\textwidth]{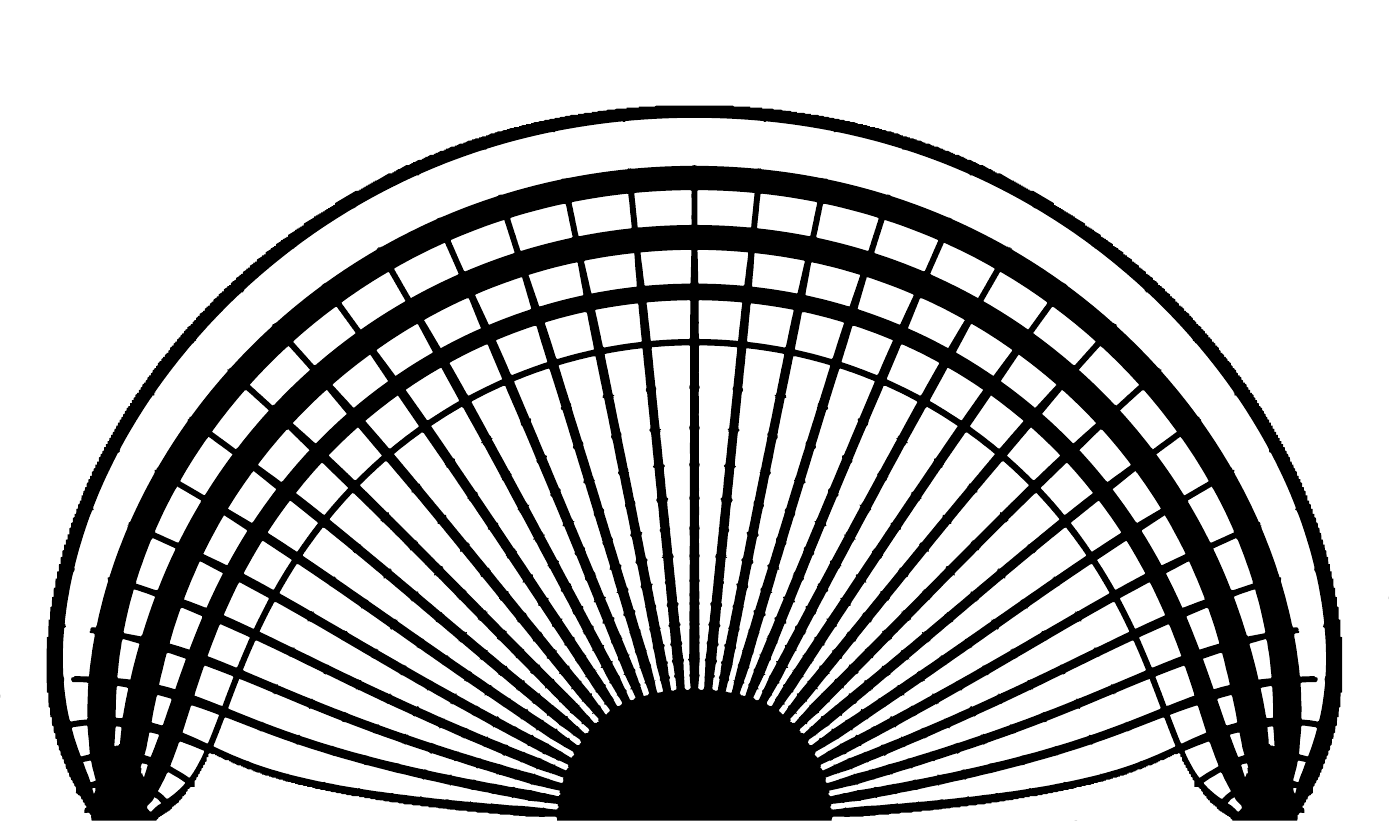}
            \caption[]%
            {{\small $\e=0.1$}}    
        \end{subfigure}
        \quad
        \begin{subfigure}[b]{0.475\textwidth}   
            \centering 
            \includegraphics[width=\textwidth]{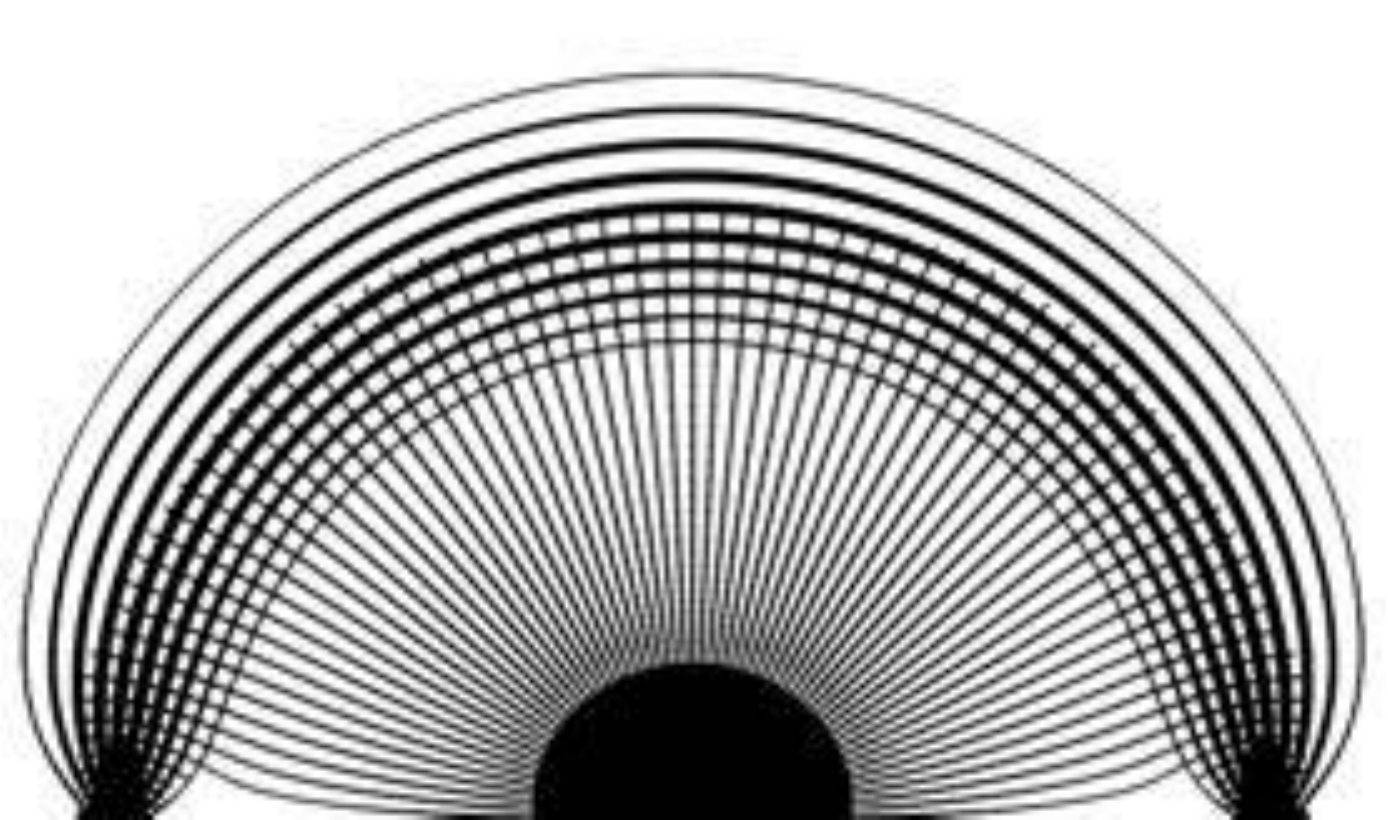}
            \caption[]%
         {{\small $\e=0.05$}}    
        \end{subfigure}
        \caption[]
        {{\small Post-processed $\bar\Omega_\e (\varphi,\bar m)$ for several $\e$ in the case of the bridge.}}
      
        \label{c5p63}
    \end{figure*}


\section{Other numerical examples} \label{sec:other numerical ex}

In this section, we show some numerical examples of the application of the whole method of this chapter.
Here, we mention that the method is not always applicable because some singularities cannot be eliminated, as we will see later.

We will show the numerical results for the optimization of a cantilever (Figure \ref{fig:cantilever case}), an MBB beam (Figure \ref{fig: MBB beam case}) and an L-beam (Figure \ref{fig: L-beam case}).
For each case, we have represented:
\begin{itemize}
    \item[(a)] 
        the optimal orientation of the periodicity cells before regularization,
    \item[(b)]
        the optimal orientation of the periodicity cells after regularization,
    \item[(c)]
        the underlying lattice on which the optimal composite is built, i.e., the projection by $\varphi$,
    \item[(d)-(f)]
        the sequence of shapes after post-processing for the case of $\e=0.2$, $0.1$ and $0.05$ respectively.
\end{itemize}

\begin{figure}[htb]
    \centering 
\begin{subfigure}{0.25\textwidth}
  \includegraphics[width=\linewidth]{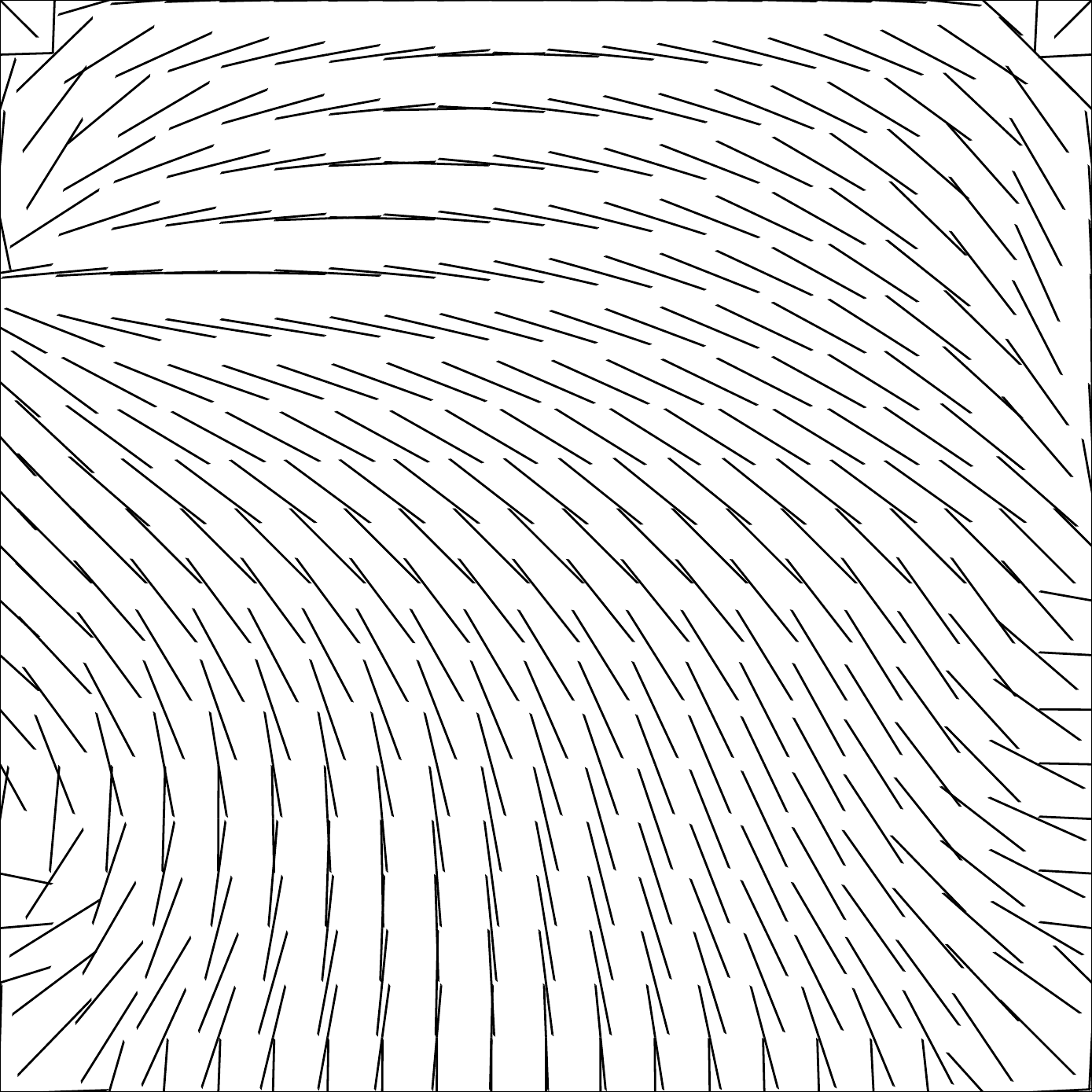}
  \caption{Opt. orientation}
\end{subfigure}\hfil 
\begin{subfigure}{0.25\textwidth}
  \includegraphics[width=\linewidth]{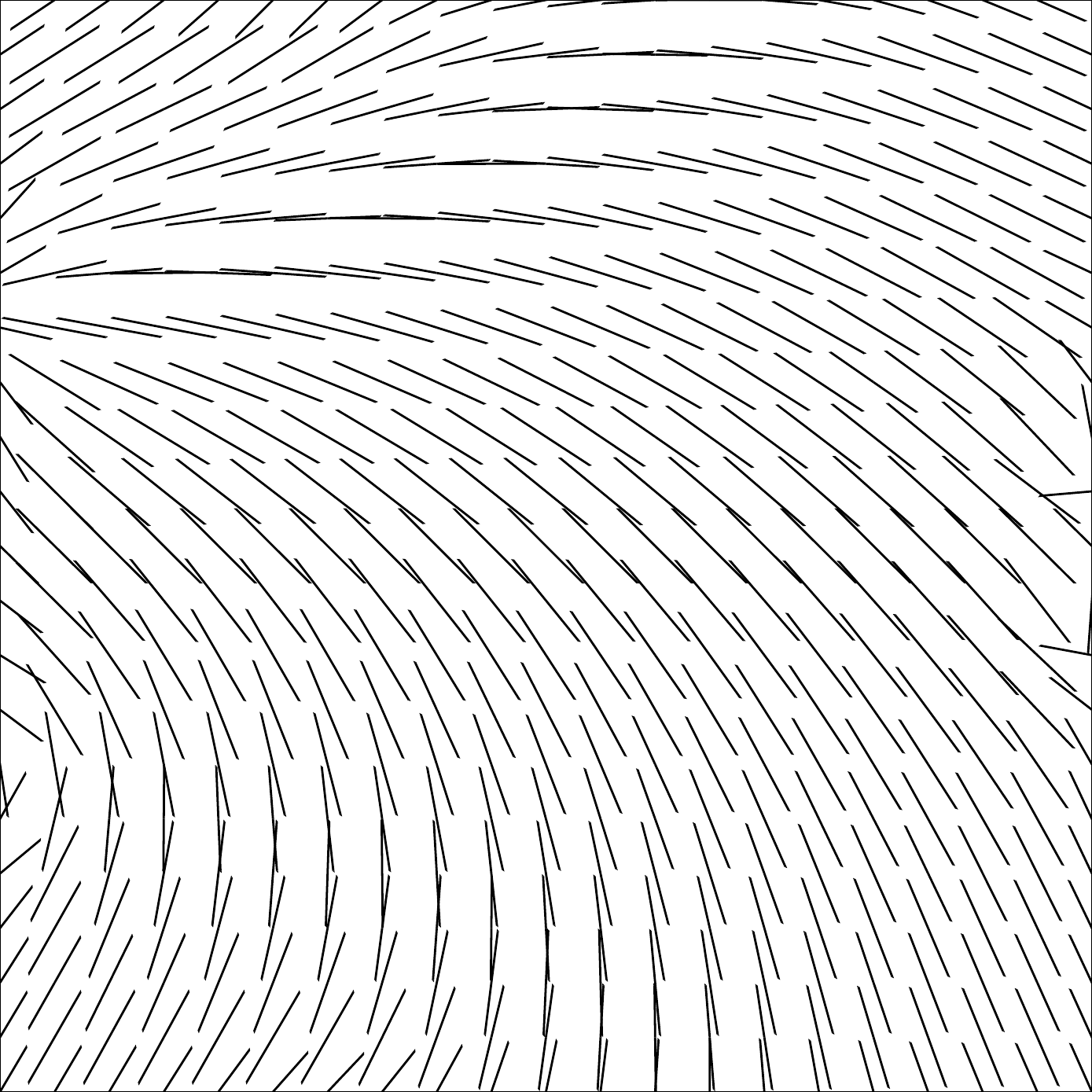}
  \caption{Regul. orientation}
\end{subfigure}\hfil 
\begin{subfigure}{0.25\textwidth}
  \includegraphics[width=\linewidth]{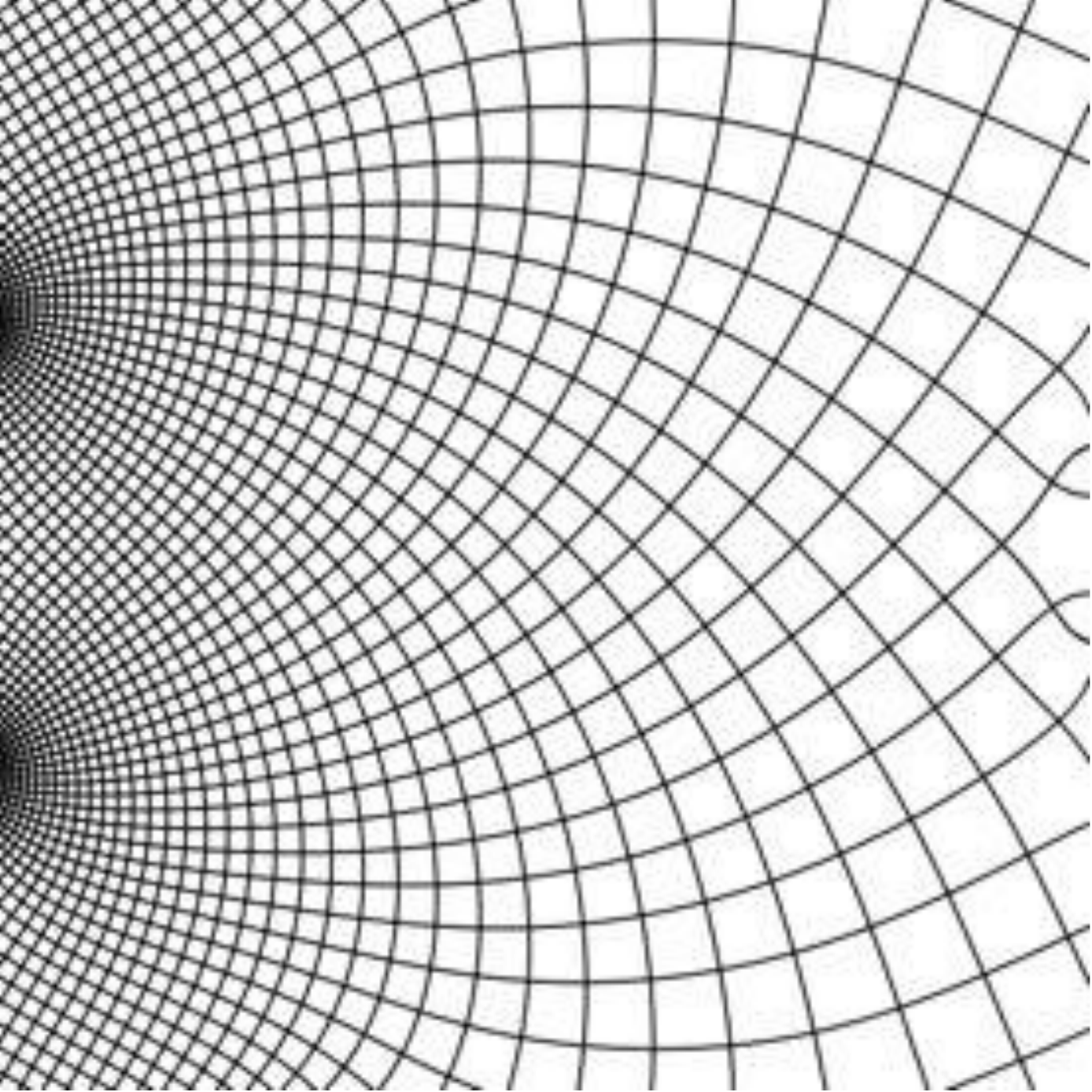}
  \caption{Projection by $\varphi$}
\end{subfigure}

\medskip
\begin{subfigure}{0.25\textwidth}
  \includegraphics[width=\linewidth]{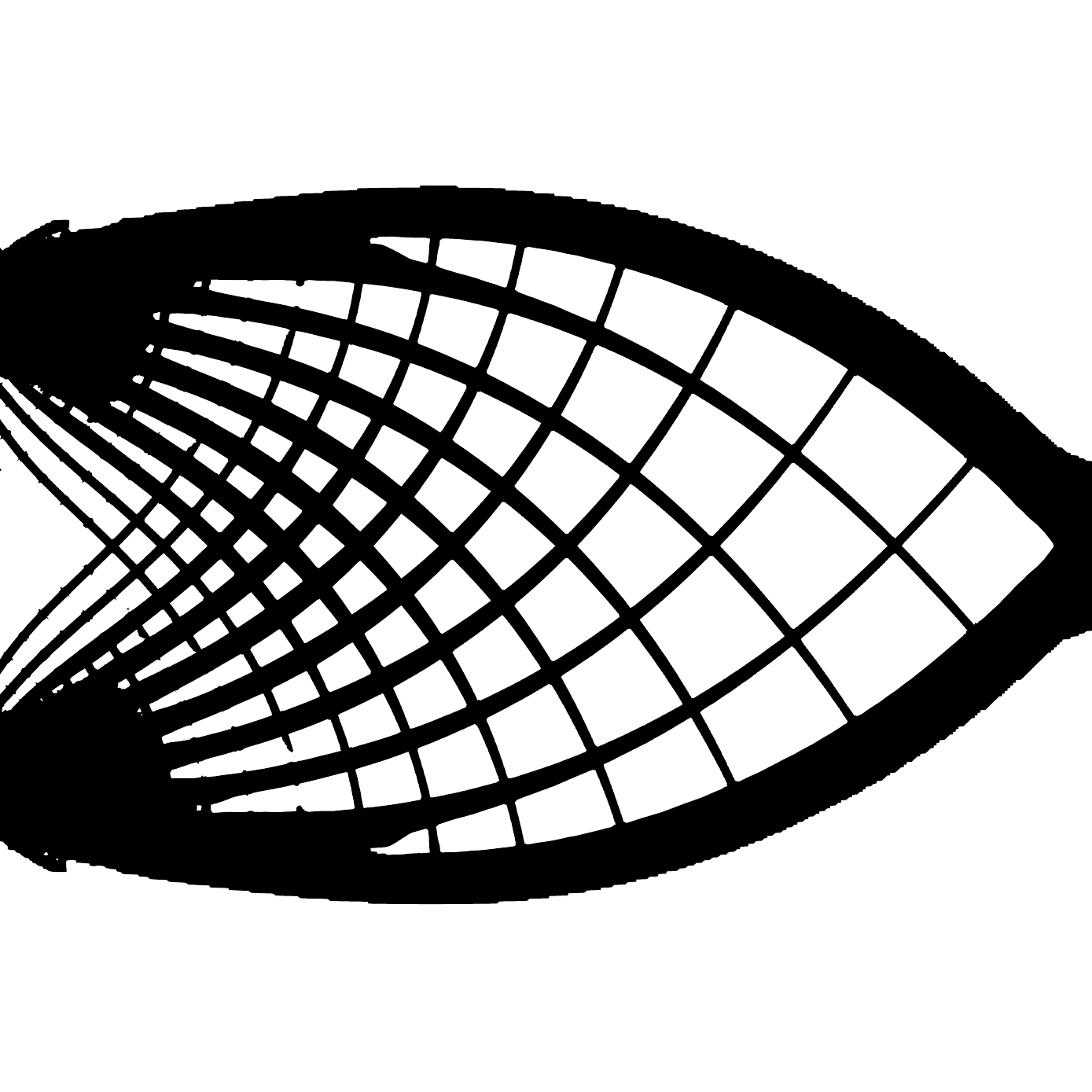}
  \caption{$\bar\Omega_\e (\varphi,\bar m)$, $\e=0.2$}
\end{subfigure}\hfil 
\begin{subfigure}{0.25\textwidth}
  \includegraphics[width=\linewidth]{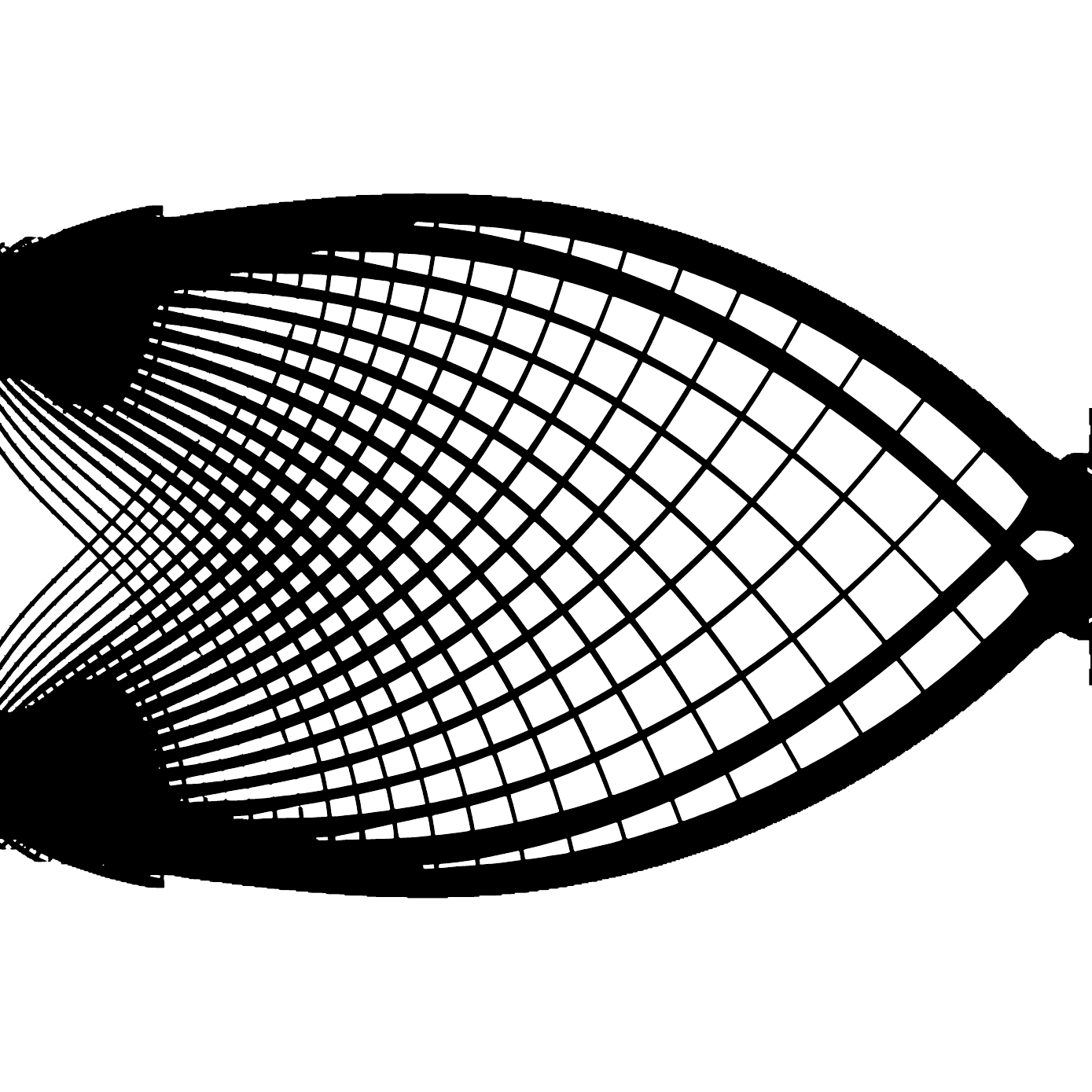}
  \caption{$\bar\Omega_\e (\varphi,\bar m)$, $\e=0.1$}
\end{subfigure}\hfil 
\begin{subfigure}{0.25\textwidth}
  \includegraphics[width=\linewidth]{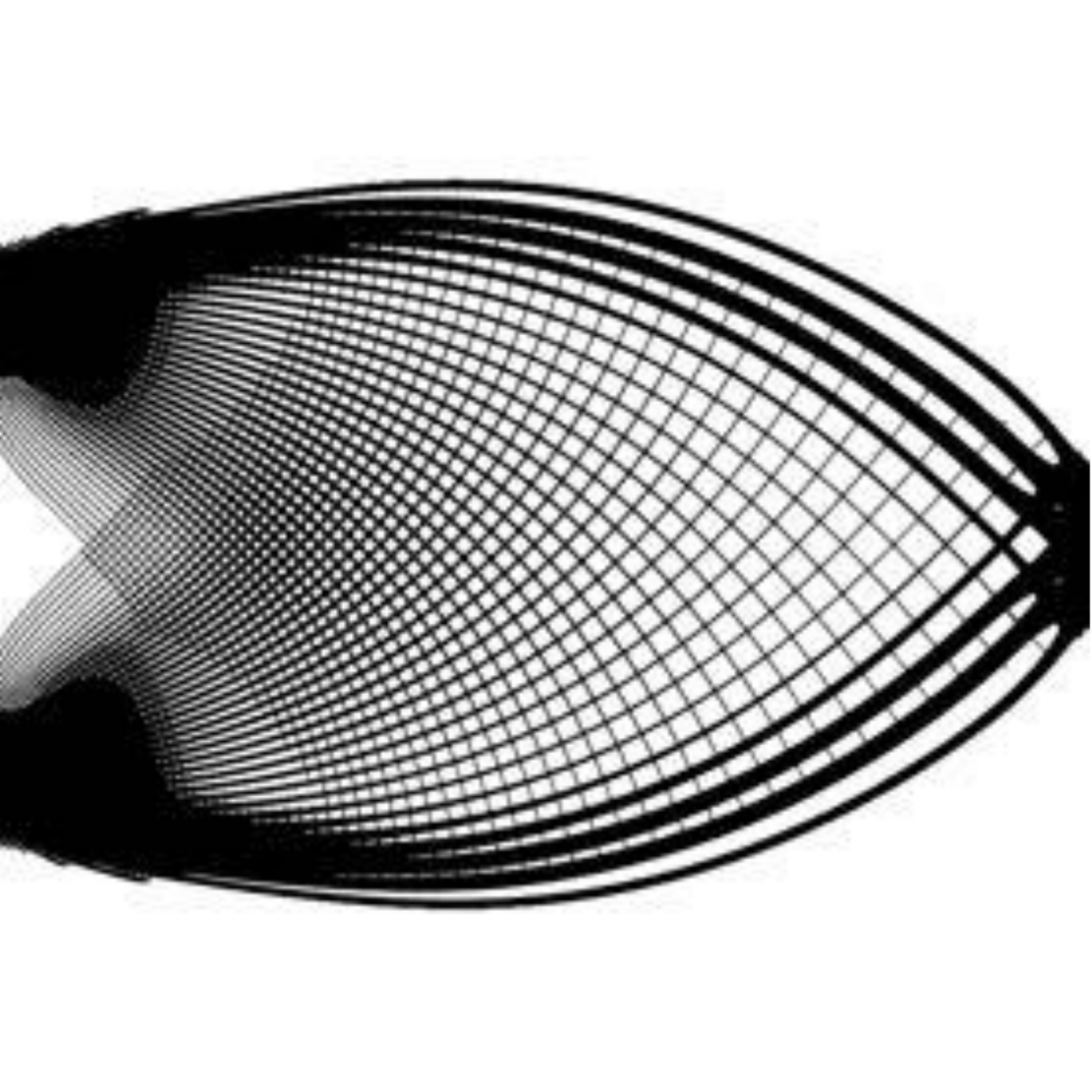}
  \caption{$\bar\Omega_\e (\varphi,\bar m)$, $\e=0.05$}
\end{subfigure}
\caption{Cantilever case.}
\label{fig:cantilever case}
\end{figure}

\begin{figure}[htb]
    \centering 
\begin{subfigure}{0.25\textwidth}
  \includegraphics[width=\linewidth]{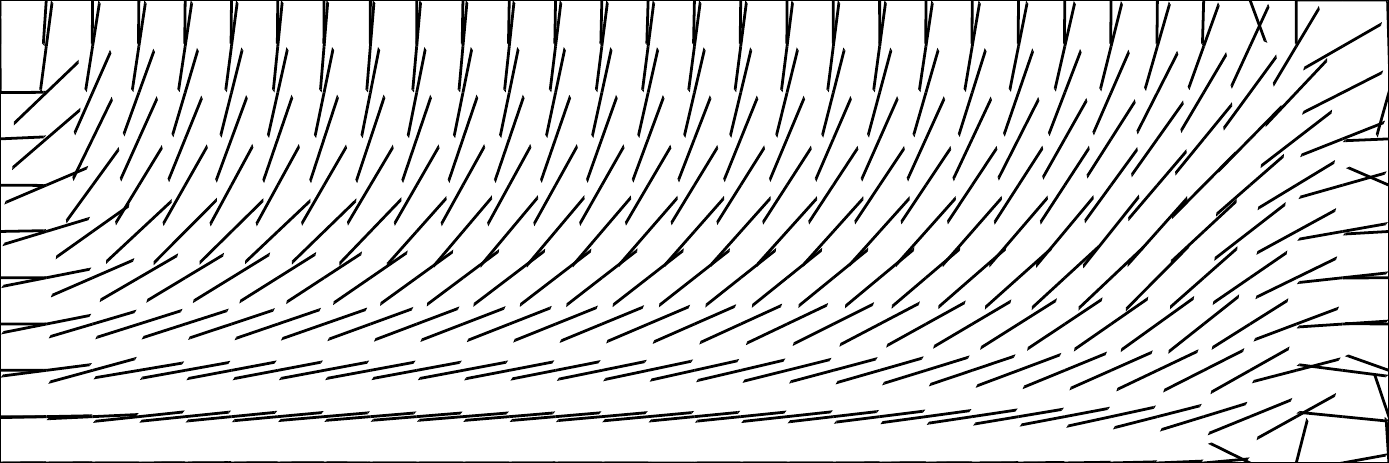}
  \caption{Opt. orientation}
\end{subfigure}\hfil 
\begin{subfigure}{0.25\textwidth}
  \includegraphics[width=\linewidth]{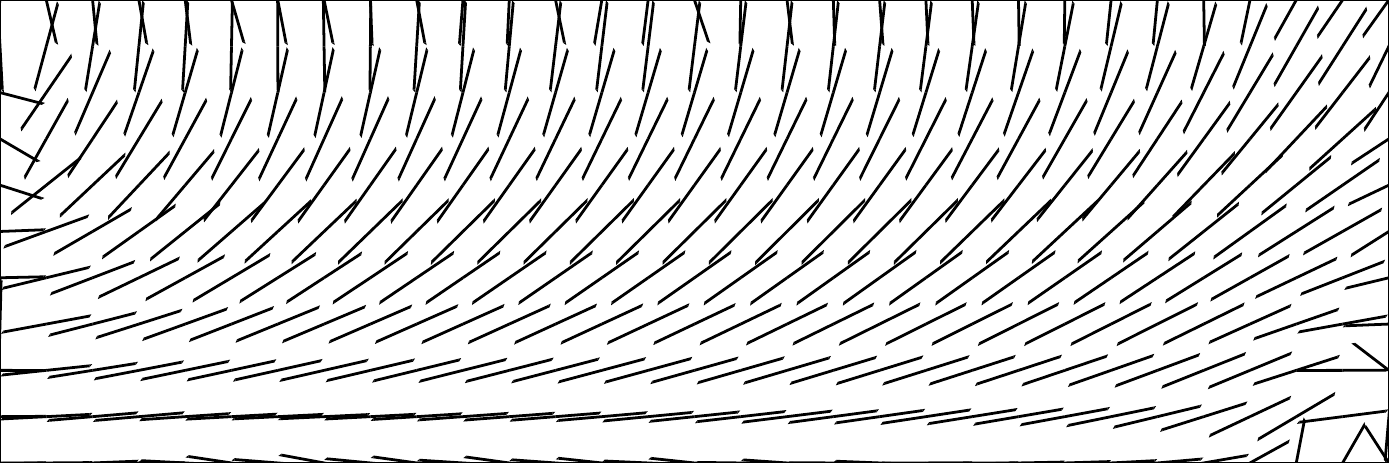}
  \caption{Regul. orientation}
\end{subfigure}\hfil 
\begin{subfigure}{0.25\textwidth}
  \includegraphics[width=\linewidth]{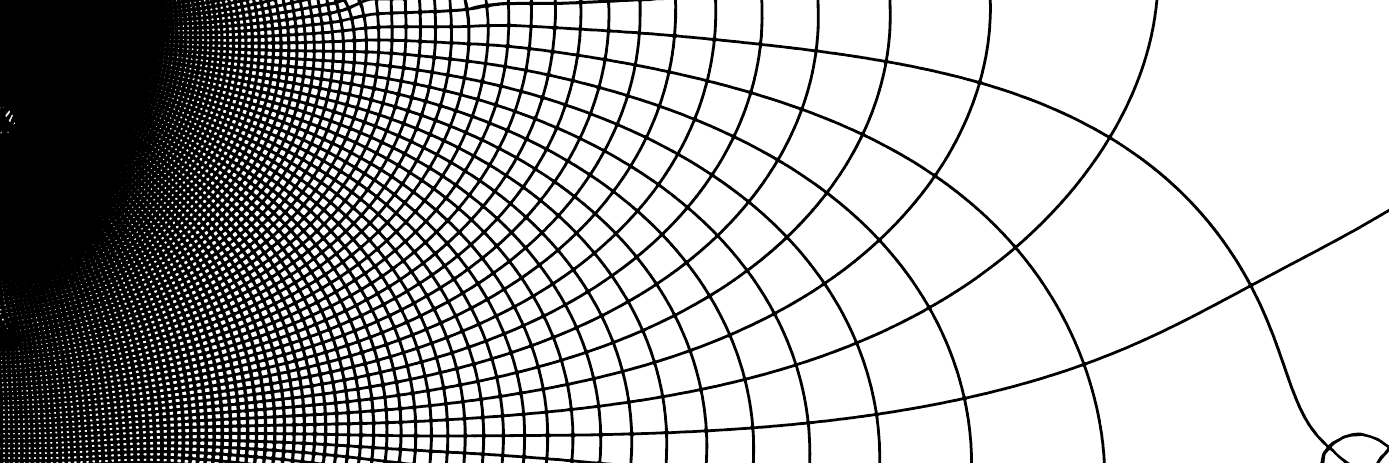}
  \caption{Projection by $\varphi$}
\end{subfigure}

\medskip
\begin{subfigure}{0.25\textwidth}
  \includegraphics[width=\linewidth]{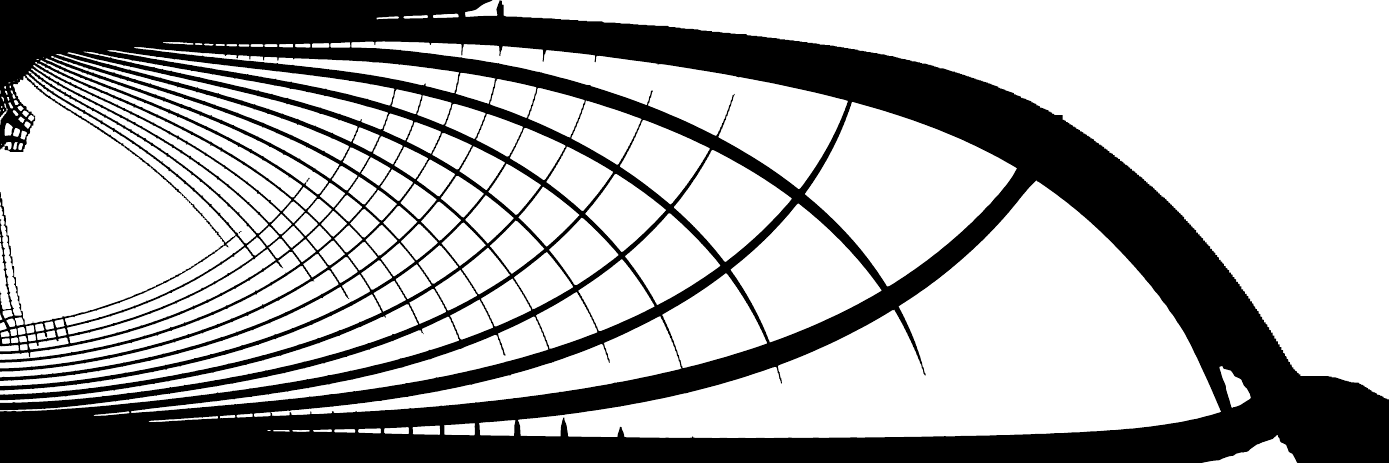}
  \caption{$\bar\Omega_\e (\varphi,\bar m)$, $\e=0.2$}
\end{subfigure}\hfil 
\begin{subfigure}{0.25\textwidth}
  \includegraphics[width=\linewidth]{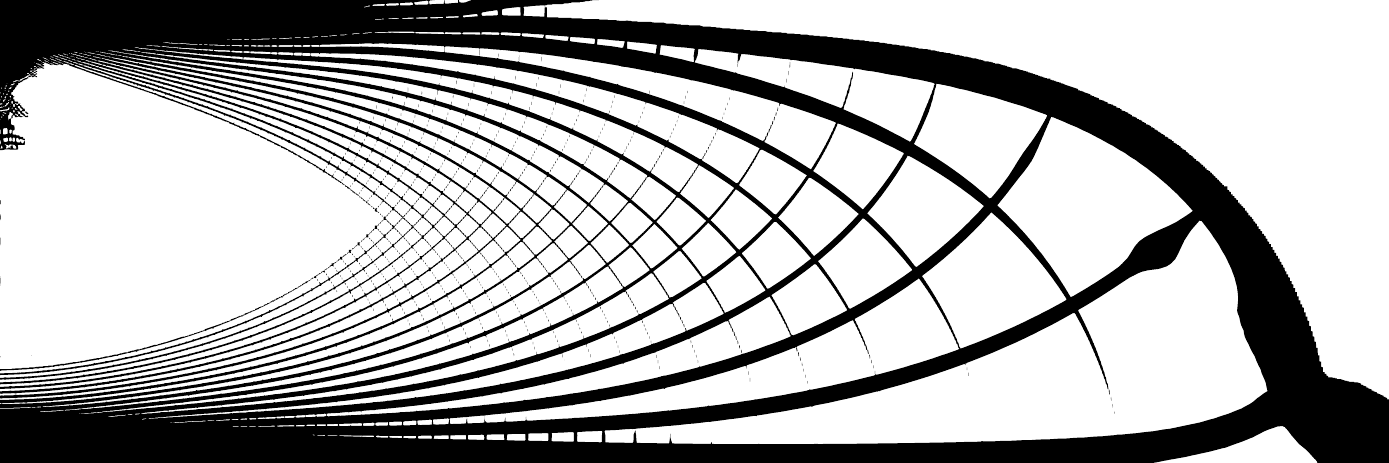}
  \caption{$\bar\Omega_\e (\varphi,\bar m)$, $\e=0.1$}
\end{subfigure}\hfil 
\begin{subfigure}{0.25\textwidth}
  \includegraphics[width=\linewidth]{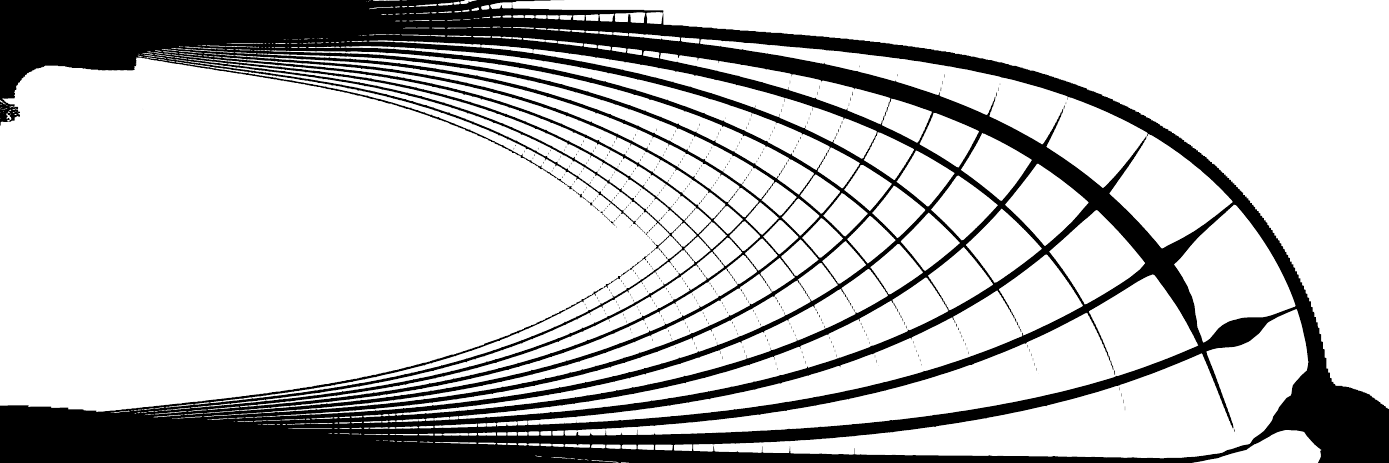}
  \caption{$\bar\Omega_\e (\varphi,\bar m)$, $\e=0.05$}
\end{subfigure}
\caption{MBB beam case.}
\label{fig: MBB beam case}
\end{figure}

\begin{figure}[htb]
    \centering 
\begin{subfigure}{0.25\textwidth}
  \includegraphics[width=\linewidth]{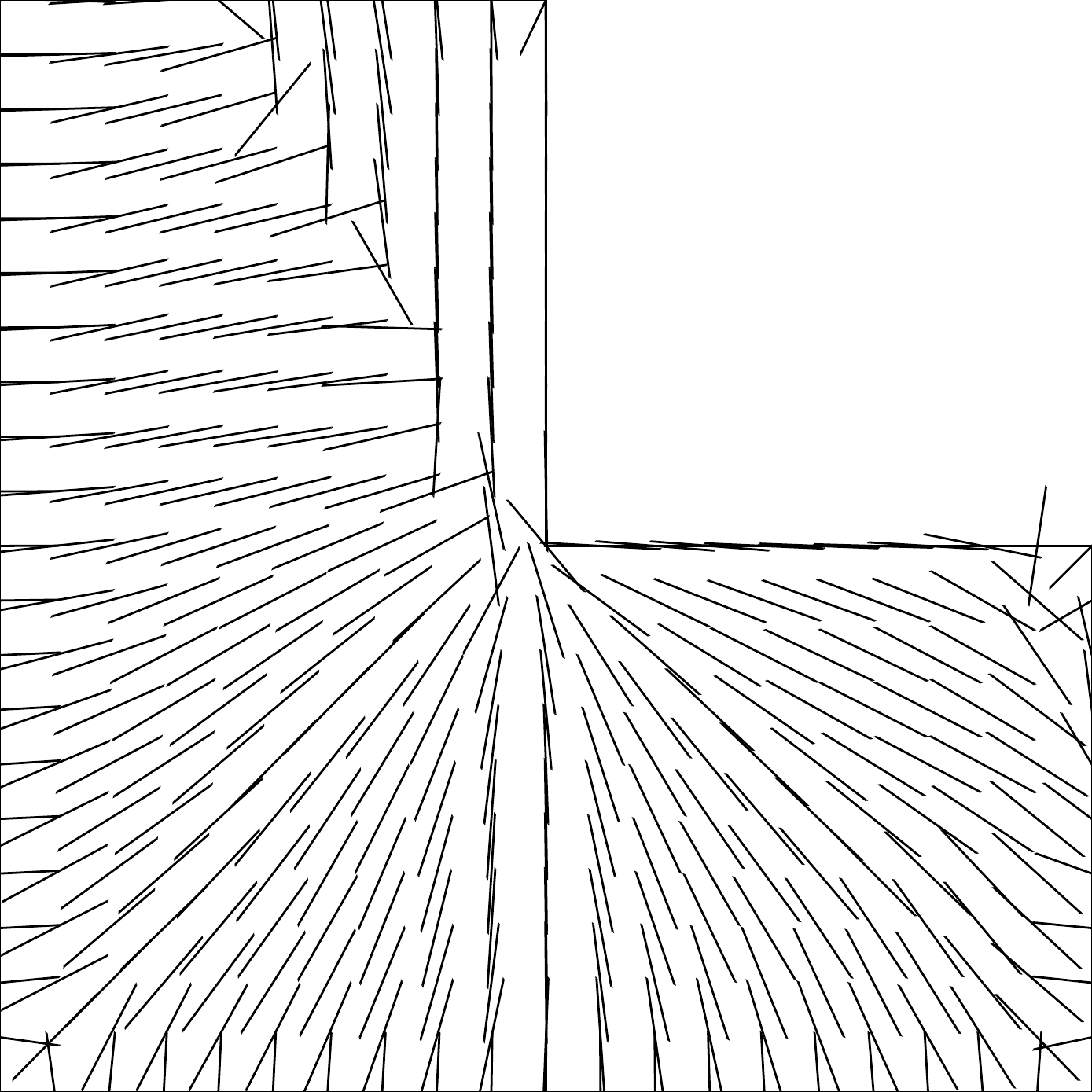}
  \caption{Opt. orientation}
\end{subfigure}\hfil 
\begin{subfigure}{0.25\textwidth}
  \includegraphics[width=\linewidth]{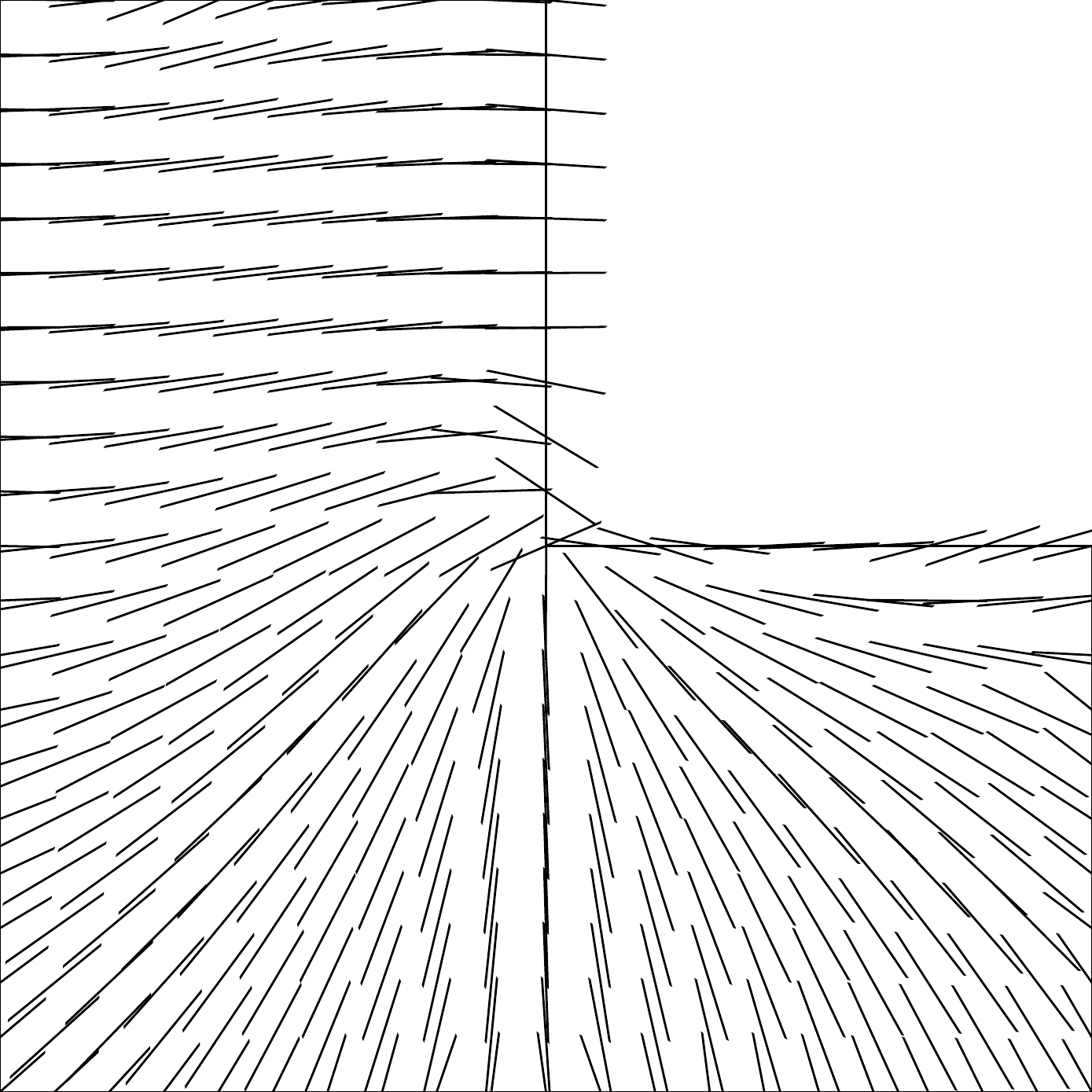}
  \caption{Regul. orientation}
\end{subfigure}\hfil 
\begin{subfigure}{0.25\textwidth}
  \includegraphics[width=\linewidth]{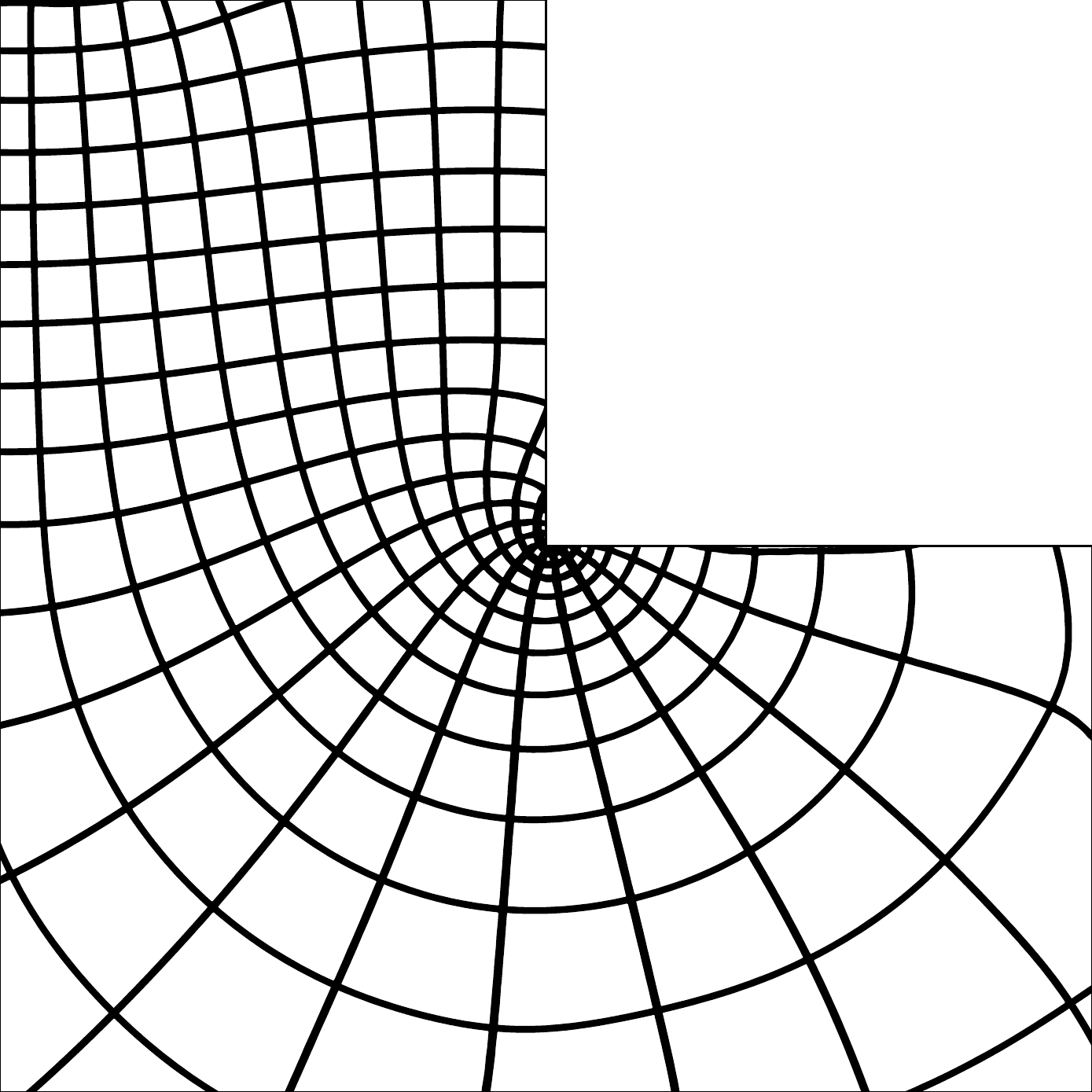}
  \caption{Projection by $\varphi$}
\end{subfigure}

\medskip
\begin{subfigure}{0.25\textwidth}
  \includegraphics[width=\linewidth]{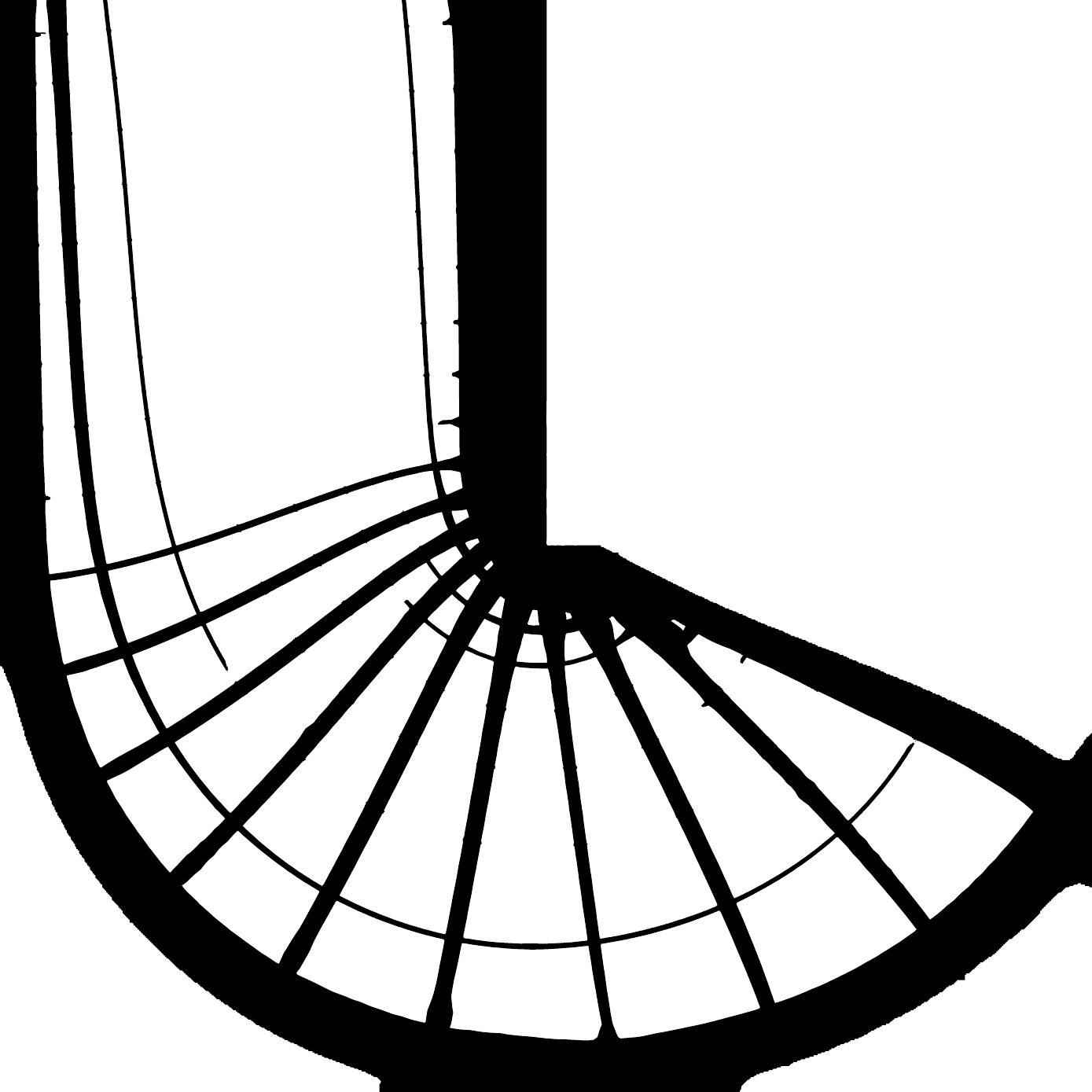}
  \caption{$\bar\Omega_\e (\varphi,\bar m)$, $\e=0.2$}
\end{subfigure}\hfil 
\begin{subfigure}{0.25\textwidth}
  \includegraphics[width=\linewidth]{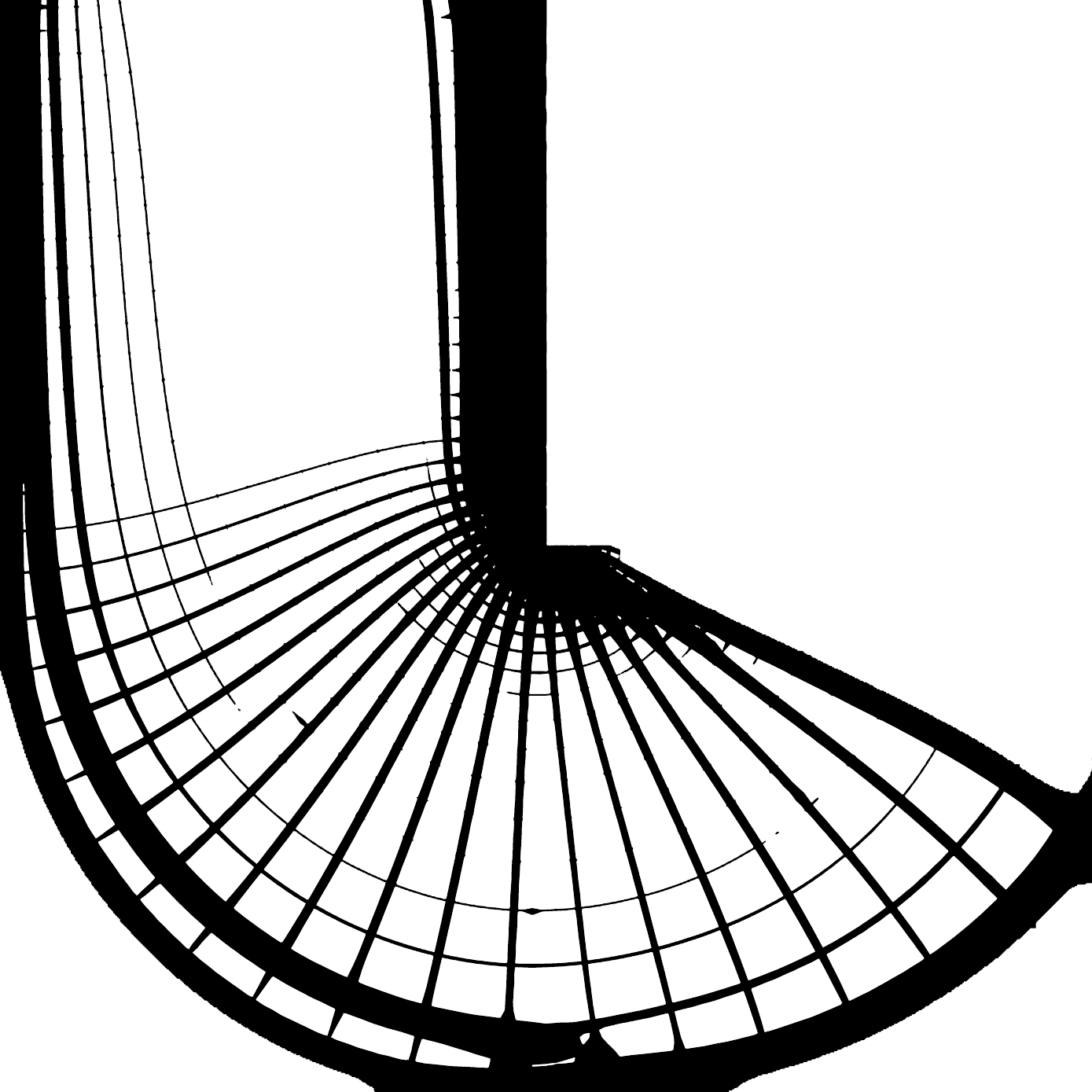}
  \caption{$\bar\Omega_\e (\varphi,\bar m)$, $\e=0.1$}
\end{subfigure}\hfil 
\begin{subfigure}{0.25\textwidth}
  \includegraphics[width=\linewidth]{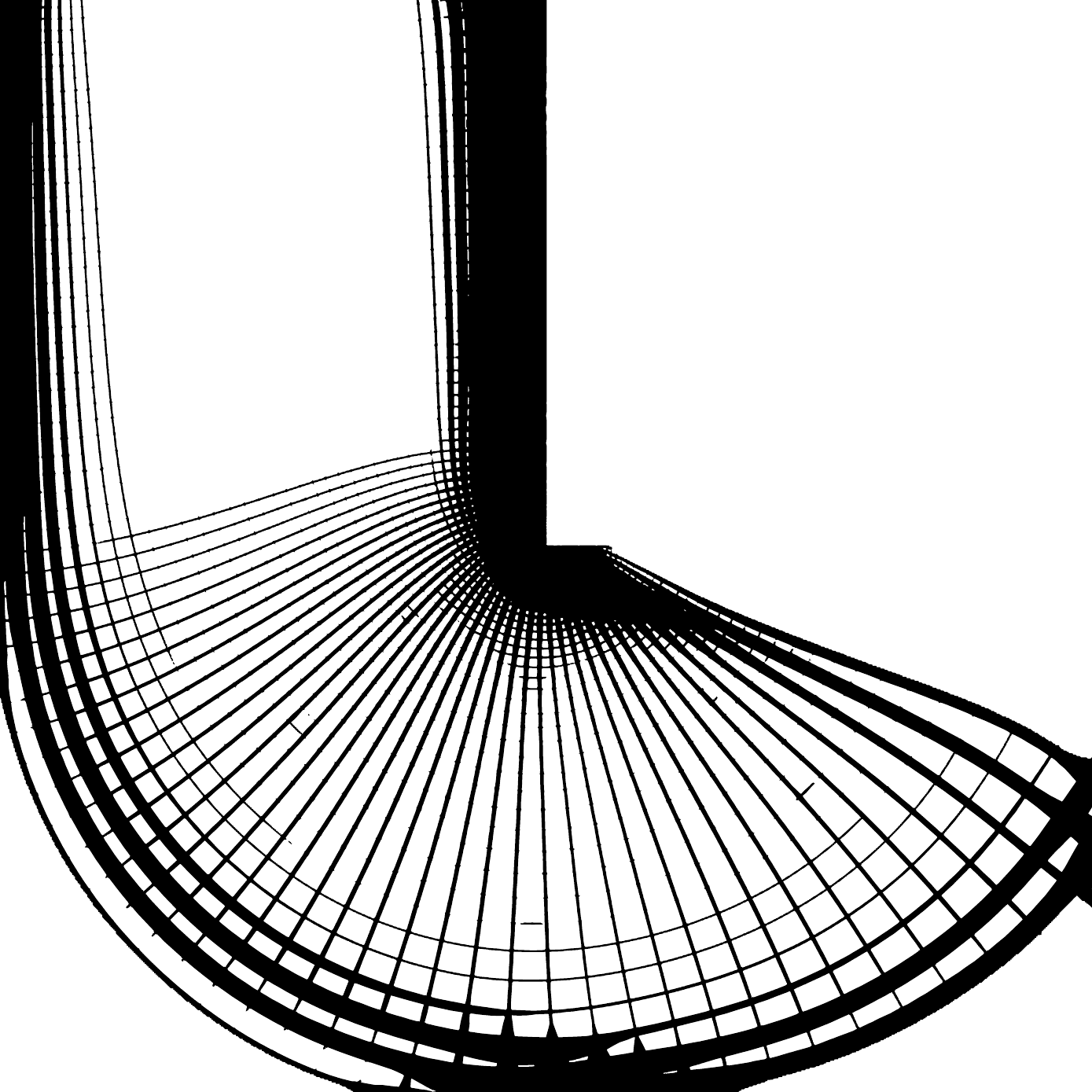}
  \caption{$\bar\Omega_\e (\varphi,\bar m)$, $\e=0.05$}
\end{subfigure}
\caption{L-beam case.}
\label{fig: L-beam case}
\end{figure}

As shown in the case of L-beam (Figure~\ref{fig: L-beam case}), the singularities which appear in Figure~\ref{fig: L-beam case}-(a) are removed during the regularization step (Figure~\ref{fig: L-beam case}-(b)).
This is a necessary condition in order to apply our method.
Indeed, as we mentioned, there is a case where the singularity cannot be removed, the so-called electrical mast (see Figure~\ref{fig:elec mast}).
Figure \ref{fig:elec mast}-(a) shows that two negative singularities, located inside the domain, cannot neither be removed nor pushed toward the boundary during the regularization step.

\begin{figure}[tbp]
    \centering 
\begin{subfigure}{0.3\textwidth}
  \includegraphics[width=\linewidth]{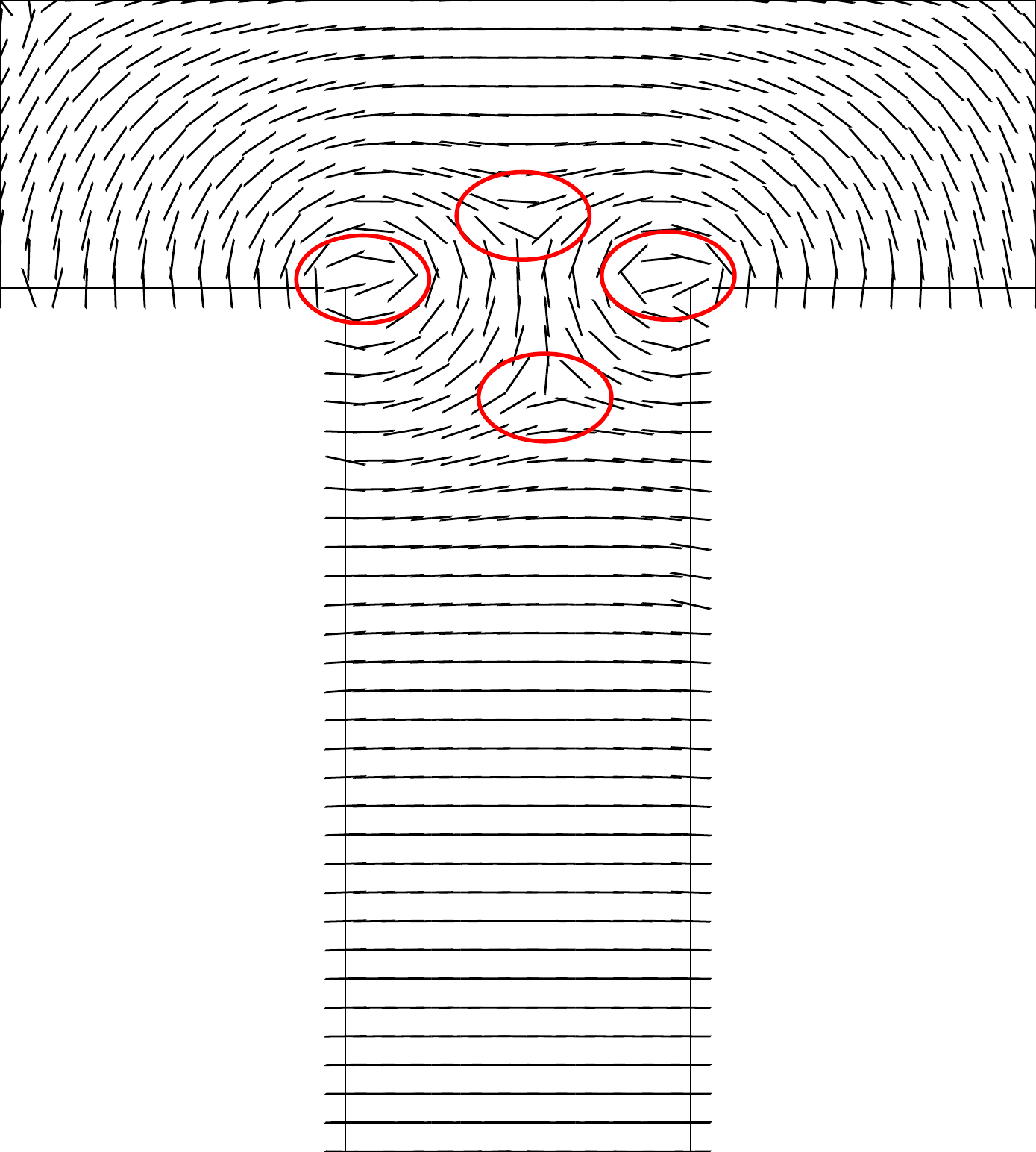}
  \caption{Singularities appear}
\end{subfigure}\hfil 
\begin{subfigure}{0.3\textwidth}
  \includegraphics[width=\linewidth]{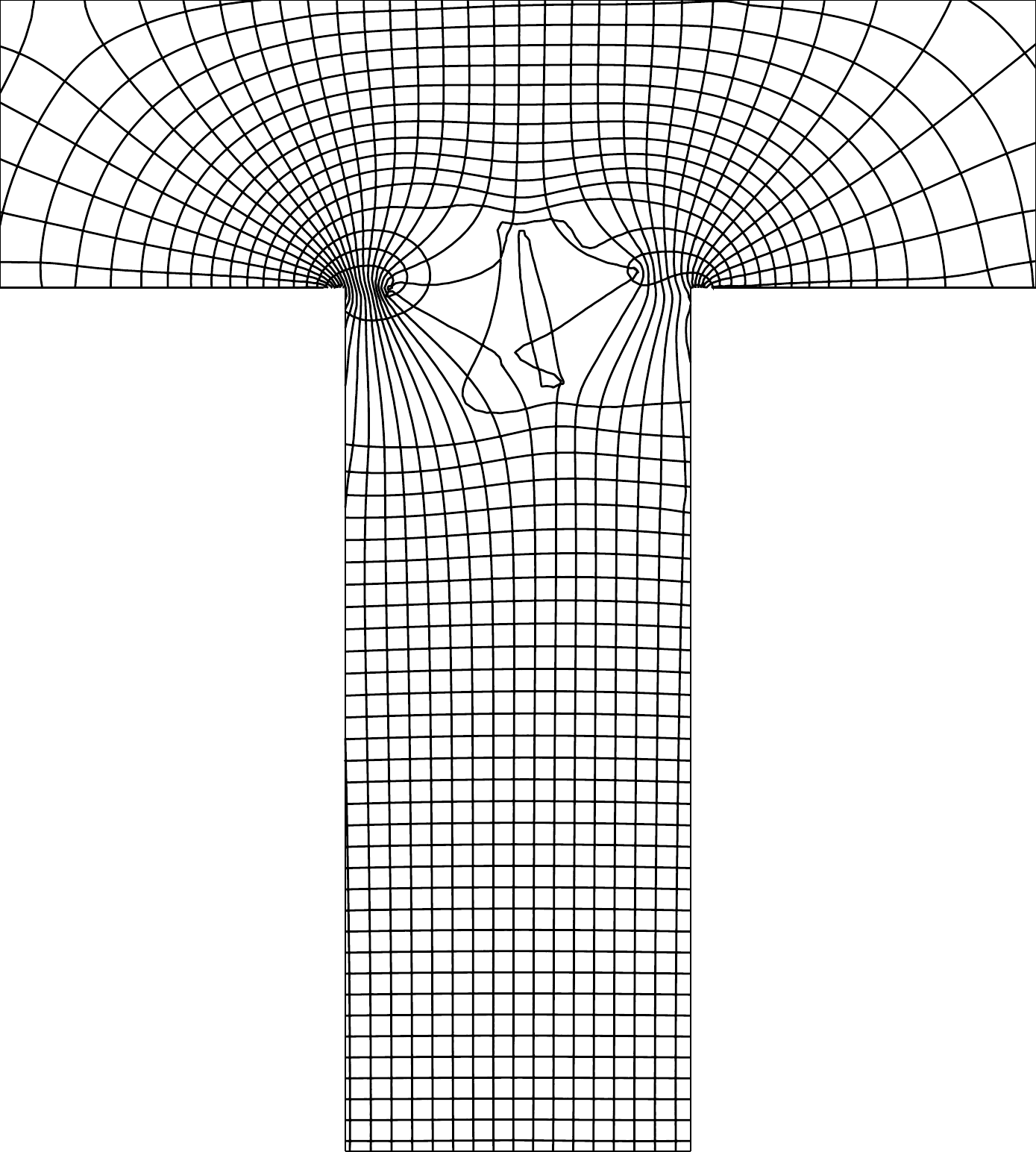}
  \caption{Projection by $\varphi$}
\end{subfigure}\hfill

\medskip

\begin{subfigure}{0.3\textwidth}
  \includegraphics[width=\linewidth]{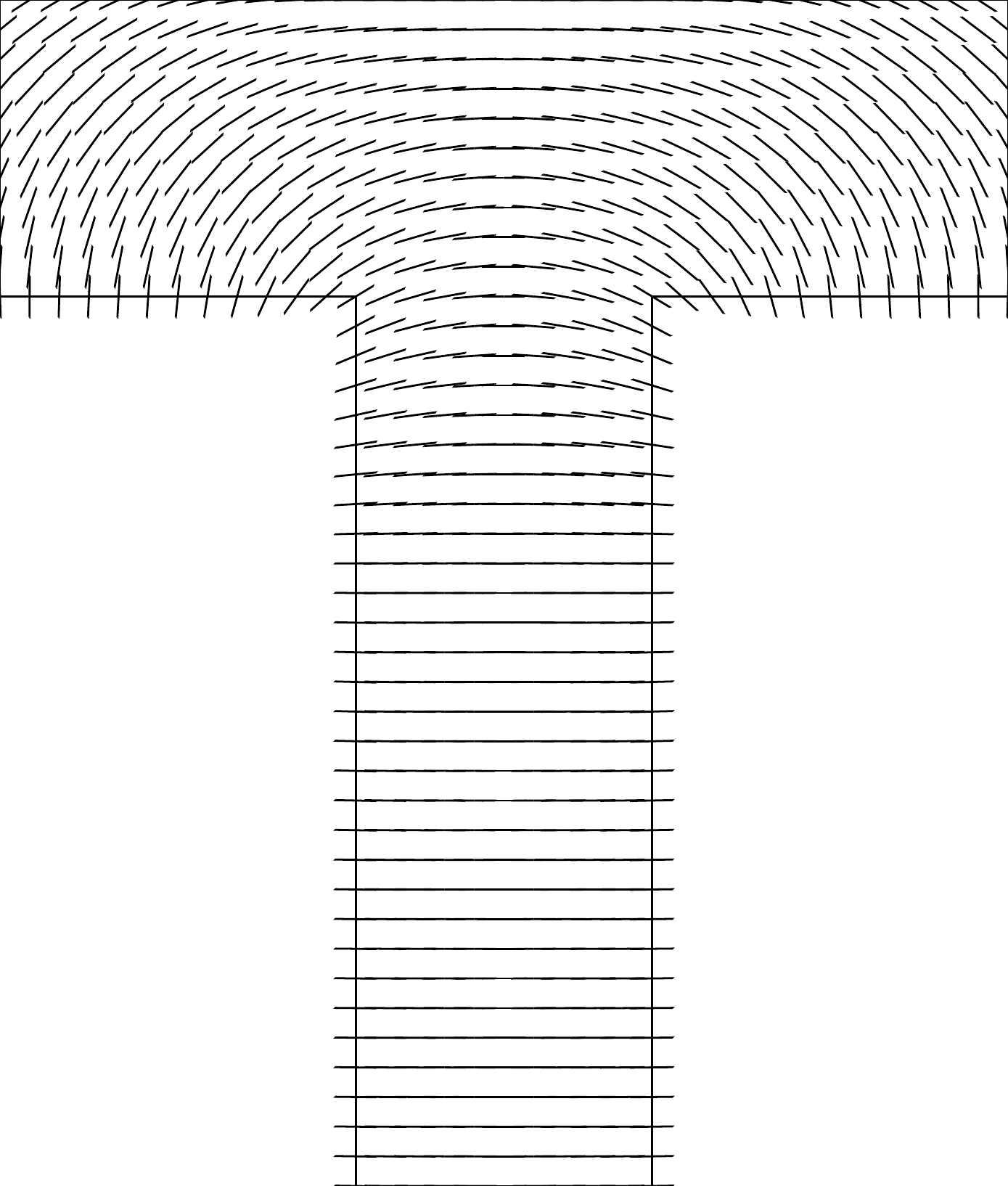}
  \caption{Improved reg. orientation}
  \label{Improved reg. orientation without singularities}
\end{subfigure}\hfil 
\begin{subfigure}{0.3\textwidth}
  \includegraphics[width=\linewidth]{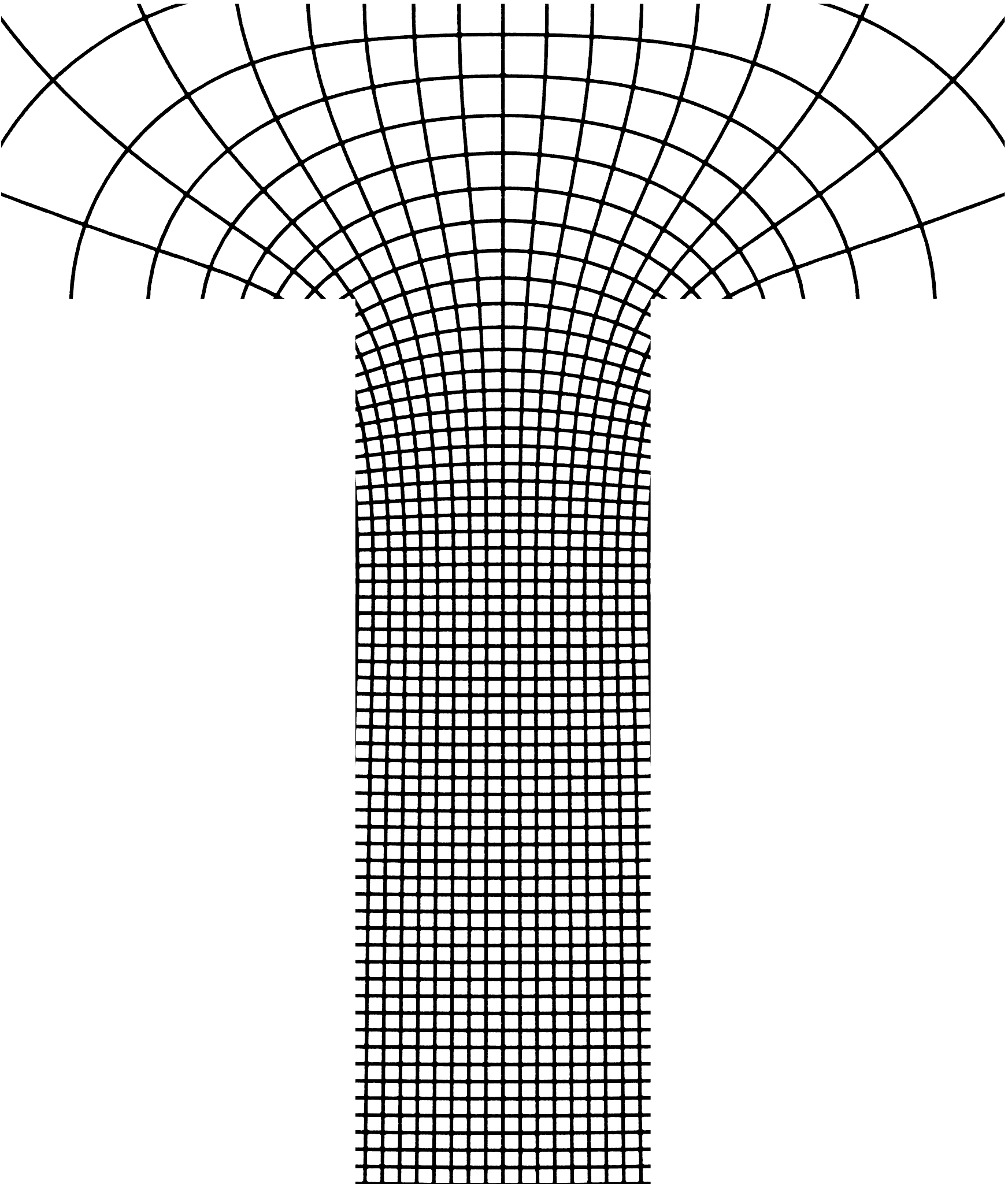}
  \caption{Reconstructed grid}
\end{subfigure}\hfill

\caption{Improved regularized orientation in the case of the electrical mast.}
\label{fig:elec mast}
\end{figure}

Hence, the computed grid shown in Figure~\ref{fig:elec mast}-(b) is clearly not correctly aligned with the optimal orientation of the cells in the vicinity of the singularities.
To overcome this problem, at least two different strategies can be considered.
One consists in modifying the regularization process in a way that forces more effectively the singularities to be eliminated (see Figure~\ref{fig:elec mast}-(c) and (d)).
Another approach is to adapt the projection step so that it is able to take singularities into account, and will be discussed in the next section.
\section{Further issues}


As we mentioned in Section~\ref{sec:3rd step}, the problem of removing the singularity might persist. 
To overcome this problem, some preliminary remedies are proposed in the PhD thesis of Perle Geoffroy \cite{geoffroy} by either trying to eliminate them by a Ginzburg-Landau approach or compute a map $\varphi$ with the previous approach and an enriched discontinuous finite element space.
In \cite{geoffroy}, one may find others extensions to cases, such as different objective functions, multiple loads and 3-d problems.





\section{Exercises}

\begin{problem}
Consider a compliance minimization problem (for one test case like cantilever, bridge, MBB beam, etc.) for the homogenized formulation. 
Choose an homogenized tensor corresponding to a non-isotropic microstructure (for example a square cell with a rectangular hole with fixed size). Then minimize the compliance with respect to the sole orientation of the microstructure (using Pedersen result).
\end{problem}

\begin{problem}
For the same compliance minimization problem as in the previous exercise, fix now the orientation and let the parameters of the microstructure (for example, the lengths $m_1$ and $m_2$ of the rectangular hole, see Figure \ref{rectangular hole}) become the optimization variables. Then minimize the compliance with respect to these parameters (with fixed orientation). Compare the optimal designs and the attained minimal compliances with the previous exercise. In particular, notice that the absence of orientation optimization yields a self-penalizing effect, namely the obtained 
designs feature almost no intermediate densities (somehow similarly to the SIMP method). 
\end{problem}

\begin{problem}
Combine the two previous optimization (with respect to the orientation and the size parameters) and minimize the compliance for the test case of the previous exercises.
\end{problem}

\end{document}